\documentclass[reqno,11pt]{amsart}
\usepackage{amsmath,amssymb,mathrsfs,amsthm,amsfonts,mathtools}
\usepackage[inline]{enumitem} 
\usepackage[usenames,dvipsnames]{xcolor}
\usepackage{comment}
\usepackage{caption}
\usepackage{array,longtable,booktabs,caption}
\DeclareMathAlphabet{\mathpzc}{OT1}{pzc}{m}{it}
\usepackage{stmaryrd}
\usepackage[paper=letterpaper,margin=1in, top=1in, bottom=1 in]{geometry}
\usepackage{graphicx}
\usepackage{tabularx}
\usepackage{mathrsfs}
\usepackage{listings}
\newcommand{\ind}{\mathbf{1}}
\usepackage{bbm}
\usepackage{tikz}
\usetikzlibrary{arrows}
\usepackage{hyperref} 

\hypersetup{%
	colorlinks=true, linkcolor=blue,
	citecolor=ForestGreen
}
\newtheorem{thm}{Theorem}[section]
\newtheorem{cor}[thm]{Corollary}
\newtheorem{prop}[thm]{Proposition}
\newtheorem{lem}[thm]{Lemma}
\newtheorem{lemma}[thm]{Lemma}
\newtheorem{conj}[thm]{Conjecture}
\theoremstyle{definition}
\newtheorem{defn}[thm]{Definition}

\newtheorem{remark}[thm]{Remark}
\newtheorem{ass}[thm]{Assumption}

\numberwithin{equation}{section}

\newcommand{\Z}{\mathbb{Z}}

\newcommand{\g}{\mathbf g}
\newcommand{\z}{\boldsymbol{{\boldsymbol{\zeta}}}}
\newcommand{\h}{\boldsymbol{{\boldsymbol{\eta}}}}

\newcommand{\vect}[1]{\mathbf{#1}}
\DeclareMathOperator{\gap}{gap}
\newcommand{\pN}[1]{\mathbf{p}_{N,#1}}
\newcommand{\PN}[1]{\mathbf{P}_{N,#1}}
\newcommand{\EN}[1]{\mathbf{E}_{N,#1}}

\renewcommand{\tilde}{\widetilde}

\DeclareMathOperator{\supp}{supp}

\DeclarePairedDelimiter{\norm}{\lVert}{\rVert}
\DeclarePairedDelimiter{\brac}{\langle}{\rangle}
\DeclarePairedDelimiter{\floor}{\lfloor}{\rfloor}

\newcommand{\seq}[1]{\mathbf{#1}}

\newcommand{\x}{\mathbf x}
\newcommand{\dr}{\mathrm d}
\newcommand{\y}{\mathbf y}

\newcommand{\bfa}{\mathbf a}

\newcommand{\R}{\mathbb{R}}
\newcommand{\N}{\mathbb{N}}

\renewcommand{\hat}{\widehat}

\renewcommand{\bar}{\overline}

\renewcommand{\P}{\mathbb{P}}	

\newcommand{\E}{\mathbb{E}}

\newcommand{\e}{\varepsilon}
\newcommand{\sdn}[1]{{\color{red}\small[#1]\color{black}}}

\title{A universality result for the critical 2d stochastic heat flow}
\author{Hindy Drillick}
\address{Courant Institute of Mathematical Sciences, New York University, New York, NY, USA}
\email{hindy.drillick@nyu.edu}
\author{Jonathan Hou}
\address{Courant Institute of Mathematical Sciences, New York University, New York, NY, USA}
\email{jh10648@nyu.edu}
\author{Shalin Parekh}
\address{Department of Mathematics, University of Maine, Orono, ME, USA}
\email{shalin.parekh@maine.edu}

\begin{document}

\maketitle

\begin{abstract}
 For a large class of discrete- and continuous-time models satisfying a linear flow property, we prove convergence to the critical $2d$ stochastic heat flow under diffusive scaling of space and time. Convergence is in the sense of finite-dimensional distributions in the space of measure-valued stochastic flows. The class of discrete models we consider includes random walks in space-time random environments with finite-range jumps and directed polymer models with finite-range spatial correlations. The class of continuum models we consider includes several distinct types of linear-multiplicative stochastic PDEs driven by Gaussian noise with smooth and compactly supported covariance kernel.
\end{abstract}


\tableofcontents

\section{Introduction and main results}

\subsection{Background and motivation}\label{sse:bg}
The \emph{critical $2d$ stochastic heat flow} (SHF) is a measure-valued random field that was recently constructed by \cite{CSZ_2d} and characterized axiomatically by \cite{Tsa24}. 
These works respectively established the SHF as a scaling limit of the directed polymer and the mollified stochastic heat equation.
Beyond these approximations, the SHF is thought to be a 
scaling limit of a number of other models, but little has been said of what these are, let alone what characterizes them.
We study a class of models satisfying six conditions that are sufficient criteria for convergence to the SHF. 
In doing so, we seek to clarify the near-critical behavior of many models with natural structural properties. We begin by discussing the notion of criticality.

A central goal of probability theory and statistical mechanics is to understand how disordered systems behave near their points of \emph{phase transition}. Loosely, these critical values witness stark changes to macroscopic observables as a parameter governing the model at the microscopic scale is varied continuously through them.
Concretely, we may consider the example of the nearest-neighbor directed polymer in a random environment, where the parameter is the inverse temperature $\beta \geq 0$ and the observable of interest is the point-to-plane partition function
\begin{equation*}
    Z_N^\beta := \mathbb E \Big[e^{\beta\sum_{r = 0}^{N-1} \omega_{r,R_r} - N\lambda(\beta)} \Big| \omega \Big].
\end{equation*}
Here $\mathbb E$ is with respect to the law of the simple random walk $(R_n)_{n \geq 0}$ on $\Z^d$ and $\omega=\{\omega_{n,x}\}_{(n,x) \in \Z \times \Z^d}$ is an i.i.d.\ random environment with cumulant generating function $\lambda(\beta) := \log\E\big[e^{\beta\omega}\big] < \infty$ for $\beta \geq 0$.
This model and its critical point were first studied in the physics literature \cite{HH85}.
The precise formulation of this transition is with respect to the mean-one martingale $Z_N^\beta$: one says that $\beta$ falls under the \emph{strong disorder} regime if $\lim_{N \to \infty} Z_N^\beta = 0$ and under the \emph{weak disorder} regime otherwise.
It was shown in  \cite{MR968950,MR1006293} that weak disorder holds in $d \geq 3$ for small positive $\beta$. Later on, it was shown that for all $d \geq 1$ there exists a critical $\beta_c \in [0,\infty)$ such that all $\beta < \beta_c$ witness weak disorder and $\beta > \beta_c$ strong disorder \cite{MR2271480}. The behavior of the partition function and the associated path measure in the strong disorder regime is itself a subject of great interest \cite{MR2249669,MR2917766,MR3785400,MR4089496,MR5110428}. See \cite{MR3444835} for an exposition of these disorder regimes.

It is also known that $\beta_c = 0$ when $d \in \{1,2\}$, meaning that for every fixed positive $\beta$ there is strong disorder \cite{MR2271480}.
Thus, the phase transition is trivial, however, behavior about the critical value is no less rich. The work of 
\cite{akq} showed that at $d = 1$ one recovers an \emph{intermediate disorder} regime by the scaling $\beta_N = \vartheta N^{-1/4}$, where $\vartheta>0$ is arbitrary and $N$ is the path length as above,
at which one observes features of both weak and strong disorder.

In this work, we will be interested in the case $d=2$. One naturally asks whether this intermediate disorder behavior extends to the two-dimensional directed polymer. After a series of works that studied this question at the level of moments \cite{MR1629198,
MR3719953, MR4017119, MR4032870, GQT}, the paper \cite{CSZ_2d} constructed the SHF as a suitably scaled limit of the polymer partition function,
provided $\beta=\beta_N$ is sent to zero at a precise logarithmically attenuated rate (see \eqref{crit1} below).

A significant feature of the SHF is that it is conjectured to describe the universal limiting behavior of a large class of models. Another model that is known to converge to the SHF is the mollified multiplicative-noise stochastic heat equation \cite{Tsa24}. We now discuss this model in more detail. 
Consider the stochastic PDE 
\begin{equation*}\label{eq:SHE}\tag{SHE}
    \partial_t u (t,x)= \tfrac{1}{2}\Delta u(t,x) + \beta \cdot u(t,x)\xi^\epsilon (t,x),\qquad t\ge 0 ,\; x\in \mathbb R^d,
\end{equation*}
where $\beta > 0$, and $\xi^\epsilon$ is a Gaussian noise on $[0,\infty) \times \R^d$ with covariance $\mathbb E[ \xi^\epsilon(t,x) \xi^\epsilon(s,y)] = \epsilon^{-d} J(\epsilon^{-1} (x-y)) \delta(t-s). $ Here $J\in C_c^\infty(\mathbb R^d)$ and $\int J=1.$
In statistical mechanical terms, this model can be understood as the continuum analogue of the partition function of the directed polymer model from above by the Feynman-Kac formula \cite{BC95}. Its logarithm, or \emph{free energy}, solves the Kardar-Parisi-Zhang (KPZ) equation of random growth \cite{KPZ}. The development of solution theory for the SHE in one dimension drove progress in understanding the models that characterize the $(1 + 1)$-dimensional KPZ university class \cite{FS10,Qua11,Cor12,Quastel_2015,CW17,CS20}. 

In addition to this statistical mechanics interpretation, we can also study the SHE from the point of view of singular stochastic PDEs, a field that has seen dramatic progress over the last 15 years \cite{Hai13, Hai14, GJ14,GIP15, GP17,duch}. From this perspective, the interesting question is whether there is a sensible limiting object for the SHE as $\epsilon\downarrow 0$, in other words as the spatially smooth noise $\xi^\epsilon$ approaches a space-time white noise $\xi$. Note that in every dimension, the SHE driven by space-time white noise is classically ill-posed due to
the lack of regularity of $\xi$, which is at best a tempered distribution. Thus, the existence of the $\epsilon \downarrow 0$ limit must use deeper probabilistic structure of $\xi$ and the PDE.

The answer to this question is strongly dimension dependent. See that replacing $(t,x) \mapsto (\epsilon^2 t,\epsilon x)$ in \eqref{eq:SHE} effectively changes $\beta \mapsto \beta \epsilon^{1-\frac{d}{2}}$. In the language of regularity structures, the SHE is therefore \textit{scaling-subcritical} in $d=1$,  \textit{scaling-critical} in $d=2$, and \textit{scaling-supercritical }in $d\ge 3$. Subcriticality in the context of stochastic PDEs is unrelated to phase transitions, but instead refers to the fact that the small-scale structure of the object is well-approximated by a stochastic PDE that is linear in the driving noise. 
This means that in $d=1$ there exist general theories such as regularity structures \cite{Hai14, MR3779690} or paracontrolled products \cite{GIP15, Rosati} that give meaning to equations such as the SHE as $\epsilon \downarrow 0$, while in $d\ge 3$ one does not expect any solution at all. However, for the critical dimension $d=2$ the answer is much more subtle.

For the SHE in its critical dimension $d=2$, \cite{MR1629198} first showed that the variance of the SHE converges to a nontrivial limit as $\epsilon \downarrow 0$, provided $\beta_\epsilon \to 0$ is logarithmically tuned in a precise fashion (see \eqref{crit2} below). Later, \cite{GQT} showed using resolvent techniques from the physics literature \cite{Rajeev, Dimock} that, under the same scaling, all moments converge to a non-trivial limit. Finally, \cite{Tsa24} showed that the moments can actually be used to characterize the limit points uniquely as those of the SHF, thus verifying independently of \cite{CSZ_2d} the existence of a nontrivial limit of \eqref{eq:SHE} as $\epsilon\downarrow 0$. 

This then leads to the question of whether all scaling-critical stochastic PDEs have nontrivial (meaning non-Gaussian and nonzero) scaling limits. The answer is no: \cite{MR5042137} showed that after tuning ferromagnetic Ising and $\Phi^4_d$ toward criticality at the marginal dimension $d = 4$, the limit can only be a Gaussian field.
Nonetheless, other examples exist where one does expect a nontrivial limit at the critical dimension. Perhaps the most prominent example, whose construction is still an open problem, is the Yang-Mills theory with non-abelian gauge groups in its critical dimension $d=4$ (see \cite{Thierry, Chandra, cao} for recent progress in the subcritical dimensions $d=2,3$). 
Thus, the SHF provides an exciting example of a scaling-critical theory with nontrivial correlations, for which one can actually say something meaningful and prove convergence to the continuum limit.

It is natural to ask if the SHF can arise from other types of prelimiting objects. The recent work \cite{sur26} proved that the continuous-time lattice polymer driven by spatially correlated Brownian motions also converges to the SHF. In this work, we enlarge the domain of attraction of the SHF to a broad class of models called \emph{linear stochastic flows}. 
This class of flows includes spatially correlated polymer models, random walks in random environments, stochastic PDEs with derivative noise, and diffusions in random media.

\begin{figure}[h]
    \centering
    \vspace{-3 em}
    \includegraphics[scale=0.4]{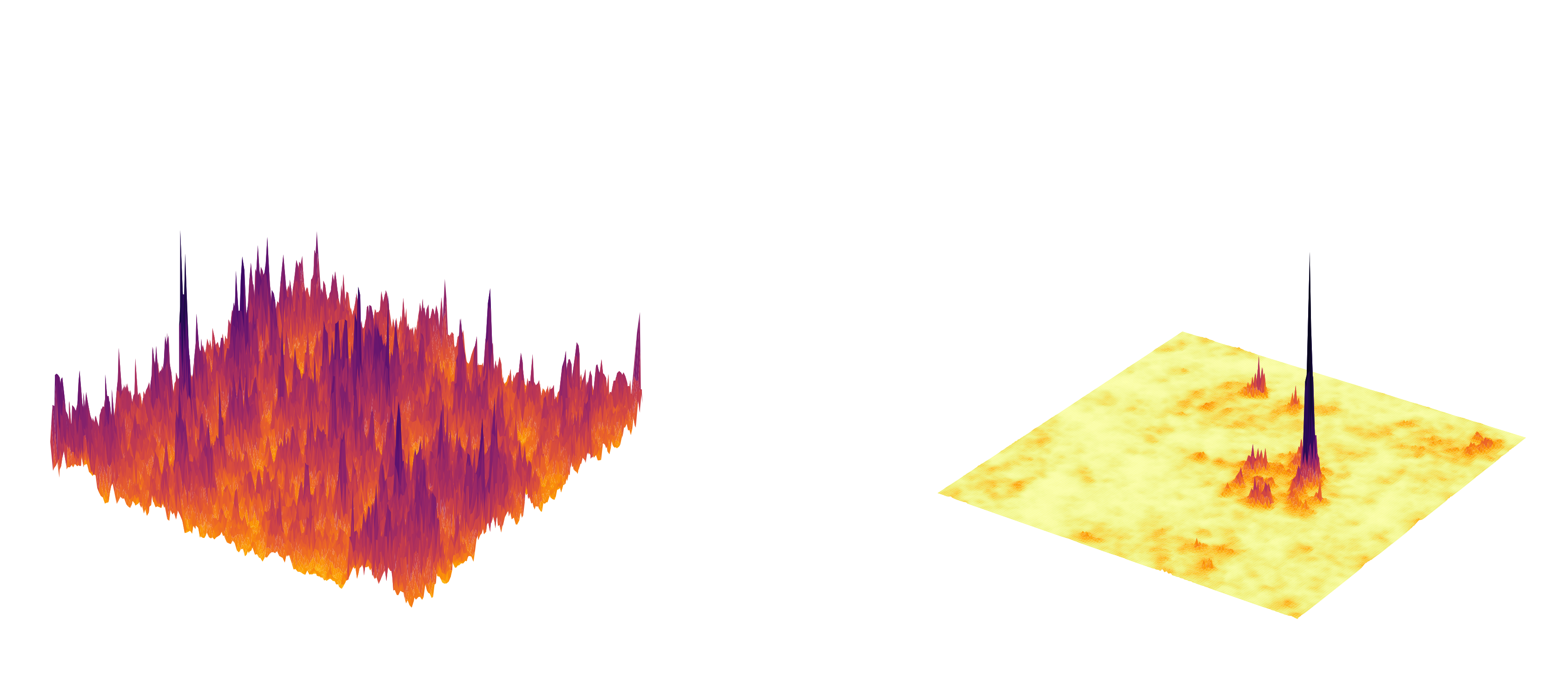}
    \caption{
        Simulations of the plane-to-point field \(H_N(x) = \sum_y \psi(y) H^N_{0,N}(y,x)\) defined from the tilted transition kernel for a random walk in a random environment with flat initial condition \(\psi \equiv 1\) (see Section \ref{5.4}).
        Here $N = 1024$, the viewport is a $512 \times 512$ box, and the noise is spatially averaged with correlations of range $4$.
        The left and right panels respectively display fine tuning as in \eqref{crit_tuning} with $\vartheta = -4$ and $\vartheta=4$.
    }
    \label{fig:rwre}
\end{figure}

\subsection{The critical $2d$ stochastic heat flow} In this section we give a self-contained definition of the SHF, which is a random field $Z^\vartheta = (Z^\vartheta_{s,t}(dx,dy))_{0 \leq s \leq t \le 1}$ with each $Z^\vartheta_{s,t}$ taking values in the space $\mathcal M_{\mathrm{loc}}(\R^2 \times \R^2)$ of positive locally finite Borel measures on $\R^2 \times \R^2$.
Here the parameter $\vartheta \in \R$ is a \textit{fine-tuning} constant chosen as desired. 
We begin with two preliminary definitions.

\begin{defn}[Basic notation]\label{j_n}
    Let $h\in \mathbb N$, let $t>0$, and let $\x=(x_1,\dots,x_h) \in (\mathbb R^2)^h$.
    We denote by $\mathbf g^{\otimes h}(t,\x):= (2\pi t)^{-h} e^{-|\x|^2 / 2t}$ the standard heat kernel on $(\mathbb R^2)^h.$

    For $\vartheta \in \mathbb R$ and $t\in [0,1]$, we also define the special function
    $G_\vartheta(t):= \int_0^\infty \frac{e^{(\vartheta-\gamma)s} st^{s-1}}{\Gamma(s+1)} ds,$ where $\gamma \approx 0.57$ is the Euler-Mascheroni constant, and $\Gamma(z) = \int_0^\infty e^{-u} u^{z-1} du$ is the Gamma function.

    We write $I \vdash S$ if $S$ is a set and $I = \{I_1,\ldots,I_k\}$ is a set partition of $S$.
    For all $i \in S$ we write $[i]_I$ for the set in $I$ containing $i$.
    A \emph{pair partition} $J \vdash \{1,\dots,h\}$ is a partition consisting of one pair and $h - 2$ singletons. 
    For all pair partitions $J \vdash \{1,\ldots,h\}$, we define the \emph{collision space} $(\R^2)^h_J := \{\mathbf x \in (\R^2)^h : x_i = x_j \textrm{ if } [i]_J = [j]_J\}$.
    See that $(\R^2)^h_J \cong (\R^2)^{h - 1}$.
\end{defn}

\begin{defn}[Moment coefficients for the SHF]\label{mom_co}
    For $\Phi, \Psi \in C_c((\mathbb R^2)^h)$, $t>0$, $\vartheta\in \mathbb R$, and $m\in \mathbb Z_{\ge 0},$ we define
    \begin{align*}
    \mathfrak I_{\mathrm{SHF}}^{(h)} (&\vartheta; m,t,\Phi,\Psi) := \sum_{\substack{J_1,\dots,J_m \vdash\{1,2,3,4\}\\ J_i\ne J_{i-1} \\ \mathrm{pair\;partitions\;only}}}  \iint_{\substack{u_1<v_1 <\dots<u_m<v_m<t \\ \prod_{i = 1}^{m+1} (\R^2)^h_{J_{i-1}} \times (\R^2)^h_{J_i}}} 
    \;\Phi(\mathbf c_0)\\ \times \prod_{i=1}^{m}&\bigg\{  
    \mathbf g^{\otimes h} (u_i-v_{i-1},\mathbf c_{i-1}, \mathbf b_i) \times \mathbf U^{J_i}_\infty( v_i - u_i, \mathbf b_i ,\mathbf c_i) 
    \bigg\} \mathbf g^{\otimes h} (t-v_m , \mathbf c_m,\mathbf b_{m+1}) \Psi(\mathbf b_{m+1})\;d\vec ud \vec v d\vec{\mathbf b}d \vec{\mathbf c}.
\end{align*}
Here we take $\mathbf U^J_\infty(s,\x,\vect y)=U^\vartheta_\infty( s, \mathbf x | _{J^{(2)}},\mathbf y| _{J^{(2)}}) \cdot \mathbf g^{\otimes (h-2)}(s, \mathbf x|_{J^{(1)}}, \mathbf y|_{J^{(1)}}),$ where $\x|_{J^{(2)}}$ corresponds to the 2-vector consisting of the $J$-paired coordinates of the vector $\x$, and $\x|_{J^{(1)}}$ corresponds to the $(h-2)$-vector consisting of the $J$-singleton coordinates of the vector $\x$, and moreover $U^\vartheta_\infty(t,\vec x,\vec y) = 4\pi G_{\vartheta}(t)\mathbf  g(\tfrac{t}{2}, y_1-x_1)$ for $\vec x,\vec y\in (\mathbb R^2)^2$. Furthermore, we take $J_0, J_{m+1}$ to be the partitions consisting of all singletons, so that $(\mathbb R^2)^h_{J_0} = (\mathbb R^2)^h_{J_{m+1}} = (\mathbb R^2)^4$ is the full space with no constraints. Finally, it should also be understood that $v_0=0$ and $u_{m+1}=t$ above. 

When $m=0$, the above integral is understood as $\mathfrak I_{\mathrm{SHF}}^{(h)} (0,t,\Phi,\Psi) = \langle \Phi, \mathbf g^{\otimes h}(t) * \Psi\rangle_{L^2((\mathbb R^2)^h)}$ where the convolution is in space only.
\end{defn}

The reason we defined these coefficients $\mathfrak I^{(h)}$ is because the infinite series $\sum_{m=0}^\infty \mathfrak I_{\mathrm{SHF}}^{(h)} (\vartheta; m,t,\Phi,\Psi)$ will describe the $h$\textsuperscript{th} moment of the SHF. When $h=1$, we see that $\mathfrak I_{\mathrm{SHF}}^{(h)} (\vartheta; m,t,\Phi,\Psi)=0$ for all $m\ge 1$, thus the $m=0$ term is the only nonzero term. Meanwhile if $h=2$ then $\mathfrak I_{\mathrm{SHF}}^{(h)} (\vartheta; m,t,\Phi,\Psi)=0$ for all $m\ge 2$, since pair partitions for a two-element set cannot be non-consecutive. Thus we see that $h=3$ is the lowest order where one sees a nontrivial contribution for all $m$. 

Now we can define the SHF. The following definition is due to Tsai \cite[Definition 1.2]{Tsa24}, based on previous work of \cite{GQT} on the moment expansions.

\begin{defn}[Critical $2d$ stochastic heat flow] \label{def:axioms}
    Let $\mathcal M_{\mathrm{loc}} ((\mathbb R^2)^2)$ denote the space of locally finite measures on $(\mathbb R^2)^2$. Equip $\mathcal M_{\mathrm{loc}} ((\mathbb R^2)^2)$ with the \textit{vague topology}, that is, the weakest topology with respect to which the maps $\mu \mapsto \int f\;d\mu$ are continuous for all $f\in C_c((\mathbb R^2)^2)$. This is a Polish space. Finally let $C([0,1]^2_{\le}, \mathcal M_{\mathrm{loc}} ((\mathbb R^2)^2))$ denote the space of continuous maps from $[0,1]^2_{\le}\to \mathcal M_{\mathrm{loc}} ((\mathbb R^2)^2))$ where $[0,1]^2_{\le}:= \{ (s,t) \in [0,1]^2: s\le t\}.$

    If $\vartheta \in \mathbb R$, then the \textit{critical $2d$ stochastic heat flow of parameter $\vartheta$} is a $C([0,1]^2_{\le}, \mathcal M_{\mathrm{loc}} ((\mathbb R^2)^2))$-valued random variable $Z^\vartheta$, written as $(s,t) \mapsto Z^\vartheta_{s,t} (dx,dy)$ with $x,y\in \mathbb R^2$, whose law is uniquely determined by the following three axioms. 

    \begin{enumerate}
        \item $Z^\vartheta_{s_j,t_j}$ are independent under $\mathbb P$ whenever $(s_j,t_j] \subset [0,1]$ are disjoint intervals.

        \item For all $f\in C_c^\infty((\mathbb R^2)^2)$ and all $\psi\in C_c^\infty(\mathbb R^2)$ with $\int_\mathbb R\psi=1$, we have that $$\int_{(\mathbb R^2)^4} \epsilon^{-2} \psi(\epsilon^{-1}(y_1-y_2))\cdot f(x,z) Z^\vartheta_{t,u}(dy_1,dz) Z^\vartheta_{s,t}(dx,dy_2)  \to \int_{(\mathbb R^2)^2} f(x,z) Z^\vartheta_{s,u}(dx,dz)$$ in probability as $\epsilon\to 0$, whenever $0\le s<t<u\le 1$.

        \item \label{item:moments} $\mathbb E \big[ \prod_{j=1}^h Z^\vartheta_{s,t} (\phi_j \otimes \psi_j)\big]  = \sum_{m=0}^\infty \mathfrak I_{\mathrm{SHF}}^{(h)} \big(\vartheta; m,t-s,\bigotimes_{j=1}^h \phi_j, \bigotimes_{j=1}^h \psi_j\big)$ for all $h \in \{1,2,3,4\}$ and all $\phi_j,\psi_j \in C_c^\infty(\mathbb R^2)$, where $\mathfrak I^{(h)}_{\mathrm{SHF}}$ are the coefficients in Definition \ref{mom_co}. 
    \end{enumerate} 
\end{defn}

We remark that in contrast to related works, here we will only consider the SHF on the time interval $0\le s<t\le 1$, as this will simplify certain aspects of the exposition and make certain formulas easier to prove. 

In the years since its construction, significant attention has been spent probing the qualitative structure of the SHF \cite{Caravenna,MR4945087,csz25,gt25,cd25,ct25,chen25b,chen25a,nak25,gn25,BCT,gt26,gn26,btz26, MR5082783}. For example, it is now known that the moments of the SHF grow doubly-exponentially \cite{GQT,gn25}. The SHF exhibits non-Gaussian intermittency \cite{gn26} and admits no representation as a Gaussian multiplicative chaos \cite{Caravenna}. The Besov-H\"older regularity of $Z^\vartheta_{0,1}$ is $C^{-\epsilon}$ for every $\epsilon>0$, and the measures $Z^\vartheta_{0,1}$ have full support on $\mathbb R^2 \times \R^2$ \cite{csz25}. Gaussian fluctuations hold for $\vartheta\to -\infty$ \cite{MR4945087}, yet the object is not a Gaussian-driven stochastic PDE in its own right \cite{gt25,cd25}. The SHF can also be used to construct continuum polymer measures \cite{MR5082783, ct25}. Furthermore, the SHF admits a martingale problem, though uniqueness is still open \cite{chen25a, nak25}. See also \cite{SHFnotes,SHFICM} for surveys and further properties.
These traits, while interesting in their own right, also offer a remarkable glimpse of how other systems behave: in particular, they show that models converging to the SHF exhibit nontrivial behavior at the critical point. 

\subsection{Results for discrete models}

\begin{defn}\label{1}
    A discrete directed \textit{linear stochastic flow on $\mathbb Z^2$} is a family of random variables $ H_{s,t}(x,y)$ where $s,t\in \mathbb Z$ with $s\le t$, and $x,y\in\mathbb Z^2$, satisfying the following assumptions: 
    \begin{enumerate}
        \item\label{1stat} (Stationary increments) $H$ is nonnegative, and strictly stationary in both variables: more precisely, $H_{s,t} \stackrel{d}{=} H_{s+h,t+h}$ for all $h\in \mathbb Z$ and all $s<t$, and also $(H_{s,t}(x+a,y+a))_{x,y\in \mathbb Z^2} \stackrel{d}{=} (H_{s,t}(x,y))_{x,y\in \mathbb Z^2}$ for all $a\in \mathbb Z^2$ and all $s<t.$

        \item\label{1indep} (Independent time increments) $H_{s_j,t_j}$ are independent if $(s_j,t_j]$ are disjoint intervals.

        \item\label{1linear} (Linear system) We have $\sum_{y\in \mathbb Z^2} H_{s,t} (x,y) H_{t,u}(y,z) = H_{s,u}(x,z)$ for all $x,z\in\mathbb Z^2$ and all integers $s<t<u$.

        \item\label{1norm} (Normalization) We have that $\sum_{y\in \mathbb Z^2} \mathbb E [ H_{0,1}(0,y)] =1.$
        \end{enumerate}
\end{defn}

Note that Item \eqref{1linear} can be viewed as a sort of cocycle property or a composition of operators, in the sense that $H_{s,t}H_{t,u} = H_{s,u}$ for $s<t<u$. Since $H_{t,t+1}$ is simply an independent copy of $H_{0,1}$, one can see that the distribution of $H_{0,1}$ actually determines all statistical properties of the full flow system $\{H_{s,t}\}_{s<t}$. The core example of such a directed flow comes from the directed polymer model, in which $H_{s,t}(x,y)$ would be the point-to-point partition function from $(s,x)$ to $(t,y)$, which means that the one-step evolution operator is $H_{t,t+1}(x,y) = e^{\omega_{t,x}} p(y-x)$ for some i.i.d.\ environment $\omega_{t,x}$ such that $\mathbb E[e^{\omega_{0,0}}]=1$, and $p$ is a deterministic probability measure on $\mathbb Z^2$.

Such stochastic flows have appeared in different contexts in the works of \cite{kun94b, TV98, lejan}, but a systematic study of the properties of interest in this work seems elusive. Before getting to concrete examples of such flows, we list some of these properties.

This notion of stochastic flow comes with a lot of structure, especially in terms of the moments of the stochastic flow. Define 
\begin{equation*}\eta^{(h)} (x_1,...,x_h): = \log \sum_{(y_1,\ldots,y_h)\in (\mathbb Z^2)^h} \mathbb E\bigg[ \prod_{j=1}^h H_{0,1}(x_j,y_j)\bigg]
\end{equation*}
and define a Markov chain $\mathbf Q_{[h]}^{(x_1,...,x_h)} $ on $(\mathbb Z^2)^h$ by the one-step transition probability 
\begin{equation*}\mathbf q_{[h]}((x_1,...,x_h), (y_1,...,y_h)) = e^{-\eta^{(h)}(x_1,...,x_h)} \mathbb E\bigg[ \prod_{j=1}^h H_{0,1}(x_j,y_j)\bigg].
\end{equation*}Then by iterating the definition of the Markov chain, it is easily seen that one has a  Bose-gas representation of the moments in terms of the Markov chain:  
\begin{equation}\label{eq:mrep}
    \mathbb E\bigg[ \prod_{j=1}^h H_{0,1}(x_j,y_j)\bigg] =  \mathbf Q_{[h]}^{(x_1,...,x_h)} \bigg[ e^{\sum_{r=0}^{t-1} \eta^{(h)}(\mathbf R_r)} \ind_{\{\mathbf R_t=\y\}}\bigg], 
\end{equation}
for all $t \in \mathbb N$ and all $x_i,y_i\in \mathbb Z^2$. Above $\mathbf R_t=(R^1_t,...,R^h_t)$ represents the realization of the Markov chain as a stochastic process, and the super-script $(x_1,...,x_h)$ on $\mathbf Q_{[h]}$ represents the starting position of the Markov chain, i.e., $\mathbf R_0=(x_1,...,x_h)$ under $\mathbf Q^{(x_1,...,x_h)}_{[h]}$. 

Suppose there exists $A>0$ for which $H_{0,1} (x,y)$ and $H_{0,1}(x',y')$ are independent whenever $|x-x'| \ge A$. This value $A$ is called the \textit{range of dependency} of the stochastic flow. Finite-range dependency does not imply that $\mu(x):=\mathbb E[H_{0,1}(0,x)]$ is a finitely supported measure, and many of our results will in fact allow infinite-range jumps as long as some exponential moment is finite. 

For $h=2$, by translation invariance, we note that $\eta^{(2)}(x_1,x_2)$ can be written as $\eta(x_1-x_2)$ for some function $\eta:\mathbb Z^2\to \mathbb R$. If a finite range of dependency $A>0$ is assumed, then $\eta$ has a finite support contained in $[-A,A]^2$. Moreover, $\eta^{(h)} (\mathbf x) = \sum_{1\le i<j\le h} \eta(x_i-x_j)$ for all $\mathbf x=(x_1,...,x_h)$ that do not have any ``non-pair collisions" among the coordinates, in the sense that $|x_i-x_j| + |x_\ell-x_k|\leq 2A$ for some $i<j$ and $k<\ell$ with $\{i,j\}\ne \{k,\ell\}$. The Markov chains $\mathbf Q_{[h]}$ are also well-behaved. For example, if $\mu(x):= \mathbb E[H_{0,1}(0,x)]$, then the one-step transitions of $\mathbf Q_{[2]}$ behave as random walk increments with distribution $\mu^{\otimes 2}$ away from a neighborhood of size $A$ of the diagonal of $(\mathbb Z^2)^2$. 


To state our main results, we will consider a \textit{sequence} $H^N$ of such stochastic flows, and we will be interested in \textit{macroscopic scaling limits} of these objects under a diffusive scaling of space and time. These ideas then lead to the following list of assumptions.

\begin{ass}[Assumptions for the discrete-time result]\label{a1}
    Let $ H^N_{s,t}(x,y)$ be a family of random variables indexed by $s,t\in \mathbb Z$ with $s\le t$, and $x,y\in\mathbb Z^2$. Suppose these satisfy the following: 
    \begin{enumerate}


        \item \label{a1flow} Each $H^N$ is a discrete linear stochastic flow on $\mathbb Z^2$, in the sense of Definition \ref{1}.
        
        \item (Markov representation of the moments) \label{a1markov} For each $N\in \mathbb N$ and $h \in \{1,2,3,4,5\}$ there exists a time-homogeneous Markov chain $\mathbf P_{N,h}^{(x_1, \ldots, x_h)}$ on $(\mathbb Z^2)^h$ with the property that 
        for all $t\in\mathbb Z_{\geq 0}$ and all $\y=(y_1,...,y_h)\in (\mathbb Z^2)^h,$ one has  \begin{align*}\mathbb E\bigg[ \prod_{j=1}^h H^N_{0,t}(x_j,y_j)\bigg] 
        =\mathbf E_{N,h}^{(x_1,\ldots,x_h)} \bigg[ e^{ \mathscr E^{N,h}_t+\sum_{r=0}^{t-1} {\boldsymbol{\eta}}_{N,h}( \mathbf R_r)  } \cdot \ind_{\{\mathbf R_t=\y\}} \bigg], 
        \end{align*}
        where $\mathbf R=(R^1,...,R^h)$ denotes the coordinate process on $(\mathbb Z^2)^h$. Furthermore, the following conditions hold:
        \begin{itemize} \item ${\boldsymbol{\eta}}_{N,h}: (\mathbb Z^2)^h \to \mathbb R$.
        \item $\lim_{N\to\infty} \sup_{t,\x} \|\mathscr E^{N,h}_t\|_{L^\infty(\mathbf P^{\x}_{N,h})} =0,$ thus $\mathscr E^{N,h}_t$ can be viewed as an error term. 
        \item $\boldsymbol{\eta}_{N,2}(x_1,x_2) = \z_N(x_1-x_2)$ where $\z_N:\mathbb Z^2\to [0,\infty)$ is nonnegative and has finite support contained in a ball of radius $A$ in $\mathbb Z^2,$ for some $A>0$ not depending on $N$. 
        \item  $\|\boldsymbol{\eta}_{N,h}\|_{L^\infty((\mathbb Z^2)^h )}\leq C/\log N,$ and furthermore
        \begin{equation*}\boldsymbol{\eta}_{N,h}(\x) = \sum_{1\le i<j\le h} \z_N(x_i-x_j) \hspace{0.2 in} \text{ for all } \hspace{0.2 in} \x= (x_1,...,x_h)\notin T^{(h)},\end{equation*} where $T^{(h)}\subset (\mathbb Z^2)^h$ is the set of \emph{non-pair collisions}, consisting of all vectors $\x=(x_1,...,x_h)$ for which $|x_i-x_j|\leq A$ and $|x_k-x_\ell|\leq A$ for some coordinate indices $i<j$ and $k<\ell$ with $\{i,j\}\ne \{k,\ell\}.$
        \end{itemize}
        \item (Translation invariance and consistency of Markov chains) \label{a1consistency}  Let $\mathbf p_{N,h}(\x,\y)$ denote the one-step transition density of the Markov chain $\mathbf P_{N,h}$ from the previous step. Then:
        \begin{itemize}\item  $\mathbf p_{N,1} (x,y) = \mathbb E[H^N_{0,1}(0,y-x)]$, so that $\mathbf P_{N,1}$ is an i.i.d.\ random walk. Equivalently $\mathscr E^{N,1} = 0$ and $\boldsymbol{\eta}_{N,1}=0.$
        \item The Markov chain $\mathbf P_{N,2}$ satisfies $\mathbf p_{N,2}((x,x'),(y,y')) = \mathbf p_{N,2} ((x+a,x'+a), (y+a,y'+a))$ for all $x,x',y,y',a\in\mathbb Z^2.$ Thus, the difference $R^1-R^2$ of the two coordinates under $\mathbf P^{\x}_{N,2}$ is Markov in its own filtration, whose law we write as $\mathbf P^x_{N,\mathrm{dif}}.$
        \item For $h \in \{1,2,3,4,5\}$, the Markov chains $\mathbf P_{N,h}$ are symmetric in their coordinates, and moreover 
        \begin{align*}\mathbf p_{N,h} ( (x_1,\dots,x_m,x_1'&,\dots,x_{h-m}') , \bullet) \\&= \mathbf p_{N,m} ((x_1,\dots,x_m),\bullet) \otimes \mathbf p_{N,h-m} ((x_1',\dots, x_{h-m}'),\bullet) 
        \end{align*} if $m\le h$ and $\min\{ |x_i-x_j'| : i\le m$ and $ j\le h-m\}>A$, where $A$ is as in Item \eqref{a1markov}.
        \footnote{We do not impose any constraints on the collision regions where $\min\{ |x_i-x_j'| : i\le m$ and $ j\le h-m\}\le A$.}
        \end{itemize}
        \item (Logarithmic asymptotics of the return probability) \label{a1logasymptotic} The transition kernels of $\mathbf P^x_{N,\mathrm{dif}}$ and $\mathbf P^x_{N,h}$ at time $r\in \mathbb Z_{\ge 0}$ will be written as $\mathbf p_{N,\mathrm{dif}}(r;x,y)$ and $\mathbf p_{N,h}(r;\x,\y),$ respectively. Thus $\mathbf p_{N,\mathrm{dif}}(r;x,y) := \sum_{a\in\mathbb Z^2} \mathbf p_{N,2} (r;(x,0), (a,a-y))$ for $x,y\in \mathbb Z^2$. We assume the following:
        \begin{itemize}
        \item There exists a measure $\nu_\infty$ of full support on $\mathbb Z^2$ such that for all fixed values of $x,y\in\mathbb Z^2$ one has the following asymptotic as $N\to \infty$:\footnote{Note that for independent, aperiodic random walks, $\nu_\infty(y) = \frac{1}{4 \pi}$.} \begin{equation*}\sum_{r=0}^N \mathbf p_{N,\mathrm{dif}}(r;x,y) = \nu_\infty(y) \log(N) +o(\log N). 
        \end{equation*}
        \item Furthermore, define a sequence of vectors in $\mathbb R^2$ by 
        \begin{equation*}d_N:= N\sum_{y\in \mathbb Z^2} y\mathbb E[H_{0,1}^N(0,y)].
        \end{equation*}If the sequence $(x^1_N,x^2_N) \in (\mathbb Z^2)^2$ is \textit{well-separated} in the sense that $N^{-1/2} (x^1_N,x^2_N)\to (x_1,x_2)$ with $x_1\ne x_2$,
        then for all $t>0$, $y\in \mathbb Z^2$, and smooth bounded $\phi:\mathbb R^2\to \mathbb R$, we impose that 
        \begin{align*}\lim_{N\to \infty} \sum_{r=0}^{Nt} &\sum_{a\in \mathbb Z^2} \phi(N^{-1/2} (a-N^{-1}d_Nr)) \mathbf p_{N,2} (r,x_N^1,x_N^2,a,a-y) \\&= 4 \pi \cdot  \nu_\infty(y) \int_0^t \mathbf g(2s,x_1-x_2) \int_{\mathbb R^2} \mathbf g(2s,z) \phi(\tfrac12 (z+x_1+x_2)) dz \,ds, 
        \end{align*}
        where $\g(s,x)=(2\pi s)^{-1} e^{-|x|^2/2s}$ is the standard heat kernel on $\mathbb R^2$.
        \end{itemize}
         

        \item (Invariance principle) \label{a1invarianceprinciple} Whenever $N^{-1/2} (x_1^N,\ldots,x_h^N)\to (a_1,\ldots,a_h)$ with $x_i^N\in\mathbb Z^2$, the process $N^{-1/2}(R_{Nt}^1-d_Nt,\ldots,R^h_{Nt}-d_Nt)$ under $\mathbf P_{N,h}^{(x_1^N,\ldots,x_h^N)}$ satisfies an invariance principle to Brownian motion in $(\mathbb R^2)^h$ started from $(a_1,\ldots,a_k)$ as $N\to\infty$, in the topology of $C([0,T], (\mathbb R^2)^h)$.
        
        \item (Heat kernel upper bound) \label{a1heatkernel} The transition densities of the Markov chains $\mathbf P_{N,h}$ satisfy the long-time heat kernel estimate $\mathbf p_{N,h} (r; \x, \y) \leq Cr^{-h} (1 + r^{-1/2}|\y - \x - N^{-1}r (d_N,\ldots,d_N)|)^{-S}$, where $r \geq 1$, $h \in \{1,2,3,4,5\}$, and $S>0$ can be taken as large as desired. Here $C$ is uniform over $N,h,r,\x,\y$, but may depend on $S$. 
    \end{enumerate}
\end{ass}

\begin{remark}  Items \eqref{a1flow}--\eqref{a1consistency} are typically simple to verify directly in models of interest. Items \eqref{a1logasymptotic}--\eqref{a1heatkernel} may appear difficult to verify directly, but simpler alternative criteria will be given in Theorem \ref{check} just below.

    A natural question is why we have assumed a Markov representation of the moments in Assumption \ref{a1}, given that such a representation always exists by \eqref{eq:mrep}. The reason is that under the construction in \eqref{eq:mrep}, it is \textit{not} automatically true that $\z_N\ge 0$, which is a nontrivial restriction in Item \eqref{a1markov}. More importantly, the choice of Markov chain from \eqref{eq:mrep} may seem like the simplest or most canonical choice to use, however, it is not the unique choice of Markov chain satisfying the representation in Item \eqref{a1markov}. In some cases of interest, it will be necessary to tilt the Markov chain in \eqref{eq:mrep} to a different one in order to verify the assumptions. 
    Such a tilt can then yield a nonzero error term $\mathscr E^{N,h}_t$ as assumed above. See Subsection \ref{5.3} for such an example in the continuum.

    One may also ask if Assumption \ref{a1} \eqref{a1logasymptotic} is equivalent to a local limit theorem for $\mathbf P_{N,2}$ as $N\to \infty$. The answer is no: it is weaker than a local limit theorem since there is a time-average. We are not actually able to prove a local limit theorem or even a heat kernel lower-bound matching the upper-bound for most models of interest, but are nevertheless able to extract the precise limit of the time average using a Krylov-Bogoliubov trick as done in \cite[proofs of Theorems 4.16--4.17]{DP25}. 
\end{remark}

\begin{defn}\label{a2}
    With $H^N$ and $d_N$ as in Assumption \ref{a1}, define the \textit{macroscopic density field} $$\mathfrak H^N_{s,t}(\phi,\psi):= N^{-1}\sum_{(x,y)\in( \mathbb Z^2)^2} \phi(N^{-1/2} (x-d_Ns)) H^N_{Ns,Nt}(x,y) \psi(N^{-1/2} (y-d_Nt)).$$
    Here $\phi,\psi \in C_c(\mathbb R^2)$ and $s,t \in N^{-1}\mathbb Z$ with $s\le t.$
\end{defn}



\begin{defn}[Critical tuning coefficients]\label{def:lambda}
    For any finitely supported function $f : \Z^2\to \R$, we define a coefficient $\Lambda_N(f)$ as follows: if $J:= \mathrm{supp}(f)$ and $\mathscr S_N(x,y):= \sum_{s=1}^N \mathbf p_{N,\mathrm{dif}}(s;x,y)$, then we take $\Lambda_N(f)$ to be the principal eigenvalue of the linear map from $\mathbb R^J\to\mathbb R^J$ given by $T_N^f h(x) = \frac1{\log N}\cdot  \sum_{z\in J}\mathscr S_N(x,z) f(z)h(z).$
\end{defn}
Then we have the following theorem.

\begin{thm}[Main result---discrete case]\label{mr}
    Suppose that we have a family $\{H^N\}_{N\ge 1}$ of discrete stochastic flows on $\mathbb Z^2$ satisfying Assumption \ref{a1}. Henceforth fix a \textit{fine-tuning constant} $\vartheta\in \mathbb R$, which is arbitrary. We impose that $\z_N$ from Item \eqref{a1markov} of Assumption \ref{a1} satisfies the \textbf{critical scaling} asymptotic 
    \begin{equation}\label{crit_tuning}
        \Lambda_N\big(e^{\z_N} - 1\big) \log N = 1+\frac{\vartheta+o(1)}{\log N}
    \end{equation}
    as $N \to \infty$, where $\Lambda_N(f)$ is as in Definition \ref{def:lambda}.
    Then, the macroscopic field $\mathfrak H^N_{s,t}(\phi,\psi)$ from Definition \ref{a2} converges to the critical 2$d$ stochastic heat flow, in the sense of finite-dimensional distributions.
    More precisely, fix $m\in \mathbb N$ and $\phi_j, \psi_j\in C_c((\mathbb R^2)^2)$, for $1\le j \le m$. Consider any sequences $0\le s_j^N<t_j^N\le 1$ with $s_j^N,t_j^N\in N^{-1}\mathbb Z$, and assume $s_j^N\to s_j$ and $t^N_j \to t_j$, as $N\to\infty$ for $1\le j \le m$. Then one has the joint convergence in distribution
    \begin{equation*}
        \big(\mathfrak H^N_{s_j^N,t_j^N} (\phi_j, \psi_j)\big)_{j=1}^m \to \big( Z_{s_j,t_j}^\vartheta (\phi_j,\psi_j)\big)_{j=1}^m
    \end{equation*}
    as $N\to \infty$, where $Z_{s,t}^\vartheta$ is the SHF with parameter $\vartheta$ as in Definition \ref{def:axioms}.
\end{thm}

One can think of \eqref{crit_tuning} as a critical tuning of the model parameters, generalizing \eqref{crit1}. The core idea for why \eqref{crit_tuning} is the correct scaling to observe the SHF is actually quite brief, and can be found in Sections \ref{s:2.1}--\ref{s:2.2}. 

One might ask whether it is easy to actually check the conditions of Assumption \ref{a1}, in particular Items \eqref{a1logasymptotic}--\eqref{a1heatkernel}. This was a problem studied in \cite[Section 4]{DP25}. It turns out \eqref{a1logasymptotic}--\eqref{a1heatkernel} automatically hold under reasonable conditions on the stochastic flow.

\begin{thm}[Simple criteria to verify Assumption \ref{a1}]\label{check}
    Suppose that $H^N$ is any sequence of directed linear stochastic flows on $\mathbb Z^2$ for which Items \eqref{a1flow}, \eqref{a1markov}, and \eqref{a1consistency} of Assumption \ref{a1} hold.  
    Assume:
    \begin{enumerate} 
    \item $\sup_N \mathbb E\big[ \big( \sum_{x\in \mathbb Z^2} e^{\epsilon |x|}H^N_{0,1}(0,x)\big)^5\big]<\infty$ for some $\epsilon>0$. 
    \item For all $h=1,2,3,4,5$, the limit of the time-one transition kernel
    \begin{equation}
        \mathbf p_{\infty,h} (\mathbf x,\mathbf y) := \lim_{N\to \infty} \mathbb E\bigg[ \prod_{j=1}^h H^N_{0,1} (x_j,y_j)\bigg]\label{limit_N}
    \end{equation}
    exists for each $\x,\y\in(\mathbb Z^2)^h$. Assume furthermore that the convergence is uniform in total variation norm: $\lim_{N\to\infty} \sup_{\x} \big\|\mathbb E\big[ \prod_{j=1}^h H^N_{0,1} (x_j,\bullet)\big] - \mathbf p_{\infty,h} (\x,\bullet) \big\|_{TV} =0.$
    \item \label{irred} The one-step transition kernel $\mathbf p_{\infty,h}$ defines an irreducible Markov chain on $(\mathbb Z^2)^h$, for all $h \in \{1,2,3,4,5\}$. 
    \item The covariance matrix of the measure $\mathbf p_{\infty,1} (0,\bullet)$ is the $2\times 2$ identity matrix.
    \end{enumerate}
    Then Items \eqref{a1logasymptotic}--\eqref{a1heatkernel} of Assumption \ref{a1} automatically hold, where $\nu_\infty$ is the invariant measure for the gap chain of $\mathbf p_{\infty,2}$, defined as the Markov chain on $\mathbb Z^2$ with one-step transition probability $\mathbf p_{\infty,\mathrm{dif}}(x,y) = \sum_{a\in \mathbb Z^2} \mathbf p_{\infty,2} ((x,0),(a,a-y)).$
    
    Furthermore, the result still holds under the weaker condition that $\lim_{N\to \infty} \|\boldsymbol{\eta}_{N,h}\|_{L^\infty} = 0$ (rather than $ \|\boldsymbol{\eta}_{N,h}\|_{L^\infty} \leq C/\log N$, as in Item \eqref{a1markov} of Assumption \ref{a1}).
\end{thm}

Note that the conditions of Theorem \ref{check} are local, in the sense that they only require one to know the properties of the stochastic flow $H^N$ on the time interval $[0,1]$. In contrast, the global conditions in Assumption \ref{a1} probe macroscopic behavior of the Markov chains (for example, by invariance principles). Thus, one can think of Assumption \ref{a1} as the set of technical conditions that make the proof work, and Theorem \ref{check} as the practical way to check those conditions.

It follows from Assumption \ref{a1} \eqref{a1markov} that $\mathrm{Var} (\sum_y H_{0,1}(0,y)) \leq C/\log N$, and therefore
\begin{equation*}
    \bigg|1- \sum_{y_1,y_2} \mathbb E\big[ H^N_{0,1}(x_1,y_1)H^N_{0,1}(x_2,y_2)\big]\bigg| = \bigg|\mathrm{Cov}\Big(\sum_y H^N_{0,1}(x_1,y) , \sum_y H^N_{0,1}(x_2,y) \Big) \bigg|\leq \frac{C}{\log N}.
\end{equation*}
Thus it is automatically true that the limit on the right hand side of \eqref{limit_N} defines a Markov chain for $k=2$, i.e.\ that $\mathbf p_{\infty,2}(\x,\bullet)$ is a probability measure for all $\x$. One might then ask if Item \eqref{irred} of Theorem \ref{check} holds automatically, but this is not the case. Irreducibility of the limiting chain can fail even in relatively simple models, see \cite{DDP24} for a discussion of the Bernoulli-weighted RWRE.

\subsection{Results for continuum models}

Next, we state a continuous version of our results.

\begin{defn}\label{2}
     We allow $\mathbb I$ to be $\mathbb Z$ or $\mathbb R$ as desired, and all integrals on $\mathbb I$ should henceforth be understood with respect to counting or Lebesgue measure. A continuum directed \textit{linear stochastic flow on $\mathbb I^2$} is a family of random variables $ H_{s,t}(x,y)$ where $s,t\in \mathbb R$ with $s < t$, and $x,y\in\mathbb I^2$, satisfying the following assumptions: 
    \begin{enumerate}
        \item (Stationary increments) $H$ is nonnegative and jointly measurable in all variables, and strictly stationary in both variables: more precisely, $H_{s,t} \stackrel{d}{=} H_{s+h,t+h}$ for all $h\in \mathbb R$, and also $(H_{s,t}(x+a,y+a))_{x,y\in \mathbb I^2} \stackrel{d}{=} (H_{s,t}(x,y))_{x,y\in \mathbb I^2}$ for all $a\in \mathbb I^2$ and all $s,t.$

        \item (Independent time increments) $H_{s_j,t_j}$ are independent if $(s_j,t_j]$ are disjoint intervals.


        \item (Linear system/Chapman-Kolmogorov) We have $\int_{\mathbb I^2} H_{s,t} (x,y) H_{t,u}(y,z)dy = H_{s,u}(x,z)$ for all $x,z\in\mathbb Z^2$ and all real numbers $s<t<u$.

        \item (Normalization) We have that $\int_{\mathbb I^2} \mathbb E [ H_{0,t}(0,y)] dy =1$ for all $0<t<1.$
        \end{enumerate}
\end{defn}

At this point, one might ask if one has an automatic representation of the moments of such $H$ in terms of a diffusion process or similar, analogously to \eqref{eq:mrep} in the discrete setup. Rather than going into such technicalities, we simply assume such a representation holds. 

\begin{ass}[Assumptions for main result---continuous case]\label{a3}
    Assume we have a family $H^N$ of random processes satisfying the following:
    \begin{enumerate}


        \item\label{a3flow} Each $H^N$ is a continuum stochastic flow on $\mathbb I^2$, in the sense of Definition \ref{2}.
        
        \item (Markov representation of the moments)\label{a3markov} For each $N\in \mathbb N$ and $h \in \{1,2,3,4,5\}$ there exists a time-homogeneous Feller Markov process $\mathbf P_{N,h}^{\x}$ on $(\mathbb I^2)^h$, with continuous paths if $\mathbb I=\mathbb R,$ and the property that 
        if $t>0$ and $\psi\in C_c^\infty(\mathbb R^2)$ then
        \begin{equation*}
            \quad\quad\quad\quad \int_{(\mathbb I^2)^h} \mathbb E\bigg[ \prod_{j=1}^h H^N_{0,t}(x_j,y_j)\bigg] \prod_{j=1}^h \psi(y_j) d\y
            =\mathbf E_{N,h}^{(x_1,\ldots,x_h)} \bigg[ e^{\mathscr E^{N,h}_t+ \int_0^t {\boldsymbol{\eta}} _{N,h}(\mathbf R_r)dr } \cdot \prod_{j=1}^h \psi(R^j_t) \bigg], 
        \end{equation*}
         where $\mathbf R=(R^1,...,R^h)$ denotes the coordinate process on $(\mathbb I^2)^h$. Furthermore, the following conditions hold:
         \begin{itemize}
         \item ${\boldsymbol{\eta}}_{N,h}: (\mathbb I^2)^h \to \mathbb R$. 
        \item $\lim_{N\to\infty} \sup_{t,\x} \|\mathscr E^{N,h}_t\|_{L^\infty(\mathbf P^{\x}_{N,h})} =0.$ 
        \item $\boldsymbol{\eta}_{N,2}(x_1,x_2) = \z_N(x_1-x_2)$ where $\z_N:\mathbb I^2\to [0,\infty)$ is nonnegative and has uniformly compact support, meaning that the support of $\z_N$ is contained in a ball of radius $A$ in $\mathbb I^2,$ for some $A>0$ not depending on $N$. If $\mathbb I= \mathbb R$, we impose the stronger constraint that $\z_N(x) = \beta_N^2 \z(x)$ for some fixed function $\z \in C_c(\mathbb R^2)$, and some sequence $\beta_N>0$.
        \item We impose that $\|\boldsymbol{\eta}_{N,h}\|_{L^\infty((\mathbb I^2)^h)}\leq C/\log N,$ and furthermore that \begin{equation*}\boldsymbol{\eta}_{N,h}(\x) = \sum_{1\le i<j\le h} \z_N(x_i-x_j) \hspace{0.2 in} \text{ for all } \hspace{0.2 in} \x= (x_1,...,x_h)\notin T^{(h)},\end{equation*} where $T^{(h)}\subset (\mathbb Z^2)^h$ is the set of \emph{non-pair collisions}, consisting of all vectors $\x=(x_1,...,x_h)$ for which $|x_i-x_j|\leq A$ and $|x_k-x_\ell|\leq A$ for some coordinate indices $i<j$ and $k<\ell$ with $\{i,j\}\ne \{k,\ell\}.$
        \end{itemize}
        \item (Translation invariance and consistency of Markov chains) \label{a3consistency} Assume that the Markov chains $\mathbf P_{N,h}$ satisfy the following.
        \begin{itemize}\item $\mathbf p_{N,1}(t,x,y) = \mathbb E[H^N_{0,t} (0,y-x)]$, so that $\mathbf P_{N,1}$ is a L\'evy process. Equivalently $\mathscr E^{N,1} = 0$ and $\boldsymbol{\eta}_{N,1}=0.$
        \item The Markov chains $\mathbf P^{\x}_{N,h}$ admit transition densities $\mathbf p_{N,h}(t;\x,\y)$ that are continuous in all three variables on $(0,\infty) \times (\mathbb I^2)^h\times (\mathbb I^2)^h$.
        \item $\mathbf p_{N,2}(t;(x,x'),(y,y')) = \mathbf p_{N,2} (t;(x+a,x'+a), (y+a,y'+a))$ for all $x,x',y,y',a\in\mathbb I^2$ and $t>0.$ Thus, the difference $R^1-R^2$ under $\mathbf P^{\x}_{N,2}$ is Markov in its own filtration, with density $\mathbf p_{N,\mathrm{dif}}(t;x,y) := \int_{\mathbb I^2} \mathbf p_{N,2} (t;(x,0), (a,a-y))da$.
        \item The Markov chain $\mathbf P_{N,h}$ is symmetric in its coordinates. Let $\mathbf L_{N,h}$ denote the Markov generator of $\mathbf P_{N,h}$. The domain of $\mathbf L_{N,h}$ should contain $C_c^\infty((\mathbb I^2)^h)$, and for all $f\in C_c^\infty((\mathbb I^2)^h)$ we impose that 
        \begin{align*}\hspace{0.7 in}\mathbf L_{N,h} f & (x_1,\dots,x_m,x_1',\dots,x_{h-m}') \\&= \Big(\mathbf L_{N,m} \otimes \mathrm{Id} + \mathrm{ Id} \otimes \mathbf L_{N,h-m} \Big) f(x_1,\dots,x_m,x_1',\dots,x_{h-m}'),
        \end{align*}
        whenever $\min\{ |x_i-x_j'| : i\le m$ and $ j\le h-m\}>A$, where $A$ is as in Item \eqref{a3markov} and thus independent of $f,N,x_i,x_i'$.
        \end{itemize}
        \item\label{a3logasymptotic} (Logarithmic asymptotics of the return probability) \label{a3log} 
        Assume the following two limits:
        \begin{itemize}
        \item There exists a positive function $\nu_\infty$ on $\mathbb I^2$ such that for $x,y\in\mathbb I^2$, \begin{align*}\int_1^N \mathbf p_{N,\mathrm{dif}}(t;x,y)dt &= \nu_\infty(y) \log(N) +o(\log N), 
        \end{align*}
        as $N\to\infty,$ where the error term is uniform over $x,y$ lying in any compact set. 
        \item Furthermore, define a sequence of vectors in $\mathbb R^2$ by 
        \begin{equation*}d_N:= N\int_{ \mathbb I^2} y\mathbb E[H_{0,1}^N(0,y)]dy.
        \end{equation*}If $\x_N$ is any sequence in $(\mathbb I^2)^2$ for which $N^{-1/2} \x_N\to \x=(x_1,x_2)$ with $x_1\ne x_2$ then for all $y\in \mathbb I^2$,  $t>0$, and smooth bounded $\phi:\mathbb R^2\to \mathbb R$, we impose that 
        \begin{align*}\lim_{N\to \infty} \int_1^{Nt} &\int_{\mathbb I^2}\phi(N^{-1/2} (a-N^{-1}d_Nr)) \mathbf p_{N,2} (r,x_N^1,x_N^2,a,a+y)da\; dr \\&= 4 \pi \cdot  \nu_\infty(y) \int_0^t \mathbf g(2s,x_1-x_2) \bigg(\int_{\mathbb R^2} \mathbf g(2s,z) \phi(\tfrac12 (z+x_1+x_2)) dz \bigg)ds, \end{align*}
        where the limit is uniform over $y$ in any compact set, and $\mathbf g(s,a)=(2\pi s)^{-1} e^{-|a|^2/2s}$ is the standard heat kernel on $\mathbb R^2$. \end{itemize}
        

        \item (Invariance principle)\label{a3inv} Whenever $N^{-1/2} (x_1^N,\ldots,x_h^N)\to (a_1,\ldots,a_h)$, the process $N^{-1/2}(R_{Nt}^1-d_Nt,\ldots,R^h_{Nt}-d_Nt)$ under $\mathbf P_{N,h}^{(x_1^N,\ldots,x_h^N)}$ satisfies an invariance principle to Brownian motion in $(\mathbb R^2)^h$ started from $(a_1,\ldots,a_h)$ as $N\to\infty$, in the topology of $C([0,T], (\mathbb R^2)^h)$.\footnote{Note that if $\mathbb I=\mathbb R$, we do not need the vectors $d_N$ as in the discrete case, because we can automatically subtract $d_N$ from each coordinate of the Markov process, and the resulting process will still be Markov. In other words, if $\mathbb I=\mathbb R$, then by shearing if necessary, we may assume without loss of generality that $\int_{\mathbb R^2} y \mathbb E[H^N_{0,1}(0,y)]dy=0.$}

        \item (Heat kernel upper-
        bounds)\label{a3heatkernel} The transition densities of the Markov chains $\mathbf P_{N,h}$ satisfy the long-time heat kernel estimate $ \mathbf p_{N,h} (t; \x, \y) \leq Ct^{-h} (1 + t^{-1/2}|\y-\x- N^{-1}t(d_N,\ldots,d_N)|)^{-S}$ for $t \ge 1$, $h \in \{1,2,3,4,5\}$, and any $S>0$. Here $C$ is uniform over $N,h,t,\x,\y$, but may depend on $S$.

        If $\mathbb I=\mathbb R,$ we impose a separate short-time estimate $\sup_N \mathbf p_{N,h} (t; \x, \y) \leq Ct^{-h} e^{-\alpha |\x-\y|/\sqrt t},$ uniform over $0<t<1$ and all $\x,\y.$ 
    \end{enumerate}
\end{ass}


\begin{defn}\label{a4}
    With $H^N$ and $d_N$ as in Assumption \ref{a3}, define the \textit{macroscopic density field} $$\mathfrak H^N_{s,t}(\phi,\psi):= N^{-1}\int_{(\mathbb I^2)^2} \phi(N^{-1/2}( x-d_Ns)) H^N_{Ns,Nt}(x,y) \psi(N^{-1/2}( y-d_Nt))dxdy.$$
    Here $\phi,\psi \in C_c(\mathbb R^2)$ and $s,t \in \mathbb R$ with $s\le t.$
\end{defn}

\begin{defn} \label{def:lambdacont}
    For any compactly supported continuous function $f:\mathbb I^2\to \mathbb R$, we define the coefficient $\Lambda_N(f)$ as follows: if $J:= \mathrm{supp}(f)$ and $\mathscr S_N(x,y):= \int_0^N \mathbf p_{N,\mathrm{dif}}(s;x,y)ds$ then we define $\Lambda_N(f)$ to be the principal eigenvalue of the integral operator from $L^2(J) \to L^2(J)$ given by $T_N^f h(x) = \frac1{\log N}\cdot  \int_J\mathscr S_N(x,z) f(z)h(z)dz.$ 
\end{defn}

\begin{thm}[Main result---continuous case]\label{mr2}
    Henceforth fix a \textit{fine-tuning constant} $\vartheta\in \mathbb R$, which is arbitrary. With $\{H^N\}_{N \geq 1}$ as in Assumption \ref{a3}, we impose that $\z_N$ from Item \eqref{a3markov} satisfies the following \textbf{critical scaling} asymptotic\footnote{In the continuous case, note that we simply have $\z_N$ as opposed to $e^{\z_N}-1,$ which is a consequence of the continuum time index as opposed to the discrete one, and will be explained.} as $N\to\infty$: 
    \begin{equation}\label{crit_tuning2}\Lambda_N\big(\z_N\big) \log N = 1+\frac{\vartheta+o(1)}{\log N},
    \end{equation}
    where $\Lambda_N$ is as in Definition \ref{def:lambdacont}.
    Then the macroscopic field $\mathfrak H^N_{s,t}(\phi,\psi)$ from Definition \ref{a4} converges to the critical $2d$ stochastic heat flow in the sense of finite-dimensional distributions. More precisely, the 
    convergence is in the following sense. Fix $m\in \mathbb N$ and $\phi_j, \psi_j\in C_c((\mathbb R^2)^2)$, for $1\le j \le m$. Consider any sequences $0\le s_j^N<t_j^N\le 1$ with $s_j^N,t_j^N\in \mathbb R_+$, and assume $s_j^N\to s_j$ and $t^N_j \to t_j$, as $N\to\infty$ for $1\le j \le m$. Then one has the joint convergence in distribution
    \begin{equation*}
        \big(\mathfrak H^N_{s_j^N,t_j^N} (\phi_j, \psi_j)\big)_{j=1}^m \to \big( Z_{s_j,t_j}^\vartheta (\phi_j,\psi_j)\big)_{j=1}^m
    \end{equation*}
    as $N\to \infty$, where $Z_{s,t}^\vartheta$ is the SHF with parameter $\vartheta$.    
\end{thm}

One might ask why we did not allow $\mathbb R^2$ as the spatial domain in the discrete-time setup of Theorem \ref{mr}. The answer is just to keep things simple. In reality, replacing $\mathbb Z^2$ by $\mathbb R^2$ in Assumption \ref{a1} is perfectly allowable as long as we impose similar continuity conditions on the sample paths of $H^N$ and the additive functionals $\z_N$. In fact, any locally compact additive subgroup of $\mathbb R^2$ would work just as well, and this was the point of view taken in \cite{DP25}.

Next, we will state a result analogous to Theorem \ref{check} about how to actually check the conditions of Assumption \ref{a3}. 

\begin{thm}[Simple criteria to verify Assumption \ref{a3}]\label{check2}
    Suppose that $H^N$ is any sequence of continuum directed linear stochastic flows on $\mathbb I^2$ for which Items \eqref{a3flow}, \eqref{a3markov}, and \eqref{a3consistency} of Assumption \ref{a3} hold true.  
    Assume furthermore that
    \begin{enumerate}
        \item\label{check2exptail} Assume that $\sup_N \mathbb E\big[ \big(\int_{\mathbb I^2} e^{\epsilon |x|} H^N_{0,1}(0,x)dx\big)^5\big]<\infty$ for some $\epsilon>0$. If $\mathbb I=\mathbb R$, assume the slightly stronger estimate over short times: $\mathbb E [ \prod_{j=1}^5 H_{0,t}(x_j,y_j) ] \leq Ct^{-5} e^{-\frac{\alpha |\x-\y|}{\sqrt t}} $ uniformly over $\x,\y \in (\mathbb Z^2)^5$ and $0<t<1$. 
        
        \item\label{check2hptlim} For all $h=1,2,3,4,5$, the limit of the time-one transition kernel \begin{equation*}\mathbf p_{\infty,h} (\mathbf x,\mathbf y) := \lim_{N\to \infty} \mathbb E\bigg[ \prod_{j=1}^h H^N_{0,1} (x_j,y_j)\bigg]
        \end{equation*} exists for each $\x,\y\in(\mathbb I^2)^h$. Assume furthermore that the convergence is uniform in total variation norm: $\lim_{N\to\infty} \sup_{\x} \big\|\mathbb E\big[ \prod_{j=1}^h H^N_{0,1} (x_j,\bullet)\big] - \mathbf p_{\infty,h} (\x,\bullet) \big\|_{TV} =0.$
        
        \item\label{check2irred} The one-step transition density $\mathbf p_{\infty,h}$ defines a topologically irreducible Markov chain on $(\mathbb I^2)^h$, for all $h \in \{1,2,3,4,5\}$.  
    
        \item\label{check2cov} The covariance matrix of the measure $\mathbf p_{\infty,1} (0,\bullet)$ is the $2\times 2$ identity matrix.
    \end{enumerate}

    Then, Items \eqref{a3log}--\eqref{a3heatkernel} of Assumption \ref{a3} automatically hold true, where $\nu_\infty$ is the invariant measure for the gap chain of $\mathbf p_{\infty,2}$, defined as the Markov chain on $\mathbb I^2$ with time-one transition probability $\mathbf p_{\infty,\mathrm{dif}}(x,y) = \int_{ \mathbb I^2} \mathbf p_{\infty,2} ((x,0),(a,a-y))da.$

    The result still holds true under the weaker hypothesis that $\lim_{N\to \infty} \|\boldsymbol{\eta}_{N,h}\|_{L^\infty} = 0$ (as opposed to $ \|\boldsymbol{\eta}_{N,h}\|_{L^\infty} \leq \frac{C}{\log N}$ in Item \eqref{a3markov} of Assumption \ref{a3}).
\end{thm}

\subsection{Examples of directed stochastic flows}
The following models fall under Definition \ref{1}, and therefore into the domain of attraction of the SHF. Full verifications of convergence are given later in Section \ref{sec:examples}: we begin now with brief high-level discussion.

\subsubsection{Directed polymer in a random environment}\label{ss2:dpre}

The quintessential example of a linear stochastic flow is the directed polymer in a random environment. Fix $\beta > 0$. We define the point-to-point partition function $H^\beta_{s,t}(x,y)$ of the directed polymer $(s,x) \to (t,y)$ by 
\begin{equation}\label{eq:DPRE}\tag{DPRE}
    H^\beta_{t,t+1}(x,y) = e^{\beta \omega_{t,x} - \log \mathbb E[e^{\beta \omega_{0,0}}]} p(y - x).
\end{equation}
Here $\{\omega_{t,x}\}_{(t,x) \in \mathbb Z \times \mathbb Z^d}$ is an identical random environment that may have finite range spatial correlations and $p$ is a deterministic probability measure on $\mathbb Z^2$. In this case, note that the Markov chain $\mathbf Q_{[h]}$ from \eqref{eq:mrep} reduces to $h$ independent random walks in $\mathbb Z^2$, each with increment law $p(\bullet)$.

 It is clear that Assumption \ref{a1} is satisfied if $p$ has exponential moments, is aperiodic, and centered with covariance matrix $\mathrm{Id}_{2\times 2}$. The fundamental result of \cite{CSZ_2d} shows that if the $\omega_{t,x}$ are i.i.d.\ and of unit variance, then the convergence of Theorem \ref{mr} holds, and in this case \eqref{crit_tuning} reduces to 
\begin{equation}\label{crit1}\frac{\beta_N^2}{4\pi} = \frac{1}{\log N} +  \frac{\pi^{1/2} \kappa_3 }{\log^{3/2} N} +\frac{\vartheta - c_{\mathrm{RW}}
-(\frac32 \pi -\frac12 \kappa_3^2 +\frac7{12} \kappa_4 )
}{\log^2 N},\quad\quad \vartheta \in \mathbb R,
\end{equation}
        where $\kappa_k$ is the $k$th cumulant of $\omega_{0,0}$ and $c_{\mathrm{RW}}$ is a constant that depends on the increment law $p(\cdot) $, for example $c_{\mathrm{RW}} = \gamma -\log 16 $ for the symmetric nearest-neighbor walk.


It turns out that if one tunes $\beta_N$ other than on the knife's edge of \eqref{crit1}, one either obtains a Gaussian or a trivial limit \cite{MR4945087, BCT}. Thus the emergence of anomalous behavior is even more sensitive to fine tuning than the one-dimensional power law $\beta_N = \vartheta N^{-1/4}$ discussed above. The sharp transition at $d = 2$, with its logarithmic attenuation, can be seen as a critical phenomenon, as explained in the next section.

In the general case of directed polymers with spatially correlated environments $\omega$, convergence to the SHF is treated in Section \ref{5.1}.

\subsubsection{Stochastic heat equation with multiplicative noise}\label{ss2:SHE}

In recent decades, a significant amount of effort has gone into studying singular stochastic PDEs \cite{Hai14, GIP15, GJ14, GP18,duch}. These are PDEs forced by random noises so irregular that they lack classical solutions. We consider the following regularization
\begin{equation}\label{eq:mSHE1}
    \partial_t u^{(N)}(t,x) = \tfrac12 \Delta u^{(N)} (t,x)+ \beta_N \cdot u^{(N)}(t,x) \cdot N^2 \xi(N^2 t, N x).
\end{equation}
of the multiplicative stochastic heat equation \eqref{eq:SHE} discussed earlier.
Here $t \ge 0$, $x \in \R^2,$ and $\xi$ is a centered Gaussian forcing term satisfying $\mathbb E[ \xi(t,x) \xi(s,y)] = J(x-y)\delta(t-s)$ for $J$ smooth and compactly supported, with $\int_{\R^2} J(x)dx = 1$. 
As explained earlier, the one-dimensional analog of this stochastic PDE is a popular model in $d=1$ due to its relation to the KPZ equation through the Cole-Hopf transform \cite{KPZ,Cor12,CS20,Quastel_2015}.

The stochastic PDE \eqref{eq:mSHE1} naturally defines a family of \emph{propagators}: write $H_{s,t}^N(x,y)$ for the solution to the above equation, started from Dirac initial condition at $(s,x)$, then evaluated at $(t,y)$, but coupled to the same realization of $\xi$ for different $s,t,x,y$. It turns out these satisfy the semigroup property
\begin{equation*}
    \int_{\mathbb R^2} H^N_{s,t}(x,y) H^N_{t,u}(y,z)dy  = H^N_{t,u}(x,z),
\end{equation*}
and so $H^N$ is a flow in the sense of Definition \ref{2}.

Take the mollification scale to 0: that is, send $N \to \infty$ in the display above. In Section \ref{sse:bg} we asked whether, if $\beta_N$ is tuned like \eqref{crit1}, the propagators of the regularized stochastic PDE above converge to the same SHF described in the previous example. This is motivated by the natural interpretation of \eqref{eq:mSHE1} as a continuum analog of \eqref{eq:DPRE} through the Feynman-Kac formula \cite{MR3444835}.
The works surveyed in Section \ref{sse:bg}
showed the answer is more subtle: this time the correct tuning is
\begin{equation}\label{crit2}
    \frac{\beta_N^2}{4\pi} = \frac{1}{\log N}+\frac{\vartheta+\big[\log 4 - 2\int_{(\mathbb R^2)^2}J(x)J(y)\log|x-y| dxdy\big]}{\log^2 N},\quad\quad \vartheta \in \mathbb R,
\end{equation}
for which one does obtain the same limiting field $Z^{\vartheta}_{s,t}.$

\eqref{eq:mSHE1} is a nice example of a singular stochastic PDE that is scaling-critical in the regularity structures sense \cite{Hai14}, yet still has a nontrivial (nonzero and non-Gaussian) limit. Scaling criticality in this context means that if a solution to the stochastic PDE were actually to exist after the covariance kernel $J$ is replaced by a Dirac mass, then that solution would be a scale-invariant object in $d = 2$. In other words, its distribution would be preserved under the mapping $\mathcal L(t,x) \mapsto \mathcal L(\epsilon^2 t, \epsilon x)$.
Thus, in some sense, the SHF provides a notion of solution of the SHE in $d=2$ driven with space-time white noise. However, it is not scale-invariant, and it cannot be represented as the solution to any Gaussian-driven stochastic PDE \cite{gt25, cd25}. 

We revisit the convergence of the SHE to the SHF in Section \ref{sss:molSHE}.

\subsubsection{Random walk in a random environment}
Random walks in random environments (RWRE) provide an example of other models which are now well understood to exhibit KPZ-type behavior in $d=1$ \cite{ldt2, ldt, bld, hass23,hass2023b, hass2024extreme, Par24, Par-,Par+, hass2025universal, hass2025superuniversalbehavioroutliersdiffusing}, but their critical behavior in higher dimensions remains less understood. The physics work \cite{ldt} predicted that in $d=2$, the critical scaling for the RWRE is at a location of order $x \sim t/\sqrt{\log t}$, at which point KPZ fluctuations should emerge. See also the recent work \cite{ark2025universalfluctuationstailprobability} for numerical simulations of the critical scaling in $d=2$. The \textit{subcritical} $d=2$ behavior for RWRE models was proved to exhibit Gaussian fluctuations in \cite{DP25}, and that work motivated the study of the critical window in the present work.

To define this model in $d=2$, we consider a family $\{K_t(x,\bullet)\}_{(t,x) \in \Z \times \Z^2}$ of random probability measures on $\mathbb Z^2$. We impose that the $K_t(x,\bullet)$ are independent for distinct $t \in \Z$, of finite-range dependence in $x\in\mathbb Z$, and strictly stationary in $x$. 

Notice that we do not necessarily assume that $K_t(x,\bullet)$ and $ K_t(x',\bullet)$ are independent for $x \ne x'$. Let $\omega = \{K_t(x,\bullet)\}_{(t,x) \in \Z \times \Z^2}$ and let $P^\omega_{s,t}$ be the transition density of the associated RWRE determined by $\omega$: in other words, $P^\omega_{s,t}(x,y)$ is the probability of the step $(s,x) \to (t,y)$, with increment probabilities defined in a Markovian way by the kernels $K_t$. 

Notice that this model also satisfies the flow property
\begin{equation*}
    \sum_{y\in \mathbb Z^2} P^\omega_{s,t}(x,y) P^\omega_{t,u}(y,z) = P^\omega_{s,u}(x,z), \quad\quad s<t<u.
\end{equation*}
Although it appears that there is no parameter $\beta$ that can be tuned in order to yield similar results as the previous two examples, it turns out that such a parameter can indeed be introduced by an exponential tilt.
The main question of interest for this model is then whether or not it is possible to obtain the SHF as a limit 
when one tunes $\beta=\beta_N$ according to some precise asymptotic similar to \eqref{crit1}--\eqref{crit2}. This question will be answered in Section \ref{5.4}.

\subsubsection{Diffusion in a random potential}

Among continuum models, the prime example of a stochastic flow is diffusion of many particles in a turbulent medium. Pure diffusion of a particle will be modeled by a standard Brownian motion $W=(W_t)_{t\ge 0}$ taking values in $\mathbb R^d$. To model the turbulence of the medium in which the particle is diffusing, we will introduce a centered $\mathbb R^d$-valued Gaussian vector field $V=(V_1,\ldots,V_d)$ which is completely independent of the Brownian motion $W$ and spatio-temporally homogeneous in the sense that $\mathbb E[ V_i(t,x) V_j(s,y)] = \delta(t-s) \cdot C_{ij}(x-y)$, say with $C$ smooth and compactly supported and $C(0) = \frac12 \mathrm{Id}_{d\times d}.$ In particular, $V$ is white in time but smooth in space. Thus, the equation for the diffusion of a single particle in the turbulent environment $V$ will be modeled by the stochastic differential equation (SDE) 
\begin{equation}\label{diffu0}
    dX_t = V(t,X_t)dt + dW_t,\quad\quad X_0=0,
\end{equation}
where $x\in \mathbb R^d$. Now, we will be primarily be interested in the behavior of a large number of independent particles $X^1,\ldots,X^N$ diffusing in the turbulent medium $V$. More precisely, these particles $X^i$ solve \eqref{diffu0} with different Brownian motions $W^i$ that are independent of one another, but coupled to the same realization of the drift field $V$. This model of diffusion in turbulent environments is heavily inspired by work of \cite{Kr, GH, LR, lejan, war, ew6, dom}.

We wish to model the time-evolution of the probability density function of the diffusion $X$ solving \eqref{diffu0}, given the randomness of the field $V$. This leads to the Fokker-Planck equation associated to the above diffusion: the conservative stochastic PDE given by 
\begin{equation}\partial_t u(t,x) = \tfrac12 \Delta u(t,x) - \mathrm{div} \big( u(t,x) V(t,x) \big),\quad\quad t\ge 0,\; x \in \R^d, \label{fth0}
\end{equation}
where the product of $u$ and $V$ should be interpreted in the Stratonovich sense. For the equivalent It\^o equation, $\frac12\Delta$ would be replaced by $\Delta$, because the correction is $\frac12\Delta u$: this may be verified using the fact that $C(0) = \frac12  \mathrm{Id}_{d\times d}$. 
\cite{kun94b} shows that \eqref{diffu0} actually makes sense and that a solution exists for almost every realization of $V$, and also that \eqref{fth0} admits a Duhamel solution that almost surely describes the evolution of the density of $   X_t$ started from 0. 

The stochastic PDE \eqref{fth0} is linear as a function of its initial data, and therefore it determines a family of propagators $u_{s,t}(x,y)$ exactly in the same way as described for the stochastic PDE \eqref{eq:mSHE1} in Section \ref{ss2:SHE}. Consequently, the flow property holds as well for this model, and the question again becomes if we can observe the SHF from this model. This question will be answered in Section \ref{5.5}. 

\subsubsection{Multiplicative stochastic PDEs with derivative noise}

The following model is motivated by work of \cite{Hai24} on renormalization in the presence of variance blowup, which was also studied in \cite{Par-} using a different moment-based approach. Those results were in $d=1$, and we now focus on $d=2$ which is more subtle.

Fix some $p\in\mathbb N$. Consider the It\^o stochastic PDE 
    \begin{equation}\label{dshe0}\tag{dSHE}\partial_t Z^\epsilon(t,y) = \tfrac12 \Delta_y Z^\epsilon(t,y) + \sqrt{\beta_\epsilon}  \epsilon^{p} \cdot Z^\epsilon(t,y) \cdot (-\Delta_y)^{p/2}\eta^\epsilon(t,y),\qquad t\ge 0, y\in \mathbb R^2,
    \end{equation}where $\eta^\epsilon:= \varphi_\epsilon *\eta$ for a standard Gaussian space-time white noise $\eta$ on $\mathbb R\times\mathbb R^2$ and where the convolution is in space only, and where $  \varphi_\epsilon (y):= \epsilon^{-2} \varphi(\epsilon^{-1}y)$ for some fixed nonnegative even continuous $\varphi \in C_c^\infty(\mathbb R^2)$ with $\int_{\mathbb R^2}\varphi=1$.

    Note that this stochastic PDE is similar to Example 2, but unlike that example, there is a fractional Laplacian around the noise term, which creates additional significant complications. The papers \cite{Hai24, Par-} considered the $d=1$ version of this equation and prove convergence to the KPZ equation with $\beta_\epsilon=\epsilon^{-1/2}$. We will see that the situation is quite different in $d=2$. 

    Set $\epsilon= \epsilon_N=N^{-1/2}$. Let $H^N_{s,t}(x,y):=\epsilon^2 \cdot Z_{\epsilon^2 s,\epsilon^2 t}^\epsilon(\epsilon x,\epsilon  y)$ where $Z^\epsilon_{s,t}$ are the propagators for \eqref{dshe0}. Then the question again becomes that if one tunes $\beta_N$ according to some precise asymptotic similar to \eqref{crit1}--\eqref{crit2}, is it possible to obtain the SHF as a limit? We will answer this question in Section \ref{5.3}.

\subsection{Proof ideas and open questions}

We prove Theorems \ref{mr} and \ref{mr2} using Tsai's axiomatic characterization of the stochastic heat flow \cite{Tsa24}, which is described in Definition \ref{def:axioms}. The bulk of the work goes into proving that the limiting field satisfies Axiom \eqref{item:moments} which states that the first four moments of the random field in question match with the first four moments of the stochastic heat flow. To study the second, third and fourth moments, we use the approach developed in \cite{MR4017119, CSZ_2d} to interpret the second moment as the \emph{renewal function} for a certain space-time renewal process. 

The main technical challenge in this paper is as follows.
In \cite{CSZ_2d}, the $h$\textsuperscript{th} moment is described in terms of the interactions between $h$ independent random walks on $\Z^2$. In their work, the Markov chain $\mathbf P_{N,h}$ in Assumption \ref{a1} \eqref{a1markov} is given by the law of $h$ independent random walks and does not actually depend on $N$.
By contrast, this work deals with the more general Markov chains $\mathbf P_{N,h}$. In Section \ref{sec:renewal} we study generalized renewal processes defined in terms of the transition probabilities of $\mathbf P_{N,2}$ and show that these renewal processes have the same asymptotic behavior as the ones studied in \cite{MR4017119, CSZ_2d}.
The approach to analyzing the fourth moment is then to express the higher moments in terms of second moment terms which encode pairwise interactions between particles. Here, additional challenges arise from the fact that we cannot just decompose the Markov chain $\mathbf P_{N,4}$ into something of the form $\mathbf P_{N,2} \otimes \mathbf P_{N,1}^{\otimes 2}$.  Instead, we need to carefully control the interactions between all particles. This is the subject of Section \ref{sec:highermoments}. To address this, we introduce a decomposition of our process in terms of certain stopping times that dictate exactly when the \emph{collision pattern} (a partition of the set $\{1,2,3,4\}$ representing which particles are interacting with one another at a given time) of the system changes.

Next, we discuss open problems. While our result applies to a broad class of \textit{linear stochastic flows} in the sense that the Chapman-Kolmogorov identity holds in the prelimit, it is not the case that every model that is conjectured to converge to the SHF satisfies such a linearity property. Of particular interest is the \textit{nonlinear multiplicative stochastic heat equation} given by $$\partial_t u(t,x) = \tfrac12 \Delta u(t,x) + \beta\cdot \sigma(u(t,x))\xi(t,x) ,\quad\quad t \ge 0, x\in \mathbb R^2,$$ where $\xi$ is a Gaussian noise that is white in time but smooth in space, and $\sigma:\mathbb R_+\to\mathbb R_+$ is smooth, nondecreasing, and globally Lipschitz with $\sigma(0)=0$ and $\sigma'(0)=1$. It is clear that this equation does not give rise to a model satisfying Chapman-Kolmogorov, and in fact it is not always true that one can obtain a four-parameter field from this model in the first place. Nonetheless, this stochastic PDE still defines a nonlinear stochastic flow $\Phi^\beta_{s,t}$ on the space of the initial data, satisfying $\Phi^\beta_{t,u}\circ \Phi^\beta_{s,t} = \Phi^\beta_{s,u}$ for $s<t<u$, and therefore it is still sensible to ask if this entire stochastic flow $(\Phi^\beta_{s,t})_{s<t}$ converges (in some suitable topology) to the SHF $(Z_{s,t}^\vartheta)_{s<t}$ under some logarithmic tuning of $\beta$ and diffusive scaling of space-time. In the \textit{subcritical regime}, strong indications in the direction of SHF-type fluctuations have been done by \cite{Tao, DG25}.


Finally we remark on the condition that ${\boldsymbol{\zeta}}_N \ge 0$ in Assumption \ref{a1} \eqref{a1markov} and likewise in Assumption  \ref{a3} \eqref{a3markov}, which may seem overly strict. There are easy examples of models for which ${\boldsymbol{\zeta}}_N$ does take negative values. For example, consider a Gaussian polymer whose driving noise has a covariance kernel that is not nonnegative everywhere (but whose integral over $\mathbb R^2$ is strictly positive). Another generalization could be removing the compact support assumption on ${\boldsymbol{\zeta}}_N$, which corresponds to models driven by noise with long-range spatial correlations. Our proof techniques cannot easily deal with such examples, as we heavily use the positivity and uniform compact support of $\z_N$. Thus we make the following conjecture.

\begin{conj}\label{mr3}
    The conclusions of Theorems \ref{mr} and \ref{mr2} remain true if we only assume that $|{\boldsymbol{\zeta}}_N(x)| \leq Ce^{-\alpha |x|}$ for some $C,\alpha>0$, instead of the stronger conditions above (${\boldsymbol{\zeta}}_N\ge 0,$ and ${\boldsymbol{\zeta}}_N$ has uniformly compact support).
\end{conj}

In this case, the operators from Definitions \ref{def:lambda} and \ref{def:lambdacont} still make sense on $L^2(\mathbb Z^2)$ or $L^2(\mathbb R^2),$ and we conjecture that the critical-tuning coefficients $\Lambda_N({\boldsymbol{\zeta}})$ still take the same form. The recent work of \cite{sur26} provides strong evidence of this conjecture, as the author studies a model for which $\z_N$ is not required to be compactly supported. In that work, even exponential decay is not necessary, and the truly optimal condition could be much weaker than that, see \cite[(1.3)]{sur26}.

\medskip
\noindent\textbf{Outline.} In Section \ref{sec:renewal}, we study the second moment of the field $\mathfrak H^N$ from Definition \ref{a2}. In Section \ref{sec:highermoments}, we study the higher moments of the field $\mathfrak H^N.$ In Section \ref{sec:mainResults}, we prove the main results (Theorems \ref{mr}--\ref{check} and Theorems \ref{mr2}--\ref{check2}). In Section \ref{sec:examples}, we introduce several models that satisfy the hypotheses (Assumptions \ref{a1} or \ref{a2}), and prove that there exists a scaling of parameters in the model with which one obtains the SHF as a continuum limit.

\medskip
\noindent\textbf{Notation.} Several distinct probability measures will be considered throughout the paper. The first is $\PN{h}^\x$, which denotes the measure on the path space $\big((\mathbb Z^2)^h\big)^{\mathbb Z_{\ge 0}}$ given by the Markov chain from Assumption \ref{a1} \eqref{a1markov}, started at state $\x\in (\mathbb Z^2)^h$. The symbol $\P$ denotes the measure on the underlying probability space containing the stochastic flows $H^N$ appearing in Assumption \ref{a1} \eqref{a1flow}, which represent the underlying environment. The unadorned measure $P$ and its tilted counterpart $\tilde P$ will be used for the randomness coming from the renewal variables $T^{x\to y}_{(N,j)}$ and $A^{x\to y}_{(N,j)}$, defined in Section \ref{s:2.1}. These three sources of randomness are distinct, and should not be confused.

Script letters such as $\mathscr Q_N$ and $\mathscr U_N$ will always denote operators on Banach spaces that encode the multiplicative functionals $e^{\boldsymbol{\eta}_{N,h}}$ and the Markov chains $\PN{h}^\x$ appearing in Assumptions \ref{a1} and \ref{a3}.
The standard heat kernel on $(\mathbb R^2)^h$ will be denoted in this paper by $\g^{\otimes h}(t,\x)= (2\pi t)^{-h} \exp(-|\x|^2/2t)$, where $t>0$ and $\x\in (\mathbb R^2)^h.$
The special function $G_\vartheta(t) := \int_0^\infty e^{(\vartheta -\gamma)s} \frac{st^{s-1}}{\Gamma(s+1)} ds$ for $t\in [0,1]$ and $\vartheta\in\mathbb R$ will also appear in many important formulas throughout the paper. 

We use Landau's asymptotic notation: by $f(n) = O(g(n))$ we mean there exists a constant $C > 0$ such that $|f(n)| \leq Cg(n)$ for all $n \in \N$, and by $f(n) = o(g(n))$ we mean $f(n)/g(n) \to 0$ as $n \to \infty$. We clarify dependencies of the implicit constant explicitly in the text.


Many of the random processes $z \mapsto X_z$ we consider, and functions thereof, are indexed by the integers or integer lattice.
When we take scaling limits, it will be convenient to write expressions such as $X_{tz}$ for $t \in \R$. These should be understood as $X_{\floor{tz}}$, where $\floor{tz}$ is the integer point nearest and no greater than $tz$ in any coordinate.

\medskip
\noindent\textbf{Acknowledgements.}
We thank Francesco Caravenna, Ivan Corwin, Yu Gu, Jeremy Quastel, Rongfeng Sun, Li-Cheng Tsai, and Nikos Zygouras for instructive discussions that helped to guide this project. We also thank Alex Dunlap, Simon Gabriel, Carl Mueller, Sudheesh Surendranath, and Jaeyun Yi for helpful conversations about the stochastic heat flow. 
This material is based upon work supported by the National Science Foundation under Grant No.\ DMS-2424139, while HD was in residence at the Simons Laufer Mathematical Sciences Institute in Berkeley, California during the Fall 2025 semester.
JH was partially supported by the National Science Foundation Graduate
Research Fellowship under Grant No.\ DGE-2234660.
HD and SP acknowledge the support of the Banff International Research Station (BIRS) for providing us with the opportunity for initial discussions about the project during the workshop ``Emerging Synergies between Stochastic Analysis and Statistical Mechanics'' in October 2025. HD, JH, and SP also acknowledge the support of the Simons Center for Geometry and Physics (SCGP) at Stony Brook University for hosting the authors for three weeks during the workshop  ``Algebraic Methods in Probability'' in July 2026.

\section{The renewal structure and second moment} \label{sec:renewal}

Here we introduce a renewal structure that will be used throughout the remainder of the paper. We will assume that Assumption \ref{a1} is satisfied throughout the discussion, so that we are in the discrete-time case. 

Throughout all of the proofs in Sections \ref{sec:renewal}--\ref{sec:mainResults}, we are going to assume that $d_N\equiv 0$ in Assumption \ref{a1} \eqref{a1invarianceprinciple}, because it adds unnecessary notation without illuminating any of the central ideas. It should be implicitly understood and that in the case of general $d_N\ne 0$, one actually needs to subtract $d_Nt$ from all coordinates of the diffusively rescaled Markov process $\mathbf P_{N,h}$ in the all of the appropriate places (and nothing else is actually impacted by the presence of $d_N$). 
\subsection{The renewal process}\label{s:2.1}

Fix $N \in \mathbb N$, and recall $\mathbf p_{N,\mathrm{dif}}$ from Assumption \ref{a1} Item \eqref{a1logasymptotic}. For distinct $j \in \N$ and $x,y \in \Z^2$, we will introduce independent random variables $T^{x\to y}_{(N,j)}$ defined by
\begin{equation}\label{s_n}
    P(T_{(N,j)}^{x \to y} = r) = \frac{\pN{\mathrm{dif}}(r;x,y) \ind_{\{1 \le r \le N\}}}{\mathscr S_N(x,y)}, \quad\quad
    \mathscr S_N(x,y) := \sum_{s=1}^N \pN{\mathrm{dif}}(s;x,y).
\end{equation}
Let $f \ge 0$ be a finitely supported function on $\mathbb Z^2$. Then we may define the quantity
\begin{align}\label{eq:timerenewal}
   \notag\mathcal E_{x_0}^f(N):&= \sum_{k=1}^\infty \sum_{0< r_1< \cdots <r_k \leq N} \EN{\mathrm{dif}}^{x_0} [ f(Z_{r_1}) \cdots f(Z_{r_k})] \\&=\notag  \sum_{k=1}^\infty \sum_{0< r_1< \cdots <r_k \leq N} \sum_{x_1,\ldots,x_k\in\mathbb Z^2}\prod_{j=1}^k f(x_j)\pN{\mathrm{dif}} (r_j-r_{j-1} , x_{j-1},x_j) \\&= \sum_{k=1}^\infty  \sum_{x_1,\ldots,x_k\in\mathbb Z^2}\prod_{j=1}^k f(x_j) \mathscr S_N(x_{j-1},x_j) \cdot P\big( T_{(N,1)}^{x_0\to x_1} + \cdots + T_{(N,k)}^{x_{k-1}\to x_k} \le N\big),
\end{align}
where $(Z_r)_r$ is a realization of the Markov process $\PN{\mathrm{dif}}^{x_0}$ from Assumption \ref{a1} \eqref{a1logasymptotic} and $\EN{\mathrm{dif}}^{x_0}$ is the associated expectation operator.

One should think of $\mathcal E_{x_0}^f(N)$ as a type of renewal function, as in \cite{MR4017119,CSZ_2d}. In Section \ref{sss:spacetimerenewal}, we will define a more general renewal function which depends on both the space and time variables, and therefore supersedes this one. Nevertheless, it will be instructive to analyze this simpler case first.

By Assumption \ref{a1} \eqref{a1logasymptotic}, we know that $\mathscr S_N(x_{j-1},x_j)$ behaves like $\nu_\infty(x_j)\log N + o(\log N)$ uniformly over all $x_j$ in a finite set, but it turns out that we cannot make this replacement if we want to calculate the precise asymptotic of \eqref{eq:timerenewal} as $N\to \infty$: in other words, we need the lower-order behavior of the Green's function, which is related to the sequence $\Lambda_N$ from Theorem \ref{a1}. However, the following trick allows us to avoid making any assumptions on this lower-order behavior. Write \eqref{eq:timerenewal} as 
\begin{align}\label{eq0}
    \notag \mathcal E^f_{x_0}(N)
    &= \sum_{k=1}^\infty (\log N)^k \sum_{x_1,\ldots,x_k\in\mathbb Z^2}  \prod_{j=1}^k f(x_j) \nu_\infty(x_j) \frac{\mathscr S_N(x_{j-1},x_j)}{\nu_\infty(x_j)\log N} P( T_{(N,1)}^{x_0\to x_1} + \cdots+ T_{(N,k)}^{x_{k-1}\to x_k} \le N) \\
    &= \sum_{k=1}^\infty  (\log N)^k \big( \nu_\infty(f) \big)^k E \Bigg[ \prod_{j=1}^k \frac{\mathscr S_N(X_{j-1},X_j)}{\nu_\infty(X_j)\log N} \ind_{\big\{ T_{(N,1)}^{x_0\to X_1} + \cdots + T_{(N,k)}^{X_{k-1} \to X_k} \le N\big\}}\Bigg],
\end{align}
where under the measure $E$, the random variables $X_j$ are sampled i.i.d.\ (and independently of the renewal variables $T^{x\to y}_{(N,j)}$) from the distribution
\begin{equation}\label{mu_f}
    \mu(x) = \frac{f(x) \nu_\infty(x)}{\nu_\infty(f)}
\end{equation}
on $\mathbb Z^2$, which has a finite support since $f$ does.
Note that the expectation $E$ is now also taken with respect to this additional marginal, and this should be understood of the measure $P$ going forward.
Next, we split the expectation as 
\begin{align}
    \notag E \Bigg[ \prod_{j=1}^k &\frac{\mathscr S_N(X_{j-1},X_j)}{\nu_\infty(X_j)\log N} \ind_{\big\{ T_{(N,1)}^{x_0\to X_1} + \dots + T_{(N,k)}^{X_{k-1} \to X_k} \le N\big\}}\Bigg] \\
    &= E \Bigg[ \prod_{j=1}^k \frac{\mathscr S_N(X_{j-1},X_j)}{\nu_\infty(X_j)\log N} \Bigg] \cdot  \tilde P_{N,k}\Big( T_{(N,1)}^{x_0\to X_1} + \dots + T_{(N,k)}^{X_{k-1} \to X_k} \le N \Big), \label{eq1}
\end{align}
where the tilted measure $\tilde P_{N,k}$ is defined by the Radon-Nikodym derivative proportional to that product.\footnote{ Under the tilted measures $\tilde P_{N,k}$, it is still true that $(X_j)_{j=0}^k$ forms a time-inhomogeneous Markov chain with explicit transition kernel (see \eqref{inhom}). This is not so important in the discrete-time case of Theorem \ref{mr}, but it will be important for the continuous-time case of Theorem \ref{mr2}.} 
We then define the function 
\begin{equation}\label{h_n}
    \mathfrak h_N(x,y):=\log \Big( \tfrac{\mathscr S_N(x,y)}{\nu_\infty(y) \log N}\Big),
\end{equation}
so that if $X_0 = x_0$ (\ref{eq:timerenewal}) can be rewritten as 
\begin{equation}
    \mathcal E^f_{x_0}(N) = \sum_{k=1}^\infty \big(\log N \cdot \nu_\infty(f) \big)^k \cdot E \Big[ e^{\sum_{j=1}^k \mathfrak h_N(X_{j-1},X_j) } \Big] \cdot \tilde P_{N,k}\Big( T_{(N,1)}^{x_0\to X_1} + \cdots + T_{(N,k)}^{X_{k-1} \to X_k} \le N\Big).\label{un}
\end{equation}
An analysis of the middle term now clarifies the role of the sequence $\Lambda_N$ from \eqref{crit_tuning}.

\subsection{The critical tuning coefficients $\Lambda_N(f)$}\label{s:2.2}

\begin{lem}\label{2.1}
    Let $f:\mathbb Z^2 \to [0,\infty)$ be finitely supported, let $X_j$ be sampled i.i.d.\ from $\mu$ as in \eqref{mu_f}, and let $\mathfrak h_N$ be as in \eqref{h_n}.
    There exists a sequence of real numbers $\lambda_N(f)$ converging to 1 as $N\to\infty$, such that
    \begin{equation*}
        \lim_{N\to\infty} \sup_{k\in\mathbb N} \bigg| k\log \lambda_N(f) - \log E \Big[ e^{\sum_{j=1}^k \mathfrak h_N(X_{j-1},X_j) } \Big] \bigg| =0.
    \end{equation*}
    This sequence $\lambda_N(f)$ is equal to $\Lambda_N(f)/\nu_\infty(f)$, with $\Lambda_N(f)$ as in Definition \ref{def:lambda}. Furthermore, the limit is uniform over all $f$ such that the support of $f$ is contained in some fixed finite set $J\subset \mathbb Z^2$.
\end{lem}

Note that the dependence of $\lambda_N$ on $f$ comes from the law $\mu$ of $X_j$. So far we have assumed $X_0 = x_0$, but the proof below shows that the initial value is irrelevant in the construction of $\lambda_N$ in the sense that changing the initial value of $(X_j)_j$, or in fact any finite number of values, gives the same sequence $\lambda_N$ in the construction: that is, $\lambda_N$ is not sensitive to finite perturbations of the i.i.d.\ sequence $X_1,X_2,\ldots$. Thus we may assume without loss of generality that $X_0$ is sampled from the same law $\mu$. 

\begin{proof}[Proof of Lemma \ref{2.1}] For completeness, we give two separate descriptions of these important coefficients $\lambda_N$: one as an exponentiated sum of joint cumulants and the other as the principal eigenvalue of some square matrix related to a specific Markov chain.

\medskip
    \noindent\textbf{Proof 1.} Abbreviate $\eta_j^N:=\mathfrak h_N(X_{j-1},X_j), $ thus $\eta^N$ is a stationary sequence with finite-range (in fact range-one) dependence. Note that $\delta_N:=\norm{\eta_1^N}_{L^\infty(P)} \to 0.$ This follows from the fact that $\mathfrak h_N(x,y)\to 0$ for $x,y \in \supp f$, by Assumption \ref{a1} \eqref{a1logasymptotic}.
    
    Expand the cumulant generating function
    \begin{equation*}
        \log E \Big[e^{\sum_{j=1}^k \eta_k^N} \Big] = \sum_{\ell=1}^\infty \frac{\kappa_\ell(\eta^N_1+ \cdots +\eta^N_k)}{\ell!},
    \end{equation*}
    where $\kappa_\ell$ denotes the cumulant of order $\ell$. It is not at all obvious that the right side defines a convergent series in $\ell$, and one in fact needs to exploit nontrivial cancellations. For this, we use Theorem 9.11 of \cite{KLM} to prove the stronger bound $|\kappa_\ell (\eta^N_1+ \cdots +\eta^N_k)| \leq k\cdot (2e)^\ell \cdot \ell! \delta_N^\ell$ uniformly over $N,k,\ell$ or the law of the $X_i$, as well as the estimate
    \begin{equation*}
        \sup_k\Big| \kappa_\ell (\eta^N_1+ \cdots +\eta^N_k) - k c_{\ell,N}^*\Big| \leq C^\ell \delta_N^\ell \ell!, \quad\quad c_{\ell,N}^*:= \sum_{t_1, \cdots ,t_{\ell-1}\in\mathbb Z} \kappa_{(\ell)} ( \eta^N_0 , \eta^N_{t_1} ,\ldots, \eta_{t_{\ell-1}}^N),
    \end{equation*}
    for some absolute constant $C$ not depending on $N,k,\ell$. Here $\kappa_{(\ell)}$ denotes the joint cumulant of order $\ell$, and we have applied stationarity of $\eta^N$ after using the multilinearity of joint cumulants to write $\kappa_\ell(\sum_{t=1}^k \eta^N_t) = \sum_{t_1,\ldots,t_\ell=1}^k \kappa_{(\ell)} (\eta^N_{t_1},\ldots,\eta^N_{t_\ell})$. We thus define
    \begin{equation*}
        \lambda_N(f):= \exp\bigg(\sum_{\ell=1}^\infty \frac{c_{\ell,N}^*}{\ell!}\bigg).
    \end{equation*}
    The series converges again because of Theorem 9.11 of \cite{KLM}, which shows that $|c_{\ell,N}^*|\leq A^\ell \cdot \ell!\delta_N^\ell $ for the same $A > 0$, and therefore that $\lambda_N\to 1$ (since $\delta_N\to 0$). Thus we obtain 
    \begin{align*}
        \sup_k\bigg| k\log \lambda_N(f) - \log E \Big[ e^{\sum_{j=1}^k \mathfrak h_N(X_{j-1},X_j) } \Big]\bigg| \leq \sum_{\ell=1}^\infty \frac{\sup_k \big| \kappa_\ell (\eta^N_1+ \cdots +\eta^N_k) - k c_{\ell,N}^*\big| }{\ell!}& \leq \sum_{\ell=1}^\infty C^\ell \delta_N^\ell\\& \leq C'\delta_N\to 0.
    \end{align*}
    
    \medskip
    \noindent\textbf{Proof 2.}
    The following is a longer but more elementary argument using only finite-state Markov chains, inspired by observations made in \cite{Meyn}. Define $g_N(x):= E\big[e^{\mathfrak h_N(x,X_1)}\big]$. Then we can use the Radon-Nikodym derivative $F:= e^{\sum_{j=1}^k \mathfrak h_N(X_{j-1},X_j) -\log g_N(X_{j-1})}$ to tilt the law of the process $(X_j)_j$,
    in order to rewrite the expectation of interest as 
    \begin{equation*}
        E \Big[e^{\sum_{j=1}^k \mathfrak h_N(X_{j-1},X_j) } \Big] = E_N^{MC} \Bigg[  \prod_{j=0}^{k-1} g_N(X_{j}) \Bigg] .
    \end{equation*}
    Since the $(X_j)_j$ were originally i.i.d.\ distributed according to $\mu$, we see that under this new measure $P_N^{MC}$ they form a Markov chain with transition density
    \begin{equation*}
        b_N(x,y):=\frac{e^{\mathfrak h_N(x,y)}f(y)\nu_\infty(y)}{\sum_{z\in J}  e^{\mathfrak h_N(x,z)}f(z)\nu_\infty(z)} = \frac{\mathscr S_N(x,y) f(y)}{\sum_{z\in J} \mathscr S_N(x,z)f(z)},
    \end{equation*}
 where $J:=\mathrm{supp}(f).$ 
 
 Let $B_Na(x):=\sum_y b_N(x,y)a(y)$ be the Markov operator associated to $b_N$. Further define $Q_N$ as the linear map on $\mathbb R^J$ given by $Q_N a(x) :=  g_N(x) B_Na(x)$. Then we can rewrite 
 \begin{align*}
     E_N^{MC} \Bigg[  \prod_{j=0}^{k-1} g_N(X_{j}) \Bigg] = (Q_N^{k} \ind_J) (X_0).
 \end{align*}
 Note that $\nu_\infty(f)\cdot Q_N$ is the same as the operator $T_N^f$ defined after \eqref{crit_tuning}.
 
 Using the assumption that $\mathscr S_N(x,y) = \nu_\infty(y)\log(N) + o(\log N)$, we have for all $x\in 
 J$ that $\lim_{N \to \infty} b_N(x,y) = \frac{\nu_\infty(y) f(y)}{\sum_{z \in J}\nu_\infty(z) f(z) }$ and $\lim_{N\to \infty} g_N(x)=1.$ This means that the operators $Q_N$ converge in operator norm to the rank-one linear map $Q_\infty a (x) = \frac{\nu_\infty(af)}{\nu_\infty(f)}$. Thus $Q_\infty$ has a single positive eigenvalue $1,$ associated with the eigenfunction $\ind_J$, whereas the other $|J|-1$ eigenvalues of $Q_\infty$ are all zero. Since the finite list of eigenvalues is continuous as a function of the square matrix, This means that for $N$ sufficiently large, $Q_N$ has a real eigenvalue $\lambda_N$ converging to $1$ as $N\to\infty$, while the other $|J|-1$ (possibly complex) eigenvalues converge to zero. Indeed, the largest eigenvalue must be real, since complex ones would come in conjugate pairs, and we know there cannot be two separate eigenvalues both converging to 1 (or else $Q_N$ would not converge to $Q_\infty$ in operator norm).
    
    In summary we can write $\ind_J= \psi_N + w_N$, where $\psi_N$ is the principal eigenfunction of $Q_N$ associated to $\lambda_N$ and $w_N$ lies in some $(|J|-1)$-dimensional subspace associated to the other eigenvalues. We can ensure that $\psi_N$ converges uniformly to $\ind_J$, since the latter is the principal eigenfunction for $Q_\infty$. Thus $w_N\to 0$ in $\mathbb R^J$ and in turn $|Q_N^k w_N| \leq c_Ne^{-\alpha_N k}$, with $c_N\downarrow 0$ and $\alpha_N\uparrow \infty$ as $N\to \infty$, and for fixed $N$ this bound is uniform over all $k\in \mathbb N$. All of this shows that $|Q_N^k \ind_J(x) - \lambda_N^k \psi_N(x)| \leq c_N e^{-\alpha_N k}$, that is,
    \begin{equation*}
        \bigg| E \Big[e^{\sum_{j=1}^k \mathfrak h_N(X_{j-1},X_j) } \Big] - \lambda_N^k \psi_N(X_0)\bigg| \leq c_Ne^{-k \alpha_N}
    \end{equation*}
    uniformly over all $N,k$. 
    Thus $\lambda_N^{-k} E \Big[e^{\sum_{j=1}^k \mathfrak h_N(X_{j-1},X_j) }  \Big]$ will converge to $\psi_N(X_0)$ as $k\to\infty$, for any fixed value of $N$ large enough so that $\alpha_N> |\log \lambda_N|$.
    Now since $|\log a - \log b| \leq 2|a-b|$ for $a,b\ge \frac12$, we get
    \begin{equation*}
        \bigg|\log\Big (\lambda_N^{-k} E \Big[e^{\sum_{j=1}^k \mathfrak h_N(X_{j-1},X_j) }  \Big] \Big)- \log \psi_N(X_0) \bigg| \le 2  c_N\lambda_N^k e^{-k\alpha_N} \leq 2 c_Ne^{-\frac12 k\alpha_N}
    \end{equation*}
    uniformly over $k\in \mathbb N$, for $N$ sufficiently large (where in the last bound we are using that $\alpha_N\to\infty,$ so that $\alpha_N>2 \log \lambda_N$ eventually). Consequently we obtain \begin{equation*}
        \sup_k \bigg|k\log \lambda_N - \log E \Big[e^{\sum_{j=1}^k \mathfrak h_N(X_{j-1},X_j) } \Big] \bigg| \leq |\log \psi_N(X_0)| + 2c_N.
    \end{equation*}
    Sending $c_N\downarrow 0$ and $\psi_N\to 1$ proves the claim.\footnote{
        This proof implicitly assumes $x_0\in J$. If this is not the case, one may slightly modify the definition of $J$ to $J\cup\{x_0\}$: clearly this will not change $\Lambda_N(f)$, since the Markov kernel $b_N$ is still supported on the original $J$.
    }
\end{proof}


\begin{remark} The above lemma will be most important in the case where $f=e^{\boldsymbol{\zeta}}-1$, the particular function appearing in Item \eqref{a1markov} of Assumptions \ref{a1} and \ref{a3}. For many interesting models, $\lambda_N(e^{\boldsymbol{\zeta}}-1)$ above will typically be of the form $1+\frac{q}{\log N} + o\big(\frac1{\log N})$ for some interesting constant $q \in \mathbb R$, see for example \eqref{crit1} and \eqref{crit2}. 
Lemma \ref{2.1} shows that this sequence $\lambda_N$ exists under very general assumptions, though of course it need not behave like $1+O\big(\frac{1}{\log N}\big)$ in general. That is because Item \eqref{a1logasymptotic} of Assumption \ref{a1} imposes no precise lower-order asymptotic beyond the logarithmic growth.
\end{remark}

Combining \eqref{un} and Lemma \ref{2.1}, and recalling from Lemma \ref{2.1} that $\nu_\infty(f) \lambda_N(f) = \Lambda_N(f)$, we obtain the following representation of $\mathcal E^f_{x_0}$.

\begin{thm}\label{2.3}
    We have 
    \begin{align*}
        \mathcal E_{x_0}^f(N)
        &=\sum_{k=1}^\infty \sum_{0< r_1< \cdots <r_k \le N} \EN{\mathrm{dif}}^{x_0} [ f(X_{r_1}) \cdots f(X_{r_k})] \\
        &= \sum_{k=1}^\infty \big(\log N \cdot \Lambda_N(f) \big)^k e^{\e_N^f(k)} \tilde P_{N,k}\Big( T_{(N,1)}^{x_0\to X_1} + \cdots + T_{(N,k)}^{X_{k-1} \to X_k} \le N\Big),
    \end{align*}
    where the error terms $\e_N^f(k)$ satisfies $\sup_{f:\mathrm{supp}(f) \subset J}\sup_{k\ge 1} |\e_N^f(k)| \to 0$ for any finite $J\subset \mathbb Z^2$.
    
    Letting $f = e^{\z_N} - 1$, and recalling from \eqref{crit_tuning} that $\Lambda_N(e^{\z_N} - 1) \log N = 1+\frac{\vartheta+o(1)}{\log N}$, we get
    \begin{equation*}
        \mathcal E_{x_0}^{f}(N)=\sum_{s\in (\log N)^{-1}\mathbb N}e^{\vartheta s + o(1)} \cdot \tilde P_{N,s\log N}\big( T_{(N,1)}^{x_0\to X_1} + \cdots + T_{(N,s\log N)}^{X_{s\log N-1} \to X_{s\log N}} \le N\big) ,
    \end{equation*}
    where the $\e_N^f(k)$ errors have been absorbed into the $o(1)$ term above.
\end{thm}

The final expression is suggestive of a Riemann sum with mesh $(\log N)^{-1}$, and this will be a crucial observation. More precisely, it resembles $\log N \cdot \int_0^\infty e^{\vartheta s} P(Y_s\le 1) ds$, where $(Y_s)_{s\ge 0}$ is the Dickman subordinator of Definition \ref{dickman} below. Theorem \ref{2.3} will be further strengthened by Proposition \ref{prop:RiemannSum} below, which will add a spatial component to the identity given above.

\subsection{Space-time version of the renewal process}\label{sss:spacetimerenewal}
Next, we generalize the previous section by studying a space-time version of the above result. Fix $N\in \mathbb N$. For distinct $j \in \N$ and $x,y \in \Z^2$, we again define independent random variables $(T^{x \to y}_{(N, j)}, A_{(N, j)}^{x \to y})\in \{1,\ldots,N\}\times \mathbb Z^2$ with distribution
\begin{equation}\label{(t,a)}
    P(T^{x \to y}_{(N, j)} = r, A_{(N, j)}^{x \to y}=a) =\frac{\mathbf p_{N,2} \big(r ; (x,0) , (a,a-y)\big)\ind_{\{1\le r\le N\}}}{\mathscr S_N(x,y)}.
\end{equation}
See that by summing over $a\in \mathbb Z^2$ we recover the same marginal law for $T^{x\to y}_{(N,j)}$ defined in \eqref{s_n}.
Here $\pN{2}$ is the transition density for the Markov chain $\PN{2}$ from Assumption \ref{a1} \eqref{a1markov}. $x$ and $y$ represent the \emph{gaps} between the two particles in this system.
For any deterministic sequence $\x = (x_j)_j\subset \mathbb Z^2$ of such gaps, we then define
\begin{equation*}
    (\tau_{(N, k)}^{\x}, S_{(N,k)}^{\mathbf x}):= \sum_{j=1}^k (T^{x_{j-1} \to x_j}_{(N,j)} , A^{x_{j-1} \to x_j}_{(N,j)}).
\end{equation*}

Throughout this section we will assume that $t\in \mathbb N$ with $t\le N$. Then we can define the following renewal function, 
which captures the energy of two particles which travel from positions $x,x'$ to $y, y'$ in time $t > 0$ while interacting through the kernel $f:\mathbb Z^2 \to \mathbb [0,\infty)$ when they are nearby.
\begin{defn}\label{def:renewal}
    Let $f:\mathbb Z^2\to\mathbb R$ be finitely supported. For $t\in \mathbb Z_{\ge 0} , x,x',y,y'\in\mathbb Z^2$, define the \textbf{renewal function}
    \begin{equation}\label{eq:renewal}
        U_N(f; t,x,x',y,y') = \sum_{k=0}^\infty \sum_{0 < r_1< \cdots<r_k < t} \mathbf E^{(x,x')}_{N,2} \Bigg[\prod_{j=1}^k f(R^1_{r_j}-R^2_{r_j}) \ind_{\{ R^1_{t}= y, R^2_{t}=y'\}}\Bigg],
    \end{equation}
    with the convention that $U_N(f;0,x,x',y,y') = \ind_{\{ y=x, y'=x'\}}$ and that the $k=0$ term of the above series is $E^{(x,x')}_{N,2}[ \ind_{\{ R^1_{t}= y, R^2_{t}=y'\}}]$. Here $(R^1_r,R^2_r)_r$ is a realization of the Markov chain $\mathbf P_{N,2}$.
\end{defn}
So, in the notation of Section \ref{s:2.1}, we have the identity $\mathcal E_{x_0}^f(N)=\sum_{y,y'}U_N(f;N;0,x_0,y,y') - 1$. Setting $f=e^{\boldsymbol{\zeta}_N}-1$, we obtain the following.
\begin{prop} \label{prop:RiemannSum}
    Let $h:\mathbb Z^2 \to \mathbb R$ be bounded. With $U_N$ as in Definition \ref{def:renewal}
    , we have
    \begin{align*}
        \sum_{y' \in \mathbb Z^2} U_N(e^{\boldsymbol{\zeta}_N}-1;& t; x, x', y, y')(e^{\z_N(y-y')} -1)h(y')\\
        &= \sum_{s\in (\log N)^{-1}\mathbb N} e^{\vartheta s + o(1)} \cdot \widetilde E_{N,s\log N} \bigg[ h(y- X_{s \log N}) \ind_{\big\{ \tau^{\mathbf X}_{(N, s\log N)} = t, S^{\mathbf X}_{(N, s \log N)} = y-x\big\}}\bigg] 
    \end{align*}
    with the $X_j$ sampled from the tilted measure $\tilde P_{N,k}$ as in \eqref{eq1}.
\end{prop}

\begin{proof}
With $r_{k+1} = t,\, r_0 = 0,\, x_0 = x- x',\, x_{k+1} = y - y'$, and $a_0 = x', a_{k+1} = y$, we expand $U_N$ as
\begin{align}
    U_N(&f;t,x,x',y,y')\notag \\
    &= 
        \sum_{k=0}^\infty  \sum_{0 <  r_1 < \cdots < r_k <  t} \sum_{\substack{x_1,\ldots,x_k\in\mathbb Z^2\\ a_1,\ldots,a_k\in \mathbb Z^2}}\prod_{j=1}^{k} f(x_j) \prod_{j=1}^{k+1} \mathbf p_{N,2}(r_j-r_{j-1} ; (a_{j-1}+x_{j-1}, a_{j-1}) , (a_j,a_j-x_j) ). \label{trigger}
\end{align}
Note that in \eqref{eq:renewal} we multiply by factors of $f(R^1_{r_j} - R^2_{r_j})$ only for times $r_j < t$. Sometimes we will need to tack on an extra factor of $f(R^1_{t} - R^2_{t})$ at the terminal time $t$. We will also often want to sum out the terminal gap $y'-y$ at time $t$.
This leads us to study sums of the following form.
If $0<t<N$ and $h: \mathbb Z^2 \to \mathbb R$ then for any fixed $y\in \mathbb Z^2$ we have 
\begin{align}\label{h_identity}
    \notag \sum_{y' \in \Z^2} &U_N(f; t,x,x',y,y') f(y-y') h(y') \\
    &= \notag \sum_{x_{k+1} \in \Z^2} U_N(f; t,x,x',y,y - x_{k+1}) f(x_{k+1}) h(y - x_{k+1}) \\&= \notag\sum_{k=0}^\infty  \sum_{\substack{x_1,\ldots,x_{k+1}\in\Z^2}}\prod_{j=1}^{k+1} f(x_j)\mathscr S_N(x_{j-1},x_j) P\big( \tau_{(N,k+1)}^{\x} = t, S_{(N,k+1)}^{\mathbf x} = y-x\big) h( y - x_{k+1}) \\
    &=  \sum_{k=1}^\infty  \sum_{\substack{x_1,\ldots,x_{k}\in \Z^2}}\prod_{j=1}^{k} f(x_j)\mathscr S_N(x_{j-1},x_j) P\big( \tau_{(N,k)}^{\x} = t, S_{(N,k)}^{\mathbf x} = y-x\big) h(y - x_{k}).
\end{align}
The second equality follows because $t < N$ by assumption, and also because for fixed $x_{j-1}$ and $x_j$ the quantity $\pN{2} (r_j - r_{j - 1},(a_{j-1}+x_{j-1}, a_{j-1}) , (a_j,a_j-x_j))$ is a function only of $(r_j-r_{j-1},a_j-a_{j-1})$, as the difference of the two coordinates in $\PN{2}$ is Markov by Assumption \ref{a1} \eqref{a1logasymptotic}. 

We can now use a similar trick as before to rewrite \eqref{h_identity} as a sum of tilted probability measures. Sample i.i.d.\ random variables $X_1,\ldots, X_k$ from the distribution on $\mathbb Z^2$ given by
\begin{equation*}
    \mu(x) = \frac{f(x) \nu_\infty(x)}{\sum_y f(y) \nu_\infty(y)}.
\end{equation*}
Let $X_0 = x-x'$. 
Then, using exactly the same manipulation as in \eqref{eq0}--\eqref{eq1}, we rewrite \eqref{h_identity} as

\begin{align}\label{eq:RiemannSum}
    \notag\sum_{k=1}^\infty (\log N)^k \bigg( \int &f\,d\nu_\infty \bigg)^k E \Bigg[ \prod_{j=1}^k \frac{\mathscr S_N(X_{j-1},X_j)}{\nu_\infty(X_j)\log N} h(y- X_k)\ind_{\big\{ \tau_{(N,k)}^{\mathbf X} = t, S_{(N,k)}^{\mathbf X} = y-x\big\}}\Bigg] \\
    &= \sum_{k=1}^\infty \big( \log N \cdot  \Lambda_N(f) \big)^{k} e^{\varepsilon_N(k)} \cdot \tilde E_{N,k}\bigg[ h(y- X_k) \ind_{\big\{ \tau^{\mathbf X}_{(N,k)} = t, S^{\mathbf X}_{(N,k)} = y-x\big\}}\bigg] .
\end{align}
Here $\e_N(k)$ is as in Theorem \ref{2.3}. 
Specializing to the case $f = e^{\boldsymbol{\zeta}_N}-1$ and recalling from \eqref{crit_tuning} that $\Lambda_N(e^{\boldsymbol{\zeta}_N}-1)\log N = 1+\frac{\vartheta+o(1)}{\log N}$, we obtain the result.
\end{proof}




\subsection{Convergence to Dickman subordinator}

In this section we study $N\to \infty$ limits of $U_N$. First, we will record some consequences of Assumption \ref{a1} used for computations in the body. Recall that $\mathbf g(s,x) = (2\pi s)^{-1} e^{-|x|^2/2s}$ denotes the standard heat kernel on $\R^2$.

\begin{lem}\label{hku}
    The estimate
    \begin{equation*}
        \pN{\mathrm{dif}}(r,x,y) \leq Cr^{-1} (1+r^{-1/2} |x-y|)^{-S}
    \end{equation*}
    holds for $S$ as large as desired (the constant $C$ may depend on $S$).
\end{lem}

\begin{proof}
    This is immediate from union-bounding over the $h=2$ case of Assumption \ref{a1} \eqref{a1heatkernel}.
\end{proof}

\begin{defn}\label{well_sep}
    We say that a sequence $(\mathbf x_N) \subset (\Z^{2})^h$ is \textit{well-separated} if $N^{-1/2}\mathbf x_N \to \mathbf x$ and $\mathbf x = (x_1,\ldots,x_h)$ satisfies $x_i \neq x_j$ for all $i < j$.
\end{defn}

\begin{lem}\label{flem}
    Let $(x^1_N,x^2_N)$ be any sequence in $(\mathbb Z^2)^2$ for which $N^{-1/2} (x^1_N,x^2_N)\to (x_1,x_2)$. Let $s_N,t_N\in N^{-1}\mathbb Z_{\ge 0}$ such that $s_N\to s$ and $t_N\to t$ with $0<s<t$. Then
    \begin{align*}\lim_{N\to\infty} \sum_{r=Ns_N}^{Nt_N}\;  &\sum_{a\in \mathbb Z^2} \mathbf p_{N,2} (r;x^1_N,x^2_N,a,a+y) \cdot h\bigg(\frac{r}{N},\frac{a}{\sqrt N}\bigg) \\& = 4 \pi \cdot  \nu_\infty(y) \int_s^t \mathbf g(2u,x_1-x_2) \int_{\mathbb R^2} \mathbf g(2u,z) h\big(u,\tfrac12 (z+x_1+x_2)\big) dz \,du
    \end{align*}
    for all $y \in \Z^2$ and smooth bounded $h: [0,\infty) \times \R^2 \to \R\).
    The statement extends to $s = 0$ provided at least one of the following two conditions holds:
    \begin{enumerate}
        \item\label{flem1} $\mathbf x_N$ is well-separated.
        \item\label{flem2} $h$ is smooth up to the boundary and $h(0,0) = 0$, and $|x^1_N-x^2_N|$ remains bounded as $N\to \infty$.
        \end{enumerate}
\end{lem}

\begin{remark}
    In the case $x_1=x_2$, the right hand side can be equivalently expressed as
    $$4\pi \cdot \nu_\infty(y) \int_s^t   \int_{\mathbb R^2} \mathbf g(u,z)^2 h(u,z+x_1)dzdu=\nu_\infty(y) \int_s^t u^{-1} \int_{\mathbb R^2} \mathbf g(\tfrac{u}2,z)h(u,z+x_1) dz\,du,$$
    which can be proved by using the identity $\g(\tfrac{s}{2},x-\tfrac12(a_1+a_2)) \g(2s, a_1-a_2) = \g(s, x- a_1)\g(s,x-a_2)$, for instance. 
    In particular, by taking $h(s,x)=\ell(s)$ we get 
    \begin{equation*}
        \lim_{N\to\infty} \sum_{r=Ns_N}^{Nt_N} \pN{\mathrm{dif}} (r;x,y) \ell(N^{-1}r)  = \nu_\infty(y) \int_s^t u^{-1} \ell(u) du,
    \end{equation*}
    valid for all $s>0$, and also for $s=0$ if $\ell(0)=0.$
\end{remark}

\begin{proof}[Proof of Lemma \ref{flem}]
    We break down the proof into two cases, the first case will be $s>0$, and the second case will be $s=0$.

    \smallskip
    \noindent\textbf{Case 1a.} By an abuse of notation, we will suppress the subscripts on $s_N,t_N$ and just write $\sum_{Ns}^{Nt}$ rather than $\sum_{Ns_N}^{Nt_N}$. This convention will be used henceforth. 
    
    First, we address the case where the limit $s > 0$, and $h$ is of the form $h(s,u) = \phi(u).$  Write $(\mathbf Z_s)_s$ for the Markov process associated to the transition density $\mathbf p_{N,2}$ started from $(x^1_N,x^2_N).$ By the Markov property we can write
    \begin{equation*}
        \sum_{r=Ns}^{Nt} \sum_{a\in \mathbb Z^2} \mathbf p_{N,2} (r;(x,x'),(a,a+y)) \phi\bigg(\frac{a}{\sqrt N}\bigg) = \mathbf E^{(x,x')}_{N,2} \Bigg[\sum_{r=0}^{N(t-s)} \sum_{a\in \mathbb Z^2} \mathbf p_{N,2} ( r; \mathbf Z_{Ns}, (a,a+y)) \phi\bigg(\frac{a}{\sqrt N}\bigg)\Bigg].
    \end{equation*}
    Using the invariance principle in Item \eqref{a1invarianceprinciple}, together with the second limit of Item \eqref{a1logasymptotic}, we see that as $N\to\infty$ we have convergence in distribution 
    \begin{align}\label{bo}
        \notag \sum_{r=0}^{N(t-s)} \sum_{a\in \mathbb Z^2}& \mathbf p_{N,2} ( r; \mathbf Z_{Ns}, (a,a+y)) \cdot \phi(N^{-1/2} a) \\
        &\stackrel{d}{\longrightarrow} 4\pi \cdot \nu_\infty(y) \int_0^{t-s} \mathbf g(2u,B^1_s-B^2_s) \int_{\mathbb R^2} \mathbf g(2u,z) \phi\big(\tfrac12 (z+B^1_s+B^2_s)\big)dz\,du,
    \end{align}
    where $(B^1,B^2)$ is standard Brownian motion in $(\mathbb R^2)^2$. 
    
    Henceforth write $\mathbf Z := (Z^1,Z^2)$ and $Z^{\mathrm{dif}} := Z^1-Z^2$.
    Note the trivial Green's function bound $\sum_{a\in \mathbb Z^2} \mathbf p_{N,2} (r,x,x',a,a+y) \phi(N^{-1/2} y) \leq \|\phi\|_{L^\infty} \cdot \pN{\mathrm{dif}}(r,x-x',y)$. By Lemma \ref{hku}, we also have the upper bound $\mathbf p_{N,\mathrm{dif}} ( r; x, y) \leq C r^{-1} (1+ r^{-1/2} |x-y|)^{-S},$ which gives a logarithmic singularity in a macroscopic neighborhood of the spatial origin:
    \begin{equation}\label{eq:2ptgreens}
        \sum_{r=0}^{N(t-s)} \mathbf p_{N,\mathrm{dif}} ( r; Z^{\mathrm{dif}}_{Ns}, y) \leq C (1+\log_- ( N^{-1/2} (1+|Z^{\mathrm{dif}}_{Ns}|) ),
    \end{equation}
    where $\log_-x = |\log x|\ind_{\{x<1\}}$. As another consequence of Lemma \ref{hku} we find that $$\sup_N \mathbf E^x_{N,\mathrm{dif}} [\log_- ^2( N^{-1/2} (1+|Z^{\mathrm{dif}}_{Ns}|) ) ]<\infty.$$ Thus uniform integrability holds in \eqref{bo} and the convergence of the expectations is true as well. It follows that 
    \begin{align*}
        \lim_{N\to \infty}&\sum_{r=Ns}^{Nt} \sum_{a\in \mathbb Z^2} \mathbf p_{N,2} (s;x^1_N,x^2_N,a,a+y) \phi(N^{-1/2}a) \\
        &= 4\pi \nu_\infty(y) \int_0^{t-s} \E \bigg[g(2u,x_1-x_2+B^1_s-B^2_s) \int_{\mathbb R^2} \mathbf g(2s,z) \phi(\tfrac12 (z+x_1-x_2+B^1_s+B^2_s)) dz \bigg]du \\
        &= 4 \pi  \nu_\infty(y) \int_s^t \mathbf g(2u,x_1-x_2) \bigg(\int_{\mathbb R^2} \mathbf g(2u,z) \phi\big(\tfrac12 (z+x_1+x_2)\big) dz \bigg)ds,
    \end{align*}
    where in the second line the $B^i$ are independent Brownian motions in $\mathbb R^2$, starting from the origin. The second equality follows from the semigroup properties of the heat kernel.

    \smallskip
    \noindent\textbf{Case 1b.} Next, we address the slightly more general case where $s > 0$ and $h$ is of the form $h(s,u) = \ell(s)\phi(u).$ Case 1a proves the result if $h$ is of the form $\ell(s)\phi(a)$ with $\ell$ piecewise-constant and $\phi$ compactly supported. For general smooth functions $\ell$, the claim then follows by using a simple rectangular approximation of $\ell$ by piecewise-constant functions. As the mesh goes to zero, the error is easily controlled using the smoothness of $\ell$ together with the heat kernel upper bound $\sup_{x,y}\mathbf p_{N,\mathrm{dif}} (s;x,y) \leq Cs^{-1}$ from Lemma \ref{hku}. This proves the claim for $h$ of the form $\ell(s)\phi(a)$ with $\phi$ compactly supported, and the result for all compactly supported two-variable functions $h(s,a)$ follows from a density argument: take linear combinations and apply Stone-Weierstrass. To extend from compactly supported $h$ to all smooth and bounded $h$, one uses the heat kernel bounds in Item \eqref{a1heatkernel} and truncates to approximate $h$ by a compactly supported sequence $h_N\to h$.

    \medskip
    Next, we address the case where $s=0$.

    \smallskip
    \noindent\textbf{Case 2a.} If $x_1\ne x_2$, then directly from Item \eqref{a1logasymptotic} and the above discussion we obtain the result for all $h$ of the form $\ell(s)\phi(a)$, with $\ell$ piecewise-constant and $\phi$ compactly supported. From this class of functions, the argument of the previous paragraph can be repeated to prove the claim.

    \smallskip
    \noindent\textbf{Case 2b.} If $h(0,0)=0$, with $h$ smooth up to the boundary, we have $|h(s,u)| \leq C(|s|+|u|)$. Recall the estimate $\sup_N \mathbf p_{N,2} (s;\x,\y) \leq Cs^{-2} ( 1+s^{-1/2} |\x-\y|)^{-S}$ from Item \eqref{a1heatkernel}. Since $y$ from the lemma statement is a fixed element of $\mathbb Z^2$, and since and $|x^1_N-x^2_N|$ remains bounded as $N\to \infty$, we can sharpen this to $\sup_N \mathbf p_{N,2} (s;x^1_N,x^2_N,a,a+y) \leq Cs^{-2} (1+s^{-1/2} |a|)^{-S}$. Thus
    \begin{align*}
        \sum_{s=0}^{N\epsilon}\sum_{a\in \mathbb Z^2} \mathbf p_{N,2} (s;x,x',a,a+y) \bigg|h\bigg(\frac{s}{N},\frac{a}{\sqrt N}\bigg)\bigg|
        &\leq \sum_{s=0}^{N\epsilon} \sum_{a\in \mathbb Z^2}  Cs^{-2} ( 1+s^{-1/2} |a|)^{-S} \bigg( \bigg|\frac{s}{N}\bigg| + \bigg| \frac{a}{\sqrt{N}}\bigg| \bigg)\\&\leq C \sum_{s=1}^{N\epsilon} \frac1{N} +\frac{1}{\sqrt{Ns} }\leq C\sqrt{\epsilon}.
    \end{align*}
    This allows us to apply the result on $[\epsilon,t] \times \mathbb R^2$ for $\epsilon>0$ and then send $\epsilon\downarrow 0$.  
\end{proof}


It is clear that to analyze the expression in Proposition \ref{prop:RiemannSum} as an approximate Riemann sum, we will require a joint invariance principle for the pair of processes $\big(\tau^{\mathbf X}_{(N,s\log N)} , S^{\mathbf X}_{(N,s\log N)}\big)$ under the measures $\tilde P_{N,s\log N}$ as $N\to \infty$. 
To do this, we will make use of the Dickman subordinator. 

\begin{defn}[\cite{MR4017119}]\label{dickman} The \textit{Dickman subordinator} is the L\'{e}vy process with moment generating function $\mathbb E[ e^{\lambda Y_t}]= \exp\big(t\int_0^1 \frac{e^{\lambda u}-1}{u} du\big).$ The density of $Y_s$ exists and will be denoted $f_s.$
\end{defn}

\begin{thm}[Joint invariance principle for the renewal process]\label{convergence_dickman}
Consider any deterministic collection of sequences $\mathbf x_N = (x^N_1,x^N_2,\dots)$ where $x^N_j\in F$ for some finite set $F\subset \mathbb Z^2$ not depending on $N$. Then the pair of processes $\big(N^{-1}\tau^{\mathbf x_N}_{(N,s\log N)} , N^{-1/2} S^{\mathbf x_N}_{(N,s\log N)}\big)_{s \ge 0}$ from \eqref{(t,a)} converges in law to $\big(
Y_{s},\frac1{\sqrt{2}} W_{
Y_{s}}\big)_{s \ge 0}$ as $N\to \infty$, where $Y_s$ is the Dickman subordinator and $W$ is an independent standard Wiener process in $\mathbb R^2$, both started from $0$. Convergence is in the sense of finite-dimensional distributions. 

The same result holds if $\x_N=(x^N_j)_{j=1}^\infty$ are randomly sampled with all $x^N_j$ in the finite set $J$, independently of the renewal variables in \eqref{(t,a)}.
In particular, this is the case under the measures $\tilde P_{N,k_N}$ from \eqref{eq1}, say with $k_N \ge \log N$.
\end{thm}

\begin{proof}
We prove the theorem first for any fixed deterministic $\mathbf x$ restricted to a finite set, and then specializing to the case of varying $\x_N$.

For the deterministic and fixed-$\x$ case, let $(x_j)_j$ be an arbitrary sequence taking values in a finite set $J\subset \mathbb Z^2$.
It is enough to establish convergence of the Fourier transforms: for $\mu\in \mathbb R$ and $\lambda \in \mathbb R^2$, we show
\begin{equation}\label{f1}
    \mathbb E \bigg[ e^{ \frac{i\mu}{N} \sum_{j=1}^{t\log N} T^{x_{j-1} \to x_j}_{(N,j)} + \frac{i\lambda}{\sqrt N} \bullet \sum_{j=1}^{t\log N}A^{x_{j-1} \to x_j}_{(N,j)} } \bigg]
    \to \exp\bigg(t \int_0^1 \int_{\mathbb R^2} (e^{i\mu u + i\lambda \bullet a}-1) \mathbf g(\tfrac{u}2 ,a)da \frac{du}{u}\bigg).
\end{equation}
The fact that the right-hand side is indeed the correct Fourier transform of the time-$t$ distribution of the L\'evy process $(Y_s,\frac1{\sqrt 2} W_{Y_s})$ follows from \cite[Equation (2.3)]{MR4017119}. To prove the above limit, we may use independence of the summands, and thus write the left side as a product 
\begin{align}
    \notag \prod_{j=1}^{t\log N}& \mathbb E \bigg[ e^{ \frac{i\mu}{N} T^{x_{j-1} \to x_j}_{(N,j)} + \frac{i\lambda}{\sqrt N} \bullet A^{x_{j-1} \to x_j}_{(N,j)} } \bigg]
    = \prod_{j=1}^{t\log N} \bigg(\sum_{r=1}^N \sum_{a\in \mathbb Z^2} e^{\frac{i\mu}{N} r + \frac{i\lambda}{\sqrt N} \bullet a} \frac{\pN{2}(r,(x_{j-1},0),(a,a-x_j))}{\mathscr S_N(x_{j-1},x_j)} \bigg) \\
   \notag  &= \prod_{j=1}^{t\log N}\bigg(1 + \frac1{\mathscr S_N(x_{j-1},x_j)}\sum_{r=1}^N \sum_{a\in \mathbb Z^2} \big(e^{\frac{i\mu}{N}r + \frac{i\lambda}{\sqrt N} \bullet a} -1\big) {\pN{2}(r,(x_{j-1},0),(a,a-x_j))} \bigg) \\
    &= \exp\Bigg[\sum_{j=1}^{t\log N} \mathrm{Log} \bigg( 1 + \frac1{\mathscr S_N(x_{j-1},x_j)}\sum_{r=1}^N \sum_{a\in \mathbb Z^2} \big(e^{\frac{i\mu}{N}r +\frac{i\lambda}{\sqrt N} \bullet a} -1\big) \pN{2}(r,(x_{j-1},0),(a,a-x_j))\bigg)\Bigg], \label{uni0}
\end{align}
where $\mathrm{Log}$ denotes the complex logarithm with the standard branch cut along the negative real axis. 
Note that $\mathrm{Log}(1+u) = u+O(u^2)$ for $u\in \mathbb C$ with $u$ close to 0, and that the $x_j$ can only take finitely many values. Thus, in order to complete the proof of \eqref{f1}, we need only prove that for all $x_{j-1},x_j$ in the finite set $J$, the limit
\begin{equation}
    \log N\cdot \sum_{r=1}^N \sum_{a\in \mathbb Z^2} \big(e^{\frac{i\mu}{N}r + \frac{i\lambda}{\sqrt N} \bullet a} -1\big) \frac{\pN{2}(r,(x_{j-1},0),(a,a-x_j)) }{\mathscr S_N(x_{j-1},x_j) } \to \int_0^1 \int_{\mathbb R^2} (e^{i\mu u + i\lambda \bullet a}-1) \mathbf g(\tfrac{u}2,a)da \frac{du}{u}\label{uni1}
\end{equation}
holds as $N\to \infty$. Note that the right side is independent of $x_{j-1},x_j$. By Assumption \ref{a1} \eqref{a1logasymptotic}, we have $\frac{\log N}{\mathscr S_N(x_{j-1},x_j)} \to \frac1{\nu_\infty (x_j)}$, and so it remains to show convergence
\begin{align}\label{f2}
    \notag\sum_{r=1}^N \sum_{a\in \mathbb Z^2} (e^{\frac{i\mu}{N} r + \frac{i\lambda}{\sqrt N} \bullet a} -1)\mathbf p_{N,2}(&r,(x_{j-1},0),(a,a-x_j)) \\
    &\to \nu_\infty(x_j)\int_0^1 \int_{\mathbb R^2} (e^{i\mu u + i\lambda \bullet a}-1) \mathbf g(\tfrac{u}2,a)da \frac{du}{u}
\end{align}
of the numerator of \eqref{uni1} as $N\to \infty$.
But since the function $h(u,a) = e^{i\mu u + i\lambda \bullet a}-1$ is smooth and vanishes at $(0,0)$, this is immediate from Lemma \ref{flem}. 
This proves the claim for deterministic $\mathbf x=(x_j)_j$, with $x_j\in J$.

For the theorem statement, we actually need to consider a deterministic sequence $\mathbf x^N$ that itself varies with $N$, and prove \eqref{f1} again in this generality. But this is immediate from the fact that $J$ is assumed to be a finite set, so that the limit \eqref{f2} is actually uniform over all possible values $x_j$, regardless of whether they depend on $N$. 

For the randomized case, we simply disintegrate along the law of $\mathbf x_N$ and use the deterministic result, using the independence of $\mathbf X$ and the renewal variables to disintegrate.
\end{proof}

\begin{remark}
    One might ask why we used the Fourier transform in the proof above, as opposed to the Laplace transform as in \cite{MR4017119}. The reason is that the heat kernels $\pN{h}$ are not assumed to have exponential moments, rather only super-polynomial decay in Assumption \ref{a1}. 
    Moreover, we assume no local limit theorem and no $r^{-1}(1 + o(1))$ asymptotics for $\pN{2}$, working instead under the weaker conditions of Assumption \ref{a1} \eqref{a1logasymptotic}, \eqref{a1invarianceprinciple}, and \eqref{a1heatkernel}.

    We further remark that the theorem above is false in the continuous-time continuous-space case of Assumption \ref{a3}, and the statement of the theorem needs to be modified so that rather than arbitrary sequences $\x_N$, we restrict the macroscopic time interval to $[0,T]$ and consider only those sequences $\x_N=(x_1^N,x_2^N,\ldots, x^N_{\lfloor T\log N \rfloor})$ such that
\begin{equation*}\{\sqrt{T\log N} \le j \le T\log N -\sqrt{T\log N}: |x_j^N-x_{j-1}^N| \le (\log N)^{-100} \} =\emptyset.
\end{equation*}
The exponents are not sharp, but they suffice for the proof of the continuous version in Theorem \ref{mr2} given below, see \eqref{100} and the subsequent argument.
\end{remark}

\begin{cor}\label{cor:dcBound}
    For all $x,M > 0$ there exist $C>0$ and $N_0\in \mathbb N$ such that
    \begin{equation*}
        \sup_{\mathbf x}  P\big( N^{-1} \tau_{(N,s\log N)}^{\mathbf x} \leq x \big) \leq C e^{-M  s}
    \end{equation*}
    uniformly over $s \ge 0$ and $N\in \{N_0,N_0+1,\ldots\}$. Here the supremum is over deterministic infinite-length sequences $\x=(x_1,x_2,\dots)$ with all $x_j$ staying fixed within a finite set $F\subset \mathbb Z^2$ (the constant $C$ may depend on $F$).
\end{cor}

This corollary states that the left side decays to zero faster than any exponential rate, uniformly in $N$. It will be important in dominated convergence arguments. For instance, it implies that the summand in the last expression of Theorem \ref{2.3} is dominated uniformly in $N$ by some integrable function of $s$ for arbitrary choices of fine-tuning constant $\vartheta\in\mathbb R$, which is not immediately clear for $\vartheta>0$. The bound is not sharp. We expect that the optimal bound will be something like $Ce^{-\alpha s\log_+ s}$ for constants $C,\alpha$ independent of $N$, but we will not need this strong a bound, which is more difficult to prove besides. 

\begin{proof}
    Fix $K>0$ and $m\in \mathbb N$. We consider $s$ of the form $Km$. In order for $ N^{-1} \tau_{(N,Km\log N)}^{\x} \leq x$, the increments must also satisfy $N^{-1} \tau_{(N,Km\log N)}^{\x} - N^{-1} \tau_{(N,K(m-1)\log N)}^{\x} \le x$, so by independence and non-negativity of the increments we have
    \begin{equation}
        \sup_{\x} P ( N^{-1} \tau_{(N,Km\log N)}^{\x} \leq x) \leq \sup_{\x} P( N^{-1} \tau_{(N,K\log N)}^{\x} \le x)^m.
    \end{equation}
    We claim there exists $\epsilon_N\downarrow 0$
    such that 
    \begin{equation}\label{a.b.}\sup_{K\in\mathbb N} \sup_{\x} \big| P(Y_K\le x) - P( N^{-1} \tau_{(N,K\log N)}^{\x} \le x) \big| \leq \epsilon_N 
    ,\end{equation} where $Y_s$ is the Dickman subordinator. Indeed, for any fixed value of $K\in \mathbb N$ the result follows from Theorem \ref{convergence_dickman}, which implies convergence of the respective CDFs. 
    Meanwhile, to deal with large values of $K$, we note that there exists $\alpha<1$ such that both quantities satisfy the uniform bounds $P(Y_K\le x) \leq P(Y_1 \le x)^K \le \alpha^K$ and $\sup_{\x} P\big( N^{-1} \tau_{(N,K\log N)}^{\x} \le x \big) \leq \sup_{\x} P\big( N^{-1} \tau_{(N,\log N)}^{\x} \le x \big)^K \le \alpha^K$, where the last estimate uses Theorem \ref{convergence_dickman}. All of this is enough to imply the above bound \eqref{a.b.} since it implies that along any two sequences $K_N$ and $\x^N$ the above difference goes to zero: $P(Y_{K_N}\le x) - P\big( N^{-1} \tau_{(N,K\log N)}^{\x^N} \le x\big) \to 0$ as $N\to \infty$.
    Combining the previous two bounds yields
    \begin{equation*}
        \sup_{\x} P \big( N^{-1} \tau_{(N,Km\log N)}^{\x} \leq M\big) \leq \big( \epsilon_N + P(Y_K \le x) \big)^m \leq 2^m \big(\epsilon_N^m + Ce^{-c Km\log K} \big), 
    \end{equation*}
    where we used the result of \cite[Equation (1.4)]{MR4017119} which implies that $P(Y_K \leq x)\le Ce^{-c K\log K}$ for some $C,c>0$ possibly depending on $x$. Now choose $K_N\uparrow \infty$ diverging slowly enough that $\delta_N:=(2\epsilon_N)^{1/K_N} \to 0$ as $N\to \infty$: for example, $K_N:= \floor{\log\log\big(\frac1{2\epsilon_N}\big)}$ will do. Recalling that $K_N\cdot m=s$, we see that the right hand side is bounded above by $\delta_N^{K_Nm} + Ce^{K_Nm(\log 2 -\log K_N)} = \delta_N^s + Ce^{cs(\log 2-\log K_N)}$. Given $M>0$ arbitrarily large, take $N_0$ large enough that $\delta_N<e^{-M}$ and $c(\log K_N - \log 2) >M$ for $N>N_0$, making $C$ bigger if necessary. Then the claim  follows.
\end{proof}

\subsection{Nonconstant macroscopically varying starting points}
Here we study a version of the above results where the deterministic starting point $x_0$ is replaced by some value $x_N\in \mathbb Z^2$ that varies with $N$, such that $N^{-1/2}x_N \to a\in \mathbb R^2$ as $N\to \infty. $ We will see that the limit depends on $a$ in a nontrivial way. If $\x\in (\mathbb R^2)^h$ and $\Phi \in C_c^\infty((\mathbb R^2)^h)$ we will use the notation
\begin{equation*}
    \g_t^{\otimes h}(\x,\Phi):= \int_{\mathbb (\R^2)^h} \g^{\otimes h}(t,\x-\y) \Phi(\y) d\y,
\end{equation*}
recalling that $\g^{\otimes h}(t,\x)= (2\pi t)^{-h} e^{-|\x|^2/2t}$ is the standard heat kernel on $(\mathbb R^2)^h$.
\begin{thm}\label{thm:macro}
    Let $t\in (0,1]$, and let $\Psi \in C_c^\infty((\mathbb R^2)^2)$. Suppose that $\x_N \to \x=(x_1,x_2)$, where $x_1\ne x_2$.
    Then we have the limit 
    \begin{align*}
        &\lim_{N\to\infty} \sum_{k=0}^\infty \sum_{0 < r_1< \cdots < r_k < Nt} \EN{2}^{\x_N} \bigg[\prod_{j=1}^k \left(e^{\boldsymbol{\zeta}_N(R_{r_j}^1 - R^2_{r_j})}-1\right)\Psi\left(\frac{\mathbf R_{Nt}}{\sqrt{N}}\right)\bigg] = \mathbf g_t^{\otimes 2}(\x, \Psi) \\
        &+ 4\pi \int_{(\mathbb R^2)^2} \Psi(w,w') \g\big(\tfrac{t}{2},\tfrac{x_1+x_2}2 - \tfrac{w+w'}{2}\big) \int_{0<u<s<t} \mathbf g(2u,x_1 - x_2) G_{\vartheta} (s-u) \mathbf g(2(t-s),w-w') \,du\,ds\,dw\,dw'.
    \end{align*}
\end{thm}
This theorem will be the key input to computing the limit of the second moment. To prove it, we will need the following lemma about the behavior of the renewal function after summing out over the time parameter $v$ and terminal points $y,y'$. 
That is, we need to estimate the quantity
\begin{align}\label{eq:Fdef}
    \notag F_N(t, u ,a ,x)
    := \sum_{v = Nu}^{Nt} \sum_{b,y \in \mathbb{Z}^2}
    U_N\Big(e^{\boldsymbol{\z}_N}-1; v-Nu &; \sqrt{N}a, \sqrt{N} a +x, b, b+ y\Big) \\
    &\times \big(e^{\z_N(y)} - 1\big) \EN{2}^{(b,\,b+y)}\bigg[ \Psi\bigg(\frac{\mathbf R_{N(t-v)}}{\sqrt{N}}\bigg)\bigg].
\end{align}

\begin{lem}\label{lem:Fdef}
    Fix all $t \in (0,1]$, $A > 0$, $\Psi \in C^\infty_c((\R^2)^2)$.
    Uniformly over $u \in [0,t] \cap N^{-1}\Z$, $a \in (N^{-1/2} \mathbb Z)^2$ with $|a|\le A$, and $x$ in some finite subset of $\mathbb Z^2$, we have
    \begin{equation*}
        \bigg|\frac{F_N(t,u,a,x)}{\log N} - \int_{u}^t \int_{\mathbb R^2}G_{\vartheta}(v - u) \mathbf g_{\tfrac{v}{2} - \tfrac{u}{2}}(b - a) \mathbf g_{t - v}^{\otimes 2}((b,b), \Psi) \,dv\,db\bigg| \to 0
    \end{equation*}
    as $N \to \infty$.
    Here, $U_N$ is the renewal function defined in \eqref{eq:renewal}.
\end{lem}

\begin{proof}
    By convention $U_N(f;0;x,x',y,y') = \ind_{\{y = x,y' = x'\}}$,
    so that the first term in the sum \eqref{eq:Fdef} is $(e^{\z_N} -1)\EN{2}^{(\sqrt N a,\sqrt N a + x)}[\Phi(N^{-1/2}\mathbf R_{N(t - u)})] =: F_N^{(0)}(t,u,a,x)$.
    We will actually prove that
    \begin{equation*}
        \bigg|F_N(t,u,a,x) - F_N^{(0)}(t,u,a,x) - \log N \int_{u}^t \int_{\mathbb R^2}G_{\vartheta}(v - u) \mathbf g_{\tfrac{v}{2} - \tfrac{u}{2}}(b - a) \mathbf g_{t - v}^{\otimes 2}((b,b), \Psi) \,dv\,db\bigg| = o(\log N),
    \end{equation*}
    uniformly over $u,a,x$ as in the statement of the lemma.
    Since $\norm{\z_N}_\infty \leq C/\log N$ and $\z_N$ has uniformly finite support by Assumption \ref{a1} \eqref{a1markov}, we have $\norm{e^{\z_N} - 1}_\infty \leq C'/\log N$ by Taylor expansion and therefore $\sup_{u,a,x} |F_N^{(0)}(t,u,a,x)| \leq C'\norm{\Psi}_\infty/\log N$ for all $N$.
    Here the supremum is again over the set of all relevant $u,a,x$, which does depend on $N$.
    This shows in particular that proving the bound in the display above proves the lemma.
    This also proves the lemma in the case $u = t$.
    
    The purpose of splitting off $F_N^{(0)}$ is that we can analyze the remaining terms using the exact expansion from the proof of Proposition \ref{prop:RiemannSum}.
    Define 
    \begin{equation}\label{hn1} h_N(t,a,x) := \EN{2}^{(\sqrt N a,\sqrt N a - x)}[\Psi(N^{-1/2}\mathbf R_{Nt})],
    \end{equation}
    where $t\in N^{-1}\mathbb Z_{\ge 0}$, $a\in N^{-1/2} \mathbb Z^2$ and $x$ is in some fixed finite subset of $\mathbb Z^2$. Then we get the exact expansion 
    \begin{align}\label{eq:Fdefexact}
        \notag (F_N - F_N^{(0)})(t,u,a&,x) = \sum_{k = 0}^\infty
        \big(\log N \cdot \Lambda_N(e^{\z_N} - 1) \big)^ke^{\varepsilon_N(k)} \\
        &\times \tilde{E}_{N,k}^x\bigg[ \ind_{\{\tau_{(N,k)}^{\seq X} \leq N(t - u)\}}
        \cdot h_N\big(t - u - N^{-1}\tau_{(N,k)}^{\seq X},a + N^{-1/2}S_{(N,k)}^{\seq X},X_k\big) \bigg],
    \end{align}
    where $(\varepsilon_N(k))_k$ is a sequence, as in Theorem \ref{2.3}, such that $\sup_k|\varepsilon_N(k)| \to 0$,
    and $\seq X$ is sampled from the tilted measure $\tilde P_{N,k}^x = \tilde P_{N,k}(\bullet| X_0 = x)$. In the display above, we have summed over the slice in \eqref{eq:RiemannSum} given by $\{\tau_{(N,k)}^{\seq X} = v,S_{(N,k)}^{\seq X} = b - a\}$ over $Nu < v \leq Nt$ and $b \in \Z^2$, just as in the sums that define $F_N - F_N^0$.
    Summing over time puts us over the event
    $\{\tau_{(N,k)}^{\seq X} \leq N(t - u)\}$,
    while summing out space starts the process $\mathbf R_{N(t - v)}$ with law $\PN{2}^{(b,b - y)}$ at $b = {a + N^{-1/2}S_{(N,k)}^{\seq X}}$, by the Markov property.
    Then by Corollary \ref{cor:dcBound} the expectation in \eqref{eq:Fdefexact} is bounded above by $\norm{\Psi}_\infty P_{N,k}^x\big( \tau_{(N,k)}^{\seq X} \leq N(t - u)\big) \leq Ce^{-M (\log N)^{-1}k}$, uniformly over $u,a,x$. This bound will be useful for dominated convergence just below.
    
    Write $L_{N,k}(u,a,x)$ for the expectation in \eqref{eq:Fdefexact}.
    Arguing as in Proposition \ref{prop:RiemannSum}, we can replace the renewal coefficients $\big\{\big(\log N \cdot \Lambda_N(e^{\z_N} - 1) \big)^ke^{\varepsilon_N(k)}\big\}_{k \in \N}$ by $\{e^{\vartheta s}\}_{s \in (\log N)^{-1}\N}$
    uniformly over $u,a,x$.
    If we view $L_{N,k}(u,a,x)$ as the step function $L_N(s,u,a,x) := L_{N,\lceil s \log N \rceil}(u,a,x)$ for $s > 0$, then by dominated convergence
    \begin{equation*}
        \bigg|\frac{1}{\log N}\sum_{k \geq 1} e^{\vartheta(\log N)^{-1}k} L_{N,k}(u,a,x) - \int_0^\infty e^{\vartheta s} L_N(s,u,a,x)ds \bigg| \to 0,
    \end{equation*}
    uniformly in $u,a,x$.
    Therefore it remains to show that $L_N(s)$ approximates the double integral in the statement of the lemma, that is, \begin{equation*}
        L_N(s,u,a,x) \to \int_u^t \int_{\mathbb R^2} G_\vartheta(v-u) \mathbf g_{\frac{v-u}{2}} (b-a) \mathbf g^{\otimes 2}_{t-v} ((b,b),\Psi) dvdb
    \end{equation*}
    uniformly over $u,a,x$.

    To this end, we first approximate the heat kernel uniformly by $h_N$ from \eqref{hn1}.
    That is, we show that $\sup_{u,a,x} |h_N(u,a,x) - \g_u^{\otimes 2}((a,a),\Psi)| \to 0$,
    where $\g_0^{\otimes 2}((a,a),\Psi) := \Psi(a,a)$ and the supremum is taken over $u,x$ as in the statement of the lemma, but over all $a \in \Z^2$.
    We need to do this because we have used the semigroup property to derive \eqref{eq:Fdefexact}.
    We will derive a uniform control over $\Z^2$ separately on a compact subset and on its complement.
    See that the limit follows immediately from the invariance principle of Assumption \ref{a1} \eqref{a1invarianceprinciple} and continuity of Brownian motion for fixed sequences $u_N \to u$, $a_N \to a$, and $x_N = x$. We can always guarantee $x_N = x$ by passing to a subsequence, since the $x$ live in a finite set. Over compacts, we get uniformity over $u$ and $a$.
    On the other hand, if $|a|$ is big, then $(\sqrt N a,\sqrt N a + x)$ is at least distance $c\sqrt N |a|$ away from each of the $O(N^2)$ lattice points in $\sqrt N \supp \Psi$.
    Therefore the heat kernel upper bound of Assumption \ref{a1} \eqref{a1heatkernel} shows that
    $|h_N(u,a,x)| \leq C\norm{\Psi}_\infty N^2(Nu)^{-2}(1 + (Nu)^{-1/2}c\sqrt N|a|)^{-S} = C\norm{\Psi}_\infty u^{-2}(1 + u^{-1/2}c|a|)^{-S} \leq C'|a|^{-S}$ for $0 < u \leq t$ and large enough $S$.
    At $u = 0$ we have $h_N(0,a,x) = \Psi(a,a + x/\sqrt N) \to 0$ as $|a| \to \infty$.
    So by the analogous tail estimates on $\g_u^{\otimes 2}((a,a),\Psi)$ we get
    \begin{equation*}
        \lim_{\lambda \to \infty} \sup_N \sup_{u \in [0,t]} \sup_{|a|>\lambda } \sup_{x\in F} \;(|h_N(u,a,x)| + |\g_u^{\otimes 2}((a,a),\Psi)|) \to 0,
    \end{equation*}
    where $F$ is any fixed finite subset of $\mathbb Z^2)$.
    The claimed uniformity follows from this estimate.
    Then Corollary \ref{cor:dcBound} and dominated convergence allow us to swap $h_N$ in \eqref{eq:Fdefexact} for $\g^{\otimes 2}$. Thus
    \begin{align}\label{eq:FdefUnif}
        \notag(F_N &- F_N^0)(t,u,a,x) = \log N \int_0^\infty e^{\vartheta s} \tilde E_{N,s \log N}^x\bigg[ \ind_{\big\{ N^{-1}\tau_{N,s \log N} \leq t - u \big\}} \\
        &\times \g^{\otimes 2}_{t - u - N^{-1}\tau^{\seq X}_{(N,s\log N)}}\Big((a + N^{-1/2}S_{(N,s\log N)},a + N^{-1/2}S_{(N,s\log N)}),\Psi\Big) \bigg] \, ds + o(\log N)
    \end{align}
    uniformly over relevant $u,a,x$, and from here it remains to take a joint limit of the renewal process $(\tau^{\seq X}, S^{\seq X})$.
    Indeed, compactness again allows us to pass to a subsequence such that $u_N \to u$, $a_N \to a$, and $x_N = x$.
    Then by the joint invariance principle of Theorem \ref{convergence_dickman} we have convergence of the tilted renewal expectations in the integrand in \eqref{eq:FdefUnif} to the quantity
    \begin{equation}\label{eq:FdefIntegrand}
        E\bigg[ \ind_{\{Y_s \leq t - u\}} \g^{\otimes 2}_{t - u - Y_s}\bigg( \Big(a + \frac{1}{\sqrt 2}W_{Y_s},a +  \frac{1}{\sqrt 2}W_{Y_s}\Big),\Psi \bigg) \bigg],
    \end{equation}
    where the expectation is against the joint law of the Dickman subordinator and its subordinated Brownian motion.
    Here, the application of Theorem \ref{convergence_dickman} is valid because the only discontinuity occurs at the point $Y_s = t - u$, which has probability $0$ since $Y_s$ has a density $f_s$.
    Then by the uniform bound over $s \geq 0$ in Corollary \ref{cor:dcBound} and dominated convergence,
    the entire integral in \eqref{eq:FdefUnif} converges to the quantity
    \begin{equation}
        \int_0^\infty e^{\vartheta s} E\bigg[ \ind_{\{Y_s \leq t - u\}} \g^{\otimes 2}_{t - u - Y_s}\bigg( \Big(a + \frac{1}{\sqrt 2}W_{Y_s},a +  \frac{1}{\sqrt 2}W_{Y_s}\Big),\Psi \bigg) \bigg] \, ds,
    \end{equation}
    uniformly over $u,a,x$.
    It remains to identify this limit.
    But the same bounds used for dominated convergence justify applying Fubini's theorem, so that \eqref{eq:FdefIntegrand} is equal to
    \begin{equation*}
        \int_0^{t - u} f_s(r) \int_{\R^2} \g_{\frac{r}{2}}(b - a) \g^{\otimes 2}_{t - u -r}((b,b),\Psi) \,db \,dr \\
        = \int_0^{t - u} \int_{\R^2} G_\vartheta(r)\g_{\frac{r}{2}}(b - a)\g^{\otimes 2}_{t - u -r}((b,b),\Psi) \,db \,dr,
    \end{equation*}
    as claimed.
\end{proof}

\begin{proof}[Proof of Theorem \ref{thm:macro}]
We will break this proof into several steps. 

\smallskip
\noindent \textbf{Step 1. $k=0$ term.}
In this case, it follows from the invariance principle in Assumption \ref{a1} \eqref{a1invarianceprinciple} that  
\begin{equation}\label{g1}
     \lim_{N\to\infty}\sum_{y_1, y_2 \in \mathbb Z^2}  \EN{2}^{\x_N}\bigg[\ind_{\big\{ R^1_{Nt}= y_1, R^2_{Nt}=y_2 \big\}}\bigg] \Psi\big( \y/\sqrt N\big)
     = \lim_{N\to\infty} \EN{2}^{\x_N} \bigg[\Psi\left(\frac{\mathbf R_{Nt}}{\sqrt{N}}\right)\bigg]
     =  \mathbf g_t^{\otimes 2}(\x, \Psi).  
\end{equation}

\smallskip
\noindent\textbf{Step 2. $k=1$ term.}
This term will be negligible in the limit, but it is notationally convenient to handle it separately from the rest of the sum. It can be written 
\begin{equation}\label{g0}
    \sum_{u=0}^{Nt} \sum_{a,x\in \mathbb Z^2} \pN{2} (u, (x^1_N, x^2_N), (a,a+x)) \big(e^{\z_N(x)}-1\big) \EN{2}^{(a,a+x)} \bigg[ \Psi\bigg(\frac{\mathbf R_{Nt - v}}{\sqrt{N}}\bigg)\bigg].
\end{equation}
Use $|\psi|\leq \|\psi\|_{L^\infty} $ and $|e^{\z_N}-1| \leq \frac{C}{\log N} \ind_{\{|x|\le A\}},$ for some $A>0$ not depending on $N$. Recall the heat kernel upper bound from Assumption \ref{a1} \eqref{a1heatkernel}, which implies that for fixed $x$ one has $\mathbf p_{N,2} (u; (x^1_N, x^2_N), (a,a+x)) \leq Cu^{-2} (1+u^{-1/2}|a|)^{-S}$. Then we see that \eqref{g0} is upper bounded by $C/\log N$ as $N\to \infty$, and thus vanishes in the limit.

\smallskip
\noindent\textbf{Step 3. $k \geq 2$ terms.}
We claim the remaining terms decompose as follows:
\begin{align}
    \notag \sum_{k=2}^\infty \sum_{0 < r_1< \cdots <r_k < N t} &\EN{2}^{(x^1_N,x^2_N)} \Bigg[\prod_{j=1}^k \bigg(e^{\boldsymbol{\zeta}_N(R_{r_j}^1 - R^2_{r_j})}-1\bigg)\Psi\bigg(\frac{\mathbf R_{Nt}}{\sqrt{N}}\bigg)\Bigg] \\
    \notag &\!\!\!\!\!\!\!\!\!\!\!\! = \sum_{\substack{0 < u < v < Nt \\ a, x, b, y \in \Z^2}}\pN{2}(u, (x^1_N, x^2_N),(a ,a + x)) \big(e^{\boldsymbol{\zeta}_N (x)}-1\big) \\
    &\quad\quad\quad\;\; \cdot U_N(e^{\boldsymbol{\zeta}_N}-1; v -u; a, a + x, b , b +y ) \big(e^{\boldsymbol{\zeta}_N (y)}-1\big) \mathbf E^{(b, b + y )}_{N,2} \bigg[\Psi\bigg(\frac{\mathbf R_{Nt - v}}{\sqrt{N}}\bigg)\bigg], \notag
\end{align}
where $U_N(e^{\boldsymbol{\zeta}_N}-1; v -u; a, a + x, b , b +y ) $ is the renewal function defined in \eqref{eq:renewal}. 
To see this, let $u = r_1, v = r_k$. What we are doing here is starting from two points that are a macroscopic distance apart and running the system until they are a microscopic distance apart. We then run the renewal function up to another microscopic distance, at which point we again separate to a macroscopic distance tested against $\psi$.
Finally, note that we need to tack on  $\left(e^{\boldsymbol{\zeta}_N (x)}-1\right)$ and $ \left(e^{\boldsymbol{\zeta}_N (y)}-1\right)$ since $U_N$ does not by definition contain factors of $f = e^{\boldsymbol{\zeta}_N}-1$ evaluated at the starting and ending sites. 

Then the expression above can be rewritten in the form
\begin{equation}\label{g2}
     \sum_{\substack{0 < u < Nt \\ a, x \in \mathbb Z^2}} \pN{2}\big((u\, (x^1_N, x^2_N), (a, a + x)\big) \big(e^{\z_N(x)} - 1\big) F_N\Big(t, \tfrac{u}{N} ,\tfrac{a}{\sqrt{N}} , x\Big),
\end{equation}
whereby Lemma \ref{lem:Fdef} we have 
\begin{equation*}
    F_N(t, u ,a ,x) = \log N  \int_{u}^t \int_{\mathbb R^2}G_{\vartheta}(v - u) \mathbf g_{\tfrac{v}{2} - \tfrac{u}{2}}(b - a) \mathbf g_{t - v}^{\otimes 2}((b,b), \Psi) \,dv \,db + o(\log N)
\end{equation*}
where the error is uniform over all macroscopic variables lying in some compact set. Let us also define 
\begin{equation*}
    F(t, u, a):= \int_{u}^t \int_{\mathbb R^2}G_{\vartheta}(v - u) \mathbf g_{\tfrac{v}{2} - \tfrac{u}{2}}(b - a) \mathbf g^{\otimes 2}_{t - v}((b,b), \Psi) \,dv \,db.
\end{equation*}
Using Lemma \ref{flem}, and the fact that $|e^{\z_N}-1| \leq \frac{C}{\log N}$, we see that \eqref{g2} equals
\begin{align*}
    \log N &\cdot  4 \pi  \cdot \nu_\infty\big(e^{\boldsymbol{\zeta}_N} - 1\big) \int_0^t \mathbf g(2s,x_1-x_2) \int_{\mathbb R^2} \mathbf g(2s,z) F(t, s, \tfrac12 (z+x_1+x_2)) \,dz\, ds +o(1)\\
   &= \log N \cdot 4\pi \cdot \nu_\infty(e^{\boldsymbol{\zeta}_N}-1) \int_0^t \int_s^t \int_{\mathbb{R}^2} \int_{\mathbb{R}^2} \mathbf g(2s, x_1-x_2)\, \mathbf g(2s, z)  \, \\ 
   &\quad\quad\quad\quad\quad\quad\quad\quad\quad\quad \times G_\vartheta(v-s)\, \mathbf g_{\frac{v-s}{2}}\!\left(b - \tfrac{z+x_1+x_2}{2}\right) \mathbf g_{t-v}^{\otimes 2}((b,b),\Psi)\, db\, dz\, dv\, ds +o(1).
\end{align*}
Finally, note that 
$\nu_\infty\left(e^{\boldsymbol{\zeta}_N} - 1\right) = \frac{1}{\log N}(1 + o(1))$
by \eqref{crit_tuning} and Lemma \ref{2.1}.
Putting this all together, we get that the limit of the sum over terms corresponding to $k \geq 2$ is  
\begin{align*}
    &4\pi \int_0^t \int_s^t \int_{\mathbb{R}^2} \int_{\mathbb{R}^2} \mathbf g(2s, x_1-x_2)\, \mathbf g(2s, z) G_\vartheta(v-s)\, \mathbf g_{\frac{v-s}{2}}\!\left(b - \tfrac{z+x_1+x_2}{2}\right) \mathbf g_{t-v}^{\otimes 2}((b,b),\Psi) \, db\, dz\, dv\, ds \\
    &= 4\pi \int_{(\mathbb R^2)^2} \Psi(w,w') \mathbf g(\tfrac{t}{2},\tfrac{x_1+x_2}2 - \tfrac{w+w'}{2}) \int_0^t \int_0^s \mathbf g(2s,x_1 - x_2) G_{\vartheta} (v-s) \mathbf g(2(t-v),w-w') \,dv\, ds \, dw \,dw'.
\end{align*}
Combining this with Steps 1 and 2 yields the result.
\end{proof}



\begin{thm}[Bounds for dominated convergence]\label{limitofu2} The estimate
\begin{equation*}
    \sum_{y_1, y_2 \in \mathbb Z^2} U_N(e^{\boldsymbol{\zeta}_N}-1; N t_N, x_1,x_2,y_1,y_2) - 1 \leq H(N^{-1/2} |x_1-x_2|)
\end{equation*}
holds uniformly over $t\in [0,T]$ and $x_1,x_2\in \mathbb Z^2$, where $H:(0,\infty) \to [0,\infty)$ is a decreasing function satisfying $H(x) \le C |\log x|$ for $|x|<1,$ and $H(x)\to 0$ superpolynomially fast as $x\to\infty.$
\end{thm}

\begin{proof}
In the infinite series defining $U_N$, note that subtracting 1 simply gets rid of the $k=0$ term (since we are summing over $y_1,y_2$). Thus we need to bound the $k\ge 1$ terms. 

To proceed, let us use the same decomposition of $U_N$ as the sum of the last expressions in \eqref{g1}--\eqref{g0}--\eqref{g2}.
We have already showed that the terms in \eqref{g1}--\eqref{g0} are upper bounded by a constant, so it suffices to upper bound the right side of \eqref{g2}.
Setting $\Psi \equiv 1$ in that expression, we just need to show that
\begin{equation*}
    \sum_{\substack{0 < u < Nt \\ a, x \in \Z^2}}\pN{2}\big(u, (x_1,x_2),(a ,a + x)\big) \left|e^{\boldsymbol{\zeta}_N (x)}-1\right|   F_N\Big(t, \tfrac{u}N, \tfrac{a}{\sqrt N},x\Big)\leq H(N^{-1/2} |x_1-x_2|),
\end{equation*}
with $H$ satisfying the bounds of the theorem statement.
    To prove this, we use the following facts:
    \begin{itemize}[leftmargin=15pt]
        \item By Assumption \ref{a1} \eqref{a1heatkernel}, we have $\pN{2}\big(u,(x_1,x_2),(a ,a + x)\big) \leq \frac{C}{u^2} \Big(1 + \frac{1}{\sqrt u} (|a-x_1|+|a+x-x_2|)\Big)^{-S}$.
        \item By Assumption \ref{a1} \eqref{a1markov}, we have \begin{equation*}
            \left|e^{\boldsymbol{\zeta}_N (x)}-1\right| \leq \frac{C}{\log N} \ind_{\{|x|\le A\}},
        \end{equation*}
    uniformly in $N$ for the diameter $A$ of $\supp \z_N$.
    \item By Lemma \ref{lem:Fdef}, we have the bound
    $F_N\big(t, \tfrac{u}N, \tfrac{a}{\sqrt N},x\big)  \leq C \log N$,
    where $C$ is uniform over all $N,t,u,a$ and $x$ in a finite set.
    \end{itemize}
Combining these bullets, we obtain the bound 
    \begin{align*}
        \sum_{\substack{0 < u  < Nt \\ a, x \in \Z^2}}\pN{2}\big(u,(x_1,x_2)&,(a ,a + x)\big) \left|e^{\boldsymbol{\zeta}_N (x)}-1\right|   F_N\Big(t, \tfrac{u}N, \tfrac{a}{\sqrt N},x\Big) \\[-10pt]
        &\leq \sum_{u=1}^{Nt} \sum_{a,x \in \mathbb Z^2} C u^{-2} \big(1 + u^{-1/2} (|a-x_1|+|a+x-x_2|)\big)^{-S} \ind_{\{|x|\le A\}} \\&\leq C' \sum_{u=1}^{Nt} u^{-1} \big( 1+ u^{-1/2} |x_1-x_2| \big)^{-S}.
    \end{align*}
    This expression can be viewed as a Riemann sum upper bounded by $H(N^{-1/2}|x_1-x_2|)$, where $H(x):= \int_0^1 u^{-1} (1+xu^{-1/2})^{-S}du$. Since $S$ can be taken as large as desired, the claim follows.
\end{proof}

\subsection{Convergence of the first and second moments of $\mathfrak H^N$}

\begin{thm}\label{1stmom}
Under Assumption \ref{a1}, we have the limit
\begin{equation*}
    \lim_{N\to\infty}\mathbb E \left[\mathfrak H^N_{0,t}(\phi,\psi)\right]  = \int_{(\R^2)^2} \phi(x) \psi(y) \g(t, x- y) \,dx\, dy.
\end{equation*}
\end{thm}

\begin{proof}
Simply by definition of the macroscopic field in Definition \ref{a2} and the conditions of Assumption \ref{a1}, we have
 \begin{align*}
     \mathbb E[\mathfrak H^N_{0,t}(\phi,\psi)] &=N^{-1}  \sum_{x \in \mathbb Z^2} \phi \left(\frac{x}{\sqrt{N}}\right) \EN{1}^{x} \bigg[ 
     \psi\left(\frac{R_{Nt}}{\sqrt{N}}\right)  \bigg].
 \end{align*}
By the invariance principle of Item \eqref{a1invarianceprinciple}
, this expression converges to 
$\int_{\R^2 \times \R^2} \phi(x) \psi(y) \mathbf g(t, x- y) \,dx\, dy$.\end{proof}

Next we move on to the second moment.

\begin{thm}\label{2ndmom}
    Consider any sequence $t_N\in N^{-1}\mathbb Z_{\ge 0}$ such that $t_N\to t\in [0,1]$. Let $\phi_1,\phi_2,\psi_1,\psi_2 \in C_c(\mathbb R^2)$. Under Assumption \ref{a1}, we have the limit
    \begin{equation*}
        \lim_{N\to\infty} \mathrm{Cov}(\mathfrak H^N_{0,t_N}(\phi_1,\psi_1),\mathfrak H^N_{0,t_N}(\phi_2,\psi_2))  = \int_{(\mathbb R^2)^4} \phi_1(z)\phi_2(z')\psi_1(w)\psi_2(w') K_t^{\vartheta} (z,z',w,w') \,dz\,dz'\,dw\,dw'.
    \end{equation*}
    Above, $\phi,\psi \in C_c(\mathbb R^2)$, and for $z,z',w,w'\in\mathbb R^2$ the kernel $K^\vartheta$ is defined by \begin{equation*}
        K_t^{\vartheta} (z,z',w,w') := {4}\pi \mathbf g(\tfrac{t}{{2}},\tfrac{z+z'}2 - \tfrac{w+w'}{2}) \int_{0<u<s<t} \mathbf g(2u,z-z') G_{\vartheta} (s-u) \mathbf g(2(t-s),w-w') \,du\,ds,
    \end{equation*}
    and $G_{\vartheta}(u) = \int_0^\infty e^{\vartheta s} f_s(u) ds$, with $f_s$ being the marginal density of the Dickman subordinator $Y_s$ from Definition \ref{dickman} above.
\end{thm}
We note that the kernel $K_\vartheta$ has logarithmic singularities along $z=z'$ and $w=w'$ which cause various complications in the proof. In particular, neither $\phi$ nor $\psi$ can be taken to approach a Dirac mass. We will use dominated convergence in conjunction with Theorem \ref{thm:macro} to deal with these.

For the nearest-neighbor discrete case, this was first shown by \cite[Proposition 3.6]{CSZ_2d}, which in turn relies heavily on the results of \cite{MR4017119}. We are going to use a similar approach, substituting generalized results such as Theorem \ref{limitofu2} for local descriptions of the renewal function.

\begin{proof}
 Set $\Psi = \psi_1\otimes \psi_2$ in Theorem \ref{thm:macro}. Then we have
 \begin{align*}
     \mathbb E[\prod_{j\in \{1,2\}}\mathfrak H^N_{0,t}(\phi_j,\psi_j)^2] &= N^{-2} \sum_{\x \in (\Z^2)^2} \phi_1\otimes \phi_2 \bigg(\frac{\x}{\sqrt{N}}\bigg) \EN{2}^{\x} \bigg[ e^{\sum_{r=1}^{Nt-1} \z_N( R^1_r - R^2_r) + \mathscr E^{N,h}_t } \psi_1 \otimes \psi_2\bigg(\frac{\mathbf R_{Nt}}{\sqrt{N}}\bigg)\bigg].
 \end{align*}
Use the crucial identity\footnote{
    Note that this is exactly where we see a discrepancy with the continuous-time case. In the continuous-time setting, $e^{\z}-1$ would be replaced by just $\z$, since Taylor expansion gives $e^{\int_0^m s_j\;dj} = \sum_k \int_{t_1< \dots <t_k\le m} s_{t_1} \cdots s_{t_j} dt_1\cdots dt_j,$ if the sum is replaced by an integral.
}
$e^{\sum_{j=1}^m s_j} = \prod_{j=1}^m (1+ (e^{s_j}-1))=\sum_k \sum_{t_1< \dots <t_k\le m} (e^{s_{t_1}}-1) \cdots (e^{s_{t_k}}-1)$ to expand the expectation as the following 
series:
\begin{align*}
    \EN{2}^{\x} \bigg[ &e^{\sum_{r=1}^{Nt-1} \z_N( R^1_r - R^2_r) + \mathscr E^{N,2}_t} \psi_1 \otimes \psi_2\bigg(\frac{\mathbf R_{Nt}}{\sqrt{N}}\bigg)\bigg]
    \\&= \EN{2}^{\x} \bigg[ \prod_{r=1}^{Nt-1}(1 + ( e^{\z_N( R^1_r - R^2_r)} - 1)) e^{  \mathscr E^{N,2}_t } \psi_1 \otimes \psi_2\bigg(\frac{\mathbf R_{Nt}}{\sqrt{N}}\bigg) \bigg] \\
    &= \sum_{k=0}^{\infty} \sum_{0 < r_1 < \cdots < r_k < Nt} \EN{2}^{\x}\bigg[ \prod_{i=1}^k (e^{{\boldsymbol{\zeta}}_N(R^1_{r_i} - R_{r_i}^2)}-1)e^{ \mathscr E^{N,2}_t}\psi_1 \otimes \psi_2\bigg(\frac{\mathbf R_{Nt}}{\sqrt{N}}\bigg)\bigg]\\
    &= (1+o(1))\mathscr K_N^{\psi_1\otimes \psi_2}(t,\x),
\end{align*}
where if $\Psi\in C_c((\mathbb R^2)^2)$ then 
\begin{equation*}
    \mathscr K_N^{\Psi}(t,\x) := \sum_{k=0}^{\infty} \sum_{0 < r_1 <\cdots < r_k < Nt}\EN{2}^{\x}\bigg[ \prod_{i=1}^k (e^{{\boldsymbol{\zeta}}_N(R^1_{r_i} - R_{r_i}^2)}-1) \Psi\bigg(\frac{\mathbf R_{Nt}}{\sqrt{N}}\bigg) \bigg].
\end{equation*}
Note that $e^{\mathscr E^{N,2}_t}$ has been replaced by 1, and that the $1 + o(1)$ error term is uniform over all variables and reflects only the fact that $\mathscr E^{N,2}_t$ converges to 0 in $L^\infty$ norm by Assumption \ref{a1} \eqref{a1markov}.

It follows from Theorem \ref{thm:macro} that $\mathscr K_N^{\Psi}(t_N,\x_N) \to \mathscr K_\infty^{\Psi}(t,\x)$ whenever $N^{-1}t_N\to t>0$ and $N^{-1/2}\x_N \to \x$ is well-separated, where $\mathscr K_\infty^{\Psi}(t,x_1,x_2)$ is defined by the limit of the expression displayed in Theorem \ref{thm:macro}. Furthermore, Theorem \ref{limitofu2} shows that
\begin{equation*}
    \sup_N \mathscr K_N^{\Psi}(t,x_1,x_2) \leq C (1+\log _- (|x_1-x_2|)).
\end{equation*}
Thus from from dominated convergence and the compact support of $\phi$, it follows immediately that 
\begin{equation*}
    N^{-2} \sum_{\x \in (\Z^2)^2} (\phi_1\otimes \phi_2) \bigg(\frac{\x}{\sqrt{N}}\bigg) \mathscr K_N^{\psi_1\otimes \psi_2}(t,\x) \to \int_{(\mathbb R^2)^2} (\phi_1\otimes \phi_2)(\x) \mathscr K_\infty^{\psi_2\otimes \psi_2} (t,\x) \,d\x.
\end{equation*}
This is exactly the desired result.
\end{proof}

\subsection{A semigroup property for limit points of $\mathfrak H^N$}

Here we prove a Chapman-Kolmogorov or semigroup-type identity for the limit point, which is one of the key properties needed to use the axiomatic characterization in Definition \ref{def:axioms} to identify the limit points.

\begin{thm}\label{convo}
    Consider a smooth and compactly supported test function $\omega: \R^2\to \R,$ such that $\int_{\mathbb R^2} \omega=1.$ Define $\omega_\e(x) = \epsilon^{-2} \omega(\epsilon^{-1}x),$ and let $\mathfrak H^N(dx,dy)$ denote the measure on $(\mathbb R^2)^2$ defined so that $\mathfrak H^N_{s,t} (\phi\otimes \psi) = \mathfrak H^N_{s,t}(\phi,\psi)$, with the right hand side defined as in Definition \ref{a2}. Then we have
    \begin{equation*}
        \lim_{\epsilon\to 0} \limsup_{N\to \infty} \mathbb E \bigg[  \bigg(\int_{(\mathbb R^2)^4 } f(x,z) \omega_{\e}(y-w)\mathfrak H^N_{t,u}(dy,dz) \mathfrak H^N_{s,t}(dx,dw) - \int_{\mathbb R^4} f(x,z) \mathfrak H^N_{s,u}(dx,dz)\bigg)^2 \bigg] = 0.
    \end{equation*}
\end{thm}

\begin{proof}
    Let us rewrite 
        \begin{align*}\int_{(\mathbb R^2)^4 } f(x,z) \omega_{\e}(y-w)&\mathfrak H^N_{t,u}(dy,dz) \mathfrak H^N_{s,t}(dx,dw) - \int_{\mathbb R^4} f(x,z) \mathfrak H^N_{s,u}(dx,dz) \\&= \int_{(\mathbb R^2)^4} f(x,z) \big( \omega_\epsilon(y-w) - N\ind_{\{y=w\}} \big) \mathfrak H^N_{t,u}(dy,dz) \mathfrak H^N_{s,t}(dx,dw).
        \end{align*}
    Squaring this expression yields
        \begin{equation*}
            \int_{(\mathbb R^2)^8} \prod_{j \in \{1,2\}} \bigg\{ f(x_j,z_j) \big( \omega_\epsilon(y_j-w_j) - N\ind_{\{y_j=w_j\}} \big) \mathfrak H^N_{t,u}(dy_j,dz_j) \mathfrak H^N_{s,t}(dx_j,dw_j)\bigg\}.
        \end{equation*}
    We expand this product as a sum of four terms. Then we take $N\to \infty$ separately for each of the four terms, and verify that the limit equals the corresponding quantity for the SHF itself. This is done below.

    \smallskip
    \noindent\textbf{First term.} The expectation of the first term equals
    \begin{align*}
        \int_{(\mathbb R^2)^8}& \prod_{j \in \{1,2\}} f(x_j,z_j) \omega_\epsilon(y_j-w_j) \mathbb E \bigg[ \prod_{j \in \{1,2\}} \mathfrak H^N_{t,u}(dy_j,dz_j)\bigg] \mathbb E \bigg[ \prod_{j \in \{1,2\}} \mathfrak H^N_{s,t}(dx_j,dw_j)\bigg] \\
        &\!\!\!\!\!\!= \int_{(\mathbb R^2)^4} \bigg\{ \int_{(\mathbb R^2)^4} \prod_{j \in \{1,2\}} f(x_j,z_j) \omega_\epsilon(y_j-w_j) \mathbb E \bigg[ \prod_{j \in \{1,2\}} \mathfrak H^N_{t,u}(dy_j,dz_j)\bigg] \bigg\}\mathbb E \bigg[ \prod_{j \in \{1,2\}} \mathfrak H^N_{s,t}(dx_j,dw_j)\bigg] .
    \end{align*}
    Call the inner integral $v_N(x_1,x_2,w_1,w_2). $ By the second-moment convergence of Theorem \ref{2ndmom}, it follows that $v_N \to v_\infty$ as $N\to \infty$ where $$v_\infty(x_1,x_2,w_1,w_2) = \int_{(\mathbb R^2)^4} \prod_{j \in \{1,2\}} \prod_{j \in \{1,2\}} f(x_j,z_j) \omega_\epsilon(y_j-w_j) \tilde K_{u-t}^\vartheta(y_1,y_2,z_1,z_2) dy_1dy_2dz_1dz_2, $$
    where $\tilde K^\vartheta_t(\y,\mathbf z):= K^\vartheta_t(\y,\mathbf z) + \mathbf g^{\otimes 2} (t,\y-\mathbf z)$, which is the integration kernel for the second moment measure of the SHF (Definition \ref{def:axioms}). Moreover, the convergence is uniform on compacts, and therefore by the uniform-in-$N$ tail decay we see by the second-moment convergence of Theorem \ref{2ndmom} that \begin{equation*}
        \int_{(\mathbb R^2)^4} v_N(\x,\mathbf w) \mathbb E \bigg[ \prod_{j \in \{1,2\}} \mathfrak H^N_{s,t}(dx_j,dw_j)\bigg] \to \int_{(\mathbb R^2)^4} v_\infty(\x,\mathbf w)\tilde K_{t-s}^\vartheta(\x,\mathbf w) \,d\x \,d\mathbf w.
    \end{equation*}
    This is the desired result.

    \smallskip
    \noindent\textbf{Second term.} The expectation of the second term equals 
    \begin{align*}
        &\int_{(\mathbb R^2)^8} \prod_{j \in \{1,2\}}f(x_j,z_j) \omega_\epsilon(y_1-w_1) N\ind_{\{y_2=w_2\}} \E \bigg[ \prod_{j \in \{1,2\}} \mathfrak H^N_{t,u}(dy_j,dz_j)\bigg] \E \bigg[ \prod_{j \in \{1,2\}} \mathfrak H^N_{s,t}(dx_j,dw_j)\bigg] \\&= \int_{(\mathbb R^2)^4} \bigg\{ \int_{(\mathbb R^2)^4} \prod_{j \in \{1,2\}} f(x_j,z_j)\omega_\epsilon(y_1-w_1) N\ind_{\{y_2=w_2\}} \E \bigg[ \prod_{j \in \{1,2\}} \mathfrak H^N_{t,u}(dy_j,dz_j)\bigg]\bigg\} \E \bigg[ \prod_{j \in \{1,2\}} \mathfrak H^N_{s,t}(dx_j,dw_j)\bigg]
    \end{align*}
    Again define $\ell_N (x_1,x_2,w_1,w_2)$ to be the inner integral. Using Theorem \ref{thm:macro}, we see that $\ell_N$ converges uniformly on compacts to the function $$\ell_\infty(x_1,x_2,w_1,w_2) = \int_{\mathbb R^6} f(x_1,z_1)f(x_2,z_2) \omega_\epsilon(y_1-w_1)\tilde K^\vartheta_{u-t} (y_1,w_2, z_1,z_2) dy_1dz_1dz_2.$$
    Directly from the second moment convergence in Theorem \ref{2ndmom}, it then follows that \begin{align*}
        \int_{(\mathbb R^2)^4} &\ell_N(\x,\mathbf w) \mathbb E \bigg[ \prod_{j \in \{1,2\}} \mathfrak H^N_{s,t}(dx_j,dw_j)\bigg] \to \int_{(\mathbb R^2)^4} \ell_\infty(\x,\mathbf w)\tilde K_{t-s}^\vartheta(\x,\mathbf w) \, d\x \, d\mathbf w,
    \end{align*}
    as desired.

    \smallskip
    \noindent\textbf{Third term.}
    This term is completely symmetric to the second term.
    
    \smallskip
    \noindent\textbf{Fourth term.} Convergence of the fourth term follows immediately from Theorem \ref{2ndmom}: 
    \begin{equation*}
        \mathbb E \bigg[  \bigg(\int_{\mathbb R^4} f(x,z) \mathfrak H^N_{s,u}(dx,dz)\bigg)^2 \bigg] \to \int_{(\mathbb R^2)^4} \prod_{j \in \{1,2\}} f(x_j,z_j) \tilde K^\vartheta_{u-s} (\x,\mathbf z) \,d\x \,d\mathbf z,
    \end{equation*}
    as desired.
    
    \smallskip
    \noindent\textbf{Putting all the terms together.} Combining the result of all four terms, we have shown that 
    \begin{align*}
        \limsup_{N\to \infty} \;&\E\bigg[  \bigg(\int_{(\mathbb R^2)^4 } f(x,z) \omega_{\e}(y-w)\mathfrak H^N_{t,u}(dy,dz) \mathfrak H^N_{s,t}(dx,dw) - \int_{\mathbb R^4} f(x,z) \mathfrak H^N_{s,u}(dx,dz)\bigg)^2 \bigg] \\
        &= \int_{(\mathbb R^2)^8}  \prod_{j \in \{1,2\}} \bigg\{ f(x_j,z_j) \big( \omega_\epsilon(y_j-w_j) - \delta_0(y_j-w_j) \big) \bigg\}\tilde K^\vartheta_{u-t} (\y,\mathbf z) \tilde K^\vartheta_{t-s} (\x,\mathbf w) \,d\x \,d\y \,d\mathbf z \,d\mathbf w.
    \end{align*}
    Now taking $\epsilon\to 0$, this quantity is known to converge to 0, as it agrees with the analogous quantity for the SHF: see \cite[Propositions 2.4--2.5]{MR5082783}, or alternatively \cite{Tsa24}.
\end{proof}

\section{Limit of the higher moments of $\mathfrak H^N$} \label{sec:highermoments}

In this section, we study the third and fourth moments of the macroscopic field. Since the arguments are similar, we will only consider the fourth moment, leaving the analogous but slightly easier case of the third moment to the reader. This is strictly for notational simplicity, as our method of proof easily extends to prove similar bounds for all other moments, provided one assumes Markov representations for the corresponding moments of the prelimiting field. 

For $\Phi,\Psi:(\R^2)^4 \to \R$ and $s,t\in N^{-1} \mathbb Z_{\ge 0}$ with $s\le t$, 
we define 
\begin{equation}\label{gnst}
    \mathcal G^N_{s,t} (\Phi,\Psi):= N^{-4} \sum_{\substack{\x \in (\Z^2)^4 \\ \y \in (\Z^2)^4}} \Phi(N^{-1/2} \x) \mathbb E\Bigg[\prod_{j=1}^4 H_{Ns,Nt}^N(x_j,y_j) \Bigg]\Psi(N^{-1/2} \y),
\end{equation}
along with the analogous quantity for the SHF,
\begin{equation*}
    \mathcal G^{\mathrm{SHF}}_{s,t} (\Phi,\Psi) := \int_{(\mathbb R^2)^4} \Phi(\x) \Psi(\y)\mathbb E\Bigg[\prod_{j=1}^4 Z^\vartheta_{s,t}(dx_j,dy_j) \Bigg].
\end{equation*}

\begin{thm}[Exact limit of the fourth moment]\label{4thmom}
    Consider sequences $s_N,t_N\in N^{-1} \mathbb Z$ such that $s_N\to s$ and $t_N\to t$ with $0\le s<t\le 1$. Under Assumption \ref{a1}, we have 
    \begin{align*}\lim_{N\to \infty}  \mathcal G^N_{s_N,t_N} (\Phi,\Psi)&= \mathcal G^{\mathrm{SHF}}_{s,t} (\Phi,\Psi),
    \end{align*}
   for all nonnegative $\Phi,\Psi \in C_c\big((\mathbb R^2)^4\big)$. The analogous results hold for the third and fifth moments.
\end{thm}



Note that if $\Phi = \phi^{\otimes 4}$ and $\Psi = \psi^{\otimes 4}$, with $\phi,\psi \in C_c^\infty(\mathbb R^2)$, then $\mathcal G^N_{s,t} (\Phi,\Psi) = \mathbb E[ \mathfrak H^N_{s,t}(\phi,\psi)^4], $ where $\mathfrak H^N$ is the macroscopic field from Definition \ref{a2}. 

We will break the proof into shorter increments, split across Subsections \ref{first}--\ref{last}. 
The main idea is a stopping-time argument based on when collisions for new partition patterns first occur.

    \subsection{Setup for the higher moments calculation}\label{first}

    First we will set up some notation.

    \begin{defn}\label{precsim}
    For any vector $\x = (x_1,x_2,x_3,x_4) \in (\mathbb Z^2)^4$,
    define the relation $R_{\x} \subset \{1,2,3,4\}^2$ pairwise on the coordinates of $\mathbf x$ by $(i,j) \in R_{\x} \iff |x_i-x_j| \leq A$.  Here $A$ is the range of dependency in Assumption \ref{a1} \eqref{a1markov}. Extend this relation by transitivity, so that the resulting equivalence relation induces a partition $I$ of $\{1,2,3,4\}$. Write in this case $\mathbf x \sim I$.

    We write $I \vdash \{1,2,3,4\}$ if $I$ is a partition of $\{1,2,3,4\}$.
    We denote by $\hat{0} := \{\{1\}, \{2\}, \{3\}, \{4\}\}$ the \textit{trivial partition} consisting of all singletons, and say that $I$ is a \emph{nontrivial} partition of $\{1,2,3,4\}$ if $I \neq \hat{0}$.
    Finally, we say that $\mathbf x \precsim I$ if $x \sim I$ or $x \sim \hat{0}$. 
    
\end{defn}

\begin{remark}\label{disj}
    Our convention on the symbol $\sim$ implies that the regions $\{\x : \x \sim J\}$ are disjoint subsets of $(\mathbb Z^2)^4$ for distinct partitions $J\vdash \{1,2,3,4\}$. 
    Thus, a given vector $\x$ uniquely determines the partition $J$. This is in contrast with \cite{CSZ_2d}, where the authors adopt the convention that $\x \sim I$ implies $\x\sim J$ if the partition $J$ is a \textit{refinement} of the partition $I$.
\end{remark}

\begin{defn}[Partition update times]\label{stop}
Consider the Markov chain $\mathbf P_{N,4}$ and let $(\mathbf R_t)_{t\ge 0}$ denote its realization. Define a sequence of stopping times $0 < \tau_1< \dots <\tau_M < Nt $ such that the partition pattern of $\mathbf R_t$ updates to a new non-trivial partition at each stopping time $\tau_i$, and such that the sequence $(\tau_i)$ represents all such partition updates in the interval $[1, Nt - 1]$.
In other words, if $\mathbf R_0 \sim I$, then $\tau_1 := \min\{t \in \mathbb N : \mathbf R_t \sim J \textrm{ for some } J \notin \{\hat 0,I\}\}$, and the other $\tau_i$ are defined inductively in a similar manner. 
\end{defn}

Note that $M$ is random, and represents the total number of partition updates on the interval $[1,Nt - 1].$ Moreover, $\{M = 0\}$ is exactly the event where $\tau_1 \geq Nt$. We adopt this convention for all processes we consider. Now, if $I_1,\dots,I_M \vdash \{1,2,3,4\}$ is the corresponding random sequence of nontrivial partitions observed during the time interval $[1,Nt - 1]$, we can view the whole probability space as the disjoint union 
\begin{equation}\label{decomp}
    \bigcup_{m \ge 0} \bigcup_{\{0 < s_1 < \dots < s_m < Nt\}} \bigcup_{\substack{J_1, \dots ,J_m \vdash \{1,2,3,4\} \\ J_i\ne J_{i-1}, J_i \neq \hat{0} }} \{ M=m, \tau_j=s_j, I_j=J_j, \forall j\}.
\end{equation}
The overarching plan is to prove Theorem \ref{4thmom} by separating the sample space according to this decomposition, and then analyzing the associated terms separately. 

Thus applying the Markov representation of moments in Assumption \ref{a1} \eqref{a1markov} and using the decomposition \eqref{decomp}, we arrive at the following expansion of the fourth moment.

\begin{prop}[Partition expansion --- Version I]\label{part_v1}
For $t\in N^{-1}\mathbb Z_{\ge 0}$ and $\Phi,\Psi\ge 0$, we have
\begin{align}\label{s=1234a}
    \notag \mathcal G^N_{0,t}(\Phi,\Psi)= N^{-4} &\sum_{m\ge 0} \sum_{0<s_1< \dots <s_m < Nt} \sum_{\substack{J_1,\ldots,J_m \vdash \{1,2,3,4\} :\; \\
    J_i \ne J_{i-1}, J_i \neq \hat{0}}}\sum_{\x \in (\Z^2)^4} \Phi(N^{-1/2}\x) \\
    &\times \mathbf E_{N,4}^\mathbf x \bigg[\ind_{\{ M=m, \tau_j=s_j, I_j=J_j, \forall j=1,\ldots,m\}} e^{\mathscr E^{N,4}_t+ \sum_{r=0}^{Nt-1} \h_{N,4}(\mathbf R_r)} \Psi\big(\mathbf R_{Nt}/\sqrt N\big) \bigg],
\end{align}
where $\h_{N,4}$ is the additive functional appearing in the moment representation of Assumption \ref{a1}.
\end{prop}


Here is an important reduction that we make: without loss of generality, we can say that $\mathscr E^{N,4}_t=0$, because we know from Assumption \ref{a1} that $\mathscr E^{N,4}_t\to 0$ in $L^\infty$ uniformly over all initial conditions for $\PN{4}$, and therefore \eqref{s=1234a} can be rewritten with a $1+o(1)$ prefactor and without $\mathscr E^{N,4}_t$, where the asymptotic term is uniform over all parameters. Thus we ignore the errors $\mathscr E^{N,h}_t$ henceforth. 

An important role will be played by \emph{pair partitions}.

\begin{defn}
    We call $J\vdash \{1,2,3,4\}$ a \emph{pair partition} if $J$ contains two singletons and one pair.
\end{defn}


%

The expression \eqref{s=1234a} is still not in a form that is amenable to analysis. Our task is now to express this quantity in a more manageable way, so as to control it.
%
The main idea will be to rewrite this in terms of the renewal functions $U_N(f,t,x,x',y,y')$ which were defined in \eqref{eq:renewal}.

\begin{defn}[Notational conventions]\label{nota}
We will adopt the following conventions:
\begin{itemize}[leftmargin=15pt]
    \item Given a collection $T$ of subsets of $\{1,2,3,4\},$ we will write $T^{(r)}$ for the union of all $R \in T$ such that $|R| = r$.
    So if $J$ is a partition, the elements of $J^{(1)}$ are the singletons of $J$.
    \item Given a subset $B \subset \{1,2,3,4\}$ and $\x = (x_1,x_2,x_3,x_4) \in (\Z^2)^4$, we define the restriction operator $\x \mapsto \x|_B : (\Z^2)^4 \to (\Z^2)^{|B|}$ by $
    \x|_B = (x_n : n \in B).$ 
    \item We write $\gap_J \mathbf b$ for the gap\footnote{
        Technically the notion of gap is not well-defined: one may just as well define it as $x_2-x_1$ instead of $x_1-x_2$. But this sign difference will never matter, because the gap will always be input into an even function on $\mathbb Z^2$. For precision, we take the convention that the smaller index goes first.
    }
    between the pair in $J$. That is, if $J = \{\{1,2\}, \{3\},\{4\}\}$ and $\mathbf b= (b_1, \dots ,b_4)$, then $\gap_J \mathbf b= b_1-b_2$.
    \item  Define the set 
    \begin{equation*}
        (\mathbb{Z}^2)_J^h:=\big\{\mathbf{x} \in(\mathbb{Z}^2)^h: \mathbf x \sim J \big\}.
    \end{equation*}
    Fix $1 < q <\infty$. Let $\ell^q((\mathbb{Z}^2)_J^h)$ be the Banach space of functions $f:(\mathbb{Z}^2)_J^h \rightarrow \mathbb{R}$ such that $\norm{f}_q := (\sum_{\mathbf x \in (\mathbb{Z}^2)_J^h}|f(\mathbf x)|^q)^{1 / q}<\infty$. 
    \item If $f\in \ell^q((\mathbb{Z}^2)^4)$ and $g\in \ell^p((\mathbb{Z}^2)^4)$ with $\frac1p + \frac1q=1$, we will write $\langle f,g\rangle := \sum_{\x \in (\mathbb Z^2)^4} f(\x) \mathbf g(\x).$
\end{itemize}
\end{defn}
To provide an example of how we use this notation, let $J$ be a pair partition. Then $\mathbf b|_{J^{(1)}}$ encodes the positions of the two singletons in $J$. If $J = \{\{1,2\},\{3\},\{4\}\}$, we could then write $\mathbf p_{N, 1}^{\otimes 2}(r,\mathbf b|_{J^{(1)}},\mathbf c|_{J^{(1)}})$ for $\mathbf p_{N,1}(r,b_3,c_3)\mathbf p_{N,1}(r,b_4,c_4)$.
    

    \subsection{Operator expansions}
    In this section, we introduce operators that encode the partition expansion \eqref{part_v1}. These operators are indexed by partitions representing the collision patterns between particles. 
    There are also situations where the specific collision pattern is irrelevant. We will therefore introduce the symbol $*$ when no specific constraint is imposed on the collision pattern of $\mathbf x$. We will abuse notation, allowing for statements like $I = *$ and moreover adopting the conventions that $\mathbf x \sim *$ for all $\mathbf x \in (\mathbb Z^2)^4$ and $(\mathbb Z^2)^h_*=(\mathbb Z^2)^h$. 

Define a family of operators as follows. 

\begin{defn}\label{pij} For any partitions $I, J \vdash \{1,2,3,4\}$, let $\mathscr P_N^{I,J}(t) : \ell^p((\mathbb Z^2)_J^4)\to \ell^p((\mathbb Z^2)_I^4)$ be the operator defined by the kernel
\begin{equation*}
        \mathscr P_N^{I,J}(t,\x,\y) = \ind_{\{\x\sim I, \y\sim J\}} \mathbf E^{\x}_{N,4} \bigg[ e^{\sum_{r=0}^{t-1} \boldsymbol{\eta}_{N,4}(\mathbf R_r)} \ind_{\{\mathbf R_t=\y, \; \mathbf R_r \precsim I\; \forall 0\le r<t\}} \bigg],
\end{equation*}
where $t>0$ and $\x,\y\in (\Z^2)^4.$ Note that this definition makes sense even if $I$ or $J = \hat{0}$. 
\end{defn}


We immediately have the following result.

\begin{prop}[Partition expansion --- Version II]\label{op_exp} The quantity $\mathcal G_{0,t}^N (\Phi,\Psi)$ from \eqref{gnst} can be rewritten as
\begin{equation*}
    \sum_{m\ge 0} \sum_{0 <s_1<\dots<s_m < Nt} \sum_{\substack{J_0, \ldots,J_{m+1} \vdash \{1,2,3,4\} \\ J_i \ne J_{i-1}  \text{ for } 1 \leq i \leq m \\ J_i \neq \hat{0} \text{ for } 1 \leq i \leq m}} \brac*{\Phi_{(N)} , \mathscr P^{J_0,J_1}_N(s_1)\mathscr P^{J_1,J_2}_N(s_2-s_1) \cdots \mathscr P^{J_m,J_{m+1}}_N(Nt-s_m) \Psi_{(N)}},
\end{equation*}
where $\Phi_{(N)}(\x) = N^{-2} \Phi(N^{-1/2}\x).$ 
\end{prop}

\begin{proof}
   Use \eqref{s=1234a}, disintegrate along all possible spatial and temporal locations for the given stopping times $\tau_j$, and then apply the Markov property for $\PN{4}$. This immediately yields the claim. 
\end{proof}

We will now further decompose the transition operators defined above.

\begin{defn}[Constrained evolution operators]\label{cons}
    Let $I,J\vdash \{1,2,3,4\}$ be partitions. We define the operators $\mathscr Q_{N,\mathrm{con}}^{I,J} (t): \ell^p((\mathbb Z^2)_J^4)\to \ell^p((\mathbb Z^2)_I^4)$ and $\mathscr U_{N,\mathrm{con}}^{J} (t): \ell^p((\mathbb Z^2)_J^4)\to \ell^p((\mathbb Z^2)_J^4)$ respectively by the following kernels:
    \begin{align*}
        \mathscr Q_{N,\mathrm{con}}^{I,J} (t,\x,\y)& := \ind_{\{\x\sim I, \y\sim J\}} \mathbf E^{\x}_{N,4} \bigg[ \ind_{\{\mathbf R_t=\y, \;\mathbf R_r \precsim I\; \forall 0\le r<t\}} \bigg],\\
        \mathscr U^J_{N,\mathrm{ con}}(t,\mathbf x,\mathbf y)&:=\begin{cases} \big( e^{\boldsymbol{\eta}_{N,4}(\x)} - 1\big)\ind_{\{\x=\y\}} \ind_{\{\x\sim J\}}, & t=0 \\[5pt] \big( 1-e^{-\boldsymbol{\eta}_{N,4}(\x)}\big)
        \mathscr P_N^{J,J}(t,\x,\y)\big( e^{\boldsymbol{\eta}_{N,4}(\y)} - 1\big) \ind_{\{\x,\y\sim J\}}, & t>0.\end{cases} 
    \end{align*}
Note that $ \mathscr U^{\hat{0}}_{N,\mathrm{ con}}(t,\mathbf x,\mathbf y) = 0$ because if $\x \sim \hat{0}$ then $\big( e^{\boldsymbol{\eta}_{N,4}(\y)} - 1\big) = 0$.
\end{defn}
Here the word \emph{constrained} refers to the restriction on intermediate collision patterns by conditions such as $\ind_{\{\mathbf R_r \precsim I\}}$. We will later define \emph{free} versions of these operators without this constraint.

With this notation at hand, we have the following identity.
\begin{lem}\label{recur}For any two partitions $I,J$, we have the recursion 
\begin{equation*}\mathscr P_N^{I,J}(t) =  \mathscr Q_{N,\mathrm{con}}^{I,J}(t)  + \sum_{0\le u \le v< t} \mathscr Q^{I,I}_{N,\mathrm{con}}(u) \mathscr U^I_{N,\mathrm{con}} (v-u) \mathscr Q^{I,J}_{N,\mathrm{con}}(t-v).
\end{equation*}
\end{lem}

\begin{proof}
    By considering the two different cases for $\mathscr U^J_{N,\mathrm{ con}}$, the above recursion can be rewritten as $$\mathscr P_N^{I,J}(t) =  \mathscr Q_{N,\mathrm{con}}^{I,J}(t) + \sum_{u=0}^{t-1} \mathscr N^{I}_N(u)\mathscr Q_{N,\mathrm{con}}^{I,J}(t-u) + \sum_{0\le u <v< t} \mathscr Q^{I,I}_{N,\mathrm{con}}(u) \mathscr U^I_{N,\mathrm{con}}(v-u) \mathscr Q^{I,J}_{N,\mathrm{con}}(t-v),$$ 
    where $\mathscr N_N^{I}(t,\x,\y) = \mathscr Q^{I,I}_{N,\mathrm{con}} (t,\x,\y) \big( e^{\boldsymbol{\eta}_{N,4}(\mathbf y)} - 1\big).$ This is the version we prove.  Expand $$e^{\sum_{r=0}^{t-1} \boldsymbol{\eta}_{N,4}(\mathbf R_r)}=\prod_{r=0}^{t-1} (1+ e^{\boldsymbol{\eta}_{N,4}(\mathbf R_r)} -1)=\sum_{k=0}^\infty \sum_{0\le r_1<\dots<r_k <t} \prod_{j=1}^k (e^{\boldsymbol{\eta}_{N,4}(\mathbf R_{r_j})} -1). $$
    From this expansion, separate three different quantities: the $k=0$ term, the $k=1$ term, and the $k\ge 2$ terms. Then apply the expectation operator  $\mathbf E^{\x}_{N,4} \big[ \big( \cdot\big) \ind_{\{\mathbf R_t=\y, \; \mathbf R_r \precsim I\; \forall 0\le r<t\}} \big]$.

    Notice that the $k=0$ term corresponds exactly to the operator $\mathscr Q_{N,\mathrm{con}}^{I,J}.$ The $k=1$ term corresponds exactly to the operators $\mathscr N_N(t)$. Consequently, the $k\ge 2$ terms will correspond exactly to the remainder $\big( e^{\boldsymbol{\eta}_{N,4}(\x)} - 1\big)\mathscr P_N^{J,J}(t,\x,\y)\big( e^{\boldsymbol{\eta}_{N,4}(\y)} - 1\big)$, with the caveat that an extra factor of $e^{-\boldsymbol{\eta}_{N,4}(\x)}$ should be multiplied in order to start the sum in the exponential in the definition of the operator $\mathscr P^{J,J}_N$ at $r=0$ rather than $r=1$. 
\end{proof}

Recall that we use the symbol $*$ when we impose no restriction on the  collision pattern, i.e., $\x \sim *$ is true for all $\x \in (\Z^2)^4$.
\begin{defn}[Free heat kernel operators]\label{ufree} For any two nontrivial partitions $I,J\vdash \{1,2,3,4\}$, including $I = *$ or $J = *$, we define the free heat kernel operator
\begin{equation*}
    \mathscr Q_{N,\mathrm{free}}^{I,J}(t, \x,\y):=\mathbf p_{N,4}(t,\x,\y) \ind_{\{ \x\sim I, \y\sim J\}}.
\end{equation*}
\end{defn}
Note that $\mathscr Q^{*,*}_{N,\mathrm{free}} (t,\x,\y) = \mathbf p_{N,4}(t,\x,\y)$, which acts as the unconstrained 4-point heat kernel operator on the unrestricted space $\ell^2( (\mathbb Z^2)^4).$ It is also important to realize that the operators do \textit{not} satisfy a nice semigroup identity, that is, $\mathscr Q^{I,J}_{N,\mathrm{free}}(r)\mathscr Q^{J,K}_{N,\mathrm{free}}(t-r) \ne \mathscr Q_{N,\mathrm{free}}^{I,K}(t).$ This is due to the boundary conditions, and the equality would only be true if $J=*$.

The following set of operations on these transition kernels allows us to further reduce \eqref{op_exp}.

\begin{defn}[Chained operators]\label{def:chained}
 Let $I,J \vdash\{1,2,3,4\}$ be any two set partitions. We define 
    \begin{equation}\label{3.11}
        \mathscr Q^{I,J}_{N,\mathrm{chain}} (t, \x,\y) := \bigg(\sum_{m=0}^{\infty} \,\sum_{\substack{J_0,\dots,J_{m-1} \vdash \{1,2,3,4\} : \\ J_i \ne J_{i-1}, J_i \neq \hat{0}}} \,\sum_{\substack{0< s_0 <\dots<s_{m} \leq t}} \prod_{i=0}^{m+1} \mathscr Q_{N,\mathrm{con}}^{J_{i-1},J_i} (s_i-s_{i-1}) \bigg) (\x,\y),
    \end{equation}
    where we always adopt the convention that $J_{-1}=I$, $J_{m}=J_{m+1}=J$, $s_{-1}=0$, and $s_{m+1}=t$.
    We further define
    \begin{align*}\label{3.11}
        \mathscr Q^{*,J}_{N,\mathrm{chain}} &:= \mathscr Q^{J,J}_{N,\mathrm{con}}+\sum_{I \vdash \{1,2,3,4\}}\mathscr Q^{I,J}_{N,\mathrm{chain}} ,\qquad  \mathscr Q^{I,*}_{N,\mathrm{chain}} (t, \x,\y) := \mathscr Q^{I,I}_{N,\mathrm{con}}+\sum_{J \vdash \{1,2,3,4\}}\mathscr Q^{I,J}_{N,\mathrm{chain}}, \qquad \\
        \mathscr Q^{*,*}_{N,\mathrm{chain}} &:= \sum_{I,J \vdash \{1,2,3,4\}}\mathscr  Q^{I,J}_{N,\mathrm{chain}}.
    \end{align*}    
\end{defn}

\begin{thm}[Contraction identities]\label{contraction} Let $I,J \vdash\{1,2,3,4\}$ be any two set partitions or $*$, with $\mathscr Q^{I,J}_{N,\mathrm{chain}} (t, \x,\y)$ as in Definition \ref{def:chained}. We have the following allowable operator contractions.

\begin{enumerate}     \item\label{thm:contraction1} $\mathscr Q^{I,J}_{N,\mathrm{chain}} (t, \x,\y)= \mathscr Q^{I,J}_{N,\mathrm{free}}(t,\x,\y)$ if $I\ne J$.
  \item\label{thm:contraction2} $\mathscr Q^{I,J}_{N,\mathrm{chain}} (t, \x,\y)=\mathscr Q^{I,I}_{N,\mathrm{free}} (t,\x,\y) - \mathscr Q^{I,I}_{N,\mathrm{con}}(t,\x,\y)$ if $I = J$, which is in turn upper bounded by the kernel $\big[\sum_{r=1}^{t-1} \sum_{K\notin \{I,\hat 0\}} \mathscr Q^{I,K}_{N,\mathrm{free}}(r)\mathscr Q^{K,I}_{N,\mathrm{free}}(t-r)\big](\x,\y)$ if $I\ne *$, where the sum is over all nontrivial partitions $K\vdash \{1,2,3,4\}$ that are different from $I$. Here $\mathscr Q_{N,\mathrm{con}}^{*,*} = \sum_{K\vdash \{1,2,3,4\}} \mathscr Q_{N,\mathrm{con}}^{K,K}$ in the case where $I=*.$
    \end{enumerate}

    
\end{thm}


\begin{proof}\textbf{Proof of Item \eqref{thm:contraction1}.}
    To illustrate the main idea, we first prove a simpler identity, namely that the $m=0$ term of $\mathscr Q^{I,J}_{N,\mathrm{chain}} (\x,\y)$ is upper bounded by the right side. To be explicit, we have
    \begin{equation*}
        \sum_{s=0}^t \sum_{\mathbf a \in \mathbb (\mathbf Z^2)^4_J} \EN{4}^{\mathbf \x}\bigg[\ind_{\{\mathbf R_0 \sim I, \tau_1 = s\}}  \ind_{\{\mathbf R_{s} =\mathbf a \}} \bigg] \EN{4}^{\mathbf a} \bigg[\ind_{\{\tau_1 > t - s\}} \ind_{\{\mathbf R_{t-s} = \mathbf y\}}\bigg]\leq  \ind_{\{\mathbf x \sim I,\mathbf y \sim J\}} \mathbf p_{N,4} (t,\mathbf \x, \mathbf y). 
    \end{equation*}
   This is not immediate from the semigroup property alone, since we are also summing over $s$. What justifies the bound is the constrained evolution of $\mathscr Q_{N,\mathrm{con}}^{I,J}(s)$, which ensures that this sum over $s$ is exactly a decomposition of the first partition update time $\tau_1$. In other words, for a fixed value of $s$, the inner sum on the left hand side is the probability of  $\mathbf x \to \mathbf y$ under $\PN{4}$ such that the unique partition update in the time interval $[0,t]$ is $I \to J$ and occurs precisely at time $s$. Thus distinct $s$ correspond to disjoint events, so that summing over $s$ and using the Markov property yields the above inequality.
    
    For the full sum over all possible $m\ge 0$, the argument is very similar, except we consider all possible $m$-tuples of partition updates, and the time $s_i$ corresponds to the time of the $i$\textsuperscript{th} update. In particular, since we define $J_m = J_{m+1} = J$, we enter state $J$ for the final time at time $s_m$ and remain there until time $s_{m+1} = t$. This is significant, because if we did not take $J_m = J_{m+1}$, then we could never be in state $J$ at time $t-1$ unlike in $\mathscr Q^{I,J}_{N,\mathrm{free}}(t,\x,\y)$. When summing over $m$, this gives us all possible sequences of intermediate collision patterns, in particular, allowing us to enter $J$ for the final time at any time in the interval $[1, t]$. Therefore we obtain $\mathscr Q^{I,J}_{N,\mathrm{free}}(t,\x,\y)$.
    
    \smallskip
    \noindent\textbf{Proof of Item \eqref{thm:contraction2}.} If $I=J$, then importantly the $m=0$ term of the series is 0, the sum being empty due to the constraint $J_i\ne J_{i-1}$.
    Given this observation, the proof of Item \eqref{thm:contraction2} is extremely similar to that of Item \eqref{thm:contraction1}. The only difference is that since the series starts from $m=1$, there must be some intermediate time $r\in \{1,\dots,t-1\}$ at which we enter a nontrivial partition pattern $K$ that is different from $I$, which is why we need to subtract $\mathscr Q^{I,I}_{N,\mathrm{con}}(t,\x,\y)$. 
    Taking a union bound over all such $K$ then yields the claimed upper bound.
\end{proof}

\subsection{The final expansion}

The observations above lead to the following expansion, which is finally in a tractable form for taking the $N\to \infty$ limit.

\begin{thm}[Partition expansion --- Version III]\label{part_v2.5}
    For pair partitions $I,J\vdash \{1,2,3,4\},$ let $\mathscr Q^{I,J}_{N,\mathrm{chain}}$ be as in \eqref{3.11}. Let $\Phi, \Psi\ge 0$. Then the quantity $\mathcal G_{0,t}^N (\Phi,\Psi)$ from \eqref{gnst} is exactly equal to $\sum_{m=0}^\infty \mathfrak I_N(m;t,\Phi,\Psi)$ where $$\mathfrak I_N(0,t,\Phi,\Psi) := \big\langle \Phi_{(N)} , \mathscr Q_{N,\mathrm{free}}^{*,*} (Nt) \;\Psi_{(N)}\big\rangle,$$ and for $m\ge 1$ one has
    \begin{align*}
        \mathfrak I_N&(m;t,\Phi,\Psi)\\
        := &\sum_{\substack{u_1\le v_1 < \dots \\ < u_m\le v_m<Nt} } \sum_{\substack{J_1,\dots,J_m \vdash \{1,2,3,4\} \\ J_i \neq \hat{0}}} \bigg \langle \Phi_{(N)} , \mathscr Q_{N,\mathrm{chain}}^{*,J_1} (u_1)\bigg[ \prod_{i=1}^m \mathscr U_{N,\mathrm{con}}^{J_i}(v_i-u_i) {\mathscr Q}_{N,\mathrm{chain}}^{J_i,J_{i+1}}  (u_{i+1} - v_i)\bigg]
    \Psi_{(N)}\bigg\rangle  ,
    \end{align*}
where we do\textbf{ not} impose the constraint $J_{i-1}\ne J_i$ on the summands, and we take $u_{m+1}=Nt$ and $J_{m+1}=*$. As always, $\Phi_{(N)}(\x):=N^{-2} \Phi(N^{-1/2}\x)$ and likewise for $\Psi_{(N)}.$ 
\end{thm}

\begin{proof}
    First consider the equivalent expression for $\mathcal G^N_{0,t}$ that is given in Proposition \ref{op_exp}. Then expand each copy of $\mathscr P^{I,J}_N(t)$ according to the decomposition in Lemma \ref{recur}. In this large expansion, note that if we have many consecutive copies of $\mathscr Q^{J_{i-1},J_i}_{N,\mathrm{con}}$, then these may all be concatenated into a single copy of $\mathscr Q_{N,\mathrm{chain}}$. Rearranging terms in this way, all appearances of the variables $s_i$ will then go away completely (they are ``contracted away"), leaving only the variables $u_i,v_i$, thus exactly yielding the stated bound.
\end{proof}

We remark that the value of $m$ in the above expansion no longer has the same meaning as in the previous subsection. In Propositions \ref{part_v1} and \ref{op_exp}, the value $m$ represented the number of stopping times (i.e., novel collision pattern updates) in the time interval $[1,Nt - 1]$. Here, infinitely many different values of $m$ from the original expansion will now contribute to a single value of $m$ in the new expansion, thanks to the contractions occurring at every level.

\begin{remark}[Comparing the moment expansion with \cite{GQT,CSZ_2d, MR4945087} and related works]\label{part_v5} One can further contract terms in the above expansion to simplify even further. Specifically, if we define the kernels $$\mathscr V_N^J(t,\x,\y) := \ind_{\{\x,\y\sim J\}} (1-e^{-\boldsymbol{\eta}_{N,4}(\x)}) \mathbf E^{\x}_{N,4} \bigg[ e^{\sum_{r=1}^{t-1} \boldsymbol{\eta}_{N,4}(\mathbf R_r)\ind_{\{\mathbf R_r\sim J\}}} \ind_{\{\mathbf R_t=\y \}} \bigg]  (e^{\boldsymbol{\eta}_{N,4}(\y)} -1),$$ then see that $\mathscr V^J_N$ is an unconstrained counterpart of the renewal operator $\mathscr U_{N,\mathrm{con}}$ from Definition \ref{cons}. Then we have one more contraction identity,
\begin{align}
    \notag \mathscr V_N^J&(t,\x,\y) = \mathscr U_{N,\mathrm{con}}^J(t,\x,\y) \\+ \bigg(&\sum_{m=1}^\infty \sum_{\substack{ v_0 < u_1 \le v_1 < \\\dots\le v_{m-1}< u_{m} \leq t}} \bigg[\prod_{j=0}^{m-1} \mathscr U^J_{N,\mathrm{con}} (v_j-u_{j}) \big(\mathscr Q^{J,J}_{N,\mathrm{free}}-\mathscr Q^{J,J}_{N,\mathrm{con}} \big) (u_{j+1}-v_j) \bigg] \mathscr U^J_{N,\mathrm{con}}(t-u_{m}) \bigg)(\x,\y),\label{contract2}
\end{align}
 where we agree that $u_0=0$. Proving this identity is similar to the arguments above. First, start from the original expression for $\mathscr V_N,$ and consider the same collection of stopping times as in Definition \ref{stop}. Then use these stopping times and the Markov property to generate an operator expansion for $\mathscr V_N$ similar to Proposition \ref{op_exp}. Then use the contraction identities of Theorem \ref{contraction}, and we will arrive exactly at the above series expansion \eqref{contract2} for $\mathscr V_N^J$. Note that one will only see contractions of the form $I=J$, which will then yield terms of the form $\big(\mathscr Q^{J,J}_{N,\mathrm{free}}-\mathscr Q^{J,J}_{N,\mathrm{con}} \big)$ as in Item \eqref{thm:contraction2} of Theorem \ref{contraction}.
 
    Plugging \eqref{contract2} back into Therorem \ref{part_v2.5}, we can contract all terms involving a consecutive partition, noting that $\mathscr Q^{J,J}_{N,\mathrm{chain}} = \mathscr Q^{J,J}_{N,\mathrm{free}}-\mathscr Q^{J,J}_{N,\mathrm{con}}$ by Theorem \ref{contraction}. Then we will see that the quantity $\mathcal G_{0,t}^N (\Phi,\Psi)$ from \eqref{gnst} is exactly equal to $\sum_{m=0}^\infty \widehat{\mathfrak I_N}(m;t,\Phi,\Psi)$ where $$\widehat{\mathfrak I_N}(0,t,\Phi,\Psi) := \big\langle \Phi_{(N)} , \mathscr Q_{N,\mathrm{free}}^{*,*} (Nt) \;\Psi_{(N)}\big\rangle,$$ and for $m\ge 1$ one has \begin{align*}\widehat{\mathfrak I_N}&(m;t,\Phi,\Psi)\\
    := &\sum_{\substack{u_1\le v_1<u_2 \le v_2<\dots\\ < u_m\le v_m<Nt} } \;\sum_{\substack{J_1,\dots,J_m \vdash \{1,2,3,4\} \\ J_i\ne J_{i-1}, J_i \neq \hat{0}}} \; \bigg \langle \Phi_{(N)} , \mathscr Q_{N,\mathrm{free}}^{*,J_1} (u_1)\bigg[ \prod_{i=1}^m \mathscr V_{N}^{J_i}(v_i-u_i) {\mathscr Q}_{N,\mathrm{free}}^{J_i,J_{i+1}}  (u_{i+1} - v_i)\bigg]
    \Psi_{(N)}\bigg\rangle  , \end{align*}
where we \emph{do} impose the constraint $J_{i-1}\ne J_i$ on the summands, we agree that $u_{m+1}=Nt$ and $J_{m+1}=*$. As always, $\Phi_{(N)}(\x):=N^{-2} \Phi(N^{-1/2}\x)$ and likewise for $\Psi_{(N)}.$ 

This expansion for $\widehat{\mathfrak I_N}$ is more similar to the expansions seen in works such as \cite{CSZ_2d}, and we now explain why we will use the uncontracted version in Theorem \ref{part_v2.5} rather than this one with $\widehat{\mathfrak I_N}$. The reason is that before taking limits of these expressions, it will first be necessary to \textit{upper-bound} the summands in the series expansion. It seems very difficult to find a direct argument to upper-bound $\mathscr V_N$ from its definition, without going through the expansion \eqref{contract2} anyways. The reason for this is that we do not necessarily have $\mathbf P_{N,4} =\mathbf P_{N,2}^{\otimes 2}$ for the models that satisfy Assumption \ref{a1}. Thus, in deriving bounds, we always work with $\mathscr U_{N,\mathrm{con}}$ rather than $\mathscr V_N,$ and because of this, it is easier to leave $\mathscr V_N$ expanded as \eqref{contract2}.
\end{remark}

\subsection{Upper bounds}

While Theorem \ref{part_v2.5} is an exact expansion that will be useful later, we will first need to obtain upper bounds that will be useful for dominated convergence arguments. We pursue this now for the next few subsections.

\begin{defn}[Free evolution operators]\label{ufree1}
If $J\vdash \{1,2,3,4\}$ is a \textit{pair partition}, then we also define the free renewal operator:
\begin{align*} &\mathscr U_{N,\mathrm{free}}^J(t,\mathbf x,\mathbf y)\\&:=\begin{cases}(e^{\z_N(\gap_J(\x))}-1) \ind_{\{\x=\y\}} \ind_{\{\x\sim J\}} ,& t=0,\\\left( e^{\z_N(\mathrm{gap}_{J} \mathbf x)} - 1\right)U_N( e^{\z_N} - 1,t, \mathbf x | _{J^{(2)}},\mathbf y| _{J^{(2)}}) \left( e^{\z_N(\mathrm{gap}_{J} \mathbf y)} - 1\right)\mathbf p_{N, 1}^{\otimes 2}(t, \mathbf x|_{J^{(1)}}, \mathbf y|_{J^{(1)}}), &t>0.
\end{cases}
\end{align*}
Here $U_N$ is the renewal function from \eqref{eq:renewal}, and $J^{(i)}$ denotes the collection of $i$-element sets in $J$, as in Definition \ref{nota}.

For nontrivial \textit{non-pair partitions} $J\vdash\{1,2,3,4\}$, we also define the quantity 
\begin{align*} &\mathscr U_{N,\mathrm{free}}^J(t,\mathbf x,\mathbf y)\\&:=\begin{cases}(e^{|\boldsymbol{\eta}_{N,4}(\mathbf x)|}-1) \ind_{\{\x=\y\}} \ind_{\{\x\sim J\}} ,& t=0,\\ \ind_{\{\x,\y\sim J\}}(e^{|\boldsymbol{\eta}_{N,4}(\mathbf x)|}-1)  \mathbf E_{N,4}^{\mathbf x} \bigg[ e^{\sum_{r=1}^{t-1} |\boldsymbol{\eta}_{N,4} (\mathbf R_r)| \ind_{\{\mathbf R_r \sim J\}}}  \ind_{\{\mathbf R_t=\mathbf y\}} \bigg]  (e^{|\boldsymbol{\eta}_{N,4}(\mathbf y)|}-1), &t\ge 0,\end{cases} 
\end{align*}
\end{defn}

It should be emphasized that the above definitions are somewhat artificial because of the absolute values, in the sense that they are designed for upper bounds rather than exact expansions as we derived in the previous section.

\begin{lem}\label{u_pair} For\textrm{ pair} partitions $I,J \vdash \{1,2,3,4\}$, we can express $$\mathscr P_N^{I,J}(t,\x,\y) = \ind_{\{\x\sim I, \y\sim J\}} \mathbf E^{\x}_{N,4} \bigg[ e^{\sum_{r=0}^{t-1} \z_N(\gap_J(\mathbf R_r))} \ind_{\{\mathbf R_t=\y, \;\mathbf R_r \precsim I\; \forall 0\le r<t\}} \bigg].$$
\end{lem}

\begin{proof}
    Here we are only looking at pair partitions, so Assumption \ref{a1} implies that the additive functional $\boldsymbol{\eta}_{N,4}$ splits according to the formula $\boldsymbol{\eta}_{N,h}(\x) = \sum_{1\le i<j\le 4} \z_N(x_i-x_j)$ for any $\x=(x_1,x_2,x_3,x_4)$ that lies in the trajectory of the 4-tuple of paths $\mathbf R$. For a given \textit{pair} partition, only one of the terms of this sum will be nonzero, precisely corresponding to the pair of $J$.
\end{proof}

We have an easy upper bound.
\begin{lem} For any nontrivial partition $J\vdash \{1,2,3,4\}$, one has 
\begin{equation}
    \label{tric} \mathscr U^J_{N,\mathrm{con}}(t,\mathbf x,\mathbf y)\le \mathscr U^J_{N,\mathrm{free}}(t,\mathbf x,\mathbf y),
\end{equation}
\end{lem}
\begin{proof} If $J$ is not a pair partition, then trivially the expression $\mathscr U^J_{N,\mathrm{free}}$ has been defined with absolute values so that it upper-bounds the constrained version $\mathscr U^J_{N,\mathrm{con}}$ by construction. 

Thus, we only need to consider pair partitions. Since we are looking only at pair partitions, notice that the definition of the operator $\mathscr U^J_{N,\mathrm{ con}}$ from Lemma \ref{recur} simplifies to $$\mathscr U^J_{N,\mathrm{ con}}(t,\mathbf x,\mathbf y)=\begin{cases} \left( e^{\zeta_N(\mathrm{gap}_{J} \mathbf x)} - 1\right)\ind_{\{\x=\y\}}\ind_{\{\x,\y\sim J\}}, & t=0 \\ \left( 1-e^{-\zeta_N(\mathrm{gap}_{J} \mathbf x)} \right)\mathscr P_N^{J,J}(t,\x,\y)\left( e^{\zeta_N(\mathrm{gap}_{J} \mathbf y)} - 1\right)\ind_{\{\x,\y\sim J\}}, & t>0.
\end{cases} $$ In the definition of $\mathscr P_N^{J,J}(t,\x,\y)$ for pair partitions $J$, note that $\mathbf E_{N,4}$ can be replaced by $\mathbf E_{N,2}\otimes \mathbf E_{N,1}^{\otimes 2}$, simply by construction (it should be assumed that the copy of $\mathbf E_{N,2}$ corresponds to the pair element of the partition $J$). Then, using the expression in Lemma \ref{u_pair}, simply remove all no-collision constraints in the definition of the no-collision kernels, in other words use $\ind_{\{\mathbf R_r \precsim I\; \forall 0\le r<t\}} \le 1$. 
\end{proof}

\begin{cor}[Partition expansion upper bound --- Version I]\label{part_v3}
    Let $\Phi, \Psi\ge 0$. The quantity $\mathfrak I_N$ from Theorem \ref{part_v2.5} satisfies the upper bound for any $m\ge 0$: \begin{align*}\mathfrak I_N&(m;t,\Phi,\Psi)\\&\le \sum_{\substack{u_1\le v_1<u_2 \le v_2<\dots\\ < u_m\le v_m<Nt} } \;\sum_{\substack{J_1,\dots,J_m \vdash \{1,2,3,4\} \\ J_i \neq \hat{0}}} \; \big\langle \Phi_{(N)} , \mathscr Q_{N,\mathrm{free}}^{*,J_1} (u_1)\bigg[ \prod_{i=1}^m \mathscr U_{N,\mathrm{free}}^{J_i}(v_i-u_i) \widehat{\mathscr Q}_{N,\mathrm{free}}^{J_i,J_{i+1}}  (u_{i+1} - v_i)\bigg]
    \Psi_{(N)}\big\rangle  , 
    \end{align*}
where we do\textbf{ not} impose the constraint $J_{i-1}\ne J_i$ on the summands, we agree that $u_{m+1}=Nt$ and $J_{m+1}=*$, and where
\begin{equation}\label{qhat}\widehat{\mathscr Q}_{N,\mathrm{free}}^{I,J} (t):= \begin{cases}
   \mathscr Q_{N,\mathrm{free}}^{I,J}(t) , & I\ne J,\\  \sum_{r=0}^{t-1} \sum_{\substack{K \vdash \{1,2,3,4\} \\ K \ne I, \hat{0}}} \mathscr Q_{N,\mathrm{free}}^{I,K}(r)\mathscr Q_{N,\mathrm{free}}^{K,I} (t-r), & I=J.
\end{cases}
\end{equation}
The sum is over pair partitions $K\vdash \{1,2,3,4\}$ different from $I$.

For $m=0$, we agree that the right side of the upper bound written above should be understood as $\langle \Phi_{(N)} , \mathscr Q_{N,\mathrm{free}}^{*,*} (Nt) \; \Psi_{(N)}\rangle.$ As always, $\Phi_{(N)}(\x):=N^{-2} \Phi(N^{-1/2}\x)$ and likewise for $\Psi_{(N)}.$
\end{cor}

\begin{proof}
    We apply the previous corollary. The contraction identities of Theorem \ref{contraction} exactly state that $\mathscr Q^{I,J}_{N,\mathrm{chain}}(t,\x,\y) \leq \widehat{\mathscr Q}^{I,J}_{N,\mathrm{free}}(t,\x,\y)$. Furthermore, we use the bound \eqref{tric} to observe that $\mathscr U_{N,\mathrm{con}}^J (t,\x,\y) \leq \mathscr U_{N,\mathrm{free}}^J (t,\x,\y),$ completing the proof.
\end{proof}

\subsection{Exponential damping of the operators}
Next, we are going to continue pursuing upper bounds for the quantity $\mathfrak I_N$ from Theorem \ref{part_v2.5} by upper bounding the quantity in Corollary \ref{part_v3}. To do this, we introduce an exponential damping in the time variable, and then extend the range of summation of all the time variables to macroscopic order.
The reason for doing this is that the sums then factor into products of integral operators, which can be controlled by a uniform upper bound after the damping. 

\begin{defn}[Damped operators]\label{damp}
Let $I,J\vdash \{1,2,3,4\}$ be nontrivial. We fix $\lambda >0$, and define $$\mathsf U_{\lambda,N}^{J} := \sum_{r=0}^{2N} e^{-\lambda N^{-1}r} \mathscr U_{N,\mathrm{free}}^J(r) , \hspace{1 in} \mathsf Q^{I, J}_{\lambda, N}:=\sum_{r=0}^{2N} e^{-\lambda N^{-1}r} \widehat{\mathscr Q}_{N,\mathrm{free}}^{I,J}(r), $$
where we recall $\mathscr U_{N,\mathrm{free}}^J(r)$ and $\widehat{\mathscr Q}_{N,\mathrm{free}}^{I,J}$ from Definition \ref{ufree1} and \eqref{qhat}, respectively. We also define the operators $\mathsf Q^{I, J}_{\lambda, N}$ in the natural way if $I$ or $J$ are taken to be $*$. We interpret $\mathsf U_{\lambda,N}^{J}$ as an integral operator from $\ell^q\big(\big(\mathbb{Z}^2\big)_J^4\big)$ to $\ell^q\big(\big(\mathbb{Z}^2\big)_J^4\big)$, and likewise we interpret $\mathsf Q^{I, J}_{\lambda, N}$ as an integral operator from $\ell^q\big(\big(\mathbb{Z}^2\big)_I^4\big)$ to $\ell^q\big(\big(\mathbb{Z}^2\big)_J^4\big)$.
Let $p$ be such that $\frac1p + \frac1q = 1$. If $K: \ell^q\big(\big(\mathbb{Z}^2\big)_I^4\big) \to \ell^q\big(\big(\mathbb{Z}^2\big)_J^4\big)$, we define the operator norm

$$
\left\|K\right\|_{\ell^q\to\ell^q} :=\sup _{\substack{f, g:(\mathbb{Z}^2)_J^4 \rightarrow \mathbb{R} \\\|f\|_p=\|g\|_q=1}}\;\; \sum_{\substack{\x\in \left(\mathbb{Z}^2\right)_I^4 \\ \y \in (\mathbb{Z}^2)_J^4}} f(\x) K(\x,\y) \mathbf g(\y)
$$
Always, $\|f\|_p$ is a shorthand notation for $\|f\|_{\ell^p((\mathbb Z^2)^4_J)}$ whenever $f\in \ell^p((\mathbb Z^2)^4_J).$
\end{defn}

Rewriting the above corollary in terms of operator notation, we obtain an upper bound as follows.
\begin{prop}[Partition expansion upper bound --- Version II]\label{part_v4} Let $p,q>1$ with $\frac1p+\frac1q=1$. Let $\Phi, \Psi\ge 0$. The quantity $\mathfrak J_N$ from Theorem \ref{part_v2.5} satisfies the upper bound for $m\ge 1$:  
\begin{align}
    \mathfrak I_N&(m;t,\Phi,\Psi) \notag \\&\le\notag \frac{e^{\lambda} \| \Psi_{(N)}\|_\infty}{N} \sum_{\substack{J_1,\ldots,J_m \vdash \{1,2,3,4\} \\ J_i \neq \hat{0} }} \bigg\langle\Phi_{(N)} \;,\; \mathsf Q_{0,N}^{*,J_1}\left[\prod_{i = 1}^m\mathsf U_{\lambda,N}^{J_{i}} \mathsf Q_{0,N}^{J_{i},J_{i+1}}\right]\ind_{\sqrt N B}\bigg\rangle \\
    &\leq \frac{e^{\lambda}}{N}\sum_{\substack{J_1,\ldots,J_m \vdash \{1,2,3,4\} \\ J_i \neq \hat{0} } }\left\|\Phi_{(N)}\right\|_p \left\|\mathsf Q_{0,N}^{*,J_1}\right\|_{\ell^q\to\ell^q}\left[\prod_{i = 1}^m\big\|\mathsf U_{\lambda,N}^{J_{i}}\big\|_{\ell^q\to\ell^q} \big\| \mathsf Q_{0,N}^{J_{i},J_{i+1}}\big\|_{\ell^q\to\ell^q}\right]\left\|\Psi_{(N)}\right\|_\infty \| \ind_{\sqrt N B}\|_q. \label{operator_bound}
\end{align}
Above, we do\textbf{ not} impose the constraint $J_{i-1}\ne J_i$ on the summands, and we agree that $u_{m+1}=Nt$ and $J_{m+1}=*$, and furthermore $B\subset (\mathbb R^2)^4$ is any ball containing the support of $\Psi.$ As always, $\Phi_{(N)}(\x):=N^{-2} \Phi(N^{-1/2}\x)$ and likewise for $\Psi_{(N)}.$
\end{prop}

The advantage here over Corollary \ref{part_v3} is that the time variables $u_j,v_j$ no longer appear in the expression, thus leading to a simpler and more tractable expansion.

\begin{proof}
Going from the first bound to the second one is trivial, so we explain how to derive the first bound.

We use the same Laplace transform idea as in \cite[Proof of Theorem 6.1]{CSZ_2d}. Fix $\lambda >0$ and multiply the expression in Corollary \ref{part_v3} by $1 = e^\lambda e^{-\frac{\lambda}{N}(u_1 + (v_1 - u_1) + (u_2 - v_1) + \cdots +(v_m - a_{m - 1}) + (N - v_m))}$. In obtaining upper bounds on terms of order $m \geq 2$ in the moment expansion of Corollary \ref{part_v3}, it will suffice to take $t=1$ and then extend the range over which the $m$-block contribution is summed, from the simplex $0 < u_1 \leq v_1 < \cdots < u_m \leq v_m < N$ to the rectangle $\{(\mathbf u,\mathbf v) : u_{i + 1} - v_i \in \{1,\ldots,2N\},u_i - v_i \in \{0,\ldots,2N\}\}$. Ignore the damping terms associated to the $\mathscr Q_N$-factors, in other words upper-bound them by 1. This then \textit{almost} yields that the expression in Corollary \ref{part_v3} is upper-bounded by the first expression above. 


The only missing step is that we do not yet have a valid argument for replacing the final undamped operator $\mathscr Q_N^{J_m,*}$ by its damped version $\mathsf Q^{J_m,*}$ appearing above, and we have also not explained the factor $\frac1N$ appearing as a prefactor in the statement of the proposition. To explain these, we claim (analogously to \cite[(6.11)]{CSZ_2d}) that one has $$ \mathscr Q_{N,\mathrm{free}}^{J,*} (N-v,\y,\ind_{\sqrt N B} ) \le \frac{C}{N} \sum_{u=N}^{2N} \mathscr Q_{N,\mathrm{free}}^{J,*} (u-v, \y,\ind_{\sqrt N B}),$$
uniformly over $N>v$ and $\y\in (\mathbb Z^2)^4.$ We show something stronger: uniformly over $2N>u>N>v$ and $\y\in (\mathbb Z^2)^4$ we have 
\begin{align*}
   \mathbf p_{N,4}(u - v ,\mathbf y,  \ind_{\sqrt N B}) &\ge \sum_{\mathbf z \in (\sqrt{N} B)^4} \mathbf p_{N,4}( N-v, \mathbf y, \mathbf z) \mathbf p_{N,4}( u-N, \mathbf z,\ind_{\sqrt N B})\\& \ge c \sum_{\mathbf z \in(\sqrt{N} B)^4} \mathbf p_{N,4}( N-v, \mathbf y, \mathbf z) \\& = c \mathbf p_{N,4} (N - v ,\mathbf y, \ind_{\sqrt N B}) 
\end{align*}
where $c>0$ is independent of $N$. The first line is simply the Markov property for $\mathbf P_{N,4}$. In the second line, we are using by the invariance principle of Assumption \ref{a1} that uniformly over all $u\in \{N,\dots,2N\}$ and $\mathbf z \in \sqrt N B$, one has that $\mathbf p_{N,4}( u-N, \mathbf z,\ind_{\sqrt{N}B})>c$ with $c>0$ independent of $N$. Indeed, consider any two sequences $t_N \in \{1,\dots,N\}$ and $\mathbf z_N \in \sqrt N B$ with $N^{-1} t_N\to t\in [0,1]$ and $N^{-1/2} \mathbf z_N \to \mathbf z \in (\mathbb R^2)^4$. If $t>0$, the invariance principle says that $\mathbf p_{N,4}(t_N,\mathbf z_N, \sqrt NB)$ necessarily converges as $N\to \infty$ to a positive constant, namely $\mathbf g^{\otimes 4} (t,\mathbf z-B)$, with $\mathbf g^{\otimes 4}$ being the standard heat kernel transition density on $(\mathbb R^2)^4$. On the other hand, if $t=0$, then it is easy to see that $\mathbf p_{N,4}(t_N,\mathbf z_N, \sqrt NB)$ is asymptotically bounded below by $1/2$ (in fact, the limit is 1 if $\mathbf z$ is an interior point of $B$, whereas it could be anything between 1/2 and 1 if $\mathbf z$ is a boundary point). This completes the proof. 
\end{proof}


\subsection{Bounds on the heat kernel operators}\label{q:op}

It is not yet clear that the right side of \eqref{operator_bound} defines a quantity that is summable in $m$. To prove this, we need to study precisely how large the operator norms can be. First, we will bound the $\mathsf Q_{\lambda, N}$ operator norms. We will bound the operators on $\ell^2$ for simplicity. The arguments easily extend to $\ell^q$ with $q\ne 2$, though we will not use this. First, we need two lemmas.

\begin{lemma}\label{lem:GreenBound}
Consider the Green's function $\sum_{t=1}^{N} \mathbf p_{N,4} (t,\mathbf x,\mathbf y) $. There exists $C$ such that for all $\mathbf x, \mathbf y \in (\mathbb Z^2)^4$, we have the estimate
\begin{equation*}
    \sup_N \sum_{t=1}^{N} \mathbf p_{N,4} (t,\mathbf x,\mathbf y)  \leq 
    \frac{C}{(1+|\mathbf y-\mathbf x|^2)^{3}}.
\end{equation*}
\end{lemma}

\begin{proof}

If $\mathbf x =\mathbf y$, then then we just use the bound $\sum_{t = 1}^N \mathbf p_{N,4} (t,\mathbf x ,\mathbf y)  \leq \sum_{t=1}^N t^{-4} \leq  C$.

By Assumption \ref{a1} \eqref{a1heatkernel}, we have that $\mathbf p_{N,4} (t,\mathbf x,\mathbf y) \leq Ct^{-4} (1+t^{-1/2} |\mathbf x-\mathbf y|)^{-S},$ for all $\mathbf x,\mathbf y\in (\mathbb Z^2)^4$ and for some $S > 8$. Therefore, if $\mathbf x \ne  \mathbf y$, we have 
\begin{align*}
 \sum_{t=1}^{N} \mathbf p_{N,4} (t,\mathbf x,\mathbf y) & \leq    C \sum_{t=1}^{N} t^{-4}(1+t^{-1/2} |\mathbf x-\mathbf y|)^{-S}\\
 &\leq C  N^{-3} \int_0^1 s^{-4}(1+s^{-1/2} |\mathbf z_N|)^{-S}ds
\end{align*}
where $\mathbf z_N = (\mathbf x - \mathbf y)/N^{1/2}$. 

We set $r = s^{-1}|\mathbf z_N|^2$. We then obtain 
\begin{align*}
 N^{-3} \int_0^1 s^{-4}(1+s^{-1/2} |\mathbf z|)^{-S}ds &=  (N|\mathbf z_N|^2)^{-3} \int_{|\mathbf z_N|^2}^{\infty} r^{2}(1+r^{1/2})^{-S} dr \\
 &\leq C (N|\mathbf z_N|^2)^{-3}.
\end{align*}
The result follows.

\end{proof}

\begin{lemma}\label{lem:HLStype}
    Fix nontrivial partitions $I,J \vdash \{1,\ldots,4\}$, and require that $I\ne J$.
    If $f \in \ell^2((\mathbb Z^2)^4_I)$ and $g \in \ell^2((\mathbb Z^2)^4_J)$, then
    \begin{equation*}
        \sum_{\mathbf x \in (\mathbb Z^2)^4_I,\mathbf y \in (\mathbb Z^2)^4_J} \frac{f(\mathbf x)g(\mathbf y)}{(1 + |\mathbf x - \mathbf y|^2)^3} \leq C\norm{f}_2\norm{g}_2.
    \end{equation*}
\end{lemma}

\begin{proof}
    This follows from \cite[Lemma 6.8]{CSZ_2d} after accounting for a constant factor picked up from the slight difference in how we define the spaces $(\mathbb Z^2)^4_I$.
\end{proof}

\begin{thm}\label{thm:Qbounds}
    For any two nontrivial partitions $I,J\vdash \{1,2,3,4\}$ (which could be equal \footnote{Note that for $I=J,$ we have defined these operators differently than \cite{CSZ_2d}, see Corollary \ref{part_v3}. If we did not do this, then as pointed out in \cite{CSZ_2d}, the operators $\mathsf Q^{I,I}_{0,N}$ would \textit{not} have an operator norm that is bounded independently of $N$.}), we have that
    \begin{equation}\label{eq:QIJ}
        \sup_{N>1} \|\mathsf Q^{I,J}_{0,N}\|_{\ell^2\to\ell^2} <\infty.
    \end{equation}
    Furthermore, in the free state endpoint cases $I = *$ or $J = *$, we have the estimates
    \begin{equation}\label{eq:Q*J}
         \| \mathsf Q^{*,J}_{0,N}\|_{\ell^2\to\ell^2} \leq CN^{1/2} ,\qquad \|\mathsf  Q^{I,*}_{0,N}\|_{\ell^2\to\ell^2} \leq CN^{1/2}.
    \end{equation}
\end{thm}


\begin{proof}
    To prove \eqref{eq:QIJ}, we need to consider two cases because of the case-wise definition of $\widehat{\mathscr Q}_{N,\mathrm{free}}^{I,J}$ in Corollary \ref{part_v3}. The first case is $I\ne J$. By definition, \eqref{eq:QIJ} says there exists $C > 0$ such that
    \begin{equation*}
        \sum_{\mathbf x \in (\mathbb Z^2)^4_I, \mathbf y \in (\mathbb Z^2)^4_J} f(\mathbf x)\mathsf Q_{0,N}^{I,J}(\mathbf x,\mathbf y)g(\mathbf y) \leq C\norm{f}_2\norm{g}_2
    \end{equation*}
    for all $f \in \ell^2((\mathbb Z^2)^4_I), g\in \ell^2((\mathbb Z^2)^4_J)$,
    and this follows immediately from Lemmas \ref{lem:GreenBound} and \ref{lem:HLStype}.
    
    The second case is $I=J$. In this case, by the expression for $\widehat{\mathscr Q}_{N,\mathrm{free}}^{I,J}$ given in Corollary \ref{part_v3}, we see that one has 
    \begin{align*}\mathsf Q_{0,N}^{I,I} &= \sum_{t=1}^{2N} \sum_{r=0}^{t-1}\sum_{\substack{K \vdash \{1,2,3,4\} \\ K \ne I, \hat{0}}} \mathscr Q_{N,\mathrm{free}}^{I,K} (r) \mathscr Q_{N,\mathrm{free}}^{K,I}(t-r) \\&= \sum_{\substack{K \vdash \{1,2,3,4\} \\ K \ne I, \hat{0}}} \sum_{r=0}^{2N}\mathscr Q_{N,\mathrm{free}}^{I,K} (r) \bigg( \sum_{t=r+1}^{2N} \mathscr Q_{N,\mathrm{free}}^{K,I} (t-r) \bigg)\\&\leq \sum_{\substack{K \vdash \{1,2,3,4\} \\ K \ne I, \hat{0}}} \bigg( \sum_{r=0}^{2N} \mathscr Q_{N,\mathrm{free}}^{I,K}(r)\bigg) \bigg( \sum_{t=0}^{2N} \mathscr Q_{N,\mathrm{free}}^{K,I}(t)\bigg)= \sum_{\substack{K \vdash \{1,2,3,4\} \\ K \ne I, \hat{0}}} \mathsf Q^{I,K}_{0,N} \mathsf Q^{K,I}_{0,N}.
    \end{align*}
    where $\le$ means that one has an inequality of the respective kernels of these operators for all $(\x,\y) \in (\mathbb Z^2)^4\times (\mathbb Z^2)^4.$ Taking operator norms gives $\|Q^{I,I}_{0,N} \| \leq \sum_{\substack{K \vdash \{1,2,3,4\} \\ K \ne I, \hat{0}}} \|\mathsf Q^{I,K}_{0,N} \|\|\mathsf Q^{K,I}_{0,N}\|,$ which by the previous case is bounded independently of $N$.
    
    
To prove (\ref{eq:Q*J}), it is enough to show that $\big\{\mathsf Q_{0,N}^{I,*}\}_N$ is uniformly bounded by the claimed estimate: the complementary bound on $\norm{\mathsf Q_{0,N}^{*,J}}_{\ell^2 \to \ell^2}$ then follows from symmetry in $\ell^2((\mathbb Z^2)^4_I)$ and $\ell^2((\mathbb Z^2)^4)$. Fix $r$ and define the operator $T_r : \ell^2((\mathbb Z^2)^4) \to \ell^2((\mathbb Z^2)^4_I)$ by $$T_r  g(\mathbf x) := \sum_{y \in (\mathbb Z^2)^4} \mathbf p_{N,4}(r,\mathbf x,\mathbf y)g(\mathbf y),$$
    so that $\mathsf Q^{I,*}_{0,N} = \sum_{r = 0}^N  T_r$. It is enough to prove that $\norm{T_r}_{\ell^2((\mathbb Z^2)^4) \to \ell^2((\mathbb Z^2)^4_I)} \leq Cr^{-1/2}$, as it then follows from considering $\lambda = 0$ that $\norm{Q_{\lambda,N}^{I,*}}_{\ell^2 \to \ell^2} \leq CN^{1/2}$ uniformly over $\lambda$ for some $C > 0$.
    The Schur test gives us the bound
    \begin{equation*}
        \norm{T_r}_{\ell^2 \to \ell^2} \leq \sqrt{\sup_{\mathbf x \in (\mathbb Z^2)^4_I} \sum_{\mathbf y \in (\mathbb Z^2)^4} \mathbf p_{N,4}(r,\mathbf x,\mathbf y)} \cdot \sqrt{\sup_{\mathbf y \in (\mathbb Z^2)^4} \sum_{\mathbf x \in (\mathbb Z^2)^4_I} \mathbf p_{N,4}(r,\mathbf x,\mathbf y)},
    \end{equation*}
    the first term of which is trivially bounded above by $1$.
    To control the second term, the heat kernel upper bound is enough.
    First see that $(\mathbb Z^2)^4_I$ is isomorphic to an $M$-fattening of $(\mathbb Z^2)^3$ 
    . In particular, boxes in $(\mathbb Z^2)^4_I$ have the same asymptotic volume and boundary growth rates as in $(\mathbb Z^2)^3$, with a constant correction factor of $|\{w \in \mathbb Z^2 : |w| \leq M\}|$. Therefore the change of variables $r^{-1/2}(\mathbf x - \mathbf y) = \mathbf w$ yields
    $\sum_{\mathbf x \in (\mathbb Z^2)^4_I} \mathbf p_{N,4}(r,\mathbf x,\mathbf y) \leq \frac{C}{r^4} \cdot r^3\int_1^\infty s^6 \cdot s^{-S} ds \leq \frac{C}{r}$, as $S > 8 > 6$.
    The claimed bound on $T_r$ then follows, which concludes the proof.
\end{proof}

A more direct proof by convolution $\sum_{\mathbf y}\mathbf p_{N,4}(r,\mathbf x,\mathbf y)\mathbf p_{N,4}(r,\mathbf w,\mathbf y)$, which can be thought of as a $TT^*$ argument, also goes through. The key idea here is the separate consideration of each time slice $T_r$ of $Q_N$: without the local central limit theorem used to derive the Gaussian tail of the Green's function in \cite[Lemma 6.7]{CSZ_2d}, one cannot control the sum over far-field terms $|\mathbf x - \mathbf y| \gg \sqrt N$ in the sum-splitting argument used by the authors of that paper.

\subsection{Bounds on collision operators: pair partitions only}\label{m=2b}

 Here we bound the operator norms of the $\mathsf U^J_{\lambda,N}$ from Definition \ref{damp}, but only where $J$ is a pair partition.

\begin{thm}\label{u_bound}
    Let $J\vdash \{1,2,3,4\}$ be a pair partition. We have that $\sup_N\|\mathsf U^J_{\lambda,N}\|_{\ell^2 \to \ell^2} \leq  g(\lambda)$ where $g:[0,\infty) \to [0,\infty)$ is non-increasing, and furthermore $\lim_{\lambda\to \infty}  g(\lambda)=0$.\footnote{In \cite{CSZ_2d} the authors can show explicitly that $g(\lambda) \le \frac{C}{\log \lambda}.$} 
\end{thm}

\begin{proof}
Note that for any operator $A$ on $\ell^2(T)$ for a countable set $T$, we can represent it in the canonical basis as $A(x,y)$ with $x,y\in T$, and the Schur test says that if $A(x,y)\ge 0$ then $\|A\|_{\ell^2\to\ell^2} \leq \big( \sup_x \sum_y A(x,y) \big)^{1/2} \big(\sup_y\sum_x  A(x,y) \big)^{1/2} \leq \max\{ \sup_x \sum_y A(x,y), \sup_y \sum_xA(x,y)\}. $ Thus, after recalling the definition of the operator from Definition \ref{ufree}, we have 
\begin{align*}\|\mathsf U^J_{\lambda,N}\|_{\ell^2 \to \ell^2} \leq C \max\bigg\{ \sup_{x,x'}& {(e^{\zeta_N(x-x')}-1)}\sum^N_{t={0}} e^{-\frac{\lambda t}{N}} \sum_{y,y'} U_N(e^{\zeta_N}-1, t, x,x',y,y'){(e^{\zeta_N(y-y')}-1)},\; \\&  \sup_{y,y'} {(e^{\zeta_N(y-y')}-1)}\sum_{t=0}^N e^{-\frac{\lambda t}{N}} \sum_{x,x'} {(e^{\zeta_N(x-x')}-1)}U_N(e^{\zeta_N}-1, t, x,x',y,y')\bigg\},
\end{align*}
where the constant $C$ is to account for the fact that we have already summed out the other two coordinates corresponding to the heat kernel parts involving $\mathbf p_{N,1}$. 

\textbf{Supremum over $x, x'$.}
First we will bound the sup over $x,x'.$ 
By Proposition \ref{prop:RiemannSum} (taking $h \equiv 1$), we have 
    \begin{align*}(e^{\zeta_N(x-x')}-1)\sum_{ y, y' \in \mathbb Z^2} &U_N(e^{\boldsymbol{\zeta}_N}-1; t; x, x', y, y')(e^{\zeta_N(y-y')} -1)\\& \leq\frac{C}{\log N} \sum_{y \in \mathbb Z^2} \sum_{s\in (\log N)^{-1}\mathbb N} e^{\vartheta s +\epsilon_N(s\log N)}  \times \widetilde E_{N,s\log N} \big[\ind_{\{ \tau^{\mathbf X}_{(N, s\log N)} = t, S^{\mathbf X}_{(N, s \log N)} = y-x\}}\big] \\
    &=\frac{C}{\log N} \sum_{s\in (\log N)^{-1}\mathbb N} e^{\vartheta s +\epsilon_N(s\log N)}  \times \widetilde E_{N,s\log N} \big[\ind_{\{ \tau^{\mathbf X}_{(N, s\log N)} = t\}}\big].
\end{align*}
where $\tau^{\mathbf X}$ is the prelimiting Dickman subordinator defined with $X_0 = x- x'$. It follows that 
\begin{align}
g_N(\lambda)&:=\frac1{\log N}\sum_{t=0}^N e^{-\frac{\lambda t}{N}} \sum_{y,y'} U_N(e^{\zeta_N}-1, t, x,x',y,y')(e^{\zeta_N(y-y')}-1)\notag \\
&= \frac{1}{\log N}\sum_{s\in (\log N)^{-1}\mathbb N} e^{\vartheta s +\epsilon_N(s\log N)}  \times \widetilde E_{N,s\log N} \big[e^{-\lambda N^{-1} \tau^{\mathbf X}_{(N, s\log N)} }\ind_{\{ \tau^{\mathbf X}_{(N, s\log N)} \leq N\}}\big]. \notag \\
&= (1 + o(1))\int_{0}^{\infty}e^{\vartheta s}\widetilde E_{N,s\log N} \big[e^{-\lambda N^{-1} \tau^{\mathbf X}_{(N, s\log N)} }\ind_{\{ \tau^{\mathbf X}_{(N, s\log N)} \leq N\}}\big] ds.
\label{eq:laplaceDickman}
\end{align}
The error term is uniform over $N,\lambda,x,x'.$ We claim that $g_N(\lambda) \leq g(\lambda)$ for some non-increasing function $g:[0,\infty)\to [0,\infty)$ with the important property that $g(\lambda) \to 0 $ as $\lambda \to \infty$, as stated in the theorem. To prove this, we will show that $g_N$ converges pointwise as $N\to\infty$ and that $g_N(\lambda_N)\to 0$ for any sequence $\lambda_N\uparrow +\infty$.

By the dominated convergence theorem (justified by the super-exponential decay bound of Corollary \ref{cor:dcBound}) and the convergence result of Theorem \ref{convergence_dickman}, we have that 
\begin{align*}
    \lim_{N \to \infty}g_N(\lambda) = \int_{0}^{\infty}e^{\vartheta s}\mathbb E[e^{-\lambda Y_s }\ind_{\{Y_s \leq 1\}}]ds =:g(\lambda),
\end{align*}
where $(Y_s)_{s\ge 0}$ is the Dickman subordinator. Finally, let $\lambda_N \uparrow \infty$ as $N \to \infty$. It is known that the law of $Y_s$ for $s>0$ has no atoms, thus we have that for all $s > 0$
\begin{align*}
\limsup_{N \to \infty} \widetilde E_{N,s\log N} \big[e^{-\lambda_N N^{-1} \tau^{\mathbf X}_{(N, s\log N)} }\ind_{\{ \tau^{\mathbf X}_{(N, s\log N)} \leq N\}}\big]  \leq \mathbb P(Y_s = 0)=0.
\end{align*}
It follows that
\begin{align*}
    \lim_{N \to \infty}g_N(\lambda_N) =\int_{0}^{\infty}e^{\vartheta s}\left(\lim_{N \to \infty}\widetilde E_{N,\lfloor s\log N \rfloor} \big[e^{-\lambda_N N^{-1} \tau^{\mathbf X}_{(N, s\log N)} }\ind_{\{ \tau^{\mathbf X}_{(N, s\log N)} \leq N\}}\big]\right)ds = 0.
\end{align*}
\textbf{Supremum over $y, y'$.}
Next we bound the sup over $y,y'$. We have $$ \sup_{y,y'}\sum_{t=1}^N e^{-\frac{\lambda t}{N}} \sum_{x,x'} (e^{\zeta_N(x-x')}-1)U_N(e^{\zeta_N}-1, t, x,x',y,y') (e^{\zeta_N(y - y')}-1).$$
Note that $\sup_{y,y'} \leq \sup_y \sum_{y'} $, thus the above sum is upper bounded by 
$$\sup_{y}\sum_{t=1}^N e^{-\frac{\lambda t}{N}} \sum_{x,x'} (e^{\zeta_N(x-x')}-1)\sum_{y'} U_N(e^{\zeta_N}-1, t, x,x',y,y') (e^{\zeta_N(y - y')}-1).$$
Using Proposition 2.4, we obtain
\begin{align*}
    &\sup_{y}\sum_{t=1}^N e^{-\frac{\lambda t}{N}} \sum_{x,x'} (e^{\zeta_N(x-x')}-1) \sum_{s\in (\log N)^{-1} \mathbb N} e^{\vartheta s} \tilde P_{N,s\log N}( \tau^{\mathbf X}_{(N,s\log N)} =t, S^{\mathbf X}_{(N,s\log N)} = y-x)
    \\
    &= \sup_y \sum_{x,x'} (e^{\zeta_N(x-x')}-1) \sum_{s\in (\log N)^{-1} \mathbb N} e^{\vartheta s} \tilde E_{N,s\log N}\Big[ e^{-N^{-1} \lambda \tau^{\mathbf X}_{(N,s\log N)}} \ind_{\{ \tau^{\mathbf X}_{(N, s\log N)} \leq N,S^{\mathbf X}_{(N,s\log N)} = y-x\}}\Big] \\
    &= \sup_y \sum_{|z| \leq M} \sum_x (e^{\zeta_N(z)}-1) \sum_{s\in (\log N)^{-1} \mathbb N} e^{\vartheta s} \tilde E_{N,s\log N}\Big[ e^{-N^{-1} \lambda \tau^{\mathbf X}_{(N,s\log N)}}\ind_{\{ \tau^{\mathbf X}_{(N, s\log N)} \leq N,S^{\mathbf X}_{(N,s\log N)} = y-x\}}\Big] \\
    &= \sup_y \sum_{|z| \leq M}  (e^{\zeta_N(z)}-1) \sum_{s\in (\log N)^{-1} \mathbb N} e^{\vartheta s} \tilde E_{N,s\log N}\Big[ e^{-N^{-1} \lambda \tau^{\mathbf X}_{(N,s\log N)}}\ind_{\{ \tau^{\mathbf X}_{(N, s\log N)} \leq N\}}\Big] \\
    &\leq \frac{4C M^2}{\log N} \cdot  \sum_{s\in (\log N)^{-1} \mathbb N} e^{\vartheta s} \tilde E_{N,s\log N}\Big[ e^{-N^{-1} \lambda \tau^{\mathbf X}_{(N,s\log N)}}\ind_{\{ \tau^{\mathbf X}_{(N, s\log N)} \leq N\}}\Big].
\end{align*}
Now this is essentially the same expression as \eqref{eq:laplaceDickman}, thus that same argument shows that this is upper bounded by the same decreasing function $g(\lambda)$ with $g(\lambda)\to 0$ as $\lambda \to \infty.$
\end{proof}

\subsection{Bounds on collision operators: non-pair partitions}

Thus far, we have only analyzed the contribution of pair partitions to the overall sum in Proposition \ref{part_v2.5}. This section will be devoted to the analysis of the non-pair contributions. The difficult aspect of these non-pair collisions is that we can no longer reduce the Markov chain $\mathbf P_{N,4}$ to $\mathbf P_{N,2} \otimes \mathbf P_{N,1}^{\otimes 2}, $ instead we need to deal with the full complexity of $\mathbf P_{N,4}.$

Recall from Definition \ref{ufree} that for nontrivial \textit{non-pair partitions} $J\vdash \{1,2,3,4\}$ we defined the quantity
\vspace{-5pt}
\begin{align*}
    \mathscr U_{N,\mathrm{free}}^J&(t,\mathbf x,\mathbf y)\\[1ex] &:=
    \begin{cases}
        (e^{|\boldsymbol{\eta}_{N,4}(\mathbf x)|}-1) \ind_{\{\x=\y\}} \ind_{\{\x,\y\sim J\}} ,& t=0,\\[1ex]
        \ind_{\{\x,\y\sim J\}}(e^{|\boldsymbol{\eta}_{N,4}(\mathbf x)|}-1)  \mathbf E_{(N,4)}^{\mathbf x} \Big[ e^{\sum_{r=0}^{t-1} |\boldsymbol{\eta}_{N,4} (\mathbf R_r)| \ind_{\{\mathbf R_r \sim J\}}}  \ind_{\{\mathbf R_t=\mathbf y\}} \Big]  (e^{|\boldsymbol{\eta}_{N,4}(\mathbf y)|}-1), &t\ge 0.
    \end{cases} 
\end{align*}
By writing $e^{\sum_0^{t-1} a_r} = \prod_1^t (1+e^{a_r}-1) = \sum_{k=0}^\infty \sum_{0\le r_1<\dots<r_k<t} \prod_{j=1}^k (e^{a_{r_j}}-1),$  we see that for $t>0$, the quantity $\mathscr U^J_{N,\mathrm{free}}(t,\x,\y)$ can also be written as 
\begin{align}\label{unonpair}
    \notag\ind_{\{\x,\y\sim J\}}&(e^{|\boldsymbol{\eta}_{N,4}(\mathbf x)|}-1) \\
    \times \bigg(  &\sum_{k=0}^\infty \sum_{0\le r_1<\dots<r_k<t}\mathbf E_{(N,4)}^{\mathbf x} \bigg[ \bigg\{\prod_{j=1}^k (e^{|\boldsymbol{\eta}_{N,4} (\mathbf R_{r_j})|}-1) \ind_{\{\mathbf R_{r_j} \sim J\}} \bigg\}  \ind_{\{\mathbf R_t=\mathbf y\}} \bigg] \bigg)  (e^{|\boldsymbol{\eta}_{N,4}(\mathbf y)|}-1).
\end{align}
This is perhaps a more useful description when deriving bounds.

\begin{thm}[Irrelevance of non-pair collisions I]\label{thm:no_trips}
    For any nontrivial $J\vdash \{1,2,3,4\}$ with $|J|\le 2$ (i.e., $J$ which is neither $\hat{0}$ nor a pair partition), we have that $$\lim_{N\to \infty} \| \mathsf U^J_{\lambda,N} \|_{\ell^2 \to \ell^2} =0,$$
    for all $\lambda>0.$ Furthermore, for any nontrivial partitions $I,J$ (not necessarily pair partitions) we have the bounds $\sup_N\|\mathsf Q^{I,J}_{0,N} \|_{\ell^2\to\ell^2} <\infty$ and $\|\mathsf Q_{0,N}^{I,*}\|_{\ell^2\to\ell^2} + \|\mathsf Q_{0,N}^{*,I}\|_{\ell^2\to\ell^2} \leq CN^{1/2}.$
\end{thm}

\begin{proof}
    The operator norm bounds for the $\mathsf Q^{I,J}$ have already been proved in Subsection \ref{q:op}
    . Thus we focus on $\mathsf U^J.$

    It follows from the heat kernel upper bound in Assumption \ref{a1}, that \begin{equation}
        \sum_{\mathbf y \sim J}  \mathbf p_{N,4}(r ; \mathbf x, \mathbf y) \leq Cr^{|J|-4}.\label{prop:sumJ}
    \end{equation}
    We can again use the Schur test to guarantee that $$\| U^J_{\lambda,N} \|_{\ell^2 \to \ell^2} \leq \max \bigg\{ \sup_{\mathbf y} \sum_{\mathbf x} U^J_{\lambda,N}(\mathbf x, \mathbf y),\sup_{\mathbf x} \sum_{\mathbf y}U^J_{\lambda,N}(\mathbf x, \mathbf y)\bigg\} .$$
    To bound this maximum by a vanishing quantity, one may use a renewal function interpretation of the kernel $U_N^J$ again, however, there is one key difference compared to the case where $J$ is a pair partition. In this case, when bounding the renewal function, because of triple-collisions or double-double collisions, the decay factor of the heat kernel will not yield $(r_k-r_{k-1})^{-1}$ as before, but rather it will yield terms that are instead upper-bounded by $(r_k-r_{k-1})^{-2}$ which are absolutely summable. Thus only chaoses of finite order contribute, as opposed to chaoses of logarithmically growing order. In particular, the renewal interpretation is unnecessary and one can instead brutally use the heat kernel upper bounds.

    We first note that we have $\sup_{\mathbf z: |\mathbf z| \leq M}(e^{|\boldsymbol{\eta}_{N,4}(\mathbf z)|}-1) \leq \frac{C}{\log N}$ by Assumption \ref{a1}. 
   Then letting $\mathbf x_0 = \mathbf x$ and $\mathbf x_{k+1} = \mathbf y$, and using \eqref{unonpair} and \eqref{prop:sumJ}, we have
  \begin{align*}\sum_{\mathbf y}&\sum_{t=1}^{2N}\mathscr U_{N,\mathrm{free}}
    (t, \mathbf x, \mathbf y)\notag \\
    &\leq  
       \frac{C}{(\log N)^2} \sum_{k=0}^\infty  \sum_{0 <  r_1 <\dots< r_k <  t} \sum_{\mathbf x_1,\dots,\mathbf x_{k+1}\in (\mathbb Z^2)^4}\prod_{j=1}^{k} \left(e^{|\eta_N(\mathbf x_j)|} - 1\right) \ind_{\{\mathbf x^j  \sim J\}}  \prod_{j=1}^{k} \mathbf p_{N,4}(r_j-r_{j-1} ; \mathbf x_{j-1}, \mathbf x_j) \\
       &\leq \frac{C}{(\log N)^2} \sum_{k=0}^\infty  \left(\frac{C}{\log N}\right)^k  \sum_{0 <  r_1 <\dots< r_k <  t\le 2N}  \prod_{j=1}^{k+1} \frac{1}{(r_j - r_{j-1})^{2}} \\
       &\leq \frac{C}{(\log N)^2} \sum_{k=0}^\infty \left(\frac{C}{\log N}\right)^{k}  \bigg(\sum_{r=1}^\infty r^{-2} \bigg)^{k+1}  \\&\le \frac{C}{(\log N)^2} \sum_{k=0}^\infty \left(\frac{C'}{\log N}\right)^{k} \\&\leq \frac{C''}{(\log N)^2} .
    \end{align*}

Separately considering the (simpler) $t=0$ case, and adding it to the above bound, we have 
  \begin{align}
    \sum_{\mathbf y} \mathsf U_{\lambda, N}^J(\mathbf x, \mathbf y)\notag &\leq \sum_{\y} \mathscr U_{N,\mathrm{free}} (0,\x,\y) + \sum_{\y} \sum_{t=1}^{2N} \mathscr U_{N,\mathrm{free}}(t,\x,\y) \le \frac{C}{\log N} +\frac{C''}{(\log N)^2}. 
    \end{align}
which goes to $0$ as $N \to \infty$. 

Now we deal with the sum over $\mathbf x$. Notice that since the heat kernel upper bounds from Assumption \ref{a1} only depends on $\mathbf x - \mathbf y$, it follows that \eqref{prop:sumJ} still holds with $\sum_{\y\sim J}$ replaced by $\sum_{\x\sim J},$ thus the proof is completely analogous and yields an upper bound of $ \frac{C}{\log N} +\frac{C''}{(\log N)^2}$. \end{proof}

\begin{cor}\label{abs_conv}
    The right side of \eqref{operator_bound} defines an absolutely convergent series in $m$, for sufficiently large $\lambda$. In fact the terms of that series are upper-bounded, uniformly in $N$, by an absolutely convergent series in $m$. Consequently, the terms of the exact expansion of Theorem \ref{part_v2.5} are upper-bounded by the same series.
\end{cor}

\begin{proof}
    Theorems \ref{thm:Qbounds}, \ref{u_bound}, and \ref{thm:no_trips} together guarantee that 
    if we take $\lambda$ large enough, then we obtain a convergent series in $m$ for each fixed value of $N$, in fact with an exponential decay rate in the $m$-variable that is independent of $N$.
\end{proof}

\begin{cor}[Irrelevance of non-pair collisions II]\label{np_irr}
    In the exact expansion of Theorem \ref{part_v2.5}, define the \textbf{non-pair contribution} \begin{align*}&\mathfrak I_N^{\mathrm{non\;pair}}(m;t,\Phi,\Psi):=\\&\sum_{\substack{u_1\le v_1<u_2 \le v_2<\dots\\ < u_m\le v_m<Nt} } \;\sum_{\substack{J_1,\dots,J_m \vdash\{1,2,3,4\}\\\mathrm{not\;all\;pair\;partitions},\; J_i \neq \hat{0} }} \; \big \langle \Phi_{(N)} , \mathscr Q_{N,\mathrm{chain}}^{*,J_1} (u_1)\bigg[ \prod_{i=1}^m \mathscr U_{N,\mathrm{con}}^{J_i}(v_i-u_i) {\mathscr Q}_{N,\mathrm{chain}}^{J_i,J_{i+1}}  (u_{i+1} - v_i)\bigg]
    \Psi_{(N)}\big\rangle  , \end{align*}
    where we take $u_{m+1} := Nt$ and $J_{m+1} := *$.
    Note that we do not impose the constraint $J_{i-1}\ne J_i$ on the summands. Then we have that $$\lim_{N\to \infty} \sum_{m=1}^\infty |\mathfrak I_N^{\mathrm{non\;pair}}(m;t,\Phi,\Psi)|=0.$$
\end{cor}

\begin{proof}

         By Corollary \ref{abs_conv} and the dominated convergence theorem, we can commute the sum over $m$ with the limit $N\to \infty$.
         
         Thus, it suffices to prove the term-by-term convergence to zero: $\lim_{N\to \infty} \mathfrak I_N^{\mathrm{non\;pair}}(m;t,\Phi,\Psi)=0$ for each fixed $m\ge 1$. Since we are considering only those terms in which there is a non-pair partition, Theorem \ref{thm:no_trips} guarantees that for each $m$, there is at least one factor in the product of \eqref{operator_bound} that vanishes in the limit $N\to\infty$. Thus for each fixed $m$, the corresponding term of the series \eqref{operator_bound} indeed goes to 0 as $N\to\infty$.
\end{proof}

\subsection{Exact limits of the kernels}

With the upper bounds proved for the expansion in Theorem \ref{part_v2.5}, we now focus on extracting exact limits for the specific kernels appearing in that theorem. This will be the main step towards proving Theorem \ref{4thmom}. Thanks to Corollary \ref{np_irr}, we can fully restrict our attention to pair partitions only. First we have the following lemma, which is an important generalization of Lemma \ref{flem} to the Markov chain $\PN{4}$.

\begin{lem}\label{u.b.}
    Fix $y\in \mathbb Z^2$. Uniformly over all $\bfa\in (\mathbb Z^2)^4$ we have that $$\sum_{r=0}^{t} \sum_{\substack{\vect b \in (\Z^2)^4 : \\ \gap_J(\mathbf b) = y}} \mathbf p_{N,4} (r,\vect a,\vect b)  \leq H\big (t^{-1/2} |\gap_J(\bfa)| \big).$$ Here $H: (0,\infty) \to [0,\infty)$ is a decreasing function satisfying $H(x) \leq C\log_-(x)$ for $x\in [0,1]$ and $H(x) \to 0$ superpolynomially fast as $x\to\infty.$
\end{lem}

\begin{proof}
    Use the heat kernel upper bound in Assumption \ref{a1}, and we obtain that the desired quantity is upper-bounded by some constant multiple of $$\sum_{r=0}^{t} \sum_{\substack{\vect b \in (\Z^2)^4 : \\ \gap_J(\mathbf b) = y}} r^{-4} \big(1+ r^{-1/2} |\bfa -\mathbf b|\big)^{-S} \le C\sum_{r=0}^{t}  r^{-1} \big(1+ r^{-1/2} |\gap_J(\bfa)|\big)^{-S} \leq H_S\big(t^{-1/2} |\gap_J(\bfa)| \big),$$ where by a Riemann sum approximation, $H_S(x)$ is equal to some constant multiple of the function $x\mapsto \int_0^1 s^{-1} (1+ xs^{-1/2})^{-S}ds.$ It is clear that this function has a log singularity near $x=0$ and tends to 0 like $x^{-S}$ as $x\to\infty.$ Since $S>0$ is arbitrary, the claim follows.
\end{proof}

\begin{defn}For nontrivial $J \vdash \{1,2,3,4\}$, we define the macroscopic analog $(\R^2)^4_J := \{\mathbf x \in (\R^2)^4 : x_i = x_j \textrm{ if } [i]_J = [j]_J\}$ where $[i]_J$ is the partition element of $J$ containing $i$. We also define $(\mathbb R^2)^4_*=(\mathbb R^2)^4$.
\end{defn}

While it is true in some sense that $(\R^2)^4_J$ are the macroscopic analogues of $(\mathbb Z^2)^4_J$, it is not the case that these macroscopic versions are disjoint or share all the same properties.

\begin{prop}\label{prop:flem4}
    Fix any two partitions $I,J\vdash \{1,2,3,4\}$ where $I$ is either $*$ or a pair partition, and $J$ is a pair partition. Consider any sequence $\vect a_N \in (\Z^2)_I^4$ such that $N^{-1/2} \vect a_N \to \vect a \in (\R^2)_I^4$. 
    Then for all bounded $F \in C([0,t] \times (\R^2)^4)$ we have the following limits for $t>0$ and $\mathbf b\in (\mathbb Z^2)^4$:
    \begin{enumerate}
    \item\label{prop:flem4neq} Assume $I\ne J$. If $ \vect a \not\in (\R^2)^4_J$, then \begin{align*} \lim_{N\to\infty} 
        \sum_{r = 0}^{Nt} \sum_{\substack{\vect b \in (\Z^2)^4 : \\ \gap_J(\mathbf b) = y}} \mathscr Q^{I,J}_{N,\mathrm{free}}(r,\vect a_N,\vect b)F\Big( \frac{r}{N},\frac{\mathbf b}{\sqrt{N}} \Big) = 4\pi \cdot \nu_\infty(y) \int_0^t \int_{(\R^2)^4_J} \mathbf g^{\otimes 4}(t,\vect z - \vect a)F(u,\vect z) \,d\vect z \,du.
    \end{align*}

    \item\label{prop:flem4eq} In the case where $I=J$ and $\bfa \in (\R^2)^4_J$, we have that $$ \lim_{N\to\infty} \sum_{r = 0}^{Nt} \sum_{\substack{\vect b \in (\Z^2)^4 : \\ \gap_J(\mathbf b) = y}} \big[ \mathscr Q^{J,J}_{N,\mathrm{free}} -\mathscr Q^{J,J}_{N,\mathrm{con}}\big](r,\vect a_N,\vect b)F\Big( \frac{r}{N},\frac{\mathbf b}{\sqrt{N}} \Big) = 0,$$
    provided that $\bfa \notin (\mathbb R^2)^4_K$ for any nontrivial $K \neq J.$ 
\end{enumerate}
     The analogous claim holds for a sequence $(F_N)$ of functions such that $F_N \to F$ uniformly on compacts.

    %
    
\end{prop}



\begin{proof} \textbf{Proof of Item \eqref{prop:flem4neq}.} 
Without loss of generality, we take $J = \{\{1,2\},\{3\},\{4\}\}.$ Note that $\mathscr Q^{I,J}_{N,\mathrm{free}}$ is defined with the kernel $\mathbf p_{N,4}$; however, if it were instead defined with the kernel $\mathbf p_{N,2}\otimes \mathbf p_{N,1}^{\otimes 2}$, then the result would actually be immediate from Lemma \ref{flem} and the invariance principle in Assumption \ref{a1}. Thus, the core of our argument is to show that the difference between these two quantities goes to zero, using a direct coupling argument. 


\textbf{Special case:} We will first prove the claim under the simplifying assumption that $I=*$ and the limiting vector $\bfa$ satisfies $a_i\ne a_j$ for $i\ne j$, thus all coordinates of $\bfa_N$ are pairwise well-separated in the sense of Definition \ref{well_sep}. In this simpler scenario, we will use a coupling between the Markov chains $\mathbf P_{N,2}\otimes \mathbf P_{N,1}^{\otimes 2}$ and $\mathbf P_{N,4}$. Set 
\begin{equation}\label{uset}V:=\{ \x\in (\mathbb Z^2)^4: \x \in (\mathbb Z^2)^4_K \text{ for some nontrivial } K \vdash \{1,2,3,4\}, K\ne J\},
\end{equation} 
and define the coupled chain $\mathbf P_{N,\mathrm{coup}}$ with one-step transition kernel on the state space $(\mathbb Z^2)^4\times (\mathbb Z^2)^4$ given by $$\mathbf p_{N,\mathrm{coup}} ((\x,\x'), (\y,\y')) = \begin{cases}
    \ind_{\{ \y=\y'\}} \mathbf p_{N,4}(\x,\y), & \x=\x'\notin V,\\ \pN{4}(\x,\y)\cdot (\pN{2}\otimes \pN{1}^{\otimes 2}) (\x',\y'), & \text{otherwise}.
\end{cases}$$
In other words, particles from the two copies of $(\mathbb Z^2)^4$ travel together as long as there are no other particle collisions except from the selected pair $\{1,2\}$. Otherwise we evolve the two Markov chains independently. A calculation shows that, under the coupled chain, the $\x$-marginal evolves as $\mathbf P_{N,4}$ and the $\x'$-marginal evolves as $\mathbf P_{N,2} \otimes \mathbf P_{N,1}^{\otimes 2}$. Note here that we are using the consistency conditions from Assumption \ref{a1} \eqref{a1consistency}. 

Denote the realization of the coupled chain as $(\mathbf R_t,\mathbf R_t')_{t\ge 0}$. We can define a stopping time $\tau$ to be the first time $t \ge 0$ that $\mathbf R_t\in V$ or $\mathbf R_t'\in V$. Notice that $\mathbf R_t  = \mathbf R_t'$ for all $0 \le t \le \tau$, as long as we start from a state of the form $(\x,\x)$. In other words, the two marginals of the coupled chain evolve together up until the stopping time $\tau$.

We thus have the following series of bounds:
\begin{align}
     \notag &\bigg|\sum_{r = 0}^{Nt} \sum_{\substack{\vect b \in (\Z^2)^h_J : \\ b_1-b_2 = y}} \big(\pN{4}- \pN{2}\otimes \pN{1}^{\otimes 2} \big)(r,\vect a_N,\vect b) F\Big( \frac{r}{N},\frac{\mathbf b}{\sqrt{N}} \Big) \bigg|\\&= \notag \bigg|\sum_{r =0}^{Nt}\mathbf E^{(\bfa_N,\bfa_N)}_{N,\mathrm{coup}} \bigg[ \ind_{\{\tau = r\}} \mathbf E_{N,\mathrm{coup}}^{(\mathbf R_r, \mathbf R_r)} \bigg[ \sum_{u=0}^{N(t-s)-r } \ind_{\{R^1_u - R^2_u = y\}}F(N^{-1}u, N^{-1/2} \mathbf R_u) \\&\hspace{3in} \notag - \ind_{\{(R^1-R^2)'_u = y\}} F(N^{-1}u, N^{-1/2} \mathbf R_u')\bigg]\bigg]\bigg| \\&\leq \notag \|F\|_{L^\infty} \sum_{r=0}^{Nt} \mathbf E^{(\bfa_N,\bfa_N)}_{N,\mathrm{coup}} \bigg[ \ind_{\{\tau = r\}} \mathbf E_{N,\mathrm{coup}}^{(\mathbf R_r, \mathbf R_r)} \bigg[ \sum_{u=0}^{N(t-s)-r } \ind_{\{R^1_u - R^2_u = y\}}+ \ind_{\{(R^1-R^2)'_u  = y\}} \bigg]\bigg] \\&= \|F\|_{L^\infty} \sum_{r=0}^{Nt} \sum_{\x \in V } \mathbf E^{(\bfa_N,\bfa_N)}_{N,\mathrm{coup}} \big[ \ind_{\{\tau = r, \mathbf R_r=\x\}} \big] \mathbf E_{N,\mathrm{coup}}^{(\x,\x)} \bigg[ \sum_{u=0}^{N(t-s)-r } \ind_{\{R^1_u - R^2_u = y\}}+ \ind_{\{(R^1-R^2)'_u  = y\}} \bigg]\label{eq5a}.
\end{align}
We used the Markov property in the second line. Next, notice by the heat kernel upper-bounds in Assumption \ref{a1} that uniformly over all $\x\in (\mathbb Z^2)^3$ and $0\le s<t\le T$, one has that 
\begin{align*}
    \mathbf E_{N,\mathrm{coup}}^{(\x,\x)} \bigg[ \sum_{u=0}^{N(t-s)-r } \ind_{\{R^1_u - R^2_u = y\}}+ \ind_{\{(R^1-R^2)'_u  = y\}} \bigg] &= \sum_{u=0}^{N(t-s)-r}  \sum_{\substack{\vect b \in (\Z^2)^h_J : \\ b_1-b_2 = y}} \big(\pN{4}+ \pN{2}\otimes \pN{1}^{\otimes 2} \big)(u,\x,\vect b) \\&\leq C\sum_{u=0}^{N(t-s)}  \sum_{\substack{\vect b \in (\Z^2)^4_J : \\ b_1-b_2 = y}} u^{-4} \big( 1+u^{-1/2} |\mathbf x-\mathbf b|\big)^{-S} \\&\leq C\big(1+\log_- (N^{-1/2} |x_1-x_2|)\big).
\end{align*}
Thus \eqref{eq5a} is upper-bounded by some constant multiple of 
\begin{equation}\label{log-}\sum_{r=0}^{Nt} \sum_{\x \in V } \mathbf E^{(\bfa_N,\bfa_N)}_{N,\mathrm{coup}} \big[ \ind_{\{\tau = r, \mathbf R_r=\x\}} \big] \cdot \big(1+\log_- (N^{-1/2} |x_1-x_2|)\big).
\end{equation}
Break the sum into $|x_1-x_2|\le \lambda \sqrt N$ and $|x_1-x_2|\ge \lambda \sqrt N$, where $\lambda \in(0,1)$ is a fixed constant. Call the two sums $I_1^N(\lambda)$ and $I_2^N(\lambda)$ respectively. Trivially we have that $$\lim_{N\to\infty} I^N_2(\lambda) \leq (1+|\log \lambda|) \lim_{N\to\infty} \mathbf P_{N,4}^{\mathbf a_N} (\tau \le Nt) = 0,$$  where we are using the invariance principle (noting that 2-dimensional Brownian motion will not hit singletons almost surely), and the well-separatedness condition on the $\mathbf a_N$.

Meanwhile, we have from the heat kernel bounds that $\mathbf E^{(\bfa_N,\bfa_N)}_{N,\mathrm{coup}} \big[ \ind_{\{\tau = r, \mathbf R_r=\x\}} \big] \leq \mathbf p_{N,4} (r,\mathbf a_N,\x) \leq Cr^{-4} \big( 1+ r^{-1/2} |\mathbf a_N-\x|\big)^{-S}  ,$ thus
\begin{align*}
    I^N_1(\lambda) \leq C\sum_{r=0}^{Nt} \sum_{\x\in V: |x_1-x_2| \leq \lambda \sqrt N} r^{-4} \big( 1+ r^{-1/2} |\mathbf a_N-\x|\big)^{-S}  \cdot \big(1+\log_- (N^{-1/2} |x_1-x_2|)\big).
\end{align*}

Reinterpreting this as a Riemann sum, we see that if $N^{-1/2} \mathbf a_N\to \mathbf a \notin (\mathbb R^2)^4_J$, then 
\begin{equation}\label{vbar}\limsup_{N\to \infty} I_1^N(\lambda)\leq  C\int_{\bar V} \bigg[\int_0^tr^{-4} (1+r^{-1/2} |\mathbf a-\x| \big)^{-S}dr\bigg] \cdot \big( 1+ \log_- |x_1-x_2|\big) \cdot \ind_{\{|x_1-x_2| \leq \lambda\}} d\x,
\end{equation}
where $\bar V = \{\x\in (\mathbb R^2)^4: x_i=x_j$ for $i \in \{1,2\}$, $j \in \{3,4\}\}$, and $d\x$ is the Lebesgue measure on $\bar V$. To see why the integral converges, note that $\bfa$ satisfying $a_i\ne a_j$ means that there exists $\delta>0$ such that $|\bfa-\x| \ge \delta>0$ for all $\x\in \bar V$. From here, it is not difficult to see that the integral converges absolutely. Thus by taking $\lambda$ small, we can make the integral as small as desired, and consequently \eqref{eq5a} tends to 0 as $N\to\infty$. Thus we may replace $\mathbf p_{N,4}$ by $\mathbf p_{N,2} \otimes \mathbf p_{N,1}^{\otimes 2}$, in which case the claimed limit is justified by Lemma \ref{flem} together with the invariance principle of Assumption \ref{a1}.

\textbf{General case:} Now let us consider the general scenario in which the coordinates of the limiting vector $\bfa$ are no longer required to satisfy $a_i\ne a_j$, and likewise $I$ is no longer restricted to $*$. The idea is to run the Markov chain $\mathbf P_{N,4}$ up to time $N\epsilon$, then restart the chain using the Markov property and use the special case above. Then we bound the accumulated effect on the interval $[0,N\epsilon]$ and show it is very small if $\epsilon$ is very small. Precisely, use the Markov property to write 
\begin{align}
    \notag \sum_{r = 0}^{Nt} \sum_{\substack{\vect b \in (\Z^2)^4 : \\ \gap_J(\mathbf b) = y}} \mathscr Q^{I,J}_{N,\mathrm{free}}(r,\vect a_N,&\vect b)F\Big( \frac{r}{N},\frac{\mathbf b}{\sqrt{N}} \Big) = \sum_{r=0}^{N\epsilon} \sum_{\substack{\vect b \in (\Z^2)^4 : \\ \gap_J(\mathbf b) = y}} \mathscr Q^{I,J}_{N,\mathrm{free}}(r,\vect a_N,\vect b)F\Big( \frac{r}{N},\frac{\mathbf b}{\sqrt{N}} \Big) \\&+ \mathbf E_{N,4}^{\bfa_N} \bigg[ \sum_{r=0}^{N(t-\epsilon)} \sum_{\substack{\vect b \in (\Z^2)^4 : \\ \gap_J(\mathbf b) = y}} \mathscr Q^{I,J}_{N,\mathrm{free}}(r,\mathbf R_{N\epsilon},\vect b)F\Big( \frac{r}{N}+\epsilon,\frac{\mathbf b}{\sqrt{N}} \Big)\bigg]. \label{qlim0}
\end{align}
Let us first focus on taking the limit of the second term on the right side. By Skorohod's lemma and the invariance principle in Assumption \ref{a1}, we can assume that $N^{-1/2} \mathbf R_{N\epsilon} \to \sqrt{\epsilon } \mathbf Z$ almost surely, for a standard Normal $\mathbf Z$ in $(\mathbb R^2)^4.$ Then none of the coordinates of $\mathbf Z$ are equal, so with probability 1, we conclude from the simpler scenario above that 
\begin{align}
    \notag\lim_{N\to\infty} \sum_{r=0}^{N(t-\epsilon)} \sum_{\substack{\vect b \in (\Z^2)^4 : \\ \gap_J(\mathbf b) = y}} &\mathscr Q^{I,J}_{N,\mathrm{free}}(r,\mathbf R_{N\epsilon},\vect b)F\Big( \frac{r}{N}+\epsilon,\frac{\mathbf b}{\sqrt{N}} \Big) \\
    &= 4\pi \cdot \nu_\infty(y) \int_0^{t-\epsilon} \int_{(\R^2)^4_J} \mathbf g^{\otimes 4}(t,\vect z - \sqrt{\epsilon}\vect Z)F(u+\epsilon,\vect z) \,d\vect z \,du.
\end{align}
By Lemma \ref{u.b.} and the heat kernel upper bounds of Assumption \ref{a1}, we have uniform integrability of the left side, so we can commute the limit and the expectation, and we obtain that 
\begin{align}
    \lim_{N\to\infty} \mathbf E_{N,4} \bigg[\sum_{r=0}^{N(t-\epsilon)} \sum_{\substack{\vect b \in (\Z^2)^4 : \\ \gap_J(\mathbf b) = y}} &\mathscr Q^{I,J}_{N,\mathrm{free}}(r,\mathbf R_{N\epsilon},\vect b)F\Big( \frac{r}{N}+\epsilon,\frac{\mathbf b}{\sqrt{N}} \Big)\bigg]\notag  \\&= 4\pi \cdot \nu_\infty(y) \int_\epsilon^{t} \int_{(\R^2)^4_J} \mathbf g^{\otimes 4}(t,\vect z - \bfa )F(u+\epsilon,\vect z) \,d\vect z \,du.\label{qlim1} 
\end{align}
On the other hand, the limiting vector $\bfa \notin (\mathbb R^2)^4_J$, which means $|\gap_J(\bfa_N)| \ge dN^{1/2}$ for some $d>0$, so that Lemma \ref{u.b.} also implies that 
\begin{equation}\label{qlim2}\limsup_{N\to\infty}\sum_{r=0}^{N\epsilon} \sum_{\substack{\vect b \in (\Z^2)^4 : \\ \gap_J(\mathbf b) = y}} \mathscr Q^{I,J}_{N,\mathrm{free}}(r,\vect a_N,\vect b)F\Big( \frac{r}{N},\frac{\mathbf b}{\sqrt{N}} \Big) \leq C\|F\|_\infty H(d\epsilon^{-1/2}), 
\end{equation}
where $H(x)\to 0$ as $x\to\infty$. Combining \eqref{qlim1} and \eqref{qlim2}, and making $\epsilon$ as small as desired, then yields the result.

\smallskip
\noindent\textbf{Proof of Item \eqref{prop:flem4eq}.} 
Let us assume that $\bfa \in (\mathbb R^2)^4_J$ but $\bfa \notin (\mathbb R^2)^4_K$ for nontrivial $K \vdash \{1,2,3,4\}$ with $K \neq J$. With $V$ as in \eqref{uset}, let us denote by $\tau$ the first hitting time of $V$. Then we see that 
\begin{align*}
    \sum_{r = 0}^{Nt} &\sum_{\substack{\vect b \in (\Z^2)^4 : \\ \gap_J(\mathbf b) = y}} \big[ \mathscr Q^{J,J}_{N,\mathrm{free}} -\mathscr Q^{J,J}_{N,\mathrm{con}}\big](r,\vect a_N,\vect b)F\Big( \frac{r}{N},\frac{\mathbf b}{\sqrt{N}} \Big) \\&\le \|F\|_\infty \sum_{r=0}^{Nt} \sum_{\x \in V} \mathbf E_{N,4}^{\bfa_N} [ \ind_{\{\tau=r, \mathbf R_r=\x\}}] \bigg( \sum_{u=0}^{Nt-r} \sum_{\substack{\vect b \in (\Z^2)^4 : \\ \gap_J(\mathbf b) = y}}\mathbf p_{N,4} (u,\x,\mathbf b) \bigg) \\&\le C\|F\|_\infty \sum_{r=0}^{Nt} \sum_{\x \in V } \mathbf E_{N,4}^{\bfa_N}\big[ \ind_{\{\tau = r, \mathbf R_r=\x\}} \big] \cdot \big(1+\log_- (N^{-1/2} |x_1-x_2|)\big).
\end{align*}
Thus we have arrived exactly at \eqref{log-}, and proving that it tends to 0 is done exactly as is done there, noting that $\bfa \notin (\mathbb R^2)^4_K$ for $K \neq J$ means that $\bfa \notin \bar V$ with $\bar V$ as in \eqref{vbar}. This completes the proof.\end{proof}

\begin{defn}\label{ujsxy}
    We define the notation $$U^\vartheta_\infty(t,\vec x,\vec y) = 4\pi G_{\vartheta}(t)\mathbf  g(\tfrac{t}{2}, y_1-x_1) $$  for $\vec x,\vec y \in (\R^2)^2$ and $t>0$. For pair partitions $J\vdash \{1,2,3,4\}$, we also define $$\mathbf U^J_\infty (s,\x,\y) := U^\vartheta_\infty( s, \mathbf x | _{J^{(2)}},\mathbf y| _{J^{(2)}}) \cdot \mathbf g^{\otimes 2}(s, \mathbf x|_{J^{(1)}}, \mathbf y|_{J^{(1)}}),$$ where $s>0$ and $\x,\y\in (\mathbb R^2)^4$, where $J^{(r)}$ denotes the union over $r$-element sets in $J$ as in Definition \ref{nota}, and where $\mathbf g^{\otimes h}$ is the standard heat kernel on $(\mathbb R^2)^h$.
\end{defn}

\begin{prop}[Limit of $\mathscr U_{N,\mathrm{con}}$ and $\mathscr U_{N,\mathrm{free}}$ for pair partitions]\label{u-u}
    Fix a pair partition $J\vdash \{1,2,3,4\}$. Recall $\mathscr U_{N,\mathrm{con}}$ and $\mathscr U_{N,\mathrm{free}}$ from Definitions \ref{cons} and \ref{ufree1}, respectively. The following two limits hold. 
    \begin{enumerate}
        \item\label{u-utest} Consider any sequence $\vect a_N \in (\Z^2)_J^4$ such that $N^{-1/2} \vect a_N \to \vect a \in (\R^2)^4_J$, and let $H$ be as in Lemma \ref{u.b.}. 
        If $F \in C([0,T] \times (\R^2)^4)$ with $\sup_{0\le s\le N}F(s,N^{-1/2} \mathbf z) \leq \sum_{K\ne J} H(|\gap_K(\x)|)$, with the sum being over pair partitions not equal to $J$, then we have the following asymptotic for $t\in [0,1]$:
        \begin{align*}  
            \sum_{0\le r<Nt} \sum_{\substack{\vect b \in (\Z^2)^4_J }} \mathscr U^{J}_{N,\mathrm{free}}&(r,\vect a_N,\vect b)F\Big( \frac{r}{N},\frac{\mathbf b}{\sqrt{N}} \Big) \\&= \frac1{4\pi}(e^{\z_N(\gap_J(\bfa_N))} -1)\log N \int_0^t \int_{(\R^2)^4_J} \mathbf U^J_\infty (s,\bfa,\vect z)F(u,\vect z) \,d\vect z \,du +o(1),
        \end{align*}
        where $\mathbf U^J_\infty (s,\x,\y)$ is as in Definition \ref{ujsxy}, and the $o(1)$ term is uniform over all variables. 
         The analogous claim holds for a sequence $(F_N)$ of functions such that $F_N \to F$ uniformly on compacts. 
         
         \item\label{u-udif}  Let $F^1_N,F^2_N: (\mathbb R^2)^4 \to \mathbb R_+$ such that $ \sup_{r\le N} F^i_N(r,\x) \leq \sum_{K\ne J} H(|\gap_K(\x)|)$ where $H$ is as in Lemma \ref{u.b.}, and the sum is over pair partitions not equal to $J$. 
         Suppose also that the $F_i$ satisfy $\sup_N \|N^{-3/2} \sup_{r\le N}  F^i_N (r,N^{-1/2}\bullet) \|_{\ell^2((\Z^2)^4_J)} < \infty$.
         Then
         \begin{align*}
            N^{-4} \sum_{0\le u\le v<Nt} \sum_{\substack{\vect a,\vect b \in (\Z^2)^4_J }} F^1_N\bigg(  \frac{u}{N},\frac{\bfa}{\sqrt N} \bigg)\big[\mathscr U^{J}_{N,\mathrm{free}}&-\mathscr U_{N,\mathrm{con}}^J\big](v-u,\vect a,\vect b)F^2_N\bigg( t-\frac{v}{N},\frac{\mathbf b}{\sqrt{N}} \bigg) = 0.
        \end{align*}
     \end{enumerate}
\end{prop}

The two hypotheses in Item \eqref{u-udif} may seem somewhat confusing. Neither condition implies the other, but they are compatible in the sense that if $F(r,N^{-1/2} \bullet)$ has uniform-in-$r$ compact support and if $ F(r,N^{-1/2}\x) \leq \sum_{K\ne J} H(|\gap_K(\x)|)$ then $ \|\sup_{r\le N} N^{-3/2} F^i_N (r,N^{-1/2}\bullet) \|_{\ell^2((\Z^2)^4_J)} < \infty$. One could generalize this to a single condition that covers both bounds, but we did not do this.


\begin{proof}
    \textbf{Proof of Item \eqref{u-utest}.} Since $\mathscr U_{N,\mathrm{free}}$ is defined in terms of the two-point renewal function $U_N$ from \eqref{eq:renewal}, the proof is immediate from Lemma \ref{lem:Fdef} and the invariance principle of Assumption \ref{a1} \eqref{a1invarianceprinciple}. 
    
    \smallskip
    \noindent\textbf{Proof of Item \eqref{u-udif}.} By \eqref{tric}, we do have the upper $\mathscr U^{J}_{N,\mathrm{con}}(t,\vect a,\vect b)\le \mathscr U^{J}_{N,\mathrm{free}}(t,\vect a,\vect b)$, thus one may try to proceed by obtaining a matching lower bound. 
    The exact relationship between $\mathscr U_{N,\mathrm{con}}$ and $\mathscr U_{N,\mathrm{free}}$ can be written as the infinite series
    \begin{align}
        \notag &\mathscr U_{N,\mathrm{free}}^J(t,\x,\y) = \mathscr U_{N,\mathrm{con}}^J(t,\x,\y) \\&+ \bigg(\sum_{m=1}^\infty \sum_{J_1,\ldots,J_m \sqsupset J}
        \sum_{\substack{ v_0 < u_1 \le v_1 < \\\dots<v_{m-1}\le u_{m} <t}} \bigg[\prod_{j=0}^{m-1} \tilde{\mathscr U}^{J_j}_{N,\mathrm{con}} (v_j-u_{j}) \tilde{\mathscr Q}^{J_j,J_j}_{N,\mathrm{chain}}
    (u_{j+1}-v_j) \bigg] \tilde{\mathscr U}^{J_m}_{N,\mathrm{con}}(t-u_{m}) \bigg)(\x,\y),\label{contract3}
    \end{align}
     where we agree on the following list of conventions:
     \begin{enumerate} \item $u_0=0$, and $J_0=J_m=J$.
     \item $\tilde{\mathscr Q}^{I,J}_{N,\mathrm{chain}} (t,\x,\y) := \ind_{\{\x\sim I, \y\sim J\}} \big[ (\mathbf p_{N,2} \otimes \mathbf p_{N,1}^{\otimes 2} ) (t,\x,\y) - \ind_{\{I=J\}} \mathscr Q^{I,J}_{N,\mathrm{con}}(t,\x,\y) \big]$, which plays a role completely analogous to $\mathscr Q_{N,\mathrm{chain}}^{I,J}$ from Theorem \ref{contraction}, except that it is defined in terms of the Markov chain $\mathbf P_{N,2}\otimes \mathbf P_{N,1}^{\otimes 2}$ as opposed to $\mathbf P_{N,4}$. 
     \item We have that $\tilde{\mathscr U}^J_{N,\mathrm{ con}}(t,\mathbf x,\mathbf y):=\big( e^{\z_N(\gap_J(\x) } - 1\big)\ind_{\{\x=\y\}} \ind_{\{\x\sim J\}}$ if $t=0$, and 
     \begin{align}
         \notag\tilde{\mathscr U}^I_{N,\mathrm{ con}}(&t,\mathbf x,\mathbf y) = \big(1 - e^{-\z_N(\gap_J(\x))}\big) \\
            &\times (\mathbf E_{N,2}\otimes \mathbf E_{N,1}^{\otimes 2})^{\x} \bigg[ e^{\sum_{r=0}^t \z_N(\gap_J(\mathbf R_s))\ind_{\{\mathbf R_s \sim I\}} } \ind_{\{\mathbf R_t =\y\}} \bigg] \big( e^{\z_N(\gap_J(\y) } - 1\big) \ind_{\{\x,\y\sim J\}}
     \end{align}
     if $t>0$. We emphasize that $I$ is not required to be a pair partition here. Thus, $\tilde{\mathscr U}^I_{N,\mathrm{con}} $ are exactly like the constrained evolution operators from Definition \ref{cons}, except that the additive functional has been changed from $\boldsymbol{\eta}_{N,4}$ to $\z_N(\gap_J(\bullet))$ and the Markov chain has changed from $\mathbf E_{N,4}$ to $\mathbf E_{N,2}\otimes \mathbf E_{N,1}^{\otimes 2}$ (in fact if $I=J$ then $\tilde{\mathscr U}^I_{N,\mathrm{con}} = \mathscr U^I_{N,\mathrm{con}} $).
     \item $I\sqsupset J$ means that the partition $J$ is a refinement of the partition $I$, so that each element of $I$ is a union of elements of $J$. Since $J$ is a pair partition, note that $I\sqsupset J$ implies either that $I=J$ or that $I$ encodes higher-order collisions such as triple collisions.
     \end{enumerate}

     A proof of \eqref{contract3} is similar to the contraction identities already described earlier, and is left to the reader (see Theorem \ref{contraction} and Remark \ref{part_v5} for similar identities). 
    
     Next, we note that all of the terms of \eqref{contract3} with some $J_j\ne J$ must have a non-pair collision, and essentially repeating the operator norm bound from Theorem \ref{thm:no_trips} - Corollary \ref{np_irr} will show that these higher-order collision terms vanish.
    
     Thus, we can restrict our attention to the series \eqref{contract3} in which all $J_j=J$. We need to show 
    \begin{align}
        \notag \lim_{N\to\infty}  N^{-4}&\sum_{0\le u\le v<Nt} \sum_{\substack{\vect a, \vect b \in (\Z^2)^4_J }} F^1_N\Big (\frac{u}{N},\frac{\bfa}{\sqrt N}\Big) \bigg(\sum_{m=1}^\infty \sum_{\substack{ u\le v_0 < u_1 \le v_1 < \\\dots<v_{m-1}\le u_{m} <v}} \bigg[\prod_{j=0}^{m-1} \mathscr U^J_{N,\mathrm{con}} (v_j-u_{j}) \\&\times \big(\tilde{\mathscr Q}^{J,J}_{N,\mathrm{free}}-\mathscr Q^{J,J}_{N,\mathrm{con}} \big) (u_{j+1}-v_j) \bigg] \mathscr U^J_{N,\mathrm{col}}(v-u_{m}) \bigg)(\bfa_N,\mathbf b)F_N^2\Big( t-\frac{v}{N},\frac{\mathbf b}{\sqrt{N}} \Big)=0,\label{ser}
    \end{align}
    where $\tilde{\mathscr Q}_{N,\mathrm{free}}(t,\x,\y):=\ind_{\{\x\sim I, \y\sim J\}}  (\mathbf p_{N,2} \otimes \mathbf p_{N,1}^{\otimes 2} ) (t,\x,\y)$. Written another way, we want to show \begin{align}\notag \lim_{N\to\infty}    \sum_{m=1}^\infty \frac1{N}&\sum_{0\le u\le v<Nt} \sum_{\substack{ u\le v_0 < u_1 \le v_1 < \\\dots<v_{m-1}\le u_{m} <v}} \bigg\langle G^1_N,\bigg[\prod_{j=0}^{m-1} \mathscr U^J_{N,\mathrm{con}} (v_j-u_{j}) \\&\times \big(\tilde{\mathscr Q}^{J,J}_{N,\mathrm{free}}-\mathscr Q^{J,J}_{N,\mathrm{con}} \big) (u_{j+1}-v_j) \bigg] \mathscr U^J_{N,\mathrm{col}}(v-u_{m}) G^2_N \bigg\rangle_{\ell^2((\mathbb Z^2)^4_J)}=0, \label{ser3}
    \end{align}
    where $G^i_N(\x)= N^{-3/2}\sup_{t\le N} F^i_N(t, \x/\sqrt N)$.  To prove this, we proceed in two steps. First we show that the terms decay exponentially fast in the variable $m$. Second, we show that for each fixed value of $m$, the associated term goes to 0 as $N\to \infty$. 
    
    \smallskip
    \noindent\textbf{Step 1. Exponential decay of terms in the $m$-variable.} We will exponentially damp the operators exactly as in Definition \ref{damp} and Proposition \ref{part_v4}. 
    Let $\mathsf U_{\lambda}$ and $\mathsf Q_{\lambda}$ denote the damping of $\mathscr U^J_{N,\mathrm{con}} $ and $\big(\tilde{\mathscr Q}^{J,J}_{N,\mathrm{free}}-\mathscr Q^{J,J}_{N,\mathrm{con}} \big)$, respectively. Thanks to the extra factor of $\frac1N$, we see that the $m$-term of the above series \eqref{ser3} is upper-bounded by 
    \begin{align*} e^{\lambda}\bigg\langle  G^1_N ,\big( \mathsf U_\lambda \mathsf Q_0)^{m} \mathsf U_\lambda G^2_N\bigg\rangle_{\ell^2((\mathbb Z^2)^4_J)}& \le  e^{\lambda} \|G^1_N \|_{\ell^2} \|\mathsf U_\lambda \|_{\ell^2\to\ell^2}^{m+1} \|\mathsf Q_0\|_{\ell^2\to\ell^2}^{m} \|G^1_N\|_{\ell^2} \leq e^{\lambda} C^{m} g(\lambda)^{m+1}.
    \end{align*}
    On the right side, $g(\lambda)$ is exactly as in Proposition \ref{u_bound}, and we are also using Proposition \ref{thm:Qbounds} to bound the $\mathsf Q_0$-operator norms. We are also using the hypotheses of the proposition that say that $\|G^i_N\|_{\ell^2}$ are bounded independently of $N$. By making $\lambda$ so large that $g(\lambda)<1/C$ we see that the terms do decay exponentially fast in the variable $m$. 
    
    \smallskip
    \noindent\textbf{Step 2. Each term goes to zero.} The fact that for fixed $m$, the associated term of the series on the left side \eqref{ser} goes to 0 follows immediately from an induction argument, repeatedly applying the result of Item \eqref{u-utest} of this proposition, together with Item \eqref{prop:flem4eq} of Proposition \ref{prop:flem4} in successive iteration. Here we are also using the bound $F^i_N(t,N^{-1/2}\x) \leq \sum_{K\ne J} H(|\gap_K(\x)|)$ which is important for applying those results.
\end{proof}

\subsection{Putting everything together}\label{last}
 Now we can finally prove Theorem \ref{4thmom}. The idea is to use the exact expansion of Theorem \ref{part_v2.5}, then commute the limit $N\to\infty$ with the sum over $m$. The reason that we are allowed to commute the two limits is Corollary \ref{abs_conv} and the dominated convergence theorem. Then we will calculate the limit of the individual terms for each value of $m$, using Theorem \ref{convergence_dickman} and induction. 

\begin{thm}\label{lim=}
    Recall the exact expansion of $ \mathfrak I_N(m,t,\Phi,\Psi)$ from Theorem \ref{part_v2.5}. For each $m\in \mathbb N$ we have that $\lim_{N\to \infty} \mathfrak I_N(m,t,\Phi,\Psi) = \mathfrak I_{\mathrm{SHF}} (m,t,\Phi,\Psi),$ where 
    \begin{align*}
    \mathfrak I_{\mathrm{SHF}} &(m,t,\Phi,\Psi) := \sum_{\substack{J_1,\dots,J_m \vdash\{1,2,3,4\}\\ J_i\ne J_{i-1} \\ \mathrm{pair\;partitions\;only}}}  \iint_{\substack{u_1<v_1 <\dots<u_m<v_m<t \\ \prod_{i = 1}^{m+1} (\R^2)^4_{J_{i-1}} \times (\R^2)^4_{J_i}}} 
    \;\Phi(\mathbf c_0)\\
     \times &\prod_{i=1}^{m}\bigg\{  
    \mathbf g^{\otimes 4} (u_i-v_{i-1},\mathbf c_{i-1}, \mathbf b_i) \times \mathbf U^{J_i}_\infty( v_i - u_i, \mathbf b_i ,\mathbf c_i) 
    \bigg\}\mathbf g^{\otimes h} (t-v_m , \mathbf c_m,\mathbf b_{m+1})\Psi(\mathbf b_{m+1})\;d\vec ud \vec v d\vec{\mathbf b}d \vec{\mathbf c},
\end{align*}
where $\mathbf U^J_\infty(s,\x,\vect y)$ is as in Definition \ref{ujsxy}, and where we take $J_0=J_{m+1} = *$ in the above expression, as well as $v_0=0$ and $u_{m+1}=t$.
\end{thm}

\begin{proof}
\textbf{Step 1. Splitting the sum.}
Recalling the expansion in Theorem \ref{part_v2.5}, define the \textit{pair contribution} to be the part of that expression only coming from pair partitions: \begin{align}\notag &\mathfrak I_N^{\mathrm{pair}}(m;t,\Phi,\Psi)\\&:= \sum_{\substack{u_1\le v_1<u_2 \le v_2<\dots\\ < u_m\le v_m<Nt} } \;\sum_{\substack{J_1,\dots,J_m \vdash\{1,2,3,4\} \\ \mathrm{pair\;partitions\;only}}} \; \big \langle \Phi_{(N)} , \mathscr Q_{N,\mathrm{chain}}^{*,J_1} (u_1)\bigg[ \prod_{i=1}^m \mathscr U_{N,\mathrm{con}}^{J_i}(v_i-u_i) {\mathscr Q}_{N,\mathrm{chain}}^{J_i,J_{i+1}}  (u_{i+1} - v_i)\bigg]
    \Psi_{(N)}\big\rangle  ,  \label{jpair}\end{align}
    where we do\textbf{ not} impose the constraint $J_{i-1}\ne J_i$ on the summands, we agree that $u_{m+1}=Nt$ and $J_{m+1}=*$. By Corollary \ref{np_irr}, it suffices to prove that $\lim_{N\to \infty} \mathfrak I_N^{\mathrm{pair}}(m,t,\Phi,\Psi) = \mathfrak I_{\mathrm{SHF}} (m,t,\Phi,\Psi).$ 

    Let us further split this sum into \textbf{consecutive} and \textbf{nonconsecutive} summands: $$\mathfrak I_N^{\mathrm{pair}}(m;t,\Phi,\Psi) = \mathfrak I_N^{\mathrm{C}}(m;t,\Phi,\Psi)+\mathfrak I_N^{\mathrm{NC}}(m;t,\Phi,\Psi),$$ where the first term on the right side corresponds to the sum over those pair partitions $J_1,\dots,J_m$ which satisfy $J_i=J_{i-1}$ for \textit{some} $i\le m$, and the second term on the right side corresponds to the sum over those pair partitions $J_1,\dots,J_m$ which satisfy $J_i\ne J_{i-1}$ for \textit{all} $i\le m$. We will show that $$\lim_{N\to \infty} \mathfrak I_N^{\mathrm{NC}}(m,t,\Phi,\Psi) = \mathfrak I_{\mathrm{SHF}} (m,t,\Phi,\Psi),\qquad \mathrm{and}\qquad\lim_{N\to \infty} \mathfrak I_N^{\mathrm{C}}(m,t,\Phi,\Psi) =0,$$ which will complete the proof.
    
\smallskip
\noindent\textbf{Step 2. Limit of $\mathfrak I_N^{\mathrm{NC}}$.} For the non-consecutive part, since $J_i \neq J_{i+1}$, it follows from the contraction identity of Theorem \ref{contraction} that every copy of $\mathscr Q_{N,\mathrm{chain}}$ can be replaced by $\mathscr Q_{N,\mathrm{free}}$. 

Thanks to Item \eqref{u-udif} of Proposition \ref{u-u}, we can replace every single instance of $\mathscr U_{N,\mathrm{con}}$ with the simpler operator $\mathscr U_{N,\mathrm{free}}$ of Definition \ref{ufree1}. Indeed if we consider any expansion of the form $$\sum_{\substack{u_1\le v_1<u_2 \le v_2<\dots\\ < u_m\le v_m<Nt} } \big \langle \Phi_{(N)}, \mathscr Q_{N,\mathrm{free}}^{*,J_1} (u_1)\big[ \prod_{i=1}^{m} \mathscr U_{N,\mathrm{con}}^{ J_{i}}(v_i-u_i) {\mathscr Q}_{N,\mathrm{free}}^{ J_{i}, J_{i+1}}  (u_{i+1} - v_i)\big]
    \Psi_{(N)}\big\rangle,$$ then by Lemma \ref{u.b.} the terms (respectively) to the left and to the right of any given copy of $\mathscr U_{N,\mathrm{con}}$ will satisfy the bounds needed (respectively) by $F_1$ and $F_2$ in Proposition \ref{u-u}. More precisely, letting $A^\star$ denote the $\ell^2$-adjoint of $A$, we can isolate the individual operator $ \mathscr U_{N,\mathrm{con}}^{ J_{i}}(v_i-u_i)$ by rewriting the expression as 
    \begin{align*}
        \sum_{\substack{u_1\le v_1<u_2 \le v_2<\dots\\ < u_m\le v_m<Nt} }\bigg\langle \Big(\prod_{j<i} &\mathscr U_{N,\mathrm{con}}^{ J_{j}}(v_j-u_j) {\mathscr Q}_{N,\mathrm{free}}^{ J_{j}, J_{j+1}}  (u_{j+1} - v_j) \Big)^{\star} \Phi_{(N)}, \mathscr U_{N,\mathrm{con}}^{J_i} (v_i-u_i) \\&\times \bigg({\mathscr Q}_{N,\mathrm{free}}^{ J_{i}, J_{i+1}}  (u_{i+1} - v_i) \Big[\prod_{j>i}^{m} \mathscr U_{N,\mathrm{con}}^{ J_{j}}(v_j-u_j) {\mathscr Q}_{N,\mathrm{free}}^{ J_{j}, J_{j+1}}  (u_{j+1} - v_j)\Big] 
        \Psi_{(N)} \bigg)\bigg\rangle
    \end{align*}
    Thus, by setting 
    \begin{align*}     F^1_N(u_i,\bfa )&:=\sum_{\substack{u_1\le v_1<u_2 \le v_2<\dots\\ < u_{i-1}\le v_{i-1}<u_i} }\Big(\prod_{j<i} \mathscr U_{N,\mathrm{con}}^{ J_{j}}(v_j-u_j) {\mathscr Q}_{N,\mathrm{free}}^{ J_{j}, J_{j+1}}  (u_{j+1} - v_j) \Big)^{\star} \Phi (\bfa),  \\
       F^2_N (v_i,\bfa ) &:= \sum_{\substack{v_i<u_{i+1}\le v_{i+1}< \dots\\ < u_m\le v_m<Nt} }
     {\mathscr Q}_{N,\mathrm{free}}^{ J_{i}, J_{i+1}}
    (u_{i+1} - v_i) \Big[\prod_{j>i}^{m}
    \mathscr U_{N,\mathrm{con}}^{ J_{j}}(v_j-u_j)
    {\mathscr Q}_{N,\mathrm{free}}^{ J_{j}, J_{j+1}}
    (u_{j+1} - v_j)\Big] 
    \Psi(\bfa),
    \end{align*} 
    we wish to apply Proposition \ref{u-u}. To do so, we first need to show that the $F^i_N$ satisfy the bound $\sup_N \|N^{-3/2} \sup_{0\le r\le N}F^i_N(r,N^{-1/2}\bullet) \|_{\ell^2((\Z^2)^4_J)} < \infty$
    . This follows immediately from damping the operators, and then applying the operator bounds for the damped operators $\mathsf U_{\lambda,N}, \mathsf Q_{\lambda,N}$ in Propositions \ref{thm:Qbounds} and \ref{u_bound} (even taking $\lambda=0$ would suffice here), noting that $\mathscr U_{N,\mathrm{con}} \le \mathscr U_{N,\mathrm{free}}$. Meanwhile, if we repeatedly apply the result of Item \eqref{u-utest} of Proposition \ref{u-u} and Item \eqref{prop:flem4eq} of Proposition \ref{prop:flem4} in successive iteration, we will obtain the bound $F^i_N(r,N^{-1/2}\x) \leq \sum_{K\ne J} H(|\gap_K(\x)|)$ as desired.  

    By an inductive (in $i$) application of Proposition \ref{u-u}, we replace $\mathscr U^{J_i}_{N,\mathrm{con}}$ one by one with $\mathscr U^{J_i}_{N,\mathrm{free}}$ in the expressions for $F_1$ and $F_2$, where at each successive step the argument above still holds after this replacement (noting that some instances of $\mathscr U_{N,\mathrm{con}}$ might be replaced by instances of $\mathscr U_{N,\mathrm{free}}$ after the first few replacements, which does not affect the arguments since we use the bound $\mathscr U_{N,\mathrm{con}} \le \mathscr U_{N,\mathrm{free}}$ anyways).
    
    Hence, in the above expression, if we write $\mathscr U_{N,\mathrm{con}}^{J_i} (v_i-u_i) = \mathscr U_{N,\mathrm{con}}^{J_i} (v_i-u_i)-\mathscr U_{N,\mathrm{free}}^{J_i} (v_i-u_i)+\mathscr U_{N,\mathrm{free}}^{J_i} (v_i-u_i)$, then the part corresponding to $\mathscr U_{N,\mathrm{con}}^{J_i} (v_i-u_i)-\mathscr U_{N,\mathrm{free}}^{J_i} (v_i-u_i)$ will vanish by Item \eqref{u-udif} of Proposition \ref{u-u}, thus leaving only $\mathscr U_{N,\mathrm{free}}^{J_i} (v_i-u_i).$ Thus, we will henceforth replace every instance of $\mathscr U_{N,\mathrm{con}}$ by $\mathscr U_{N,\mathrm{free}}$ in the original expansion \eqref{jpair}.
 
For the limit of these non-consecutive terms with $\mathscr U_{N,\mathrm{con}}$ replaced by $\mathscr U_{N,\mathrm{free}}$, the sum over partitions is a finite sum, so we just need to prove that the limit equals the correct quantity for any \textit{fixed} sequence $ J_1, J_2, J_3,\dots,J_m$ of nonconsecutive pair partitions of $\{1,2,3,4\}$. Thus we will henceforth fix such a sequence $ (J_i)_{i= 1}^m$. Explicitly, the goal is then to show that 
\begin{align*}
    &\lim_{N\to\infty} \sum_{\substack{u_1\le v_1<u_2 \le v_2<\dots\\ < u_m\le v_m<Nt} } \;\big \langle \Phi_{(N)} , \mathscr Q_{N,\mathrm{free}}^{*,J_1} (u_1)\bigg[ \prod_{i=1}^m \mathscr U_{N,\mathrm{free}}^{J_i}(v_i-u_i) {\mathscr Q}_{N,\mathrm{free}}^{J_i,J_{i+1}}  (u_{i+1} - v_i)\bigg]
    \Psi_{(N)}\big\rangle \\&=  \iint_{\substack{u_1<v_1 <\dots<u_m<v_m<t \\ \prod_{i = 1}^{m+1} (\R^2)^4_{J_{i-1}} \times (\R^2)^4_{J_i}}} 
    \;\Phi(\mathbf c_0) \prod_{i=1}^{m}\bigg\{  
    \mathbf g^{\otimes 4} (u_i-v_{i-1},\mathbf c_{i-1}, \mathbf b_i) \\&\hspace{1 in}\times \mathbf U^{J_i}_\infty( v_i - u_i, \mathbf b_i ,\mathbf c_i) 
    \bigg\}\mathbf g^{\otimes h} (t-v_m , \mathbf c_m,\mathbf b_{m+1})\Psi(\mathbf b_{m+1})\;d\vec ud \vec v d\vec{\mathbf b}d \vec{\mathbf c}.
\end{align*}
We can actually prove a more general claim in which $\Phi_{(N)}$ on the left side is replaced by $N^4 \ind_{\{\mathbf z_N\}}$ where $N^{-1/2} \mathbf z_N\to \mathbf z$ with $\gap_{J_1}(\mathbf z)>0$. Then $\Phi$ on the right side would be replaced by $\delta_{\mathbf z}$. 

Proving this essentially just follows from an inductive application of Item \eqref{flem1} of Proposition \ref{prop:flem4} and Item \eqref{u-udif} of Proposition \ref{u-u}. To be precise, one starts from the terminal partition $J_m$ and works backwards, first applying Item \eqref{prop:flem4eq} of Proposition \ref{prop:flem4}, which will yield a logarithmic singularity as in Lemma \ref{u.b.}. Then this is followed by an application of Item \eqref{u-utest} of Proposition \ref{u-u}, which will yield a bounded function. Then we reach the next partition $J_{m-1}$ going in backwards order, and we deal with $J_{m-1}$ in a similar fashion, and so on. 

These logarithmic singularities that are produced upon each application of Proposition \ref{prop:flem4} will not cause any problem \textit{precisely because} the partitions are non-consecutive, which is what allows us to apply Proposition \ref{u-u} immediately afterwards. If instead we looked at consecutive partitions $J \to J$ then the functions $F_i$ in Proposition \ref{u-u} would not satisfy the required upper bound as there would be a logarithmic singularity due to a term $H(|\gap_J(\x)|)$. 
\textbf{}

 \smallskip
\noindent\textbf{Step 3. Limit of $\mathfrak J_N^{\mathrm{NC}}$.} For the consecutive part, use the same exact induction as in Step 2. However, since there is at least one factor $\mathfrak J_N^{\mathrm{NC}}(t,m,\Phi,\Psi)$ for which there is a repeated pair partition, Theorem \ref{contraction} says that a factor $ \mathscr Q^{I,I}_{N,\mathrm{chain}} = \mathscr Q^{I,I}_{N,\mathrm{free}}-\mathscr Q^{I,I}_{N,\mathrm{con}}$ will appear within the operator expansion for $\mathfrak J_N$. Instead of using Item \eqref{prop:flem4neq} in Proposition \ref{prop:flem4}, we instead use Item \eqref{prop:flem4eq} whenever we encounter the first such factor.
This means that the contribution of this quantity $\mathscr Q^{I,I}_{N,\mathrm{free}}-\mathscr Q^{I,I}_{N,\mathrm{con}}$ on macroscopic scales will be \textit{vanishing}, as opposed to converging to a meaningful strictly positive limit (as was the case for all terms in Step 2). While performing the induction, the contribution of all other terms will converge to some bounded limit (which could also be zero). Thus, the result of the induction will reveal that the limit equals zero. This completes the proof.
\end{proof}


Finally, we can prove Theorem \ref{4thmom}.

\begin{proof}[Proof of Theorem \ref{4thmom}]
By Definition \ref{def:axioms}, we have that $\sum_{m\ge 0} \mathfrak I_{\mathrm{SHF}}(m, t, \Phi,\Psi) = \mathcal G_{0,t}^{\mathrm{SHF}}(\Phi,\Psi)$, where the latter was defined just before Theorem \ref{4thmom}. 
Using this expression, the result of Theorem \ref{lim=} immediately gives the result for the fourth moment, with Corollary \ref{abs_conv} and dominated convergence justifying the interchange of the large-$N$ limit with the infinite sum over $m$.

    For the third and fifth moment, note that there is no real obstruction to simply replacing $\{1,2,3,4\}$ with the more general set $\{1,\ldots,h\}$ for any $h$ such that the conditions of Assumption \ref{a1} hold. Every part of the proof generalizes readily. We worked with $h=4$ just for notational simplicity and concreteness.
\end{proof}

\section{Proofs of main results} \label{sec:mainResults}

In this section, we will prove Theorems \ref{mr}--\ref{check} and \ref{mr2}--\ref{check2}. Before getting to the main results, we first discuss the topology in which we prove tightness.

\subsection{Tightness and proof of Theorem \ref{mr}}

Recall that the macroscopic density field $\mathfrak H^N$ is defined by
\begin{equation*}
    \mathfrak H^N_{s,t}(\phi,\psi):= N^{-1}\sum_{(x,y)\in( \mathbb Z^2)^2} \phi(N^{-1/2} x) H^N_{Ns,Nt}(x,y) \psi(N^{-1/2} y),
\end{equation*}
with $H^N$ as in Assumption \ref{a1}, $\phi,\psi \in C_c(\mathbb R^2)$, and $s,t \in N^{-1}\mathbb Z$ with $s\le t.$
Since we are only interested in finite-dimensional convergence in the time trajectory, we may pass to the constant interpolation $\mathfrak H^N_{s,t} = \mathfrak H^N_{N^{-1} \lfloor Ns\rfloor, N^{-1} \lfloor Nt\rfloor}$ between points of $N^{-1}\mathbb Z$. Since $H^N$ is nonnegative, we can view $\mathfrak H^N_{s,t}$ for fixed $s,t$ as a random variable valued in the space $\mathcal{M}_{\mathrm{loc}}(\R^2 \times \R^2)$ of locally finite Radon measures on $\R^2\times \R^2$ equipped with the \emph{vague topology}. 
    
\begin{defn} A sequence of deterministic measures $(\mu_n)$ on $\R^d$ is said to converge \emph{vaguely} to $\mu$ if $\int \varphi \,d \mu_n \to \int \varphi \,d\mu$ for all $\varphi \in C_c(\R^d)$.
We denote this convergence by $\mu_n \xrightarrow{v} \mu.$
\end{defn}

Give a sequence $(\mu_n)$ of $\mathcal{M}_{\mathrm{loc}}(\R^2 \times \R^2)$-valued random measures, it makes sense to ask about convergence in distribution to a limiting random measure $\mu$. This mode of convergence is called \emph{convergence in distribution with respect to the vague topology} and is denoted by $\mu_n   \xrightarrow{vd} \mu$. 

\begin{thm}[Theorem 4.11 in \cite{MR3642325}] \label{thm:vdconvergence}
    The space $\mathcal M_{\mathrm{loc}} (\mathbb R^2\times \mathbb R^2)$ equipped with the vague topology is completely metrizable and thus Polish. Furthermore, showing that $\mu_n   \xrightarrow{vd} \mu$ is equivalent to showing that $\mu_n(\varphi)   \xrightarrow{d} \mu(\varphi)$ for all $\varphi \in C_c(\R^2 \times \R^2)$. 
\end{thm}

In Polish spaces, Prokhorov's theorem gives easy criteria for tightness in terms of compact sets.

We want to take a limit of $\mathfrak H^N_{s,t}$ as a measure-valued stochastic process indexed by $s,t$. Using Theorem \ref{thm:vdconvergence}, showing that $\mathfrak H^N_{s,t}$ converges in finite-dimensional distributions in the $s,t$ variables where each $\mathfrak H^N_{s,t}\in \mathcal M_{\mathrm{loc}}((\mathbb R^2)^2)$ to a limiting process $\mathfrak H_{s,t}$ is equivalent to showing that we have 
\begin{equation}
    \bigg( \int_{(\mathbb R^2)^2} f_i(x,y) \mathfrak H^N_{s_i,t_i}(dx,dy)\bigg)_{i=1}^k \xrightarrow{d}   \bigg(\int_{(\mathbb R^2)^2} f_i(x,y) \mathfrak H_{s_i,t_i}(dx,dy)\bigg)_{i=1}^k
\end{equation}
for all $k$ and all collections $f_i \in C_c^\infty ((\mathbb R^2)^2)$ and all times $s_i\le t_i \; (1\le i \le k)$.

\begin{prop}[Tightness]\label{tight}
    Fix any $s<t$. The collection $\mathfrak H^N_{s,t}$ is tight as $N\to \infty$ with respect to the vague topology on $\mathcal M_{\mathrm{loc}}(\mathbb R^2\times \mathbb R^2).$
\end{prop}

\begin{proof}
    Using Theorem \ref{thm:vdconvergence}, we just need to show tightness of $\mathfrak H^N_{s,t}(\phi,\psi)$ for fixed $\phi,\psi$. But this is immediate from the first moment convergence in Theorem \ref{1stmom}.
\end{proof}

Finally, we use the axiomatic result of \cite{Tsa24} to conclude the proof of the theorem. We first state the main result of \cite{Tsa24}.

\begin{thm}[{\cite[Theorem 1.9]{Tsa24}}]
    \label{tsai} Let $G$ be a countable dense subset of $\mathbb R.$ Suppose that $\{M_{s,t} \}_{s<t}$ is a family of $\mathcal M_{\mathrm{loc}}(\R^2 \times \R^2)$-valued random variables indexed by $s,t \in G$ and assume that these random variables satisfy the following assumptions:
    \begin{enumerate}
        \item $\{M_{s_j,t_j}\}_{j = 1}^k$ are independent under $\mathbb P$ whenever $\{(s_j,t_j)\}_{j = 1}^k$ are disjoint intervals.

        \item For all $f\in C_c^\infty(\R^2 \times \R^2)$ and all $\psi\in C_c^\infty(\mathbb R^2)$ with $\int_\mathbb R\psi=1$, we have convergence
        \begin{equation*}
            \int_{(\R^2)^4 } \epsilon^{-2} \psi(\epsilon^{-1}(y_1-y_2))\cdot f(x,z) M_{t,u}(dy_1,dz) M_{s,t}(dx,dy_2)  \to \int_{(\mathbb R^2)^2} f(x,z) M_{s,u}(dx,dz)
        \end{equation*}
        in probability as $\epsilon\to 0$, whenever $s < t < u$.

        \item $\mathbb E \big[ \prod_{j=1}^h M_{s,t} (\phi_j \otimes \psi_j)\big] = \mathbb E \big[ \prod_{j=1}^h Z^\vartheta_{s,t} (\phi_j \otimes \psi_j)\big]$ for all $h \in \{1,2,3,4\}$ and all $\phi_j,\psi_j \in C_c^\infty(\mathbb R^2)$, where $Z^\vartheta_{s,t} (\phi,\psi)$ is the SHF with parameter $\vartheta \in \R$ defined in Definition \ref{def:axioms}.

    \end{enumerate}
    Then for any finite collection $\{(s_j,t_j)\}_{j = 1}^m$ of indices in $G$ with $s_j \leq t_j$, the collection $\{M_{s_j,t_j}\}_{j=1}^m$ has the same joint law as that of $\{Z^\vartheta_{s_j,t_j}\}_{j=1}^m$.
    Each coordinate here is equipped with the vague topology on $\mathcal M_{\mathrm{loc}}(\R^2 \times \R^2)$. 
\end{thm}

\begin{proof}[Proof of Theorem \ref{mr}] Let $G \subset \R$ be any countable dense subset containing the set of indices $\{(s_j,t_j)\}_{j=1}^m$ specified in the theorem statement. Let $\mathfrak H^\infty$ denote a realization from any limit point of $\mathfrak H^N$ in the infinite product space $\mathcal M_{\mathrm{loc}}(\R^2 \times \R^2)^G$ equipped with its product $\sigma$-algebra (such limit points exist by Theorem \ref{tight} and Kolmogorov's extension theorem, together with a diagonal argument). By Theorem \ref{4thmom}, the fifth moment of $ \mathfrak H^\infty_{s,t}(\phi,\psi)$ converges, giving uniform integrability of the first four moments. Then by Theorems \ref{1stmom}, \ref{2ndmom}, and \ref{4thmom} 
the first four moments of $ \mathfrak H^\infty_{s,t}(\phi,\psi)$ match those of the SHF.
Moreover, by Theorem \ref{convo}, $\mathfrak H^\infty$ has the convolution property needed for Theorem \ref{tsai}. The independence assumption of Theorem \ref{tsai} is trivially checked, as it holds in the prelimit. Thus an application of Theorem \ref{tsai} completes the proof.
\end{proof}

\subsection{Proof of Theorem \ref{mr2}}

When going from the discrete case to the continuous case, the first thing to note is that the definition of the renewal variables change, and it is important to start the time-integral from $t=0$ rather than $t=1$, otherwise the critical tuning \eqref{crit_tuning2} would be wrong.

\begin{defn}[Renewal variables for the continuous-time case]\label{cont_inc}
    We define the random renewal increments $(T^{x \to y}_{(N, j)}, A_{(N, j)}^{x \to y}) \in [0,N] \times \mathbb I^2$ by the joint density
\begin{equation}\label{(t,a)''}
    f_{(N,j)}^{x \to y}(r,a) =\frac{\mathbf p_{N,2} \big(r ; (x,0) , (a,a-y)\big)\ind_{\{0< r< N\}}}{\mathscr S_N(x,y)}, \qquad \mathscr S_N(x,y):= \int_0^N \mathbf p_{N,\mathrm{dif}} (s;x,y)ds.
\end{equation}
Note that if $\mathbb I=\mathbb R$, then $\mathscr S_N$ is no longer $L^\infty$ because of a logarithmic singularity along $x = y$ (see the short-time estimate in Assumption \ref{a3} \eqref{a3heatkernel}). Thus if $x=y$ and $\mathbb I=\mathbb R$, then we adopt the natural definition $(T^{x\to x}_{(N,j)}, A^{x\to x}_{(N,j)}) = (0,0)$. See that by integrating over $a\in \mathbb I^2$ we recover the marginal law for $T^{x\to y}_{(N,j)}$ given by the density $g_{(N,j)}^{x \to y} (r) = \frac{\mathbf p_{N,\mathrm{dif}}(r,x,y) \ind_{\{0< r< N\}} }{\mathscr S_N(x,y)}.$
For any deterministic sequence $\x = (x_j)_j\subset \mathbb Z^2$ of such gaps, we then define the partial sums
\begin{equation*}
    (\tau_{(N, k)}^{\x}, S_{(N,k)}^{\mathbf x}):= \sum_{j=1}^k (T^{x_{j-1} \to x_j}_{(N,j)} , A^{x_{j-1} \to x_j}_{(N,j)}).
\end{equation*}
\end{defn}

With this definition, we can already prove the discrete-space case of Theorem \ref{mr2}. 

\begin{proof}[Proof of Theorem \ref{mr2} in the case $\mathbb I=\mathbb Z$]
In continuous time, the overall strategy of replacing $\sum_{r=0}^{Nt}$ with $\int_0^{Nt} dr$ suffices everywhere. Every step of the proof is essentially identical after this replacement, using the renewal variables defined just above in place of the discrete versions from \eqref{(t,a)}. 

One thing that we should explain in some depth is the discrepancy between \eqref{crit_tuning2} and \eqref{crit_tuning}, namely that $\z_N$ appears in the critical tuning instead of $e^{\z_N} - 1$. Let us explain where this discrepancy appears in the proof. That is because a Taylor expansion of the exponential yields that
\begin{equation}
    e^{\int_0^t \boldsymbol{\eta}_{N,4} (\mathbf R_r) dr } = \sum_{k=1}^\infty \int_{0< r_1<\dots<r_k<t} \prod_{j=1}^k \boldsymbol{\eta}_{N,4} (\mathbf R_{r_j}) dr_1\cdots dr_k, \label{c1}
\end{equation}
whereas in the discrete case we instead have by $e^{\sum_j a_j} = \prod_j (e^{a_j}+1-1) = \sum_{j_1<\dots<j_k} \prod_1^k(e^{a_{j_i}}-1)$ that
\begin{equation}
    e^{\sum_{r=0}^t \boldsymbol{\eta}_{N,4} (\mathbf R_r)  } = \sum_{k=1}^\infty \sum_{0< r_1<\dots<r_k<t} \prod_{j=1}^k (e^{\boldsymbol{\eta}_{N,4} (\mathbf R_{r_j})}-1) dr_1\cdots dr_k. \label{c2}
\end{equation}
This change should be propagated into all renewal function calculations, in particular into \eqref{eq:timerenewal}--\eqref{un} and Lemma \ref{2.1}, as well as Proposition \ref{prop:RiemannSum}, and each of the operator definitions in Section \ref{sec:highermoments}. In the proofs of the second and fourth moment convergence results, all factors of $e^{\z_N(x)}-1$ should be replaced simply by $\z_N(x)$ throughout the entire discourse. These replacements will justify the required scaling \eqref{crit_tuning2}.
\end{proof}

Next we discuss the case $\mathbb I=\mathbb R$, which has several additional complications. 
We indicate only those places in the proof that need to be adjusted when $\mathbb I=\mathbb R$.



The first proof of Lemma \ref{2.1} no longer works since $\mathscr S_N$ are no longer bounded due to the log singularity. A truncation argument fixes this, but it is easier to adapt the second proof. We now pursue this approach after restating Lemma \ref{2.1} in the continuous setting.

\begin{lem}[Continuous-space version of Lemma \ref{2.1}]\label{2.1'}
    Let $f:\mathbb R^2 \to [0,\infty)$ be compactly supported, let $X_j$ be sampled i.i.d.\ from $\mu$ as in \eqref{mu_f}, and let $\mathfrak h_N$ be as in \eqref{h_n}.
    There exists a sequence of real numbers $\lambda_N(f)$ converging to 1 as $N\to\infty$, such that
    \begin{equation*}
        \lim_{N\to\infty} \sup_{k\in\mathbb N} \bigg| k\log \lambda_N(f) - \log E \Big[ e^{\sum_{j=1}^k \mathfrak h_N(X_{j-1},X_j) } \Big] \bigg| =0.
    \end{equation*}
    This sequence $\lambda_N(f)$ is equal to $\Lambda_N(f)/\nu_\infty(f)$, with $\Lambda_N(f)$ as in Definition \ref{def:lambdacont}. 
    \end{lem}

    \begin{proof}
    In this case, we cannot use the eigendecomposition for $n \times n$ matrices as written since now we have the integral operators $Q_Nh(x) = \frac1{\log N}\cdot  \frac{\int_J\mathscr S_N(x,z) f(z)h(z)dz}{\int_J\nu_\infty(z)f(z)dz},$ which act on the infinite-dimensional space $L^2(J)$ rather than the finite dimensional space $\mathbb R^J$. Nevertheless, the spectral theorem for compact operators still applies, as these integral operators are given by kernels with only a logarithmic singularity. The limit $Q_\infty$ is still a rank-one operator, and the convergence $Q_N\to Q_\infty$ is still in operator norm. Therefore, the principal eigenvalue converges to a strictly positive value exactly as in the finite-dimensional case, while the remainder of the spectrum converges to zero uniformly as $N\to\infty$. Moreover, since $\mathscr S_N(x,y) \leq C(1+\log_-|x-y|)$, it follows that every $Q_N$ maps $L^2(J)$ boundedly into $C(J)$, thus it follows that all eigenfunctions of $Q_N$ are continuous and converge uniformly to 1 on $J,$ and consequently all of the same arguments that were given in the proof of Lemma \ref{2.1} still apply.
\end{proof} 

\begin{cor}[Renewal function expansion - continuum analogue of Proposition \ref{prop:RiemannSum}]\label{rie2}
    Let $f:\mathbb R^2\to\mathbb R$ be compactly supported. For $t>0 , x,x',y,y'\in\mathbb R^2$, define the \textbf{renewal function}
    \begin{equation*}
        U_N(f; t,x,x',y,y')dydy' = \sum_{k=0}^\infty \sum_{0 < r_1< \cdots<r_k < t} \mathbf E^{(x,x')}_{N,2} \Bigg[\prod_{j=1}^k f(R^1_{r_j}-R^2_{r_j}) \ind_{\{ R^1_{t}= dy, R^2_{t}=dy'\}}\Bigg].
    \end{equation*}
The density on the left side is well-defined, and moreover if $h:\mathbb R^2 \to \mathbb R$ is bounded and measurable, then
    \begin{align*}
        \int_{\mathbb R ^2} U_N(e^{\boldsymbol{\zeta}_N}-1;& t; x, x', y, y')\z_N(y-y')h(y')dy'\\
        &= \sum_{s\in (\log N)^{-1}\mathbb N} e^{\vartheta s + o(1)} \cdot \widetilde E_{N,s\log N} \bigg[ h(y- X_{s \log N}) \ind_{\big\{ \tau^{\mathbf X}_{(N, s\log N)} = t, S^{\mathbf X}_{(N, s \log N)} = y-x\big\}}\bigg] 
    \end{align*}
    with the $X_j$ sampled from the tilted measure $\tilde P_{N,k}$ as in \eqref{eq1}.
\end{cor}

\begin{proof} The proof essentially follows that of Proposition \ref{prop:RiemannSum}, however, a few points need to be clarified. 

The first point is the very existence of the density $U_N(f; t,x,x',y,y')$ above. This essentially follows from applying the Markov property and then expanding all of the integrals in the definition of the renewal functions, exactly as in \eqref{trigger}. The important thing to note here is that, in the first place, Assumption \ref{a3} imposes the existence of transition kernels $\mathbf p_{N,h}(t,\x,\y)$, so that the expansion \eqref{trigger} explicitly proves the existence of the density.

The second point is the error term $o(1)$ appearing in the exponent. Note that, unlike Lemma \ref{2.1}, it is no longer the case that the limit in Lemma \ref{2.1'} is \textit{uniform} over all $f$ whose support is contained in a fixed compact subset of $\mathbb R^2$ (for example, taking $\mu$ in \eqref{mu_f} to approach a Dirac mass would cause a serious problem due to the log singularity of $\mathscr S_N$). Instead, it only remains true for a \textit{fixed} continuous function $f$ of compact support. However, the uniformity is no longer required because of our stronger constraint on the functions $\z_N$, namely, that they are of the form $\beta_N^2 \z(x)$ for a \textit{fixed} function $\z$ which does not depend on $N$, see the third point in Assumption \ref{a3} \eqref{a3markov}. This means that the law $\mu$ from \eqref{mu_f} does not depend on $N$. This is important 
because a lack of uniformity would potentially affect the error term appearing in the exponent of the statement of the corollary (causing it to possibly blow up as $N$ grows large).
\end{proof}

Next, we note that there is a significant subtlety that arises in the proof of the Dickman convergence result (Theorem \ref{convergence_dickman}), and in fact that theorem is false in its stated form if $\mathbb I=\mathbb R$. Indeed, $\mathscr S_N(x,y) = +\infty$ if $x=y$, which is why we took the convention that $T^{x\to x}_{(N,j)}=0$. This does impact the Dickman convergence result because if all gaps are equal to the same value $(x\to x)$, then clearly the convergence will fail.

Thus, the statement of Theorem \ref{convergence_dickman} needs to be modified so that rather than arbitrary sequences $\x_N$, we consider only some special sequences in which the consecutive gaps are sufficiently far apart. We fix some $T>0$ and consider only those sequences $\x_N=(x_1^N,\ldots, x^N_{\lfloor T\log N\rfloor })$ such that 
\begin{equation}\label{100}\{\sqrt{T\log N} \le j \le T\log N -\sqrt{T\log N}: |x_j^N-x_{j-1}^N| \le (\log N)^{-100} \} =\emptyset.
\end{equation}
These exponents are not sharp, but they suffice for the proofs below.

\begin{thm}[Dickman convergence result in the continuum setting]\label{dick2} Let $T>0$, and let $\x_N=(x_1^N,\ldots, x^N_{\lfloor T\log N\rfloor })$ be any collection of sequences satisfying \eqref{100}. Then the pair of processes $\big(N^{-1}\tau^{\mathbf x_N}_{(N,s\log N)} , N^{-1/2} S^{\mathbf x_N}_{(N,s\log N)}\big)_{s \in [0,T]}$ given in Definition \ref{cont_inc} converge in law to $\big(
Y_{s},\frac1{\sqrt{2}} W_{
Y_{s}}\big)_{s \in [0,T]}$ as $N\to \infty$, where $Y_s$ is the Dickman subordinator and $W$ is an independent standard Wiener process in $\mathbb R^2$, both started from $0$. Convergence is in the sense of finite-dimensional distributions.
\end{thm}

\begin{proof}
Much of the proof of Theorem \ref{convergence_dickman} goes through verbatim, but there is a substantial problem in the calculations leading up to \eqref{f2}. In that proof, every instance of $\sum_{r=1}^N$ needs to be replaced by $\int_0^N dr$. The convergence \eqref{f2} remains uniform here, which is easy to check thanks to $(u,a) \mapsto e^{i\mu a+i\lambda \bullet a}-1$ vanishing at the origin, and the uniformity in Assumption \ref{a3} \eqref{a3logasymptotic}. However, the convergence $\frac{\log N}{\mathscr S_N(x_{j-1},x_j)} \to \frac{1}{\nu_\infty(x_j)}$ is no longer uniform! It is actually possible that $\frac{\log N}{\mathscr S_N(x_{j-1},x_j)} =0$ if $x_{j-1}=x_j$, for example. Thus, if we have a sequence $\x=(x_j)$ such that the gaps $|x_{j-1} - x_j|$ are very small for too many values of $j$, then this will cause a serious problem in the convergence \eqref{uni0}.\footnote{It is tempting to believe that one may be able to replace instances of $\int_0^Ndr$ by instances of $\int_1^Ndr$ in the definition of $\mathscr S_N$ and the renewal variables, however this is not permissible because the time component of the renewal variables must start at zero, otherwise the constant $\vartheta$ in the critical tuning \eqref{crit1} would change by an $O(1)$ shift. Thus, the integrals \textit{must} start from 0, and the singularity is inevitable.} 

However, under the hypothesis \eqref{100}, we can actually estimate the error. First note that $\mathscr A_N(x,y):= \int_0^1 \mathbf p_{N,\mathrm{dif}}(t,x,y)dt \leq C(1+\log_-|x-y|)$ thanks to the short-time heat kernel estimate in Assumption \ref{a3}. Second, note that $\mathscr B_N(x,y) = \int_1^N \mathbf p_{N,\mathrm{dif}}(t,x,y)dt$ satisfies $\frac{1}{\log N} \mathscr B_N(x,y) \to \nu_\infty(y)$ uniformly over $x,y\in J$ as $N\to \infty$, by Assumption \ref{a3} \eqref{a3logasymptotic} (the idea here is that starting at time 1 removes the singularity along $x=y$). Thus, if $|x-y| \ge (\log N)^{-100}$, we see that $\mathscr A_N(x,y) \leq C(1+100 \log\log N)$, and thus $\frac1{\log N}\mathscr S_N(x,y) =\frac{1}{\log N} \big( \mathscr A_N(x,y) + \mathscr B_N(x,y)\big) = \frac{\mathscr B_N(x,y)}{\log N} +O (\tfrac{\log\log N}{\log N}) \to \nu_\infty(y)$ uniformly in the set of such $x,y$ satisfying this constraint. This discussion means that \eqref{uni1} is actually uniform over the set of all $x_{j-1},x_j$ satisfying $|x_{j-1}-x_j| \ge (\log N)^{-100}$ (all sums appearing there are now replaced by integrals). Thus we see that if \eqref{100} holds, then the exponential in \eqref{uni0} will still converge to the correct limit, with a multiplicative error term that is at worst $e^{O((\log N)^{-1/2})},$ 
thus proving the result among this class of $\mathbf x_N$ satisfying \eqref{100}.
\end{proof}

We are not finished with the discussion of the Dickman convergence, because now we need to show that the assumption \eqref{100} actually holds with high probability for the gap variables $X_j$ under the tilted measures $\tilde P_{N,k}$ from \eqref{eq1}. Thus we need to study the precise behavior of the successive gaps $(X_j)_{j=1}^k$ under $\tilde P_{N,k}$, which we avoided in the discrete-space cases. 

\begin{thm}[Estimating the difference of consecutive gaps under $\tilde P_{N,k}$]\label{dick3} Under the change of measure from \eqref{eq1}, we have an upper bound of the form 
\begin{equation}\label{99}\sum_{j=\sqrt{T\log N}}^{T\log N-\sqrt{T\log N}} \tilde E^{x_0}_{N,T\log N} \bigg[ \ind_{\{|X_j-X_{j-1}| \leq (\log N)^{-100} \}} \bigg] \leq C(\log N)^{-49},
\end{equation}where $C$ is uniform over $N$, and we are now using the superscript $x_0$ to specify the starting position $X_0=x_0.$ Thus by Markov's inequality ($P(X\ge 1) \le E[X])$, this bound implies that if $(x_j)_j$ are sampled from $\tilde P_{N,T\log N}$, the identity \eqref{100} holds with probability greater than $1-C(\log N)^{-49}$. 
\end{thm}

\begin{proof} To prove \eqref{99}, we prove more generally that 
\begin{equation}
    \label{98} \tilde P_{N,k}^{x_0} (|X_j-X_{j-1}| <\epsilon) \leq C(\sqrt \epsilon+ke^{-\alpha \sqrt k}),
\end{equation}
where $C,\alpha$ are uniform over $\sqrt k\le j \le k-\sqrt{k}, x_0\in J, \epsilon\in (0,1),$ and $N,k\in\mathbb N.$ To prove this, let $\mu$ be as in \eqref{mu_f}. 
Note by the definition \eqref{eq1} that $\tilde P^{x_0}_{N,k}$ is defined by tilting $\mu^{\otimes k}$ by the Radon-Nikodym derivative $\frac1{Z}e^{\sum_{j=1}^k \mathfrak h_N(x_{j-1},x_j)}$ where $Z$ is the normalizing constant and $\mathfrak h_N$ are as in \eqref{h_n}. Thus, under $\tilde P_{N,k}$ one may verify that $(X_j)_{j=0}^k$ form a time-inhomogeneous Markov chain $(X_0,\ldots,X_k)$ with transition probability of $\{X_t=x\}\to \{X_{t+1}=y\}$ given by the formula 
\begin{equation}\tilde p_{N,k} (t,t+1;x,y) = \frac{e^{\mathfrak h_N(x,y)} \mu(y) E [ e^{\mathfrak h_N(y,X_{t+2}) + \sum_{j=t+3}^{k} \mathfrak h_N(X_{j-1},X_j)}]}{E [ e^{\mathfrak h_N(x,X_{t+1}) + \sum_{j=t+2}^{k} \mathfrak h_N(X_{j-1},X_j)}]},\label{inhom}
\end{equation}where $x,y\in J$ and $t=0,1,\ldots,k-1.$ In the above expectations, $X_j$ are sampled i.i.d.\ from $\mu$, and $\mathfrak h_N(X_{j-1},X_j)$ is understood as 0 if $j>k.$ To prove \eqref{98}, the idea is to show that for $j \in [\sqrt k, k-\sqrt k]$ this Markov chain is well-approximated by a time-homogeneous one with a positive spectral gap (independent of $N$) such that at its equilibrium $\varrho_N$, this Markov chain satisfies $\varrho_N(|X_j-X_{j-1}|<\epsilon)\le C\sqrt \epsilon.$ We have included a buffer time of order $\sqrt k$ at both ends of the length-$k$ interval because it takes some time to reach equilibrium at the beginning, and likewise the time-homogeneous approximation will fail close to the end of the interval (obviously $\sqrt k$ here is not optimal, since time to equilibrium should be $O(1)$).\footnote{The intuition for the Markov property of $\tilde P_{N,k}$ is simply that we have a Gibbs measure with nearest-neighbor Hamiltionan $\mathfrak h_N(x,y)$ and base measure $\mu^{\otimes k}$. Far from the boundaries, the finite-volume Gibbs measure will be approximable by its thermodynamic (large-volume) limit as the size of the interval grows large, since we are on the 1d lattice $\mathbb Z.$ This thermodynamic limit would certainly be a time-homogeneous chain. With this intuition, one could alternatively use the cluster expansion in \cite[Theorem 6.38-Proposition 6.39]{FV} to prove the required bound \eqref{98}.}

To prove this homogeneous approximation, we study the ratio of the expectations in \eqref{inhom}. Define $V_N(x,y):=e^{\mathfrak h_N(x,y)}$, so that $V_N(x,y)\to 1$ for $x\ne y$ as $N\to \infty$, and moreover $V_N(x,y) \leq C(1+\log_-|x-y|)$. Let $\boldsymbol{\psi}_1^N:J\to \mathbb R$ be the top eigenfunction for $\mathbf V_N f(x) := \int_J f(y) V_N(x,y)\mu(y)dy,$ say with top eigenvalue $E_1(N).$ Note that since $\mathbf V_N$ is converging in operator norm to a rank-one projection operator $f\mapsto \langle f,\ind_J\rangle \ind_J,$ it follows that $E_1(N)\to 1$ and $\boldsymbol{\psi}_1^N \to \ind_J$ in $L^2(J)$ as $N\to\infty$ (and therefore uniformly on $J$, since $\boldsymbol{\psi}_1^N$ is an eigenvector and $\mathbf V_N$ maps $L^2$ boundedly into $C(J))$. Furthermore, all other eigenvalues of $\mathbf V_N$ are converging uniformly to 0. Combining these observations, and noting that $E [ e^{\mathfrak h_N(x,X_{t+1}) + \sum_{j=t+2}^{k} \mathfrak h_N(X_{j-1},X_j)}] = \mathbf V_N^{k-t}\ind_J (x),$ by eigenfunction expansion we see immediately that the ratio of expectations in \eqref{inhom} satisfies 
\begin{equation}\label{k-t}
    \bigg| \frac{ E [ e^{\mathfrak h_N(y,X_{t+2}) + \sum_{j=t+3}^{k} \mathfrak h_N(X_{j-1},X_j)}]}{E [ e^{\mathfrak h_N(x,X_{t+1}) + \sum_{j=t+2}^{k} \mathfrak h_N(X_{j-1},X_j)}]}  - \frac{\boldsymbol{\psi}_1^N(y)}{E_1(N) \boldsymbol{\psi}_1^N(x) }\bigg| \leq Ce^{-\alpha (k-t)}, 
\end{equation}
where $C,\alpha>0$ are independent of $N,k,t,x,y$. Note that if $t<k-\sqrt{k}$ then the right side is upper-bounded by $Ce^{-\alpha \sqrt{k}}. $ In particular, if we define the time-homogeneous Markov chain $\tilde Q^x_N$ with one-step transition density 
\begin{equation}\label{qc2}
    \tilde q_N(x,y) := \frac{e^{\mathfrak h_N(x,y)} \mu(y) \boldsymbol{\psi}_1^N(y) }{E_1(N)\boldsymbol{\psi}_1^N(x)},
\end{equation}
then clearly \eqref{k-t} implies a total variation bound 
\begin{equation}
    \label{tvpq}
\sup_{x\in J}\| \tilde p_{N,k} (t,t+1,x,\bullet) - \tilde q_N(x,\bullet)\|_{TV} \leq Ce^{-\sqrt k} \qquad \mathrm{for} \qquad 0\le t \le k-\sqrt k.\end{equation}
This means that we can couple the Markov chain $(X_j)_j$ sampled from $\tilde P_{N,k}^{x} $ with the Markov chain $(Y_j)_j$ sampled from $\tilde Q^x_N$ so that the trajectories coincide for all times $1\le t \le k-\sqrt k$, with probability greater than $1-Cke^{-\alpha \sqrt k}$, where the extra factor $k$ comes from a simple union bound over all the possible one-step transitions up to time $k-\sqrt k \le k$. 

As we already observed, the Markov chain \eqref{qc2} has transition density converging uniformly to $\mu(y)$ as $N\to \infty$, which implies that all other eigenvalues except for $E_1(N)$ are converging uniformly to zero, thus the Markov chain clearly has a spectral gap and mixing time independent of $N$. The invariant measure of the Markov chain $\tilde Q_N$ is equal to $\boldsymbol{\pi}_N(x) = \boldsymbol{\psi}_1^N(x)^2$, by direct calculation, which means that $\boldsymbol{\pi}_N \to 1$ uniformly. Since the Markov chain thus has a mixing time independent of $N$, this means we can couple $(Y_j)_j$ sampled from $\tilde Q^x_N$ with the stationary version of the same Markov chain $(Z_j)_j$ such that $Y_j=Z_j$ for all $j\ge \sqrt k$ with probability greater than $1-Ce^{-\alpha \sqrt k}$.

Super-imposing both couplings $(X,Y)$ and $(Y,Z)$, we can thus couple the original Markov chain $(X_j)$ sampled from $\tilde P_{N,k}^{x} $ with  the stationary chain $(Z_j)_j$ so that $X_j = Z_j$ for all $\sqrt k \le j\le k-\sqrt k, $ with probability greater than $1-Cke^{-\alpha \sqrt k}$, where $C,\alpha$ are independent of $N,k$. If we calculate the analogue of the left side of \eqref{98} for the stationary chain $(Z_j)_j$, we find that 
\begin{equation}
    \tilde Q^{\boldsymbol{\pi}_N}_N(|Z_j-Z_{j-1}| \le \epsilon ) = \int_{J\times J} \boldsymbol{\psi}_1^N(x)^2 \tilde q_N(x,y) \ind_{\{|y-x| \le \epsilon\}} dydx \leq\int_{J} \boldsymbol{\psi}_1^N(x)^2 \cdot \sqrt{\epsilon} dx \leq C\sqrt \epsilon,
\end{equation}
where we note that $\sqrt \epsilon$ is not an optimal upper bound for the inner integral over $y$, but suffices for us (note that $\tilde q_N$ has a log singularity along $x=y$). By the coupling, this implies that \eqref{98} is upper bounded by $C\sqrt \epsilon + Cke^{-\alpha \sqrt k},$ where the exponential term comes from the possibility of the coupling $X_j=Z_j$ failing to hold for some $j\in [\sqrt k,k-\sqrt k]$, completing the proof.
\end{proof}

\begin{cor}\label{dick4}
    Suppose that $\mathbf x_N$ are randomly sampled from $\tilde P_{N,\lfloor T\log N\rfloor}$, independently of the renewal variables from \eqref{(t,a)''}. The pair of processes $\big(N^{-1}\tau^{\mathbf x_N}_{(N,s\log N)} , N^{-1/2} S^{\mathbf x_N}_{(N,s\log N)}\big)_{s \in [0,T]}$ given in Definition \ref{cont_inc} converge in law to $\big(
Y_{s},\frac1{\sqrt{2}} W_{
Y_{s}}\big)_{s \in [0,T]}$ as $N\to \infty$, where $Y_s$ is the Dickman subordinator and $W$ is an independent standard Wiener process in $\mathbb R^2$, both started from $0$. Convergence is in the sense of finite-dimensional distributions.
\end{cor}

\begin{proof}
    By independence, we can disintegrate along the law of the gap variables $x_j$, which as stated are sampled from $\tilde P_{N,\lfloor T\log N\rfloor}$. Consider the ``bad event" $B_N$ in which \eqref{100} does not hold, and we see by Theorem \ref{dick3} that the probability of $B_N$ is upper-bounded by $C(\log N)^{-49}. $ 
    
    On the other hand, on $B_N^c$, the \textit{quenched} characteristic function of the pair $(N^{-1}\tau^{\x_N}, N^{-1/2} S^{\x_N})$ given the gap variables $x_j$ converges in probability to the correct (deterministic) limit by Theorem \ref{dick2}. Taking expectation over the gap variables and applying the bounded convergence theorem then yields that the unconditional characteristic function also has the correct limit.
\end{proof}


Now we are finally ready to prove the main theorem for the case $\mathbb I=\mathbb R$. 

\begin{proof}[Proof of Theorem \ref{mr2} in the case $\mathbb I=\mathbb R$]

Thanks to the renewal representation of Corollary \ref{rie2} and the Dickman convergence result of Corollary \ref{dick4}, all of the calculations that were done in the discrete setting in Theorem \ref{thm:macro} - Theorem \ref{limitofu2} go through unchanged in the continuum setting. Since these are the crucial inputs for the second-moment convergence in Theorem \ref{2ndmom} and the semigroup property in Theorem \ref{convo}, the proofs of those theorems therefore go through directly in the present setting. 
Thus, what remains to be shown is how to adapt the proof of the higher moments convergence results from Section 3. The most difficult part is to carefully adapt the definition of the operators to the continuum spatial setting, and then derive the resulting contraction identities. Thus, adjusting these definitions accordingly is the aspect that we focus on. 

In the case $\mathbb I=\mathbb R$, we need to adapt the stopping time arguments in the fourth-moment convergence result of Theorem \ref{4thmom}. Recall the partition update times from Definition \ref{stop}. These stopping times are hitting times of closed sets which may share boundaries, so in principle the number of partition updates could be infinite in the continuous-time setting, which means that the stopping times could themselves be undefined (similar to the zero set of standard one-dimensional Brownian motion being an uncountable set). Moreover, it is not yet clear how to generalize the discrete-time setup to continuous time when we disintegrate along the possible values of the stopping times and the possible spatial locations of the Markov chains at those stopping times (as in Proposition \ref{s=1234a}). For this, we adjust the definition of the relation $\precsim$ so that $\x\precsim I$ means that $\mathrm{dist}(\x, \{\y: \y\sim I \}) < 1$ or $\x \sim \hat 0$, then adjust the definition of the stopping times so that the next partition update only occurs when $\mathbf R_t\not\precsim I$ if the previous partition update was associated to the partition $I$. Then the number of partition updates on any finite time horizon $[0,t]$ is necessarily finite, since a continuous path cannot traverse a fixed distance infinitely many times in a bounded interval.\footnote{
    More precisely, if $f: [0,1] \to \mathbb R^d$ is continuous, then there is no increasing sequence $t_n\uparrow t \in [0,1]$ for which $|f(t_i)- f(t_{i-1})| \geq 1$ for all $i\in\mathbb N.$
} Then we would accordingly adjust the definition of the operators from Definition \ref{pij} to 
    $$\mathscr P_N^{I,J}(t,\x,d\y):= \ind_{\{\x\sim I, \y \sim J\}} \mathbf E^{\x}_{N,4} \bigg[ e^{\int_0^t \boldsymbol{\eta}_{N,4} (\mathbf R_r) dr } \ind_{\{\mathbf R_t =d\y, \mathbf R_r \precsim I, \forall 0\le r<t\}} \bigg].$$ 
    This expression is purely formal because of the ``$d\y$'' on the right side, but it can be rigorously understood as a regular conditional probability, defined by the hitting density of the Markov chain $\mathbf R_t$ which is killed as soon as it leaves the set $\{\x\precsim I\}$. Then the associated operator is still well-defined and given by this kernel. 
    
    We now define the ``collision spaces" $\mathbf C_J := \{ \x \in (\mathbb R^2)^4: \x\sim J\},$ so that $\mathscr P_N^{I,J}$ can be viewed as operators from $\mathbf C_J\to \mathbf C_I$ by the formula $\mathscr P_N^{I,J} f(\x) = \int_{\mathbf C_J} f(y) \mathscr P^{I,J}_N (\x,d\y)$ for $\x\in \mathbf C_I.$ These spaces $\mathbf C_I$ are the natural analogue of the discrete collision spaces $(\mathbb Z^2)^4_J$ in this continuum setting.

    Likewise, the operators $\mathscr Q_{N,\mathrm{con}}$ and $\mathscr U_{N,\mathrm{con}}$ should be defined as regular conditional probabilities in a similar fashion, more or less exactly as in Definition \ref{cons}:
    \begin{align*}
        \mathscr Q_{N,\mathrm{con}}^{I,J} (t,\x,d\y)& := \ind_{\{\x\sim I, \y\sim J\}} \mathbf E^{\x}_{N,4} \bigg[ \ind_{\{\mathbf R_t=d\y, \;\mathbf R_r \precsim I\; \forall 0\le r<t\}} \bigg],\\
        \mathscr U^J_{N,\mathrm{ con}}(t,\mathbf x,d\mathbf y)&:=\begin{cases} \boldsymbol{\eta}_{N,4}(\x) \delta_{\x}(d\y) \ind_{\{\x\sim J\}}, & t=0 \\[5pt] \boldsymbol{\eta}_{N,4}(\x)
        \mathscr P_N^{J,J}(t,\x,d\y)\boldsymbol{\eta}_{N,4}(\y) \ind_{\{\x,\y\sim J\}}, & t>0.\end{cases} 
    \end{align*}
    Note that $e^{\boldsymbol{\eta}}-1$ has been replaced by $\boldsymbol{\eta}$ here for the same reasons explained in \eqref{c1}-\eqref{c2}.
    
    Then the recursion of Lemma \ref{recur} still holds true, though the integral should have an extra Dirac component along $u=v:$
    \begin{equation*}\mathscr P_N^{I,J}(t) =  \mathscr Q_{N,\mathrm{con}}^{I,J}(t)  + \int_{0\le u \le v< t} \mathscr Q^{I,I}_{N,\mathrm{con}}(u) \mathscr U^I_{N,\mathrm{con}} (v-u) \mathscr Q^{I,J}_{N,\mathrm{con}}(t-v) \Big[ dudv + \delta_u(dv)du\Big].
\end{equation*}
Now, we make the following important observation: despite the fact that all of these kernels were defined via regular conditional probabilities, the measures $\mathscr Q_{N,\mathrm{con}}^{J,J}(t,\x,\bullet)$ and $\mathscr U^J_{N,\mathrm{con}} (t,\x,\bullet)$ actually have densities with respect to Lebesgue measure. In other words, we can write $\mathscr Q_{N,\mathrm{con}}^{J,J} (t,\x,d\y) = \mathscr Q_{N,\mathrm{con}}^{J,J}(t,\x,\y)d\y$\footnote{This will not be true for $\mathscr Q^{I,J}_{N,\mathrm{con}}$ with $I\ne J$, which in general could only hope to have some hitting density that is absolutely continuous with respect to Hausdorff measure on the boundaries of $\mathbf C_J$. But we will observe that these objects with $I\ne J$ are irrelevant for the final expansion.} and $\mathscr U^J_{N,\mathrm{con}} (t,\x,d\y) = \mathscr U^J_{N,\mathrm{con}} (t,\x,\y)d\y,$ where the kernels are continuous in both variables. This can actually be proved directly from the fact that the free versions of these operators have densities with respect to Lebesgue measure, and (even stronger than just the existence of the densities) this argument will show that we still have the upper bounds 
\begin{equation*}
    \mathscr U^J_{N,\mathrm{con}} (t,\x,\y) \le \mathscr U_{N,\mathrm{free}} (t,\x,\y) , \qquad \mathscr Q^{J,J}_{N,\mathrm{con}}(t,\x,\y) \le \mathscr Q^{J,J}_{N,\mathrm{free}} (t,\x,\y),
\end{equation*}
for all $t>0$, where the free versions are defined in the obvious way in this context (e.g., by adapting Definitions \ref{ufree} and \ref{ufree1}). Again, we make no such claim for $\mathscr Q^{I,J}_{N,\mathrm{con}}(t,\x,d\y)$ with $I\ne J$, and in fact the analogous claim is unnecessary and would be nonsensical in this context anyways.

    In this context, Proposition \ref{op_exp} and the contraction identities of Theorem \ref{contraction} still hold, and so one still obtains the final expansion of Theorem \ref{part_v2.5}. In this final expansion, there are no remaining instances of $\mathscr Q^{I,J}_{N,\mathrm{con}}(t,\x,d\y)$ with $I\ne J$ because they have all been ``contracted away." Only free versions remain when $I\ne J$. Thus, all of the remaining operators have kernels that admit a density with respect to the Lebesgue measure on $\mathbf C_J$ (except for instances of $\mathscr U_{N,\mathrm{con}}(0,\x,d\y)$ which have Dirac components; these are easily handled separately, as they have operator norm $O(1/\log N))$ and are irrelevant for the limit, as we have already observed in the discrete case). This means that once we have reached Theorem \ref{part_v2.5}, there is no need to worry about any kernels that have singular components with respect to Lebesgue measure, and we can proceed exactly as in the discrete case.

    Consequently, all of the operator norm bounds and estimates still go through unchanged, though every instance of the spaces $(\mathbb Z^2)^4_J$ should now be replaced by its continuous-space counterpart $\mathbf C_J$. Likewise, the exact limits of the kernels still go through verbatim. Then the tightness arguments and the identification of limit points also goes through verbatim.
\end{proof}

\subsection{Theorems \ref{check} and \ref{check2}}

For Theorems \ref{check} and \ref{check2}, we will rely on the theory of \emph{SRI chains} that was developed in \cite[Section 4]{DP25}. We will state a more general theorem that will allow us to verify both theorems at once.


\begin{ass}\label{B}
    Let $I$ be any locally compact subgroup of $\mathbb R^2$, and let ``Lebesgue measure on $I$" refer to the Haar measure normalized so as to have density 1 at infinity. 
    Let $\{\mathbf P_{N,h}\}_{N\ge 1}$ be a sequence of Markov chains on $I^h$ and let $\mathbf P_{\infty,h}$ be another Markov chain on $I^h$. These Markov chains may be discrete- or continuous-time.  The time $t=1$ transition probabilities will be denoted by $\mathbf p_{N,2}(\x,d\y)$ and $\mathbf p_{\infty,2}(\x,d\y)$. Suppose that this sequence of Markov chains satisfies the following conditions for all $1 \leq h \leq 5$:
    \begin{enumerate}
        \item\label{Bstrongfeller} The transition kernels $\mathbf p_{N,h}(\x,d\y)$ and $\mathbf p_{\infty,h}(\x,d\y)$ have a globally bounded and continuous density with respect to Lebesgue measure on $I$, which decays to 0 exponentially fast in $|\x-\y|$, at a rate that is independent of $N$. In particular, the Markov chains $\mathbf P_{N,h}$ and $\mathbf P_{\infty,h}$ have the strong Feller property.
        \item\label{BconvTV} The transition probabilities $\pN{h}(\x,\bullet)$ converge uniformly in total variation distance: for all $\x=(x_1,\dots,x_h)$
        \begin{align*}
            \lim_{N\to\infty} \sup_{\x\in I^h} \norm*{\mathbf p_{N,h} (\x,\bullet) - \mathbf p_{\infty,h} (\x,\bullet)}_{TV} &=0.
        \end{align*}
        \item\label{Bexpmoments} We have exponential moments for the one-step transition probabilities: for some $\epsilon > 0$,
        $$
        \sup_N \sup_{\x\in I^h} \int_{I^h} e^{\epsilon|\y-\x|} \mathbf p_{N,h} (\x,d\y) <\infty.
        $$
        \item\label{Birred} The Markov chains $\mathbf P_{N,h}$ and $\mathbf P_{\infty,h}$ are topologically irreducible, and for each $N$ there exists a measure $\mu_N$ on $I$ converging weakly to a measure $\mu_\infty$ with covariance given by the $2 \times 2$ identity matrix and satisfying for all $\x=(x_1,\dots,x_h)$ the estimate
        \begin{align*}
            \sup_N&\norm*{\mathbf p_{N,h} (\x,\bullet) - \mu_N^{\otimes h} (\x-\bullet)}_{TV} \leq Ce^{-\alpha\min_{1\le i<j\le h} |x_i-x_j|},
        \end{align*}
        which should be interpreted for $h=1$ as $\mathbf p_{N,1} (x,\bullet) = \mu_N(x-\bullet)$ for all $x\in I$. Here the constants $C,\alpha>0$ do not depend on $N$.
        \item\label{Btransinv} Assume that for all $N$ we have $\mathbf p_{N,h} (\x,\y ) = \mathbf p_{N,h} (\x+(a,\dots,a), \y+(a,\dots,a))$.
        \item\label{Bconsistency} 
        The $\mathbf P_{\infty,h}$ are projective in the sense that any subset of $m<h$ coordinates of $\mathbf P_{\infty,h}$ is distributed as $\mathbf P_{\infty,m}$ (but this is not necessarily true in the prelimit). 
    \end{enumerate}
\end{ass}

\begin{thm}\label{A} In the setting of Assumption \ref{B}, the following results hold. 
    \begin{enumerate}
    \item\label{Ainvarianceprinciple} (Invariance principle) Put $d_N:= N\int_I y\mu_N(dy)$. If $N^{-1/2} \y_N \to \y$, then in the topology of $C[0,T]$ the process $N^{-1/2}(\mathbf R_{Nt} - (d_Nt,\dots,d_Nt))_{t \ge 0}$ under $\PN{h}^{\y_N}$ converges in law to Brownian motion in $(\R^2)^h$ started from $\y$.
    
    \item\label{Aheatkernel} (Heat kernel upper-bounds) The kernels $\mathbf p_{N,h}$ have densities with respect to Lebesgue measure on $I^h$, and moreover these densities satisfy the upper bound
    \begin{equation*}
        \pN{h}(t, \x,\y) \leq C t^{-k} \big(1+t^{-1/2} |\y-\x-N^{-1}(d_N,\dots,d_N)t| \big)^{-S},
    \end{equation*}
    where $S$ can be taken as large as desired.
    
    \item\label{Auniqueinvmeas} Up to scalar multiplication, there exists a unique invariant measure $\nu_\infty$ of constant growth at infinity for the gap chain of $\mathbf p_{\infty,2}$, defined as the Markov chain on $I$ with one-step transition probability $\mathbf p_{\infty,\mathrm{dif}}(x,y) := \int_{I} \mathbf p_{\infty,2} ((x,0),(a,a-y))da.$ Moreover, the measure $\nu_\infty$ has a density with respect to Lebesgue measure on $I$.

    \item\label{Adiaglimits} (Diagonal limits) In the discrete-time case, there exists a unique constant multiple of $\nu_\infty$ such that the following limits hold.
    Let $\mathbf p_{N,\mathrm{dif}}(r;x,y):= \int_{I} \mathbf p_{N,2} (r,(x,0), (a,a-y))da$.
    For all fixed values of $x,y\in I$ we have
    \begin{align*}
        \sum_{r=0}^N \mathbf p_{N,\mathrm{dif}}(r;x,y) &= \nu_\infty(y) \log(N) + o(\log N) 
    \end{align*}
    as $N\to\infty.$ Furthermore, if $(\x_N) \subset (\Z^2)^2$ is well-separated with limit $N^{-1/2}\x_N \to (x_1,x_2)$, then for all  $t>0$, $y\in \mathbb Z^2$, and smooth bounded $\phi:\mathbb R^2\to \mathbb R$ we have 
    \begin{align*}
        \lim_{N\to \infty} \sum_{r=0}^{Nt} &\int_{I} \phi(N^{-1/2} (a-N^{-1}d_Nr)) \mathbf p_{N,2} (r,x_N^1,x_N^2,a,a-y)da \\&= 4 \pi \cdot  \nu_\infty(y) \int_0^t \mathbf g(2s,x_1-x_2) \int_{\mathbb R^2} \mathbf g(2s,z) \phi(\tfrac12 (z+x_1+x_2)) 
        \,dz\,ds.
    \end{align*}
    If we are instead in the continuous-time case, then the same limits hold verbatim, except with $\sum_{r=0}^N$ replaced by $\int_1^N dr$ on the left-hand side.
    \end{enumerate}
\end{thm}

\begin{proof}First we note that \begin{align}\liminf_N \inf_{\x\in I^h}&\int_{I^h} |y_i-y_j-(x_i-x_j)| \mathbf p_{N,h} (\x,d\y) >0 .\label{delta_irr}
\end{align}
We prove this for $h=2$. The general case then follows from the total variation convergence in Item \eqref{BconvTV}, the consistency condition of Item \eqref{Bconsistency}, and the exponential moment bounds in Item \eqref{Bexpmoments}. For $h=2$, proving the claim is equivalent by Item \eqref{Btransinv} to showing that
\begin{equation*}
    \liminf_N \inf_{x\in I} \int_I |y-x| \mathbf p_{N,\mathrm{dif}} (x,dy) >0.
\end{equation*}
If we let $g_N(x):= \int_I |y-x| \mathbf p_{N,\mathrm{dif}} (x,dy) >0,$ then from the strong Feller property and the exponential moment bounds, it is easy to see that $g_N$ is continuous. Furthermore, Item \eqref{Birred} above ensures that $|g_N(x) - c_N| \leq Ce^{-\alpha|x|/2}$ where $c_N:= \int_I\int_I |u-v| \mu_N(dx) \mu_N(dy).$ Item \eqref{Birred} also ensures that $c_N \to c_\infty>0$, where $c_\infty: =\int_I\int_I |u-v| \mu_\infty(du) \mu_\infty(dv).$ Consequently we just need to show that there exists no bounded sequence $x_N\to x \in \mathbb R^2$ such that $g_N(x_N) \to 0.$ If such a sequence were to exist, it would follow that $\int_I |y-x| \mathbf p_{\infty,\mathrm{dif}}(x,dy)=0$, so that $x$ would be an absorbing state for $\mathbf P_{\infty,\mathrm{dif}}$, contradicting its irreducibility and thus proving \eqref{delta_irr}.

Letting $2\delta$ be the infimum in \eqref{delta_irr}, we proceed to the proof of the main result. In the language of \cite[Definition 4.1]{DP25}, we see that for sufficiently large $N$, the Markov chain $\mathbf P_{N,h}$ is a $\delta$-repulsive SRI chain. The invariance principle of Item \eqref{Ainvarianceprinciple} follows immediately from \cite[Theorem 4.6]{DP25}.

The heat kernel upper upper bound of Item \eqref{Aheatkernel} follows from \cite[Theorem 4.7]{DP25}. That theorem yields bounds on balls of radius one, i.e., $\int_{I^k} \ind_{\{|\bfa |\le 1\}} \mathbf p_{N,k}(t,\x,d\bfa) \leq Ct^{-k} (1+|\bfa -\x - N^{-1} (d_N,\ldots,d_N)t|)^{-S},$ but we can use the property $\mathbf p_{N,k} (t,\x,\y) = \int_{I^k} \mathbf p_{N,k} (1,\bfa,\y) \mathbf p_{N,k}(t-1,\x,\bfa )d\bfa $ to easily go from a bound on balls to a pointwise bound on the transition density that is uniform over all $\x,\y$. 

Existence of a unique invariant measure in Item \eqref{Auniqueinvmeas} then follows from \cite[Theorem 4.10]{DP25}, with the scalar multiple chosen to yield the unit normalization of \cite[Definition 4.12]{DP25}. That theorem does not prove the existence of a continuous and bounded density, but this follows from using the invariance property to write $\nu_\infty(y) = \nu_\infty(g_y)$, where $g_y(x) = \mathbf p_{\infty,\mathrm{dif}}(x,y)$, which also shows that $\nu_\infty$ has a density with respect to Lebesgue measure on $I$.

To prove the diagonal limit of Item \eqref{Adiaglimits}, we may then apply \cite[Theorems 4.16--17]{DP25}. 
Letting $\nu_\infty$ denote two-dimensional Lebesgue measure yields the following weak version of Item \eqref{Birred}: for $x\in\mathbb R^2$ and $f_N\to f$ uniformly with $\sup_N| f_N(x)| \leq Ce^{-\alpha|x|}$, one has \begin{align*}
    \sum_{r=0}^N  \mathbf E^x_{N,\mathrm{dif}}[f_N(X_r)] &= \nu_\infty(f) \log(N) +o(\log N), 
\end{align*}
as $N\to\infty,$ where $\mathbf p_{N,\mathrm{dif}}$ denotes the transition density associated to $\mathbf P_{N,\mathrm{dif}}$ given by \eqref{aepsilon}, and the estimate is uniform over $x$ lying in any compact set. Furthermore, the cited result also shows that, if $\x_N$ is well-separated as in Item \eqref{Adiaglimits}, then for all $\phi \in C_c(\mathbb R^2)$ and $f_N\to f$ uniformly with $\sup_N| f_N(x)| \leq Ce^{-\alpha|x|}$ and $t>0$, we have 
\begin{align*}
    \lim_{N\to \infty} \sum_{r=0}^{Nt} & \mathbf E^{\x_N}_{N,2} [f_N(R^1_r-R^2_r) \phi(N^{-1/2} (R^1_r - d_Nr) ] \\&=4 \pi \cdot  \nu_\infty(f) \int_0^t \mathbf g(2s,x_1-x_2) \int_{\mathbb R^2} \mathbf g(2s,z) \phi(\tfrac12 (z+x_1+x_2)) \,dz \,ds.
\end{align*}
Now to complete the proof, we need to upgrade the above statement to the case where $f$ is a Dirac mass, for which we will use the strong Feller property. We know that $\nu_\infty$ is the invariant measure for $\mathbf P_{\infty,\mathrm{dif}}$, thus by the invariance property we can write $\nu_\infty(y) = \nu_\infty(g_y)$ where $g_y(x) = \mathbf p_{\infty,\mathrm{dif}}(x,y)$, which was also used to show that $\nu_\infty$ has a density with respect to Lebesgue measure on $I$. On the other hand, by the Markov property we can also write $\int_1^N \mathbf p_{N,\mathrm{dif}} (t;x,y)dt = \int_1^N \mathbf E^x_{N,\mathrm{dif}} [ g_y^N(X_{t-1})] dt,$ where $g_y^N(x) = \mathbf p_{N,\mathrm{dif}}(t=1,x,y).$ The strong Feller property and the heat kernel bound in Item \ref{Bstrongfeller} of Assumption \ref{B} then implies that if $y_n\to y\in \mathbb R^2,$ then $\sup_N g^N_{y_N}$ decays exponentially fast at infinity, and $g^N_{y_N}\to g_y$ uniformly. Thus Item \eqref{Adiaglimits} follows.
\end{proof}

\begin{thm}
    Under the assumptions of Theorems \ref{check} or \ref{check2}, the six assumptions of Theorem \ref{A} hold, and the statements of Theorems \ref{check} and \ref{check2} follow.
\end{thm}

\begin{proof}
    Items \eqref{Bstrongfeller} and \eqref{BconvTV} of Assumption \ref{B} are assumed outright in Theorems \ref{check} and \ref{check2}. Item \eqref{Btransinv} just follows from translation invariance of the flows $H^N$.

    To check Item \eqref{Bexpmoments}, we need to check that $\int_{I^h} e^{\epsilon|\y-\x|} \mathbb E [ \prod_{j=1}^h H_{0,1}^N(x_j,y_j)] d\y <\infty$ for some $\epsilon>0$. Let $I$ be either $\mathbb Z^2$ or $\mathbb R^2$ depending on the case of interest. By H\"{o}lder's inequality and the triangle inequality $|\y-\x| \leq \sum_1^h |y_j-x_j|$, we have 
    \begin{align*}
        \int_{I^h} e^{\epsilon|\y-\x|} \mathbb E [ \prod_{j=1}^h H_{0,1}^N(x_j,y_j)] d\y &\le  \mathbb E \bigg[ \prod_{j=1}^h \bigg(\int_{I} e^{\epsilon|y_j-x_j|} H_{0,1}^N(x_j,y_j) dy_j\bigg) \bigg] \\&\leq \prod_{j=1}^h \mathbb E \bigg[  \bigg(\int_{I} e^{\epsilon|y_j-x_j|} H_{0,1}^N(x_j,y_j) dy_j\bigg)^h \bigg]^{1/h}  \\&= \mathbb E \bigg[  \bigg(\int_{I} e^{\epsilon|y|} H_{0,1}^N(0,y) dy\bigg)^h \bigg] ,
    \end{align*}
    where in the last line we used the translation invariance assumption on $H^N$. By the assumptions of Theorem \ref{check}, the last term is a finite quantity independent of $\x=(x_1,\dots,x_h)$ for $h\le 5.$

    To check Item \eqref{Birred} of Assumption \ref{B}, note that topological irreducibility is automatically assumed in Theorems \ref{check} and \ref{check2}, and so it suffices to check exponential decay of the total variation distance. In the discrete-time setting of Theorem \ref{check}, this is trivial because we assumed the consistency of kernels in Item \eqref{Bconsistency} of Assumption \ref{a1}. However, in the continuous-time setting of Theorem \ref{check2}, this is not immediate from the consistency condition given in Assumption \ref{a3}. Nonetheless, the total variation bound can be proved by a simple and explicit coupling construction as follows. 
    Let $G:= \{ \x \in I^h : \min_{i<j} |x_i-x_j| \leq A\}$ where $A>0$ is fixed as in Item \eqref{a3markov} of Assumption \ref{a3}. By Item \eqref{check2exptail} of Theorem \ref{check2}, note that $\mathbf P_{N,1}^{\otimes h}$ is just a L\'{e}vy process with exponential moments, and so the running maximum of its modulus has exponential tail decay. A simple union bound using this exponential tail decay shows that, starting from $\x=(x_1,\dots,x_h)$, the probability that any $\PN{1}^{\otimes h}$-sample path $(\x+\sigma \vec W(t))_{t\in [0,1]}$ hits $G$ is bounded above by $C\sum_{1\le i<j\le k} e^{-\alpha |x_i-x_j|},$ where $C,\alpha>0$ are independent of $N$. Meanwhile, the  $\mathbf P_{N,h}^{\x}$-sample path $(\mathbf R(t))_{t\in [0,1]}$ agrees with $(\x+\sigma \vec W(t))_{t\in [0,1]}$ on the event that $(\x+\sigma \vec W(t))_{t\in [0,1]}$ does \textit{not} hit $G$, which provides the explicit high-probability coupling of the law of $\mathbf R(1)$ with the law of $\x+\sigma \vec W(1)$. This shows that the total variation distance in Item \eqref{Adiaglimits} of Theorem \ref{A} is bounded above by $\sum_{1\le i<j\le h} e^{-\alpha |x_i-x_j|}$, which in turn is upper bounded by $\frac12 h(h-1)e^{-\alpha \min_{1\le i<j\le h} |x_i-x_j|}$.

    Finally, we check the consistency condition in Item \eqref{Bconsistency}. This amounts to showing that for all $\x,\y\in I^h$ we have $\int_{I^m} \mathbf p_{\infty,h}(\x,\y)dy_1\cdots dy_m = \mathbf p_{\infty,h-m}( (x_{m+1},\dots,x_h) , (y_{m+1},\dots,y_h)). $ By the existence of continuous densities that is assumed, both sides of this equation are continuous in the $(y_{m+1},\dots,y_k)$ variable for fixed $\x$. Thus it suffices to prove the relation for almost every $(y_{m+1},\dots,y_h)$. This amounts to checking
    \begin{align*}
        \int_{I^{h-m}} \psi(&y_{m+1},\dots,y_h) \bigg[ \int_{I^m} \mathbf p_{\infty,h}(\x,\y)dy_1\cdots dy_m\bigg] dy_{m+1} \cdots dy_h\\& = \int_{I^{h-m}} \psi( y_{m+1},\dots,y_h)\mathbf p_{\infty,m-k}( (x_{m+1},\dots,x_h) , (y_{m+1},\dots,y_h))dy_{m+1}\cdots dy_h
    \end{align*}
    for all $\psi\in C_c^\infty (\mathbb R^{h-m})$,
    which reduces to
    \begin{align*}
        \lim_{N\to\infty} \int_{I^h} \psi(& y_{m+1},\dots,y_h)\mathbb E \Big[ \prod_{j=1}^m H^N_{0,1} (x_j,y_j)\Big] d\y \\
        &= \lim_{N\to\infty} \int_{I^{h-m}} \psi( y_{m+1},\dots,y_h)\mathbb E \Big[ \prod_{j>m} H^N_{0,1} (x_j,y_j) \Big]dy_{m+1}\cdots dy_h,
    \end{align*}
    or equivalently
    \begin{equation*}
        \mathbb E \bigg[ \bigg(  \int_{I^m} \prod_{j\le m} H^N_{0,1}(x_j,y_j) dy_1\cdots dy_m \;-\; 1\bigg) \cdot \int_{I^{h-m}} \psi( y_{m+1},\dots,y_h)\prod_{j>m} H^N_{0,1} (x_j,y_j)dy_{m+1}\cdots dy_h\bigg]\to 0.
    \end{equation*}
    By H\"{o}lder's inequality, we just need to show that $\lim_{N\to \infty} A_N B_N=0$ where 
    \begin{align*}
    A_N&= \mathbb E \bigg[ \bigg(  \int_{I^m} \prod_{j\le m} H^N_{0,1}(x_j,y_j) dy_1\cdots dy_m \;-\; 1\bigg)^{h/m}\bigg]^{m/h}, \\ B_N&=\|\psi\|_{L^\infty} \mathbb E \bigg[ \bigg( \int_{I^{h-m}} \prod_{j>m} H^N_{0,1} (x_j,y_j)dy_{m+1}\cdots dy_h\bigg)^{\frac{h}{h-m}} \bigg]^{1-\frac{m}{h}}.
    \end{align*}
    But $B_N$ is upper bounded by a constant independent of $N$ by Item \eqref{a1markov} of Assumption \ref{a1} (which is assumed in Theorem \ref{check}): indeed $\int_{I^{h-m}} \prod_{j>m} H^N_{0,1} (x_j,y_j)dy_{m+1}\cdots dy_h = \prod_{j>m} \int_I H^N_{0,1}(x_j,y) dy $, and from here another application of H\"{o}lder's inequality gives us the uniform upper bound $B_N \leq \prod_{j>m} \mathbb E \big[ \big(\int_I H^N_{0,1}(x_j,y) dy\big)^h \big]^{1/h} = \mathbb E [ \big(\int_I H^N_{0,1}(0,y) dy\big)^h\big]^{1-\frac{m}{h}},$ which is a constant.
    
    It remains only show that $A_N \to 0$. By the identity $a_1\cdots a_m -1 = \sum_{S\subset \{1,\dots,m\}: S\ne \emptyset } \prod_{j\in S} (a_j-1)$ and another application of H\"{o}lder's inequality, we have 
    \begin{align*}
        A_N = \bigg\| \prod_{j\le m} \Big( \int_I H_{0,1}^N(x_j,y) dy\Big) - 1 \bigg\|_{L^{h/m}} &\le \sum_{\emptyset \ne S\subset \{1,\dots,m\}} \bigg\| \prod_{j\in S} \Big( \int_I H_{0,1}^N(x_j,y)dy -1\Big)  \bigg\|_{L^{h/m}} \\&\leq \sum_{\emptyset \ne S\subset \{1,\dots,m\}}  \prod_{j\in S} \bigg\|  \int_I H_{0,1}^N(x_j,y)dy -1\bigg\|_{L^{h|S|/m}} \\&\le \sum_{\emptyset \ne S\subset \{1,\dots,m\}} \bigg\| \int_I H^N_{0,1} (0,y)dy -1\bigg\|_{L^h}^{|S|},
    \end{align*}
    where in the last line we are using that $|S| \le m$, so that by monotonicity of $L^p$ norms, the $L^{h|S|/m}$ norm is upper bounded by the $L^h$ norm. Now for $h\le 5$, the limit $\big\| \int_I H^N_{0,1} (0,y)dy -1\big\|_{L^h}^h \to 0$ is immediate from Item \eqref{a1markov} of Assumption \ref{a1} and a direct computation. In fact, it is upper bounded by $C/\log N$. 
\end{proof}

\section{Examples of models satisfying the hypotheses}\label{sec:examples}


In this section, we provide examples of models that satisfy the assumptions of the main results.

\subsection{Discrete directed polymer}\label{5.1}

\begin{ass}\label{4.1}We will consider a random environment $\omega=\{\omega_{t,x}\}_{(t,x) \in \Z \times \Z^2}$ defined on some complete probability space. These weights are independent in $t$, of finite-range dependence in $x$, and strictly stationary in $x$. For simplicity, we assume $\omega_{0,0}$ is nondeterministic but deterministically bounded i.e., $|\omega_{0,0}|\leq K$ for some nonrandom $K > 0.$ assume that the spatial covariance function is nonnegative, that is,
\begin{equation}\label{zeta_poly}
    \zeta_{1,1}(x):= \mathrm{Cov}(\omega_{0,0}, \omega_{0,x})  \ge 0 ,\qquad x\in \mathbb Z^2.
\end{equation} We also assume that $\zeta_{1,1}(x)$ is \textit{strictly} positive for all $x$ such that $\omega_{0,0}$ and $\omega_{0,x}$ are dependent. 
\end{ass}
This is true for example in the case of nearest-neighbor interaction, or of an environment given by a simple averaging of i.i.d.\ spatial neighbors. We also assume that the probability space admits a $\mathbb Z$-valued random walk $(R(t))_{t\in \mathbb Z_{\ge 0}}$ that is independent of the environment $\omega$. In other words, the probability measure splits as $\mathbb P = P^\omega\otimes \mathbf P_{\mathrm{RW}}$, where $P^\omega$ denotes the law of the environment and $\mathbf P_{\mathrm{RW}}$ denotes the law of the random walk. 

\begin{ass}\label{rwass}
    We assume that the $\mathbb Z^2$-valued random walk $(R(t))_{t\in \mathbb Z_{\ge 0}}$ starts at the origin, is aperiodic, has bounded increments of mean 0 and variance 1, and $\mathbb E[ e^{\epsilon|R(1)|}] <\infty $ for some $\epsilon > 0$.
\end{ass}

\begin{defn}
    Given a sequence $\beta_N>0$, define the \textit{directed polymer partition function} 
    \begin{equation}\label{zn}
        H^N_{s,t}(x,y):= e^{-(t-s)\lambda(\beta_N) } \mathbb E \big[e^{\beta_N  \sum_{u=s}^{t-1} \omega_{u,x+R(u-s)} }\ind_{\{R(t-s)=y-x\}}\big| \omega \big], 
    \end{equation}
    where $s,t,x,y\in \mathbb Z $ with $s\le t$ and $N\in \mathbb N$. Here $\lambda(\beta) = \log \mathbb E[e^{\beta \omega_{0,0}}].$
\end{defn}
The expectation is over the random walk path $(R(t))_{t\in \mathbb Z_{\ge 0}}$, conditional on the environment $\omega$ which we again emphasize is always assumed to be \textit{independent} of the random walk $R$. 
Thus \eqref{zn} can be viewed as a weighted sum over paths from $(s,x)$ to $(t,y)$, the weight being determined by the environment $\omega$.

Thus for $s,t \in N^{-1} \mathbb Z$ with $s<t$, the macroscopic field $\mathfrak H^N_{s,t}$ from Definition \ref{a2} can be written  $$\mathfrak H^N_{s,t}(\phi,\psi) = N^{-1}\sum_{x\in \mathbb Z^2} \phi(N^{-1/2}x) 
e^{-N(t-s)\lambda(\beta_N) }
\mathbb E[e^
{\beta_N  \sum_{u=Ns}
^{Nt-1}\omega_{u,x+R(u-Ns)}
}
\psi(N^{-1/2}(R(N(t-s))+x))\big| \omega \big]. $$

\begin{thm}[Verification of hypotheses for the directed polymer]\label{poly_verify} Suppose Assumptions \ref{4.1} and \ref{rwass} hold. Then there exists a sequence $\beta_N\asymp (\log N)^{-1/2}$ such that the family $H^N$ in \eqref{zn} satisfies Assumption \ref{a1}. The measure $\nu_\infty$ is $\frac1{4\pi}$ times counting measure on $\mathbb Z^2$ for this model.
\end{thm}

\begin{proof}
\textbf{Verification of Items \eqref{a1flow}, \eqref{a1markov}, and \eqref{a1consistency}.} For now let $\beta_N$ be any sequence of positive values converging to 0. For $y\in \mathbb Z^2$ let $p(y) = \mathbf P_{\mathrm{RW}} (R(1) -R(0) =y)$, so that $p$ is the increment law of the random walk. Then we see that this model satisfies the stochastic flow property, with $H_{0,1}^N(x,y) = e^{\beta_N \omega_{t,x} - \log \mathbb E[e^{\beta_N \omega_{0,0}}]} p(y-x)$. Taking products and then expectations, the consistency of kernels in Item \eqref{a1consistency} follows trivially from the finite-range assumption on $\{\omega_{0,x}\}_{x\in \mathbb Z^2}$.

In this case, the Markov chain $\PN{h}$ is just given by the law of $h$ independent random walks of increment law $p(\bullet)$ as above, and in particular does not actually depend on $N$. Indeed, using \eqref{zn} and putting $G_N^{(h)}(x_1,\dots,x_h):=\mathbb E [ e^{\beta_N ( \omega_{t,x_1} + \dots+\omega_{t,x_h})} ]/\mathbb E[e^{\beta_N\omega_{0,0}}]^h$,  we obtain
\begin{align*}
    \mathbb E\bigg[ \prod_{j=1}^h H^N_{0,t} (x_j,y_j)\bigg] = (\mathbf E_{\mathrm{RW}}^{\otimes h})^{\x} \bigg[ \ind_{\{ \mathbf R_{t} = \y\}} \prod_{u=0}^{t-1} G_N^{(h)}(\mathbf R_u) \bigg],
\end{align*}
where the expectation is taken with respect to $h$ independent random walks $\mathbf R = (R^1,\dots,R^h)$ started from $\x=(x_1,\dots,x_h)$. Notice that $\log G^{(2)}(x_i,x_j) = \z_N(x_i-x_j) $  for some function $\z_N$, by translation invariance of the field $\{\omega_{t,x}\}_{t,x}$.
Taylor expanding the cumulant generating function $\log G_N$ gives
\begin{align*}
    \log G_N^{(h)}(x_1,\dots,x_h) &= \sum_{p,q=1}^\infty \frac{\beta_N^{p+q}}{p!q!}  \sum_{1\le i<j\le h} \zeta_{p,q} (x_i-x_j)\;+\; \beta_N^3 \boldsymbol\varepsilon_N(\x) \\& = \sum_{1\le i<j\le h} \log G^{(2)}_N(x_i-x_j) \;+\; \beta_N^3 \boldsymbol\varepsilon_N(\x),
\end{align*}where $\zeta_{p,q}:\mathbb Z^2\to\mathbb R$ are given by 
\begin{equation*}
    \zeta_{p,q}(x_i-x_j) = \kappa_{(p+q)} (\underbrace{\omega_{t,x_i},\dots,\omega_{t,x_i}}_{p\;\text{terms}}, \underbrace{\omega_{t,x_j},\dots,\omega_{t,x_j}}_{q\;\text{terms}})
\end{equation*}
and $\kappa_{(k)}$ denotes joint cumulant of order $k$. Meanwhile, the remainder $\boldsymbol\varepsilon_N(\x)$ corresponds to those terms in the cumulant expansion consisting of joint cumulants that contain at least three separate indices $i<j<\ell$, and is thus supported on triple intersections in the sense that there exists $A>0$ so that
\begin{equation*}
    | \boldsymbol\varepsilon_N(\x)| \leq C \sum_{ 1\le i<j<\ell \le k} \ind_{\{ |x_i-x_j| \le A , |x_j-x_\ell|\le A\}}.
\end{equation*}
The fact that the cumulant generating function above converges absolutely is nontrivial, and follows from \cite[Theorem 9.11]{KLM}.
    
To check Item \eqref{a1markov}, we define\footnote{
    It is tempting to believe that $\z(\beta,x)$ could be replaced by its leading term $\beta^2 \zeta_{1,1}(x),$ but this will lead to the wrong answer, which was already noticed in early works such as \cite{MR4017119}.
}
$\z(\beta,x):= \sum_{p,q=1}^\infty \frac{\beta^{p+q}}{p!q!} \zeta_{p,q} (x) = \log \mathbb E [e^{\beta(\omega_{0,0} + \omega_{0,x})}] - 2\log \mathbb E [e^{\beta \omega_{0,0} }],$
so that $\z_N:=\z(\beta_N,\bullet)$.\footnote{
    Note that in the nearest-neighbor case $\z(\beta,x)$ is nonzero only when $x=0$, and $\z(\beta,0) = \lambda(2\beta)-2\lambda(\beta).$ This is consistent with the scaling of $\beta$ in \cite{CSZ_2d}.
} 

Thus we see that $\h_{N,h} (\x):= \log G_N^{(h)} (\x)$ will satisfy Assumption \ref{a1} \eqref{a1markov}. It is not trivial that $\z_N = \z(\beta_N,\,\bullet\,)$
is nonnegative: however, this follows for sufficiently large values of $N$ from our assumption after \eqref{zeta_poly}, which concerns the leading-order term in the expression for $\z(\beta,x)$ (note that all $\zeta_{p,q}$ are supported on the same finite set). Then we obtain exactly the representation needed for Item \eqref{a1markov} of Assumption 1.1.

\medskip
\noindent\textbf{Verification of Items \eqref{a1logasymptotic}--\eqref{a1heatkernel}.} In this case $\nu_\infty$ is just counting $\frac1{4\pi}$ times the counting measure on $\mathbb Z^2,$ which follows from the local central limit theorem. Proving the logarithmic asymptotics and heat kernel bounds also follows from the local central limit theorem. Alternatively, an easy application of Theorem \ref{check} will quickly verify Items \eqref{a1logasymptotic}--\eqref{a1heatkernel}. 

\medskip
\noindent\textbf{Verification of the critical tuning.} Thus far we have considered arbitrary $\beta_N\to 0$. We now justify the existence\footnote{In this example and later ones, we do not prove uniqueness of the solution $\beta_N$, though we believe it to be true. With more work, we could actually go further and prove a full Taylor expansion of the top eigenvalue 
using explicit formulas for the Fr\'echet derivative of the top eigenvalue. 
Such local expansion formulas for the top eigenvalue can be found in \cite[Theorem 2.6]{Kloeckner2019Effective} or \cite[Chapter 3]{Kato}. For brevity, we do not calculate these expansions. } of a particular $\beta_N \downarrow 0$ satisfying the critical tuning \eqref{crit_tuning}.

To be more precise, let $J$ denote the support of $\zeta_{1,1}$, which is a finite set since we are in the discrete setting. We write $T_{N,\beta}$ for the linear map on $\R^J$ from Definition \ref{def:lambda}, which is more explicitly given by
\begin{align*}
    T_{N,\beta} h(x) &= \frac1{\log N} \sum_{z\in J} \bigg( \sum_{r=0}^N \mathbf p_{\mathrm{dif}} (r,x,y)\bigg) (e^{\z(\beta, y)}-1) h(y),
\end{align*}
where $\mathbf p_{\mathrm{dif}} (1,x,y) :=  \sum_{a\in \mathbb Z^2} p^{\otimes 2} (a, a+y-x)$.

Let $\boldsymbol{\lambda}_N(\beta)$ denote the top eigenvalue of $T_{N,\beta}: \mathbb R^J\to \mathbb R^J$. We need to show that there exists a solution to the equation $ \boldsymbol{\lambda}_N(\beta) \cdot \log(N) = 1+\frac{\vartheta+o(1)}{\log N}$, which then defines $\beta_N$. To prove this, it suffices to show a local expansion 
\begin{equation}\label{eigen_exp0}
    \boldsymbol{\lambda}_N(\beta) = \beta^{2} f_N(\beta) , \qquad \limsup_{N\to\infty} \sup_{0\le \beta<\epsilon}| f_N(\beta) - c_{\mathrm{mean}}| < \frac{c_{\mathrm{mean}}}{2},
\end{equation}
where $c_{\mathrm{mean}} := \sum_{a\in \mathbb Z^2} \zeta_{1,1}(a) >0$ and moreover each $f_N$ is a continuous function on $[0,\epsilon]$. Here $\epsilon>0$ is some positive constant not depending on $N$. 
This would imply that $\beta \mapsto \boldsymbol{\lambda}_N(\beta)$ is continuous on $[0,\epsilon]$ and $\frac{c_{\mathrm{mean}}}{2}\beta^2 \leq \boldsymbol{\lambda}_N(\beta) \leq \frac{3c_{\mathrm{mean}}}{2}\beta^2,$ so that the intermediate value theorem would immediately justify the existence of $\beta_N \asymp (\log N)^{-1/(2p)}$ such that $ \boldsymbol{\lambda}_N(\beta) \cdot \log(N) = 1+\frac{\vartheta+o(1)}{\log N}$.
        
First let us recall from \eqref{phoo} that $e^{\z(\beta,y)}-1$ Taylor expands as $\beta^{2p} \z_{1,1}(y) +\beta^{2p+1} \ell(\beta,y)$ where $\sup_{0<\beta<1} \|\ell(\beta,\bullet)\|_{L^\infty}<\infty$. Thus $f_N(\beta)$ is the top eigenvalue of the linear map
\begin{equation*}
    S_{\beta,N} h(x) = \frac{1}{\log N} \sum_{r=0}^N \sum_{y\in J}\mathbf p_{\mathrm{dif}} (r,x,y) \big(\zeta_{1,1}(y) +\beta\ell(\beta,y)\big) h(y)
\end{equation*}
on $\mathbb R^J$. We conclude the proof of Theorem \ref{poly_verify} by checking \eqref{eigen_exp0} in steps.

\smallskip
\noindent\textbf{Step 1.} Let $S_{0,\infty}:= \lim_{N\to\infty} S_{0,N}$, so that $S_{0,\infty} = \langle \bullet , \zeta_{1,1} \rangle_{L^2(J)} \ind_J$ is manifestly a rank one operator whose unique non-zero eigenvalue is $\brac{\zeta_{1,1},\ind_J}_{L^2(J)} =c_{\z}$. For any $\delta>0$, we claim there exists $\epsilon>0$ and $N_0 \in \mathbb N$ such that $\|S_{N,\beta} - S_{0,\infty}\| < \delta$ whenever $\beta<\epsilon$ and $N\ge N_0$, where the operator norm is always understood to be on $L^2(J)$. Indeed, for any sequence $\beta_N\downarrow 0$ and $x,y\in J$, the local CLT (or Theorem \ref{check}) guarantees that $\frac1{\log N} \sum_{r=0}^N \mathbf p_{\mathrm{dif}} (r,x,y)dt \to 1$. This means that, along any sequence $\beta_N\downarrow 0$, the integral kernel of $S_{\beta_N,N}$ converges uniformly to that of $S_{0,\infty}$, where for $x,y\in J$ the kernels are respectively given by $K_{\beta_N,N}(x,y) = \frac1{\log N} \sum_{r=0}^N \mathbf p_{\mathrm{dif}}(r,x,y)\cdot (\zeta_{1,1}(y)+\beta_N\ell(\beta_N,y))$ and $K(x,y)=\zeta_{1,1}(y)$. The claim follows.

\smallskip
\noindent\textbf{Step 2.} Next we prove continuity of $f_N$ on $[0,\epsilon]$. Define $\tau_0:= \frac{\|\zeta_{1,1}\|_{L^2(J)}}{\langle \ind_J ,\zeta_{1,1}\rangle_{L^2(J)}}$ and $\gamma_0:=\langle \zeta_{1,1},\ind_J\rangle_{L^2(J)}$. Both quantities are positive since $\zeta_{1,1}>0$ on $J$. Using \cite[Theorems 2.3 and 2.6]{Kloeckner2019Effective}, there exists a map $A\mapsto \lambda_{\mathrm{top}}(A)$ which sends a compact operator $A$ on $L^2(J)$ to its eigenvalue of maximal modulus, and the function $\lambda_{\mathrm{top}}$ is well-defined and uniformly continuous (in the operator norm topology) on the domain $\{A: \|A-S_{0,\infty}\| <1/(12\tau_0 \gamma_0) \}.$ Notice that the function $\beta\mapsto f_N(\beta)$ is the composition of the two maps $$\beta \mapsto S_{N,\beta} \mapsto \lambda_{\mathrm{top}}(S_{N,\beta}).$$
The first map is clearly continuous on $[0,\epsilon]$, and it lands in the domain of the second map whenever $\beta<\epsilon$ and $N\ge N_0$ by Step 1. The second map is continuous as just discussed, and thus the composed map is continuous.

\smallskip
\noindent\textbf{Step 3.} It remains to prove the estimate on the $\limsup$ in \eqref{eigen_exp0}. In fact, we prove the stronger statement that $f_N(\beta_N) \to c_{\mathrm{mean}}$ for any sequence $\beta_N\downarrow 0$,
which clearly implies the claim. But this is immediate from the result of Step 1 and the continuity of the map $\lambda_{\mathrm{top}}$ in Step 2.
\end{proof}

\begin{remark}
    We comment here on how to match our notation with that of \cite{CSZ_2d} in the case of the directed polymer with i.i.d.\ weights. To avoid periodicity issues, we used an aperiodic version of the model, from \cite{SHFnotes}. 

\begin{center}
    \begin{tabular}{|c|c|c|}
        \hline
        \rule{0pt}{2.5ex}
        \textbf{Quantity} & \textbf{Notation in \cite{CSZ_2d}} & \textbf{Our notation} \\[2.5pt]
        \hline
        
        \rule{0pt}{2.5ex}
        Noise strength
        &
        $\sigma_N^2 := e^{\lambda(2\beta_N)-2\lambda(\beta_N)}-1 $
        &
        $ e^{\z(\beta_N,0)}-1 $
        \\[6pt]
        
        Replica overlap
        &
        $R_N$
        &
        $ \mathscr S_N(0,0)
        $
        \\[8pt]
        
        Critical tuning
        &
        $ \frac{1}{\log N}R_N\sigma_N^2 $
        &
        $ \Lambda_N\big(e^{\z(\beta_N, \bullet)}-1\big) $
        \\
        \hline
    \end{tabular}
\end{center}
Putting this together, in \cite{SHFnotes}, the critical tuning is $\sigma_N^2 = \frac{1}{R_N}\big(1 + \frac{\vartheta + o(1)}{\log N}\big).$ This is equivalent to the eigenvalue tuning $\Lambda_N\big(e^{\z(\beta_N, \bullet)}-1\big) \log N = 1+\frac{\vartheta+o(1)}{\log N}$ from \eqref{crit_tuning}. 
\end{remark}

\subsection{Mollified (2+1)-dimensional multiplicative SHE}\label{sss:molSHE}

We set $\epsilon=\epsilon_N=N^{-1/2}$ and consider the It\^o stochastic PDE 
\begin{equation*}\label{mshe}\tag{mSHE}
    \partial_t Z^\epsilon(t,y) = \tfrac12 \Delta_y Z^\epsilon(t,y) + \beta_\epsilon \cdot Z^\epsilon(t,y) \cdot \eta^\epsilon(t,y),\qquad t\ge 0,\; y\in \mathbb R^2,
\end{equation*}
where $\eta^\epsilon:= \varphi_\epsilon *\eta$ for a standard Gaussian space-time white noise $\eta$ on $\mathbb R\times\mathbb R^2$ and the convolution is in space only, and where $  \varphi_\epsilon (y):= \epsilon^{-2} \varphi(\epsilon^{-1}y)$ for some fixed nonnegative continuous even function $\varphi \in C_c^\infty(\mathbb R^2)$. 
    
\begin{defn}
    Set $\epsilon = \epsilon_N := N^{-1/2}$ and then define $H^N_{s,t}(x,y) := \epsilon^2 \cdot z_{\epsilon^2s,\epsilon^2 t}(\epsilon x,\epsilon y)$,
    where $z_{s,t}$ are the propagators for \eqref{mshe}.
\end{defn}

\begin{thm}[Verification of the hypotheses for the mollified SHE]\label{mshe_verify}
    There exists a sequence $\beta_N := \beta_{\epsilon_N}\asymp (\log N)^{-1/2}$ such that the family $H^N$ defined above satisfies Assumption \ref{a1}.
    In this case ${\boldsymbol{\zeta}}_N(x)=\beta_N^2\cdot  \varphi * \varphi(x)$, and $\nu_\infty$ is $\frac1{4\pi}$ times Lebesgue measure on $\R^2.$
\end{thm}

This result is not new, and follows from a culmination of the works of \cite{MR1629198, GQT, Tsa24}, but we can recover it in the present framework.

\begin{proof}
    By the Feynman-Kac formula \cite{BC95}, we can represent \begin{equation*}
        \int_{\mathbb R^2} H^N_{0,t} (x,y) \psi(y) dy = \mathbf E_{\mathrm{BM}}^x \bigg[ e^{\beta_\epsilon \int_0^t \eta^\epsilon(s,B_s)ds  -\frac{\beta_\epsilon^2}2 (\varphi_\epsilon *\varphi_\epsilon)(0)t} \psi(B_t) \bigg| \eta^\epsilon \bigg]
    \end{equation*}
    for $\psi\in C_c^\infty(\mathbb R^2)$, from which one easily sees that
    \begin{equation*}
        \int_{(\mathbb R^2)^h} \mathbb E \bigg[\prod_{j=1}^h H_{0,t}(x_j,y_j) \bigg]\Psi(\y) d\y = (\mathbf E_{\mathrm{BM}}^{\otimes h})^{\x} \bigg[ e^{\int_0^t \z_N(B_s^i-B_s^j) ds} \Psi(\mathbf B_t) \bigg],
    \end{equation*}
    for all $\Psi \in C_c^\infty ((\mathbb R^2)^h),$ with $\z_N$ as in the theorem statement. 
    Since the noise is Gaussian, all cumulants of order three and up vanish identically, thus the process $\boldsymbol{\eta}_{N,h}$ is identically equal to $\sum_{i<j} \z_N(x_i-x_j)$. In this case $\nu_\infty$ is just Lebesgue measure on $\R^2.$ Checking the logarithmic asymptotics and heat kernel bounds of Assumption \ref{a3} amounts to scale invariance and asymptotics for the standard heat kernel on $\mathbb R^2$. 

    To ensure that the critical tuning is satisfied in this case, we simply set
    \begin{equation*}
        \beta_N^2 := \frac1{\Lambda_N(\varphi * \varphi) \log N}\bigg(1+\frac{\vartheta+o(1)}{\log N}\bigg).
    \end{equation*}
    We note here that we can unambiguously let $\Lambda_N(f)$ denote the top eigenvalue as in Definition \ref{def:lambdacont}.
\end{proof}

\subsection{SHE with derivative-type noise}\label{5.3} 

The following model is motivated by work of \cite{Hai24} on renormalization in the presence of variance blowup, which was also studied in \cite{Par-} using a different moment-based approach. Those results were in $d=1$: we now focus on $d=2$, which is more subtle.

Fix some $p\in\mathbb N$. Consider the It\^o stochastic PDE 
\begin{equation*}\label{dshe}\tag{dSHE}
    \partial_t Z^\epsilon(t,y) = \tfrac12 \Delta_y Z^\epsilon(t,y) + \beta_\epsilon^{1/2}  \epsilon^{p} \cdot Z^\epsilon(t,y) \cdot (-\Delta_y)^{p/2}\eta^\epsilon(t,y),\qquad t\ge 0,\;y\in \mathbb R^2,
\end{equation*}
where $\eta^\epsilon := \varphi_\epsilon *\eta$ for a standard Gaussian space-time white noise $\eta$ on $\mathbb R\times\mathbb R^2$ and the convolution is in space only,
and where $\varphi_\epsilon (y):= \epsilon^{-2} \varphi(\epsilon^{-1}y)$ for some fixed even function $\varphi \in C_c^\infty(\mathbb R^2)$ with $\int_{\mathbb R^2}\varphi=1$. 

Note that unlike \eqref{mshe}, there is a fractional Laplacian around the noise term, which creates significant additional complications. The papers \cite{Hai24, Par-} consider the $d=1$ version of this equation and prove convergence to the KPZ equation with $\beta_\epsilon=\epsilon^{-1/2}$. Here we will prove the $d=2$ version of the result, with convergence to the SHF when $\beta_\epsilon= O((\log(1/\epsilon)^{-1/2})$. We remark that there is an extra square root on $\beta_\epsilon$ appearing in \eqref{dshe} as compared to \eqref{mshe}, which lacks the derivative $(-\Delta)^{p/2}$ acting on the noise. The reason for this will become clear from the proof. 

\begin{defn}
    Here, we set $\epsilon = \epsilon_N := N^{-1/2}$. Let $H^N_{s,t}(x,y):=\epsilon^2 \cdot Z_{\epsilon^2 s,\epsilon^2 t}^\epsilon(\epsilon x,\epsilon  y)$,
    where $Z^\epsilon_{s,t}$ are the propagators for \eqref{dshe}.
\end{defn}



\begin{thm}[Verification of hypotheses for the SHE with derivative-type noise]\label{dshe_verify}
    There exists a sequence $\beta_N:= \beta_{\epsilon_N} \asymp (\log N)^{-1/2}$ for which
    $H^N$ as defined above satisfy Assumption \ref{a3}. In this case ${\boldsymbol{\zeta}}_N(x): = \beta_N^2 |\nabla \varphi * \Delta^{p-1}\varphi(x)|^2$, and $\nu_\infty$ is $\frac1{4\pi}$ times Lebesgue measure on $\R^2.$
\end{thm}

\begin{proof}
    \textbf{Verification of Items \eqref{a3flow}, \eqref{a3markov}, and \eqref{a3consistency}.} First let us verify Item \eqref{a3markov} of Assumption \ref{a3}.
    Let $\Phi_\epsilon := \varphi_\epsilon * \varphi_\epsilon$, with $\varphi_\epsilon$ as above, and let $\Phi:=\Phi_1$. For now, consider any sequence $\beta_\epsilon\downarrow 0$; we will specify one later. If $\phi: \mathbb R^2\to\mathbb R$ is bounded and continuous, and $\x = (x_1,\dots,x_h) \in \mathbb R$, we have 
    \begin{equation*}
        \int_{(\mathbb R^2)^h} \mathbb E \bigg[ \prod_{j=1}^h Z_{0,t}^\epsilon (x_j, y_j) \bigg]\prod_{j=1}^h \phi(y_j) dy\cdots dy_n = (\mathbf E_{\mathrm{BM}}^{\otimes h})^{\x} \bigg[  e^{\beta_\epsilon (-\epsilon^2)^p \sum_{1\le i<j\le h}\int_0^t \Phi_\epsilon^{(2p)} (W^i_s-W^j_s)ds} \prod_{j=1}^h \phi(W^j_t)\bigg].
    \end{equation*}
    Here the expectation on the right side is with respect to an $h$-dimensional Brownian motion $\mathbf W$ started from $\x$, and $\Phi^{(2p)}:= \Delta^p \Phi$ denotes the $p$\textsuperscript{th} power of the standard Laplacian on $\mathbb R^2$ applied to $\Phi$. The formula above is easily checked, since $(-1)^p\Phi_\epsilon^{(2p)}$ is the covariance kernel of the noise $(-\Delta_y)^{p/2} \eta^\epsilon$. 
    
    Thus in checking Item \eqref{a3markov} it is tempting to say $\mathbf P_{N,h} = \mathbf P_{\mathrm{BM}}^{\otimes h}$ and $\z_N = \beta_\epsilon (-\epsilon^2)^p \Phi_\epsilon^{(2p)}$, but this is wrong! Indeed, $\Phi_\epsilon^{(2p)}$ fails to stay nonnegative (it integrates to 0), as required in Assumption \ref{a3}. Instead, we will further manipulate the above expectation. By It\^o's formula, we can rewrite the quantity in the exponent above in terms of a stochastic integral, so that
    \begin{align}
        \notag &\int_{(\mathbb R^2)^h } \mathbb E  \bigg[ \prod_{j=1}^h Z_{0,t}^\epsilon (x_j, y_j) \bigg] \prod_{j=1}^h \phi(y_j) dy\cdots dy_h \\
        &= \notag (\mathbf E_{\mathrm{BM}}^{\otimes h})^{\x} \bigg[ e^{\beta_\epsilon(-\epsilon^{2} )^p\sum_{1\le i<j\le h} \big(\Phi_\epsilon^{(2p-2)}(W^i_t-W^j_t)-\Phi_\epsilon^{(2p-2)}(x_i-x_j) -\int_0^t \nabla \Phi_\epsilon^{(2p-2)}(W^i_s-W^j_s) d(W^i-W^j)_s\big) } \prod_{j=1}^h \phi( W^j_t )\bigg] \\
        &= \mathbf E^{\x}_{\mathrm{Diff}(\epsilon)} \bigg[e^{ \mathbf T^\epsilon_t(X)+\sum_{1\le i<j\le h} \big[\beta_\epsilon(-\epsilon^{2})^p\big(\Phi_\epsilon^{(2p-2)}(X^i_t-X^j_t)-\Phi_\epsilon^{(2p-2)}(x_i-x_j)\big) +\beta_\epsilon^2\epsilon^{4p} \int_0^t 
        |\nabla \Phi_\epsilon^{(2p-2)}(X^i_s-X^j_s)|^2 ds\big] } \prod_{j=1}^h \phi( X^j_t) \bigg] \label{apx}
    \end{align}
    where by Girsanov's theorem the last expectation is no longer with respect to Brownian motion, but rather the diffusion process on $(\mathbb R^2)^h$, started from initial condition $\x$, with dynamics given by 
    \begin{equation}\label{diffu}
        dX^i = \beta_\epsilon(-\epsilon^2)^p \sum_{j\ne i} \mathrm{sign}(i-j) \nabla \Phi_\epsilon^{(2p-2)}(X^i-X^j)ds + dW^i, \qquad 1\le i \le h,
    \end{equation}
    and where the process of triple intersections is given by 
    \begin{align*}
        \mathbf T_t^\epsilon(X)&:= \frac{\beta_\epsilon^2}2 \epsilon^{4p}\sum_{i=1}^h \sum_{\substack{j_1,j_2\neq i\\ j_1\ne j_2}} \int_0^t \mathrm{sign}(i-j_1)\mathrm{sign}(i-j_2) \sum_{a_1,a_2=1}^h \prod_{k\in \{1,2\}} \partial_{a_k} \Phi^{(2p-2)}_\epsilon(X^i_s -X^{j_k}_s) ds.
    \end{align*}
    Translating back to microscopic coordinates, we therefore see that, in Assumption \ref{a3}, the relevant Markov process $\mathbf P_{N,h}$ will be given by the law of $\mathbf A_t:= \epsilon^{-1} \mathbf X_{\epsilon^2 t}$, with $\mathbf X=(X^1,\ldots,X^h)$ solving \eqref{diffu}. Thus we see that $\mathbf P_{N,h}$ solves the microscopically rescaled SDE 
    \begin{equation}\label{aepsilon}
        dA^i = \beta_\epsilon \sum_{j\ne i} \mathrm{sign} (i-j) \nabla \Phi^{(2p-2)}(A^i-A^j) ds +  dW^i,\qquad 1\le i \le h.
    \end{equation}
    We thus set $\z_N(\x):= \beta_N^2{\boldsymbol{\zeta}}(x)$ with $\z(x) := |\nabla \varphi * \varphi^{(2p-2)}(x)|^2$, and in turn \begin{equation*}
        \boldsymbol{\eta}_{N,h} (\mathbf x) :=  \sum_{1\le i<j \le h} \z_N(x_i-x_j) +\frac{\beta_\epsilon^2}{2} \sum_{i=1}^h \sum_{\substack{j_1,j_2\neq i\\ j_1\ne j_2}} \mathrm{sign}(i-j_1)\mathrm{sign}(i-j_2) \sum_{a_1,a_2=1}^h \prod_{k\in \{1,2\}} \partial_{a_k} \Phi^{(2p-2)}(x_i-x_{j_k}),
    \end{equation*}
    and we see from above that Item \eqref{a3markov} of Assumption \ref{a3} is satisfied. Since the drift term in \eqref{aepsilon} is smooth and compactly supported, the existence of a smooth bounded transition density with Gaussian tails at $t=1$ for the diffusion \eqref{diffu} is classical from \cite[Theorem A]{Zhang97}, and the consistency condition in Item \eqref{a3consistency} is also clear from the exact form of \eqref{diffu}.
    It remains to check that this family of processes satisfies the claimed heat kernel bounds and logarithmic asymptotics.
    
    There is still an extra term $\mathscr E^{N,h}_t:=(-1)^p\beta_\epsilon\sum_{1\le i<j\le h}  \big(\Phi^{(2p-2)}(X^i_t-X^j_t) - \Phi_\epsilon^{(2p-2)}(X^i_0-X^j_0)\big)$ appearing in \eqref{apx} that has not yet been explained. This term is an error that will appear in the moment expansion, and it is not part of the additive functional $\boldsymbol{\eta}_{N,h}$. Since it has vanishing $L^\infty$ norm, it will be irrelevant in the limit, but is allowable as shown in Assumption \ref{a3} \eqref{a3markov}.
    
    We still have not shown that $\|\boldsymbol{\eta}_{N,h}\|_{L^\infty} \leq C/\log N, $ i.e., $\beta_N^2\le C/\log N.$ This will be done when we study the exact asymptotics of $\beta$ below. However, we still have $\|\boldsymbol{\eta}_{N,h}\|_{L^\infty}\to 0$ for arbitrary sequences $\beta_N\to 0$: by the final statement in Theorem \ref{check2}, this will be enough to check Items \eqref{a3log}--\eqref{a3heatkernel}. 
    
    \smallskip
    \noindent\textbf{Verification of Items \eqref{a3logasymptotic}--\eqref{a3heatkernel}.} When $h=2$ in \eqref{aepsilon} above, we indeed see that the difference $X^1-X^2$ is Markov since $d(A^1-A^2) = \beta_\epsilon \nabla \Phi^{(2p-2)} (A^1-A^2) dt + \sqrt 2 dW.$ We need to prove that this diffusion satisfies all of the relevant limiting formulas and estimates.
    In this SDE, the drift term $\nabla \Phi^{(2p-2)}$ is a fixed function not depending on $\epsilon,$ and $\beta_\epsilon \to 0$. Thus as $\epsilon\to 0$ it is clear that $\beta_\epsilon \nabla \Phi^{(2p-2)}\to 0$ uniformly. Consequently, this family of Markov chains $\mathbf P_{N,h}$ converges on compact time intervals to a standard Brownian motion in $(\mathbb R^2)^h$ as $N\to \infty$ (equivalently, as $\epsilon\to 0$). Immediately we see that Items \eqref{check2hptlim}--\eqref{check2cov} of Theorem \ref{check2} are satisfied. To satisfy Item \eqref{check2exptail}, we need to make sure that the time $t=1$ density of $\mathbf A$ started from $\mathbf A_0=\mathbf 0$ has exponentially decaying tails. This follows from the main result of \cite{Zhang97} on Gaussian heat kernel bounds for diffusions with smooth and bounded coefficients. This verifies the conditions of Theorem \ref{check2}, and thus items \eqref{a3logasymptotic}--\eqref{a3heatkernel} of Assumption \ref{a3}.

    \smallskip
    \noindent\textbf{Verification of the critical tuning.} Thus far we have verified Assumption \ref{a2}, and we see that the proof is valid for any sequence $\beta_\epsilon\to 0$. However, we still need to find the precise asymptotics of $\beta_\epsilon$ by verifying \eqref{crit_tuning2}, which is the hardest assumption to check for this particular model. 
    
    Take $h=2$ in \eqref{aepsilon}, and note that $\nabla \Phi$ is an odd function. Then we see that the difference process $A:=A^1-A^2$ is just $\sqrt{2}$ times a Brownian motion, regardless of the value of $\beta>0$. Consequently, the Markov chain $\mathbf p_{N,\mathrm{dif}}$ will be independent of $N$, and we can unambiguously let $\Lambda_N(f)$ denote the top eigenvalue as in Definition \ref{def:lambdacont}. Thus to verify the critical tuning, we can just take 
    $$\beta_N^2:= \frac1{\Lambda_N(\z) \log N} \bigg( 1+\frac{\vartheta+o(1)}{\log N}\bigg),$$
    where as always $\z(x) := |\nabla \varphi * \varphi^{(2p-2)}(x)|^2$.
\end{proof}

\subsection{Random walk in a random environment}\label{5.4}

\begin{ass}\label{4.6} Consider a family of random probability measures $\{K_t(x,\bullet)\}_{(t,x) \in \mathbb Z \times \mathbb Z^2}$ on $\mathbb Z^2$. We impose that $K_t(x,\bullet)$ are independent for distinct $t\in \mathbb Z$, of finite-range dependence in $x\in\mathbb Z^2$, and strictly stationary in $x$ (in the sense that $K_0(x,\bullet+x)$ is a stationary process in $x$). 

Fix some $p\in \mathbb N$, as well as some unit vector $v_{\mathrm{unit}}\in \mathbb R^2$. Suppose the first $p-1$ moments of $K_0(0,\bullet)$ are deterministic, whereas the $p$\textsuperscript{th} one is genuinely random (in the sense that the quantity $\sum_{(a_1,a_2)\in \mathbb Z^2} a_1^{k_1} a_2^{k_2} K_{0}(0 , (a_1,a_2))$ is deterministic for all positive integers $k_1,k_2$ such that $k_1+k_2< p$, but has positive variance for some $k_1,k_2$ such that $k_1+k_2=p$). 

We assume that $\mu(x):= \mathbb E[K_0(0,x)]$ has covariance given by the $2 \times 2$ identity matrix and has finite support. 
We also assume that the $p$\textsuperscript{th} moments have nonnegative spatial covariance function
\begin{equation}\label{zeta_rwre}
    \zeta_{p,p}(x):= (p!)^{-2} \mathrm{Cov}\bigg(\sum_{z\in \mathbb Z^2} (v_{\mathrm{unit}}\bullet z)^p K_0(0,z), \sum_{z\in \mathbb Z^2} (v_{\mathrm{unit}}\bullet (z-x))^pK_0(x,z) \bigg) \ge 0 ,\qquad x\in \mathbb Z^2, 
\end{equation}
and that $\zeta_{p,p}(x)$ is strictly positive whenever $K_t(0,\bullet), K_t(x,\bullet)$ are dependent.
We furthermore assume that the Markov chain on $(\mathbb Z^2)^h$ with one-step transition density 
\begin{equation}\label{qchains}
    \mathbf q^{(h)} (\x,\y) = \mathbb E \bigg[ \prod_{j=1}^k K_t(x_j,y_j)\bigg]
\end{equation}
is irreducible. \end{ass}

The condition \eqref{zeta_rwre} holds e.g.\ in the spatially independent case where $K_t(x,\bullet), K_t(x',\bullet)$ are independent for $x \ne x'$.  \eqref{zeta_rwre} also holds when the kernels $K_t$ are given by spatially averaging such i.i.d.\ kernels on some finite-width box. See \cite{hass2025superuniversalbehavioroutliersdiffusing, DP25} for further discussion of the significance of the integer $p$ and the condition \eqref{zeta_rwre}. Irreducibility of $\mathbf q^{(h)}$ is not automatic: consider $K_t(0,\bullet)$ supported on a single but random point almost surely (this is exactly the Bernoulli RWRE, see the introduction of \cite{DDP24} for a deeper discussion of this model). 


\begin{defn} Let $\omega= \{K_t(x,\bullet)\}_{(t,x) \in \Z \times \mathbb Z^2}$ and then let $P^\omega_{s,t}$ denote the transition density of the associated RWRE. That is, $P_{s,t}^\omega(x,y)= K_{s+1}\cdots K_t(   x,y),$ where the product is defined by $K_1K_2(   x,z) = \sum_y K_1(x,y)K_2(y,z)$, and onward by associativity.

Let $\varsigma_N = \beta_N v_{\mathrm{unit}}$ where $\beta_N>0$ can be any positive sequence for now.\footnote{
    The location strength vectors $\varsigma_N$ could be taken to be of a more general form, having some vanishing angular component as well, and the results would still hold true with some modified $\z(\beta,x)$, but this would clutter the discussion with more notation, so we do not pursue this.
}
We also put $\mu(x):= \mathbb E[K_0(0,x)]$ and write $M(\lambda):= \sum_{z\in \mathbb Z^2} e^{\lambda \bullet z} \mu(z)$ for its moment generating function.
Define the location vectors $d_N:= N\nabla \log M(\varsigma_N) = \sum_y y\mathbb E[H^N_{0,1}(0,y)]$,
and from here
\begin{equation}\label{hn_rw}
    H^N_{s,t}(x,y) := e^{\varsigma_N \bullet(y-x
    ) -N(t-s) \log M(\varsigma_N)} P^\omega_{s,t} (x, y).
\end{equation}
\end{defn}
This is exactly the random field that was studied in \cite{DDP24, DP25}.
Samples of this field with respect to a certain spatially-averaged random environment are visualized in Figure \ref{fig:rwre}.

\begin{thm}[Verification of hypotheses for the RWRE]
    Suppose Assumption \ref{4.6} holds. Then there exists a sequence $\beta_N\asymp (\log N)^{-1/(2p)}$ along which the family
    $H^N$ defined above satisfies Assumption \ref{a1}.
    In this case, $\nu_\infty$ is the invariant measure of the gap process for the 2-point motion generated by these Markov kernels.
\end{thm}

\begin{proof}
    \textbf{Verification of Items \eqref{a1flow}, \eqref{a1markov}, and \eqref{a1consistency}.}
    The fact that $H^N$ as defined in \eqref{hn_rw} satisfies the flow property $H_{s,t}^N H_{t,u}^N = H_{s,u}^N$ for $s<t<u$ follows from direct computation and the fact that $P^\omega_{s,t}$ clearly satisfies the flow property. 
    
    For now, we will consider an arbitrary sequence $\beta_N\to 0$. Let 
    \begin{equation*}
        \boldsymbol{\eta}_{N,h} (\x):= \log \sum_{y_1,\dots,y_k \in \Z^2} \mathbb E \bigg[ \prod_{j=1}^k H^N_{0,1} (x_j,y_j)\bigg],
    \end{equation*}
    from which we may define, as in \eqref{eq:mrep}, the Markov chain $\mathbf P_{N,h}$ by the one-step transition probability \begin{equation}\label{pnh} \mathbf p_{N,h} (\x,\y) = e^{-\boldsymbol{\eta}_{N,h} (\x)}\mathbb E \bigg[ \prod_{j=1}^k H^N_{0,1} (x_j,y_j)\bigg].
    \end{equation}
    This is exactly the same Markov chain as in \cite{DP25}, albeit written somewhat more compactly.
    Since $\beta_N\to 0$, one may easily verify that $\|\boldsymbol{\eta}_{N,h}\|_{L^\infty} \to 0$.
    The finite-range assumption on the kernels $K_t$ clearly implies that $\z_N(x):= \boldsymbol{\eta}_{N,2}(0,x)$ has finite support independent of $N$, and moreover that $\boldsymbol{\eta}_{N,h}(\x) = \sum_{1\le i<j\le h} \z_N(x_i-x_j) $ as long as $\x \notin \{\y: |y_i-y_j| +|y_j-y_k|\le A\}$, where $A$ is the range of dependency specified in Item \eqref{a1markov}.
    
    Here we need to do a more careful analysis of how exactly the Markov chains $\mathbf p_{N,h}$ \eqref{pnh} are related to the Markov chains $\mathbf q^{(h)}$ from \eqref{qchains}. A direct calculation shows that the exact relation is \begin{equation}\label{tilt_dens}\mathbf p_{N,h}(\x,\y) = \frac{e^{   \varsigma \bullet \sum_{i=1}^h(   y_i-   x_i)} \mathbf q^{(h)}(\mathbf x,\mathbf y)}{\sum_{\bfa\in (\mathbb Z^2)^h} e^{   \varsigma\bullet  \sum_{i=1}^h(   a_i-   x_i)} \mathbf q^{(h)}(\mathbf x,\mathbf a)},
    \end{equation}
    where $\varsigma = \varsigma_N = \beta_N v_{\mathrm{unit}}$. This is straightforward to show directly, and we refer to \cite[Proposition 2.6]{DP25} for a detailed proof.
    
    Thus we obtain precisely that 
    \begin{align}\label{phoo} 
        \notag \boldsymbol{\eta}_{N,h}(\x)&=\log \sum_{\bfa\in (\mathbb Z^2)^h}e^{   \varsigma_N\bullet  \sum_{i=1}^h(   a_i-   x_i)} \mathbf q^{(h)}(\mathbf x,\mathbf a) - h\log M(\varsigma_N) \\&=  \sum_{1\le i<j\le h} \sum_{m,n=p}^\infty \frac{\beta_N^{m+n}}{m!n!} \zeta_{m,n}(x_i-x_j) + \boldsymbol\varepsilon(\varsigma_N,\x) ,
    \end{align}
    where $\zeta_{m,n}:\mathbb Z^2\to\mathbb R$ are given by
    \begin{equation*}
        \zeta_{m,n}(x^i-x^j):=\kappa_{(m+n)} (\underbrace{v_{\mathrm{unit}}\bullet (R^i_1-x^i) , \dots, v_{\mathrm{unit}}\bullet (R^i_1-x^i)}_{m\;\text{terms}}, \underbrace{v_{\mathrm{unit}}\bullet (R^j_1-x^j),\dots,v_{\mathrm{unit}}\bullet (R^j_1-x^j)}_{n\;\text{terms}})
    \end{equation*}
    and the joint cumulants are understood to be under the measures $\mathbf q^{(h)}(\x,\bullet)$ with $\x=(x^1,\dots,x^h).$ We are using the fact that the joint moments up to order $p-1$ of the kernels $K_t$ are deterministic, which is why the above sum starts from $m,n=p$ and the leading order is $\beta_N^{2p}$ (all lower-order terms will cancel exactly). Moreover, the error term satisfies
    \begin{equation*}
        \boldsymbol\varepsilon(\varsigma,x) \leq C|\varsigma|^{2p+1}\sum_{ 1\le i<j<\ell \le h} \ind_{\{ |x_i-x_j| \le A , |x_j-x_\ell|\le A\}}
    \end{equation*}
    for some $A>0$, because it contains exactly those cumulants with three or more distinct indices. The above series converges absolutely and defines an analytic function of $\beta$ by e.g.\ the finite-range assumption on $K_t(x,\bullet)$ and \cite[Theorem 9.11]{KLM}.
    
    Thanks to this Taylor expansion, the condition \eqref{zeta_rwre} makes it clear that $\z_N$ will necessarily be nonnegative for large $N$ as long as $\beta_N\to 0.$ The fact that $H^N$ is translation invariant in space-time then implies that $\mathbf p_{N,2} ( (x+a,x'+a), (y+a,y'+a)) = \mathbf p_{N,2} ((x,x'),(y,y'))$, so that the difference process of $\mathbf P_{N,2}$ is indeed Markov.
    
    We still have not shown that $\|\boldsymbol{\eta}_{N,h}\|_{L^\infty} \leq C/\log N.$ This will be accomplished when we study the exact asymptotics of $\beta$ below. However, for arbitrary sequences $\beta_N\to 0$ we still have $\|\boldsymbol{\eta}_{N,h}\|_{L^\infty}\to 0$: again, invoking the final statement of Theorem \ref{check}, this will be enough to verify Items \eqref{a1logasymptotic}--\eqref{a1heatkernel}.

    \smallskip
    \noindent\textbf{Verification of Items \eqref{a1logasymptotic}--\eqref{a1heatkernel}.} This is most easily done by verifying the conditions of Theorem \ref{check}. As $N\to \infty$ we know that $\|\boldsymbol{\eta}_{N,h}\|_{L^\infty} \to 0$ which means by \eqref{tilt_dens} that the Markov chains $\PN{h}$ defined above converge as $N\to \infty$ to the Markov chains $\mathbf q^{(h)}$ from \eqref{qchains}, which are irreducible by assumption. Since $\mu = \mathbf q^{(1)}(0,\bullet)$ has finite support and covariance given by the $2 \times 2$ identity matrix, it then follows that the assumptions of Theorem \ref{check} are true.
    
    \smallskip
    \noindent\textbf{Verification of the critical tuning.} 
    In light of \eqref{phoo}, we define $$\z(\beta,x):= \sum_{m,n=p}^\infty \frac{\beta^{m+n}}{m!n!} \zeta_{m,n}(x) = \log \sum_{\bfa\in (\mathbb Z^2)^2}e^{   \beta v_{\mathrm{unit}}\bullet  \sum_{i=1}^2(   a_i-   x_i)} \mathbf q^{(2)}(\mathbf x,\mathbf a) - 2\log M(\beta v_{\mathrm{unit}}).$$ 
    To be more precise, let $J$ denote the support of $\zeta_{p,p}$ which is a finite set since we are in the discrete setting. Now, we let $T_{N,\beta}$ denote the linear map on $\mathbb R^J$ given by Definition \ref{def:lambda}: more explicitly
    \begin{equation*}
        T_{N,\beta} h(x) = \frac1{\log N} \sum_{z\in J} \bigg( \sum_{r=0}^N \mathfrak p^\beta (r,x,y)\bigg) (e^{\z(\beta, y)}-1) h(y),
    \end{equation*}
    where
    \begin{equation*}
        \mathfrak p^\beta (1,x,y) =  \frac{\sum_{a\in \mathbb Z^2} e^{   \beta v_{\mathrm{unit}} \bullet \sum_{i=1}^2(   y_i-   x_i)} \mathbf q^{(2)}((0,x), (a,a+y))}{\sum_{\bfa\in (\mathbb Z^2)^2} e^{   \beta v_{\mathrm{unit}}\bullet  \sum_{i=1}^2(   a_i-   x_i)} \mathbf q^{(2)}((0,x),\mathbf a)}.
    \end{equation*}
    Let $\boldsymbol{\lambda}_N(\beta)$ denote the top eigenvalue of the (finite-dimensional) operator $T_{N,\beta}: \mathbb R^J\to \mathbb R^J$. We need to show that there exists a solution to the equation $ \boldsymbol{\lambda}_N(\beta) \cdot \log(N) = 1+\frac{\vartheta+o(1)}{\log N}$, which then defines $\beta_N$. To prove this, it suffices to show a local expansion 
   \begin{equation}\label{eigen_exp1}
        \boldsymbol{\lambda}_N(\beta) = \beta^{2p} f_N(\beta) , \qquad \limsup_{N\to\infty} \sup_{0\le \beta<\epsilon}| f_N(\beta) - c_{p}| < \frac{c_{p}}{2},
    \end{equation}
     where $c_{p} := \sum_{a\in \mathbb Z^2} \zeta_{p,p}(a) >0$ and each $f_N$ is continuous on $[0,\epsilon]$. Here $\epsilon>0$ is some positive constant not depending on $N$. 
     This would imply that $\beta\mapsto \boldsymbol{\lambda}_N(\beta)$ is continuous on $[0,\epsilon]$ and $\frac{c_{p}}{2}\beta^{2p} \leq \boldsymbol{\lambda}_N(\beta) \leq \frac{3c_{p}}{2}\beta^{2p}.$
     Once again, the intermediate value theorem immediately shows existence of a sequence $\beta_N \asymp (\log N)^{-1/(2p)}$ such that $ \boldsymbol{\lambda}_N(\beta) \cdot \log(N) = 1+\frac{\vartheta+o(1)}{\log N}$.
        
    First let us recall from \eqref{phoo} that $e^{\z(\beta,y)}-1$ Taylor expands as $\beta^{2p} \z_{p,p}(y) +\beta^{2p+1} \ell(\beta,y)$ where $\sup_{0<\beta<1} \|\ell(\beta,\bullet)\|_{L^\infty}<\infty$. Thus, we note that $f_N(\beta)$ is simply the top eigenvalue of the linear map on $\mathbb R^J$ given by $$S_{\beta,N} h(x) = \frac{1}{\log N} \sum_{r=0}^N \sum_{y\in J}\mathfrak p^\beta (r,x,y) \big(\zeta_{p,p}(y) +\beta\ell(\beta,y)\big) h(y) . $$ 
        
    \smallskip
    \noindent\textbf{Step 1.} Let $S_{0,\infty}:= \lim_{N\to\infty} S_{0,N}$, so that explicitly $S_{0,\infty} = \langle \bullet , \zeta_{p,p} \rangle_{L^2(J)} \ind_J$ is a rank-one operator, whose unique non-zero eigenvalue is $\langle \zeta_{p,p},\ind_J\rangle_{L^2(J)} =c_{\z}$. For any $\delta>0$, we will show that there exist $\epsilon>0$ and $N_0 \in \mathbb N$ such that $\|S_{N,\beta} - S_{0,\infty}\| < \delta$ whenever $\beta<\epsilon$ and $N\ge N_0$, where as always the operator norm is understood to be on $L^2(J)$. Indeed, for any sequence $\beta_N\downarrow 0$ and $x,y\in J$, Theorem \ref{check} guarantees that $\frac1{\log N} \sum_{r=0}^N \mathbf p^{\beta_N} (r,x,y)dt \to \nu_\infty(y)$, where $\nu_\infty$ is the invariant measure of the gap process of $\mathbf q^{(2)}$.
    This means that, along any sequence $\beta_N\downarrow 0$, the integral kernel of $S_{\beta_N,N}$ converges to that of $S_{0,\infty}$,
    where for $x,y \in J$ the kernels are respectively given by $K_{\beta_N,N}(x,y) = \frac1{\log N} \sum_{r=0}^N \mathbf p^{\beta_N}(r,x,y) \cdot (\zeta_{p,p}(y)+\beta_N \ell(\beta_N,y))$ and $K(x,y)=\zeta_{p,p}(y)$. This implies the claim.

    \smallskip
    \noindent\textbf{Step 2.} Next, we prove continuity of $f_N$ on $[0,\epsilon]$. Define $\tau_0:= \frac{\|\zeta_{p,p}\|_{L^2(J)}}{\langle \ind_J ,\zeta_{p,p}\rangle_{L^2(J)}}$ and $\gamma_0:=\langle \zeta_{p,p},\ind_J\rangle_{L^2(J)}$. Both quantities are positive since $\zeta_{p,p}>0$ on $J$. Using \cite[Theorems 2.3 and 2.6]{Kloeckner2019Effective}, there exists a map $A\mapsto \lambda_{\mathrm{top}}(A)$ which sends a compact operator $A$ on $L^2(J)$ to its eigenvalue of maximal modulus, and the function $\lambda_{\mathrm{top}}$ is well-defined and uniformly continuous (in the operator norm topology) on the domain $\{A: \|A-S_{0,\infty}\| <1/(12\tau_0 \gamma_0) \}.$ Notice that the function $\beta\mapsto f_N(\beta)$ is the composition of the two maps $$\beta \mapsto S_{N,\beta} \mapsto \lambda_{\mathrm{top}}(S_{N,\beta}).$$
    The first map is clearly continuous on $[0,\epsilon]$, and by Step 1 lands in the domain of the second one whenever $\beta<\epsilon$ and $N\ge N_0$. The second map is also continuous as just discussed, and so the composed map is continuous.

    \smallskip
    \noindent\textbf{Step 3.} Now we prove the estimate on the $\limsup$ in \eqref{eigen_exp1}. Again this follows from the stronger statement that, along any sequence $\beta_N\downarrow 0$, $f_N(\beta_N) \to c_{p}$, which is immediate from the result of Step 1 and the continuity of the map $\lambda_{\mathrm{top}}$ in Step 2.
\end{proof}

\begin{remark} In the proof, we did not comment on the invariant measure $\nu_\infty$ very much, but there are many examples where it is a nontrivial measure as opposed to just counting measure on $\mathbb Z^2$. Even in $d=1$, \cite{DDP24} provides such an example of a nontrivial measure arising from a nearest-neighbor model, and consequently the noise coefficient of the KPZ equation there was calculated to be $8\sigma^2/(1-4\sigma^2)$ where $\sigma^2$ is the one-point variance of the underlying i.i.d.\ random environment $\{\omega_{t,x}\}$. The next example will also provide an example of a nontrivial measure $\nu_\infty.$
\end{remark}

\subsection{Diffusion in random medium}\label{5.5}
Here we consider another RWRE-type model in a continuous environment, 
which models diffusion in turbulent media and heavily draws from the works of \cite{kar68,Kr, GK95, GK96, GH, LR, war}. 

\begin{defn}
    Consider a vector-valued Gaussian noise $ \vec  \eta = (\eta_1,\eta_2)$ where each $\eta_i = \varphi *\xi_i$ for i.i.d.\ standard space-time white noises $\xi_1$ and $\xi_2$ on $\mathbb R\times \mathbb R^2$, with $\varphi \in C_c^\infty(\mathbb R^2)$ a fixed non-negative even function with $\int_{\mathbb R^2} \varphi =1$. Equivalently $\vec \eta$ has 
    with covariance matrix
    $\mathbb E [ \eta_i(t,x) \eta_j (s,y)] = \ind_{\{i=j\}}\Phi (x-y) \delta_0(t-s)$, where $\Phi=\varphi * \varphi$ and $*$ is spatial convolution.
    Assume 
    that $\Phi$ is compactly supported and that $\Phi(0) = \theta $, where $\theta \in (0,1)$. As explained in \cite{kun94a, kun94b}, one may sensibly construct It\^o solutions to the stochastic PDE given by 
    \begin{equation}\label{fth}\tag{aSHE}
        \partial_t u (t,x) = \tfrac12 \Delta u(t,x) + \mathrm{div} \big( u(t,x)  \vec  \eta (t,x)\big), \qquad u(0,x) = \delta_0(x),\qquad t\ge 0 ,\; x\in \mathbb R^2.
    \end{equation}
    Thus, we can consider the family of propagators $u_{s,t}(x,y)$ for \eqref{fth} with $s<t$ and $x,y\in \R^2$, collectively called diffusion in a random potential.
\end{defn}

A pathwise solution adapted to the filtration generated by $  \vec \eta$ is possible thanks to the spatial smoothness of $\vec \eta $, see \cite{kun94b}. Then \eqref{fth} turns out to be conservative, meaning that $\int_{\mathbb R^d} u(t,x)dx =1$ for all $t>0$. Note that \eqref{fth} can be viewed as the Kolmogorov forward equation associated to the formal SDE 
\begin{equation}d   X(t) =   - \vec \eta(t,   X(t))dt + \theta^{1/2} d   W(t), \hspace{1 in} X(0)=0.
\end{equation}
Here $   W(t)$ is a standard Wiener process in $\mathbb R^d$. This SDE representation looks misleading because the It\^o-Stratonovich correction in \eqref{fth} is not a constant multiple of $u$, but rather is given by $\frac12 (1-\theta) \Delta u(t,x),$ which explains why the viscosity in the It\^o equation is $\frac12$ instead of $\frac12\theta$. 

Despite $\vec \eta $ being distribution-valued, \cite{kun94b} shows that such a SDE actually makes sense and that a solution exists for a.e.\ realization of $\vec \eta$, and furthermore that \eqref{fth} describes the evolution of the density of $   X(t)$ started from 0.

\begin{defn}\label{tilt_diffre}
    Let $u_{s,t}$ denote the propagators of the stochastic PDE \eqref{fth}. Fix a unit vector $v_{\mathrm{unit}}\in\mathbb R^2$.
    Let $\varsigma_N = \beta_N v_{\mathrm{unit}}$, where $\beta_N>0$ can be any positive sequence for now. Then we may define \begin{equation*}
        H^N_{s,t}(x,y) = e^{\varsigma_N \bullet(y-x) -\frac{N}2(t-s) |\varsigma_N|^2} u_{s,t} (x,y),
    \end{equation*}
    where $\varsigma_N=N^{-1}d_N = N^{-1} \beta_Nv_{\mathrm{unit}}$.
\end{defn}
Thus, the macroscopic field $\mathfrak H^N$ from Definition \ref{a2} can be represented in this case by
\begin{align*}\mathfrak H^N_{0,t}(\phi, \psi)&:= \frac{1}{N} \int_{(\mathbb R^2)^2} \phi\left (\frac{x}{\sqrt{N}}\right)\psi\left (\frac{y-   d_Nt}{\sqrt{N}}\right) e^{\varsigma_N \bullet (y-x) - \frac{N}2t |\varsigma_N|^2}u_{0,Nt}(x, y) dxdy.
\end{align*}
In macroscopic coordinates, we note that $(t,y)\mapsto \mathfrak H^N_{0,t}(\phi,\delta_y)$ solves the stochastic PDE 
\begin{equation}\label{ashe1}
    \partial_t H (t,x) = \tfrac12 \Delta H(t,x) + (N^{-1/2}  \mathrm{div} - \varsigma_N \bullet) \big( H(t,x) \eta^N(t,x)\big),\qquad t\ge 0, x\in \mathbb R^2
\end{equation}
with initial condition $\phi$, where $\eta^N(t,x) = N \vec\eta(Nt, N^{1/2}x + d_Nt)$ is converging in law to standard Gaussian space-time white noise. 

\begin{thm}[Verification of hypotheses for the diffusion in random medium]
    There exists a sequence $\beta_N\to 0$ such that the family $H^N$ as defined above satisfies Assumption \ref{a3}. In this case ${\boldsymbol{\zeta}}_N(x)= \beta_N^2\cdot  (\varphi *\varphi)(x)$, meanwhile $\nu_\infty(x)$ is equal to some constant multiple of $\frac1{1-\varphi *\varphi(x)}. $\footnote{With more work, we could identify the constant as $\frac1{4\pi}$. For brevity we skip this calculation.}
\end{thm}

Based on \eqref{ashe1} and the result of Theorem \ref{mshe_verify} for the mollified SHE, it might seem obvious that this would converge to the same limit as the mollified SHE. That is because the divergence part of the noise in \eqref{ashe1} vanishes at rate $N^{-1/2}$, whereas the scalar part behaves like $\beta_N$, which we expect to vanish at the slower rate of $O(\log^{-1/2} N)$; and moreover $\z_N$ takes on the same form.
However, it turns out that the divergence part does contribute nontrivially to the limit! If we were to couple the noises in the natural way and take a joint limit in distribution of \eqref{ashe1} with \eqref{mshe}, then the two SHFs thus obtained would not be the same, thus indicating non-triviality of the result. The fact that $\nu_\infty$ is a nontrivial measure is another indication of this fact. See \cite{DDP24} for a discussion of this decoupling phenomenon in the $d=1$ case. See also the related discussion in \cite{dom}.

\begin{proof}
    \textbf{Verification of Items \eqref{a3flow}, \eqref{a3markov}, and \eqref{a3consistency}.} Throughout the proof $\epsilon = \epsilon_N := N^{-1/2}$. Overall the proof is quite similar to the tilting proof used for the RWRE model in the previous subsection (and in fact \cite{DP25} included both models under the same framework) but here we will give the explicit argument to be as clear as possible.
    The $h$-point motion is defined as the Markov process on $(\mathbb R^2)^h$ whose transition density is given by $\mathbf q^{(h)}(t,\x,\y) = \mathbb E \big[ \prod_{j=1}^h u_{0,t}(x_j,y_j)\big]$, where $u_{s,t}$ denotes the propagators for \eqref{fth}.
    Then the generator corresponding to $\mathbf q^{(h)}(t,\x,\y)$ is given by the following second-order elliptic operator on $C^2((\mathbb R^2)^h)$: 
    \begin{equation*}
    L^{(h)}_{\mathrm{Cen}}f(y_1,\dots,y_h):= \frac12 \sum_{i,j=1}^h \big( \theta \delta_{i,j}+ \Phi(y_i-y_j)\big)\cdot (\mathrm{div}_i\nabla_j f )(y_1,\dots,y_h),
    \end{equation*}
    where if $(y_1,\dots,y_h)\in (\mathbb R^2)^h$ then $\mathrm{div}_i$ denotes divergence only in the $i$\textsuperscript{th} coordinate, and $\nabla_j$ denotes the 2-vector consisting of partials only in the $j$-coordinate. Denote by $ \mathbf E^{\x}_{\mathrm{Cen}(\epsilon)}$ the expectation with respect to the Markov process $X$ on $(\mathbb R^2)^h$ with generator $L^{(h)}_{\mathrm{Cen}}$ and $X_0 = \x$.
    
    Let $\Phi_\epsilon := \varphi_\epsilon * \varphi_\epsilon$, with $\varphi_\epsilon(x) := \epsilon^{-2} \varphi(\epsilon^{-1}x)$. Then for $\psi:\mathbb R^2\to \mathbb R$ bounded and continuous and $x_1,\dots,x_h\in\mathbb R^2$, we claim that
    \begin{equation}\label{a7}\int_{(\mathbb R^2)^h} \mathbb E\bigg[ \prod_{j=1}^h H^N_{0,t} (x_j, y_j) \bigg]\prod_{j=1}^h \psi(y_j) dy\cdots dy_n = \mathbf E^{(x_1,\dots,x_h)}_{\mathrm{Sing}(N,h)} \bigg[ e^{  \beta_N^2\sum_{1\le i<j\le h}\int_0^t \Phi (X^i_s-X^j_s)ds}\prod_{j=1}^h \psi(X^j_s) \bigg].
    \end{equation}
    Here the expectation on the right side is with respect to a diffusion process on $(\mathbb R^2)^h$ with generator 
    \begin{equation}\label{lnh}L^{(N,h)}_{\mathrm{Sing}}f(y_1,\dots,y_h):= L^{(h)}_{\mathrm{Cen}}f(y_1,\dots,y_h) + \sum_{i\ne j} \Phi(y_i-y_j) (\varsigma_N\bullet \nabla_i f)(y_1,\dots,y_h). 
    \end{equation}
    Thus the Markov chain $\PN{h}$ needed to satisfy Assumption \ref{a3} \eqref{a3markov} will be precisely the one with generator $L_{\mathrm{Sing}}^{(N,h)}$, and furthermore 
    \begin{equation}\label{eta_diff}
    \boldsymbol{\eta}_{N,h}(\x) := \beta_N^2 \sum_{1\le i<j\le h} \Phi (x_i-x_j) = \sum_{1\le i<j\le h} \z_N (x_i-x_j).
    \end{equation} For $h=2$ we indeed see that the difference $X^1-X^2$ is Markov in its own filtration, with generator given by $L_{\mathrm{Sing}}^{(N,\mathrm{dif})} f(y)= (1-  \Phi(y)) \Delta f(y) $, thanks to a cancellation of the drift terms. To prove the above moment formula \eqref{a7}, we use observations made in \cite{war}. Consider the diffusion of particles $X^i\in\mathbb R^2$ given by \begin{equation}\label{sde}dX^i_s =-\vec\eta(s,X^i_s)ds + \theta^{1/2} dW^i_s
    \end{equation}
    where $1\le i\le n$ and the environment $\vec\eta$ is fixed (quenched), and $W^i$ are independent Brownian motions independent of $\vec\eta$. Such a process makes sense despite the temporal roughness of $\vec\eta$, in fact the entire flow of diffeomorphisms for the SDE given $\eta^\epsilon$ is well-posed, see Theorem 4.5.1 of the monograph \cite{kun94b}. As mentioned above and proved in \cite[Proposition 2.1]{ew6}, the Kolmogorov forward equation for each $X^i$ may be written as the It\^o-Walsh solution to \eqref{fth}. 
    If we average out the environment $\vec\eta$ then the $n$-point motion $(X^1,\dots,X^n)$ is a diffusion in $(\mathbb R^2)^h$ with generator $L^{(h)}_{\mathrm{Cen}}$, more precisely if we let $u_{s,t}(x,y)$ denote the propagators of \eqref{fth}, we have \begin{equation}\label{moma}\int_{(\mathbb R^2)^h} \mathbb E\bigg[ \prod_{j=1}^h u_{0,t}(x_j,y_j)\bigg] \prod_{j=1}^h \phi(y_j)dy\cdots dy_h = \mathbf E_{\mathrm{Cen}(h)}^{(x_1,\dots,x_h)} \bigg[ \prod_{j=1}^h \phi(X^j_t)\bigg].\end{equation}
    This immediately follows from \eqref{sde} since for $a\in \mathbb R^2$ one has $\mathbb E[ a\bullet (X^i_{t+dt}-X^i_t) \cdot a \bullet  (X^j_{t+dt}-X^j_t)| \mathcal F_t] = |a|^2 (\epsilon^2 \Phi(X^i_t-X^j_t) +\theta \delta_{i,j})dt$, so that $(X^1,\dots,X^h)$ solves the martingale problem for $L^{(h)}_{\mathrm{Cen}}$; we refer to the discussion of \cite[(2.23)-(2.24)]{ew6} if one desires a more explicit and rigorous calculation.

    With \eqref{moma} verified, we now apply the exponential tilting given in Definition \ref{tilt_diffre}. The exponential term $e^{\frac12 |\varsigma_N|^2 t + \varsigma_N\bullet y }$ in Definition \ref{tilt_diffre} can be interpreted as a change of measure of $\mathbf P^{\x}_{\mathrm{Cen}(h)}$ by using the martingale $e^{\varsigma_N\bullet (X^1+\dots+X^h) - \frac{1}2 \langle \varsigma_N\bullet (X^1+\dots+X^h)\rangle}.$ More precisely, recall that $\langle \varsigma_N\bullet X^i,\varsigma_N\bullet X^j\rangle_t = \beta_N^2 \int_0^t \Phi(X^i_s-X^j_s)ds$ for $i\ne j$, then use \eqref{moma} and then go through with this change of measure, and we will obtain that \begin{align*}\int_{(\mathbb R^2)^h} \mathbb E\bigg[ \prod_{j=1}^h H^N_{0,t}(x_j,y_j)\bigg] \prod_{j=1}^h \phi(y_j)dy\cdots dy_j &= e^{\frac{h}2 |\varsigma_N|^2t} \mathbf E_{\mathrm{Cen}(h)}^{(x_1,\dots,x_h)} \big[e^{\varsigma_N\bullet (X^1_t +\dots+X^h_t)} \prod_{j=1}^h \phi(X^j_t+d_N t)\big]\\&=  \mathbf E^{(x_1,\dots,x_h)}_{\mathrm{Sing}(N,h)} [ e^{  \sum_{1\le i<j\le h}\beta_N^2 \int_0^t \Phi (X^i_s-X^j_s)ds}\prod_{j=1}^h \phi(X^j_s)]
    \end{align*}
    as required.

    We have not yet shown that $\|\boldsymbol{\eta}_{N,h}\|_{L^\infty} \leq C/\log N.$ This will be accomplished when we study the exact asymptotics of $\beta$ below. However, for \textit{any sequence} $\beta_N\to 0$, we still have that $\|\boldsymbol{\eta}_{N,h}\|_{L^\infty}\to 0$, thus by the final statement in Theorem \ref{check2}, this will still be enough to verify Items \eqref{a3log}--\eqref{a3heatkernel}.
    
    \smallskip
    \noindent\textbf{Verification of Items \eqref{a3logasymptotic}--\eqref{a3heatkernel}.} This is most easily achieved by checking the conditions of Theorem \ref{check2}. As $N\to \infty$ we know that $\|\boldsymbol{\eta}_{N,h}\|_{L^\infty} \to 0$ which means by \eqref{tilt_dens} that the Markov chains $\mathbf p_{N,h}$ defined above are converging as $N\to \infty$ to the Markov chains $\mathbf q^{(h)}$ from \eqref{qchains}, which by assumption are irreducible.
    Since $\mu = \mathbf q^{(1)}(0,\bullet)$ has finite support and covariance given by the $2 \times 2$ identity matrix, the heat kernel bound needed in Item \eqref{check2exptail} of Theorem \ref{check2} follows from the work of \cite{Zhang97} on Gaussian heat kernel bounds for diffusions with drift on compact time intervals. It then follows that the conditions of Theorem \ref{check2} hold. 
    
    Recall that $L_{\mathrm{Sing}}^{(N,\mathrm{dif})} f(y)= (1-  \Phi(y)) \Delta f(y) $, which is actually independent of $N$ (similar to the other continuum examples). Thus we see that the limiting Markov chain $\mathbf P_{\infty,\mathrm{dif}}$ has the same generator $L_{\mathrm{Sing}}^{(\infty,\mathrm{dif})} f(y)= (1-  \Phi(y)) \Delta f(y)$, which indeed has invariant measure $\frac1{1-\Phi(y)}$ as claimed in the theorem statement.
    
    \smallskip
    \noindent\textbf{Verification of the critical tuning.}  So far, we have considered a general sequence $\beta_N\to 0$, but now we need to prove the existence of a precise sequence satisfying \eqref{crit_tuning2}. 
    
    As we just noted above, the Markov chain $\mathbf p_{N,\mathrm{dif}}$ will be independent of $N$, and we can unambiguously let $\Lambda_N(f)$ denote the top eigenvalue as in Definition \ref{def:lambdacont}. Thus to verify the critical tuning, we can just take 
    $$\beta_N^2:= \frac1{\Lambda_N(\Phi) \log N} \bigg( 1+\frac{\vartheta+o(1)}{\log N}\bigg).$$
\end{proof}

\appendix

\bibliographystyle{alpha}
\bibliography{refs.bib, ref.bib}
\end{document}